\documentclass[a4paper,11pt]{amsart}
\usepackage[colorlinks, linkcolor=blue,anchorcolor=Periwinkle,
citecolor=blue,urlcolor=Emerald]{hyperref}
\usepackage[all]{xy}
\SelectTips{cm}{}
\usepackage{hyperref}
\usepackage{graphicx}
\usepackage{psfrag}
\usepackage{tikz}
\usepackage{tikz-cd}
\usepackage{mathtools}
\usepackage{amsmath,amssymb,tikz-cd}
\usepackage[export]{adjustbox}
\usepackage{calligra}

\usetikzlibrary{decorations.pathreplacing}
\usetikzlibrary{matrix,arrows}

\usepackage[french,english]{babel}

\newcommand{\ko}{\: , \;}

\newcommand{\ol}[1]{\overline{#1}}

\numberwithin{equation}{subsection}
\newtheorem{theorem}[subsection]{Theorem}
\newtheorem{definition}[subsection]{Definition}
\newtheorem{classification-theorem}[subsection]{Classification Theorem}
\newtheorem{decomposition-theorem}[subsection]{Decomposition Theorem}
\newtheorem{proposition-definition}[subsection]{Proposition-Definition}
\newtheorem{definition-proposition}[subsection]{Definition-Proposition}
\newtheorem{example-definition}[subsection]{Example-Definition}
\newtheorem{periodicity-conjecture}[subsection]{Periodicity Conjecture}
\newtheorem{lemma}[subsection]{Lemma}
\newtheorem{proposition}[subsection]{Proposition}
\newtheorem{corollary}[subsection]{Corollary}

\newtheorem{example}[subsection]{Example}
\newtheorem{remark}[subsection]{Remark}

\newtheorem{notation}[subsection]{Notation}

\newtheorem{Definition-Proposition}[subsection]{D\'efinition-Proposition}
\newtheorem{Definition}[subsection]{D\'efinition}
\newtheorem{Lemme}[subsection]{Lemme}
\newtheorem{Theorem}[subsection]{Th\'eor\`eme}
\newtheorem{Corollaire}[subsection]{Corollaire}
\newcommand{\reminder}[1]{}

\renewcommand{\mod}{\mathrm{mod}}

\newcommand{\rep}{\mathrm{rep}}

\newcommand{\Mod}{\mathrm{Mod}\,}
\newcommand{\CM}{\mathrm{CM}}
\newcommand{\pd}{\mathrm{pd}}
\newcommand{\proj}{\mathrm{proj}\,}

\newcommand{\per}{\mathrm{per} }
\newcommand{\thick}{\mathrm{thick} }
\newcommand{\pvd}{\mathrm{pvd} }
\newcommand{\add}{\mathrm{add} }

\renewcommand{\Im}{\mathrm{Im} }

\newcommand{\tr}{\mathrm{tr}}
\newcommand{\cat}{\mathrm{cat}}

\newcommand{\Cone}{\mathrm{Cone}}
\newcommand{\Com}{\mathrm{Com}}
\newcommand{\sMod}{\mathrm{sMod}\,}
\newcommand{\sSet}{\mathrm{sSet}}

\newcommand{\dgcat}{\mathrm{dgcat}}
\newcommand{\Hqe}{\mathrm{Hqe}}

\renewcommand{\rep}{\mathrm{rep}}

\newcommand{\pretr}{\mathrm{pretr} }

\newcommand{\Set}{\mathrm{Set}}
\newcommand{\const}{\mathrm{const}}
\newcommand{\colim}{\mathrm{colim}}
\newcommand{\cok}{\mathrm{cok} }

\renewcommand{\ker}{\mathrm{ker} }
\newcommand{\obj}{\mathrm{obj} }

\newcommand{\R}{\mathcal R}

\newcommand{\iso}{\xrightarrow{_\sim}}

\newcommand{\Cosp}{\mathrm{Cosp}}

\newcommand{\Sq}{\mathrm{Sq}}
\newcommand{\Cat}{\mathrm{Cat}}
\newcommand{\Quiv}{\mathrm{Quiv}}
\newcommand{\Id}{\mathrm{id}}
\newcommand{\Dia}{\mathrm{Dia}}
\newcommand{\Res}{\mathrm{Res}}
\newcommand{\Iso}{\mathrm{Iso}}
\newcommand{\Ex}{\mathrm{Ex}}
\newcommand{\Ab}{\mathrm{Ab}}

\newcommand{\Def}{\mathrm{def}\kern 0.1em}

\newcommand{\D}{\mathcal {D}}
\newcommand{\A}{\mathcal {A}}
\newcommand{\B}{\mathcal {B}}
\newcommand{\C}{\mathcal {C}}
\newcommand{\E}{\mathcal {E}}
\newcommand{\F}{\mathcal {F}}

\newcommand{\J}{\mathcal {J}}
\newcommand{\I}{\mathcal {I}}
\newcommand{\M}{\mathcal {M}}
\newcommand{\N}{\mathcal {N}}
\newcommand{\U}{\mathcal {U}}
\newcommand{\V}{\mathcal {V}}
\newcommand{\W}{\mathcal {W}}
\newcommand{\s}{\mathcal S}
\newcommand{\T}{\mathcal T}
\renewcommand{\P}{\mathcal P}
\newcommand{\X}{\mathcal X}
\newcommand{\Y}{\mathcal Y}
\newcommand{\Z}{\mathcal Z}
\newcommand{\Fun}{\mathrm{Fun}}
\newcommand{\Hom}{\mathrm{Hom}}

\newcommand{\RHom}{\mathrm{RHom}}

\newcommand{\Ext}{\mathrm{Ext}}

\newcommand{\End}{\mathrm{End}}
\newcommand{\rad}{\mathrm{rad}}

\newcommand{\Mor}{\mathrm{Mor}}

\newcommand{\Map}{\mathrm{Map}}
\renewcommand{\phi}{\varphi}

\renewcommand{\tilde}[1]{\widetilde{#1}}

\usetikzlibrary{positioning}

\begin{document}

\setcounter{tocdepth}{2}%To be deleted
\setcounter{secnumdepth}{4}%To be deleted
\begin{figure}
  \includegraphics[width=200pt,left]{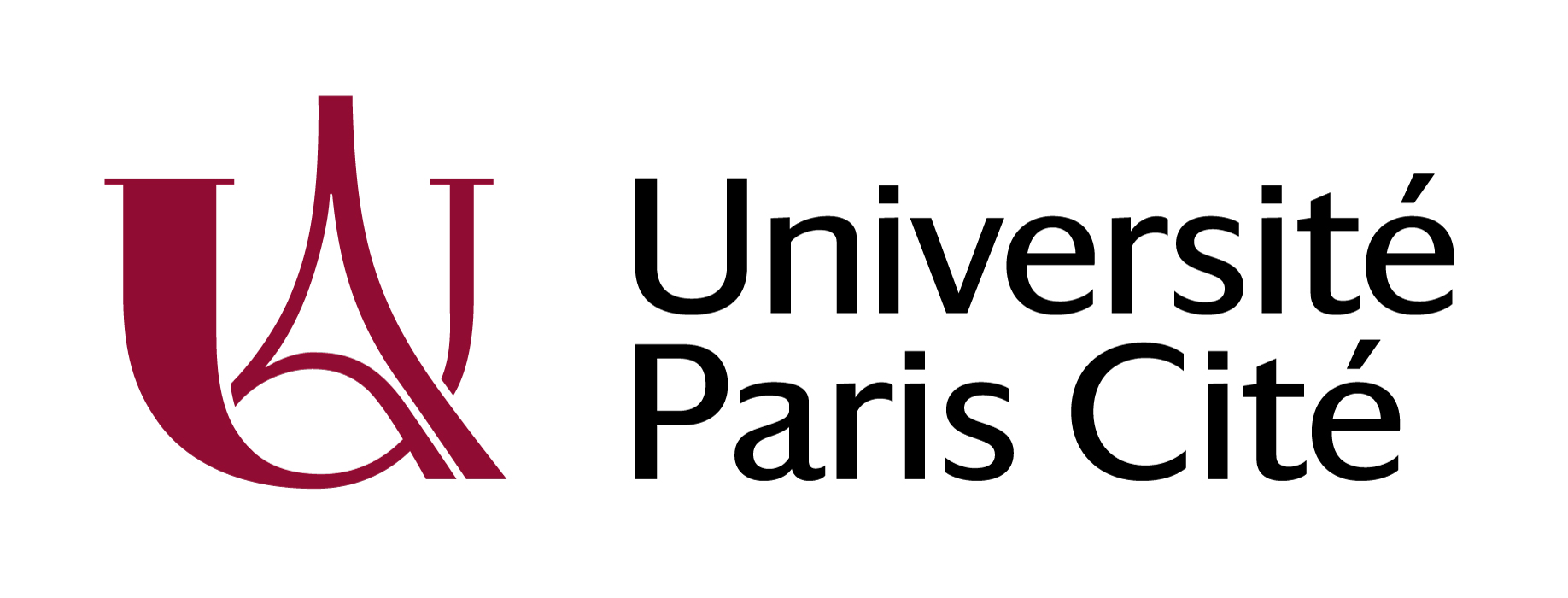}
  %\caption{Universit\'e Paris Cit\'e}
  \label{logo:UPC}
\end{figure}
\begin{center}
%\bfseries
\LARGE
{Universit\'e Paris Cit\'e}
\end{center}
\vskip 40pt
\textbf{\'Ecole Doctorale de Sciences Math\'ematiques de Paris Centres (ED 386)}\\
\textbf{Instituet de Math\'ematiques de Jussieu-Paris Rive Gauche (UMR 7586)}

%\vskip 20pt
%\hrulefill
%\vskip 20pt
%\title{On exact dg categories}
%\maketitle
%\hrulefill

\begin{center}
\rule{\textwidth}{.4pt}\par
{\huge\bfseries On exact dg categories \par}
\rule{\textwidth}{.4pt}\par
\bigbreak
\end{center}
\vskip 10pt
\begin{center}pr\'esent\'e par \textbf{Xiaofa CHEN}\end{center}

\vskip 20pt

\begin{center}TH\`ESE DE DOCTORAT DE MATH\'EMATIQUES\end{center}

\vskip 10pt

\begin{center}Dirig\'ee par \textbf{Bernhard KELLER}\end{center}

\vskip 20pt

\begin{center}Pr\'esent\'ee et soutenue publiquement le 11 Juillet 2023 \end{center}
\vskip 10pt
\begin{center}devant le jury compos\'e de:\end{center}
\vskip 10pt
\[
\begin{aligned}
&\textbf{Mme Claire AMIOT} &&\text{MCU-HDR}&&\text{Universit\'e Grenoble Alpes} &&\text{Examinatrice}\\
&\textbf{M Xiao-Wu CHEN} &&\text{Professor}&&\text{USTC} &&\text{Examinateur}\\
&\textbf{M Gustavo JASSO} &&\text{Senior Lecturer}&&\text{Lunds Universitet} &&\text{Examinateur}\\
&\textbf{M Bernhard KELLER} &&\text{Professeur}&&\text{Universit\'e Paris Cit\'e} &&\text{Directeur}\\
&\textbf{M Julian K\"ULSHAMMER} &&\text{Associate Professor}&&\text{Uppsala University} &&\text{Rapporteur}\\
&\textbf{Mme Muriel LIVERNET} &&\text{Professeure}&&
\text{Universit\'e Paris Cit\'e} &&\text{Examinatrice}\\
&\textbf{Mme Sibylle SCHROLL}&&\text{Professor}&&\text{University of Cologne} &&\text{Examinatrice}\\
&\textbf{M Bertrand TO\"EN} &&\text{DR CNRS}&&\text{Universit\'e de Toulouse} &&\text{Rapporteur}
\end{aligned}
\]
\newpage

\begin{figure}[!htb]
   \begin{minipage}{0.48\textwidth}
     \includegraphics[width=.7\linewidth]{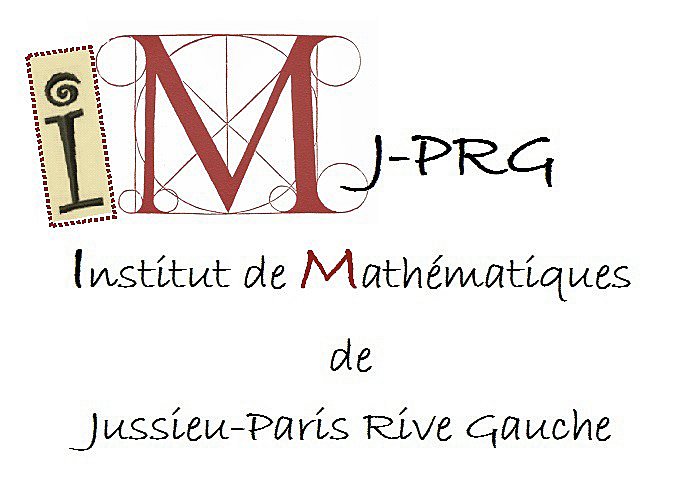}
    % \caption{Interpolation for Data 1}\label{Fig:Data1}
    
Institut de Math\'ematiques de Jussieu-Paris Rive Gauche

Universit\'e Paris Cit\'e – Campus des Grands Moulins

B\^atiment Sophie Germain, Boite Courrier 7012

8 Place Aur\'elie Nemours,

75205 Paris Cedex 13
   \end{minipage}\hfill
   \begin{minipage}{0.48\textwidth}
     \includegraphics[width=.7\linewidth]{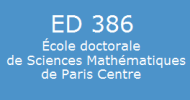}
    % \caption{Interpolation for Data 2}\label{Fig:Data2}
    
    Universit\'e Paris Cit\'e

\'Ecole Doctorale ED 386

B\^atiment Sophie Germain

Case courrier 7012

8 place Aur\'elie Nemours

75205 Paris Cedex 13
   \end{minipage}
\end{figure}
\newpage

\begin{Large}\textbf{Acknowledgments}\end{Large}
\vskip 10pt
I would like to express my heartfelt gratitude to my supervisor, Bernhard Keller. His exceptional knowledge and expertise and consistent support have been extremely helpful throughout my PhD journey. His excellent guidance, productive feedback and continuous encouragement have been invaluable in shaping my research and refining my ideas. I am truly grateful for his very kind sharing of mathematical insights. His taste in mathematics remains a treasure for me.

I would like to thank my supervisor in China Xiao-Wu Chen for the constant support and encouragement during my career in mathematics. His extensive knowledge and commitment to academic excellence have greatly enhanced the quality and rigor of this thesis. I am indebted to his critical insights in shaping the direction and scope of my research. I would like to thank him for being a member of the defense committee of this thesis.

It is my honour to have Bertrand To\"en and Julian K\"ulshammer as the rapporteurs for this thesis and I would like to thank them for accepting it and for all their valuable suggestions. I would like to thank Sibylle Schroll for inviting me to give a talk at the LAGOON seminar, for her suggestions about finding postdoc positions, and for agreeing to be a member of the defense committee. I would like to thank Gustavo Jasso for organising the FD seminar and the study group on $\infty$-categories, and among others, for his talks on exact $\infty$-categories, and for agreeing to be part of the thesis committee. I sincerely thank Claire Amiot and Muriel Livernet for being part of the thesis committee.

I am grateful to Erlend Børve, Peigen Cao, Alessandro Contu, Mikhail Gorsky, Gustavo Jasso, Miantao Liu, Amnon Neeman, Yann Palu, Yu Wang, Zhengfang Wang and Yilin Wu for many interesting discussions and useful comments. I am grateful to Antoine S\'edillot for his help during my stay in Paris.
I would like to thank Steffen Koenig for his advice on searching for a postdoc position.
 I would like to thank Steffen Koenig and Zhengfang Wang for their invitations to University of Stuttgart and their hospitality. I thank Petter Andreas Bergh for his invitation to the Abel Symposium 2022 and the hospitality. 
 
 I would like to thank my wife Yaru Shao for always being supportive and for being with me all the way.
 Finally, I would like to thank my father, my mother and my brother for their unconditional support. 
	\label{s:acknowledgment}
	\newpage

\selectlanguage{french}
\begin{Large}\textbf{R\'esum\'e}\end{Large}
\vskip 10pt
Nous introduisons la notion de dg-cat\'egorie exacte, qui g\'en\'eralise simultan\'ement les notions de 
cat\'egorie exacte au sens de Quillen et de dg-cat\'egorie pr\'etriangul\'ee au sens de Bondal--Kapranov. 
Il s'agit \'egalement d'un analogue diff\'erentiel gradu\'e  de la notion d'$\infty$-cat\'egorie exacte de Barwick et
d'un enrichissement diff\'erentiel gradu\'e de la notion de cat\'egorie extriangul\'ee de Nakaoka--Palu. 
Nos motivations viennnent par exemple des cat\'egories apparaissant dans la cat\'egorification additive des 
alg\`ebres amass\'ees \`a coefficients. Nous donnons une d\'efinition en compl\`ete analogie avec celle de Quillen, 
mais o\`u la cat\'egorie des paires noyau-conoyau est remplac\'ee par une cat\'egorie homotopique plus sophistiqu\'ee. 
Nous obtenons un certain nombre de r\'esultats fondamentaux concernant le dg-nerf, la dg-cat\'egorie d\'eriv\'ee, 
les produits tensoriels et les cat\'egories de dg-foncteurs \`a valeurs dans une dg-cat\'egorie exacte.
Par exemple, nous montrons que pour une dg-cat\'egorie $\A$ dont la cat\'egorie homotopique $H^0(\A)$ est additive,
on a une bijection entre les structures exactes sur $\A$ et les structures exactes (au sens de Barwick) sur le dg-nerf de $\A$. 
Nous montrons \'egalement l'existence de la plus grande structure exacte sur une (petite) dg-cat\'egorie $\A$ dont
la cat\'egorie homotopique est additive. Ceci g\'en\'eralise un th\'eor\`eme de Rump pour les cat\'egories exactes de Quillen.
\vskip 15pt
{\begin{Large}\textbf {Mots-cl\'es}\end{Large}}
\vskip 10pt
Cat\'egorie extriangul\'ee, dg-cat\'egorie exacte,  dg-cat\'egorie d\'eriv\'ee, 
complexes homotopiques \`a 3 termes, sous-cat\'egorie stable par extensions, cat\'egorie amass\'ee.

\selectlanguage{english}
\newpage
\begin{Large}\textbf{Abstract}\end{Large}
\vskip 10pt
We introduce the notion of an exact dg category, which is a simultaneous generalization of the notions of 
exact category in the sense of Quillen and of pretriangulated dg category in the sense of Bondal--Kapranov. 
It is also a differential graded analogue of Barwick's notion of exact $\infty$-category and a differential
graded enhancement of Nakaoka--Palu's notion of extriangulated category. It is completely different 
from Positselski's notion of exact DG-category.
Our motivations come for example from the categories appearing in the additive categorification 
of cluster algebras with coefficients. We give a definition in complete analogy with Quillen's but 
where the category of kernel-cokernel pairs is replaced with a more sophisticated homotopy category. 
We obtain a number of fundamental results concerning the dg nerve, the dg derived category, tensor 
products and functor categories with exact dg target. For example, we show that for a given dg category 
$\A$ with additive homotopy category $H^0(\A)$, there is a bijection between exact structures on $\A$ 
and exact structures (in the sense of Barwick) on the dg nerve of  $\A$. 
We also show the existence of the greatest exact structure on a (small) dg category 
with additive homotopy category. This generalizes a theorem of Rump for Quillen exact categories. 
\vskip 15pt
{\begin{Large}\textbf {Keywords}\end{Large}} 
\vskip 10pt
Extriangulated category, exact dg category, dg derived category, 3-term complex up to homotopy,
homotopy short exact sequence,  extension-closed subcategory.
\newpage

\section*{}
\topskip0pt
\vspace*{4cm}
\rightline{\begin{Large}\`A mon p\`ere\end{Large}}
\vspace*{4cm}

\newpage
\tableofcontents
\newpage
\section{Introduction (version française)}
\subsection{Aperçu}
Dans cette th\`ese, nous introduisons et \'etudions la notion de \emph{dg-cat\'egorie exacte}.
Nous illustrons la position de cette classe de cat\'egories parmi les autres types de cat\'egories que nous consid\'ererons dans cette th\`ese dans le diagramme suivant

\begin{tikzcd}
&\{\text{dg-cat.  pr\'etri.}\}\ar[rd,red]\ar[rr,"N_{dg}"{description,near start},hook]\ar[dd,hook]&&\{\text{$\infty$-cat. stable}\}\ar[dd,hook]\ar[ld,red]\\
&&\{\text{cat. tri.}\}\ar[dd,hook]&\\
\{\text{ cat. ex. (Quillen)}\}\ar[r,hook]\ar[rrd,hook]&\{\text{dg-cat.  ex.}\}\ar[rr,hook,"N_{dg}"{description,near start}]\ar[rd,red]&&\{\text{$\infty$-cat. ex. (Barwick)}\}\ar[ld,red]\\
&&\{\text{cat. extri.}\}&
\end{tikzcd}
Ici, les fl\`eches noires repr\'esentent des inclusions de classes et les fl\`eches rouges envoient une dg-cat\'egorie  (respectivement une $\infty$-cat\'egorie) $\A$ sur $H^0(\A)$ (respectivement $h\A$). Comme nous le voyons dans le diagramme, la notion de dg-cat\'egorie exacte est
\begin{itemize}
\item[1)] une version diff\'erentielle gradu\'ee de la notion d'$\infty$-cat\'egorie exacte de Barwick \cite{Barwick15};
\item[2)] une g\'en\'eralisation simultan\'ee des notions de cat\'egorie exacte au sens de Quillen \cite{Quillen73} et de dg-cat\'egorie  pr\'etriangul\'ee au sens de Bondal--Kapranov \cite{BondalKapranov90};
\item[3)] un enrichissement diff\'erentiel gradu\'e de la notion de cat\'egorie extriangul\'ee au sens de Nakaoka--Palu \cite{NakaokaPalu19} en analogie avec l'enrichissement diff\'erentiel gradu\'e de la notion de cat\'egorie triangul\'ee de Bondal--Kapranov.
\end{itemize}
Nous avons \'egalement des caract\'erisations de certaines des sous-classes pr\'esentes dans le diagramme:
\begin{itemize}
\item[1)] une cat\'egorie extriangul\'ee est une cat\'egorie exacte de Quillen si et seulement si toute inflation est monomorphe et toute d\'eflation est \'epimorphe, cf. \cite[Corollaire 3.18]{NakaokaPalu19};
\item[2)] une cat\'egorie extriangul\'ee est une cat\'egorie triangul\'ee si et seulement si elle est une cat\'egorie extriangul\'ee de Frobenius dont les projectifs-injectifs sont les objets nuls, cf. \cite[Corollaire 7.6]{NakaokaPalu19};
\item[3)] une dg-cat\'egorie exacte est quasi-\'equivalente \`a une cat\'egorie exacte de Quillen si et seulement si 
l'homologie de la dg-cat\'egorie sous-jacente est concentr\'ee en degr\'e z\'ero;
\item[4)] une dg-cat\'egorie exacte connective est quasi-\'equivalente \`a la $\tau_{\leq 0}$-troncation d'une dg-cat\'egorie pr\'etriangul\'ee si et seulement si sa cat\'egorie d'homologie de degr\'e z\'ero est canoniquement triangul\'ee, cf. Exemple~\ref{exm:exactdgpretriangulated}.
\end{itemize}
Plus de d\'etails sur ces classes et les d\'eveloppements historiques sont donn\'es dans les sous-sections suivantes.
Soulignons que notre notion de dg-cat\'egorie exacte est compl\`etement diff\'erente de la notion de Positselski de DG-cat\'egorie exacte \cite{Positselski21}. Par exemple, une structure exacte dans notre sens peut être transport\'ee le long d'une quasi-\'equivalence, cf. Remarque~\ref{truncationexactdgstructure} b), ce qui n'est pas du tout le cas pour les structures exactes au sens de Positselski. La motivation
principale de Positselski est d'axiomatiser des situations o\`u l'on a une dg-cat\'egorie
{\em pr\'etriangul\'ee} $\A$ dont la cat\'egorie $Z^0(\A)$ est en outre munie d'une structure exacte
au sens de Quillen (comme la cat\'egorie des complexes \`a composantes dans une
cat\'egorie exacte). Par contre, notre notion vise \`a axiomatiser une classe de dg-cat\'egories
{\em non n\'ecessairement pr\'etriangul\'ees} munie d'une structure suppl\'ementaire.

\subsection{Cat\'egories exactes}
La notion de cat\'egorie exacte a \'et\'e introduite par Quillen dans \cite{Quillen73} afin de d\'efinir sa $K$-th\'eorie alg\'ebrique sup\'erieure (les cat\'egories exactes Karoubiennes ont \'et\'e introduites 15 ans plus tôt par Heller \cite{Heller58}).
Grosso modo, une cat\'egorie exacte est une cat\'egorie additive $\E$ munie d'une classe de paires noyau-conoyau appel\'ees conflations
\[
\begin{tikzcd}
L\ar[r,tail,"i"]&M\ar[r,two heads,"p"]&N
\end{tikzcd}
\]
o\`u $i$ est appel\'ee inflation et $p$ est appel\'ee d\'eflation, et qui satisfait certains axiomes.
La liste des axiomes a \'et\'e r\'eduite au `minimum' par Keller dans \cite[Appendice A.1]{Keller90}.
Notamment, l'axiome `obscur' de Quillen est redondant et peut être d\'eduit des autres axiomes, cf. aussi  \cite[p.~525, Corollaire]{Yoneda60}.
Pour un traitement complet de la th\'eorie des cat\'egories exactes au sens de Quillen, nous renvoyons \`a l'article de synth\`ese de B\"uhler \cite{Buhler10}. Pour une \'etude des groupes d'extension sup\'erieurs dans une cat\'egorie exacte, nous renvoyons aux notes de cours \cite{FrerickSieg10}.
Quillen a affirm\'e et Laumon et Thomason--Trobaugh \cite{ThomasonTrobaugh90} ont prouv\'e que pour une cat\'egorie exacte essentiellement petite $\E$, il existe une \'equivalence exacte $G:\E\rightarrow \C$ sur une sous-cat\'egorie pleine stable par extensions d'une cat\'egorie ab\'elienne $\C$ telle que le foncteur $G$ soit pleinement exact, c'est-\`a-dire qu'une suite dans $\E$ est une conflation si et seulement si son image dans $\C$ est une suite exacte courte, cf. \cite[A.2 Proposition]{Keller90}.
Quillen a utilis\'e les cat\'egories exactes comme cadre pour sa remarquable construction $Q$, qui est li\'ee \`a la bicat\'egorie des ponts introduite par B\'enabou, cf. \cite[2.6]{Benabou67}. La cat\'egorie $Q(\E)$ d'une cat\'egorie exacte a les m\^emes objets que $\E$ et les morphismes de $M$ \`a $M'$ sont donn\'es par des classes d'\'equivalence de toits
\[
\begin{tikzcd}
M&N\ar[l,two heads,"p"swap]\ar[r,tail,"i"]&M'
\end{tikzcd}
\]
o\`u $p$ est une d\'eflation et $i$ est une inflation, et deux toits sont consid\'er\'es comme 
\'equivalents s'il existe un isomorphisme entre les toits qui se restreint aux identit\'es sur $M$ et $M'$.
Les compositions dans $Q(\E)$ sont donn\'ees par des pullbacks.
En fait, la liste des axiomes de Quillen correspond exactement \`a ce dont on a besoin lors de la construction $Q$.
La g\'en\'eralisation des cat\'egories ab\'eliennes aux cat\'egories exactes permet d'obtenir plus d'exemples, comme la cat\'egorie des modules libres sur un anneau, la cat\'egorie des fibr\'es vectoriels sur un sch\'ema, les cat\'egories des modules localement libres sur une alg\`ebre pr\'eprojective g\'en\'eralis\'ee \cite{GeissLeclercSchroer17}, \ldots\ .

\subsection{Des cat\'egories triangul\'ees aux dg-cat\'egories}
D'autre part, pour formuler et d\'emontrer ses extensions du th\'eor\`eme de dualit\'e de Serre \cite{Grothendieck60},
Grothendieck a invent\'e la notion de cat\'egories d\'eriv\'ees et,
avec son \'etudiant Verdier, a introduit la notion de cat\'egorie triangul\'ee \cite{Verdier96}.
Presque simultan\'ement, Puppe a \'egalement trouv\'e les axiomes (sauf l'axiome de l'octa\`edre) 
pour les cat\'egories triangul\'ees \cite{Puppe62}.
Pendant longtemps, les cat\'egories triangul\'ees ont \'et\'e consid\'er\'ees comme trop grossi\`eres pour permettre le d\'eveloppement d'une th\'eorie substantielle, principalement en raison du manque de fonctorialit\'e de la construction du c\^one. Cette perception a chang\'e gr\^ace notamment aux travaux de Neeman \cite{Neeman01,Neeman22}, mais il reste le fait que de nombreuses op\'erations sur les cat\'egories d\'eriv\'ees ne peuvent  pas \^etre effectu\'ees dans le cadre des cat\'egories triangul\'ees,
notamment les produits tensoriels de cat\'egories triangul\'ees et les cat\'egories de foncteurs avec des cibles triangul\'ees.
Une solution viable \`a ce probl\`eme est fournie par la th\'eorie des dg-cat\'egories.
La notion de dg-cat\'egorie est apparue dans le travail de Kelly \cite{Kelly65}.
Dans les ann\'ees 90, Bondal--Kapranov \cite{BondalKapranov90} ont observ\'e que
les espaces de morphismes des cat\'egories triangul\'ees apparaissant naturellement en alg\`ebre
sont les groupes de cohomologie de degr\'e z\'ero de certains complexes, et
il est un principe (popularis\'e par Grothendieck) en alg\`ebre homologique de consid\'erer le complexe lui-même
plutôt que sa cohomologie.
Par cons\'equent, ils ont propos\'e d'enrichir les cat\'egories triangul\'ees en utilisant des dg-cat\'egories,
en exigeant que la structure triangul\'ee soit compatible avec la structure dg.
En formalisant cette id\'ee, ils ont d\'efini la propri\'et\'e d'une dg-cat\'egorie d'être pr\'etriangul\'ee et
ont observ\'e que chaque dg-cat\'egorie peut engendrer librement une dg-cat\'egorie pr\'etriangul\'ee.
Peu de temps apr\`es, Keller a publi\'e son article fondamental \cite{Keller94}
sur les cat\'egories d'homotopie et les cat\'egories d\'eriv\'ees non born\'ees des dg-cat\'egories.
Drinfeld \cite{Drinfeld04} a donn\'e une construction \'el\'egante du quotient dg, dont l'existence
remonte \`a Keller \cite{Keller99}, et qui enrichit le quotient de Verdier.

\subsection{Des cat\'egories extriangul\'ees aux dg-cat\'egories exactes}
La notion de cat\'egorie extriangul\'ee a \'et\'e introduite par Nakaoka--Palu dans \cite{NakaokaPalu19}. Elle constitue une g\'en\'eralisation simultan\'ee de la notion de cat\'egorie exacte de Quillen et de la notion de cat\'egorie triangul\'ee de Grothendieck--Verdier. Leur objectif \'etait de fournir un cadre pratique pour \'enoncer des preuves qui s'appliquent \`a la fois aux cat\'egories exactes et aux cat\'egories triangul\'ees, et plus g\'en\'eralement aux sous-cat\'egories stable par extensions des cat\'egories triangul\'ees. Depuis leur introduction en 2019, la th\'eorie des cat\'egories extriangul\'ees a beaucoup \'evolu\'e et de nombreuses notions et constructions ont \'et\'e g\'en\'eralis\'ees dans ce cadre (ou dans le cadre plus g\'en\'eral des cat\'egories $n$-exangul\'ees \cite{HerschendLiuNakaoka21}). Parmi ces notions et constructions, on peut citer :
\begin{itemize}
\item les paires de cotorsion et leurs cœurs et co-cœurs, ainsi que les sous-cat\'egories 
$n$-amas-basculantes et bousculantes (silting), cf. \cite{NakaokaPalu19, LiuNakaoka19, ChangZhouZhu19, AdachiTsukamoto22, HuertaMendozaSaenzSantiago22, AdachiTsukamoto23, ZhaoZhuZhuang21, ZhuZhuang20, LiuZhouZhouZhu21};
\item les mutations et la th\'eorie des objets amas-basculants, cf.~\cite{ChangZhouZhu19, Pressland23, ZhouZhu18, LiuZhou20, LiuZhou20a, LiuZhou21, JorgensenShah22, GorskyNakaokaPalu23};
\item les Karoubianisations, cf.~\cite{WangWeiZhangZhao23, Dixy22, KlapprothMsapatoShah22};
\item la th\'eorie d'Auslander--Reiten, cf.~\cite{IyamaNakaokaPalu18, TanWangZhao23};
\item la formule d'Auslander pour la cat\'egorie des foncteurs d'une cat\'egorie extriangul\'ee, cf.~\cite{Ogawa21};
\item les sous-structures d'une cat\'egorie extriangul\'ee, cf.~\cite{HerschendLiuNakaoka21, Enomoto21, Sakai23};
\item le recollement de cat\'egories extriangul\'ees, cf.~\cite{WangWeiZhang22, HeHuZhou22, HuZHou21, MaZhaoZhuang23};
\item les groupes de Grothendieck des cat\'egories extriangul\'ees, cf.~\cite{ZhuZhuang21, Haugland21, EnomotoSaito22};
\item les alg\`ebres de Hall des cat\'egories extriangul\'ees, cf.~\cite{WangWeiZhang22a, FangGorsky22};
\item les localisations des cat\'egories extriangul\'ees, cf.~\cite{NakaokaOgawaSakai22, Ogawa22, HeHeZhou22};
\item \ldots
\end{itemize}
L'\'etude de la classe des sous-structures (= sous-bifoncteurs ferm\'es) d'une cat\'egorie exacte appara\^{\i}t d\'ej\`a dans \cite{ButlerHorrocks61}, cf.~aussi \cite{AuslanderSolberg93a, DraxlerReitenSmaloSolberg99, BrustleHassounLangfordRoy20, FangGorsky22, BaillargeonBrustleGorskyHassoun22}. Notons que la classe des sous-structures triangul\'ees d'une cat\'egorie triangul\'ee ne contient que la structure triviale.
Il serait int\'eressant de trouver une g\'en\'eralisation commune des cat\'egories exactes \`a droite et des cat\'egories suspendues, cf.~\cite{KellerVossieck87, BazzoniCrivei13}.

Par analogie avec la notion de cat\'egorie triangul\'ee alg\'ebrique \cite{Keller06d}, nous d\'efinissons une {\em cat\'egorie extriangul\'ee alg\'ebrique} comme suit.
\begin{Definition-Proposition}[Definition-Proposition~\ref{algebraic}]\label{def:algebraic}
Soit $\C$ une cat\'egorie extriangul\'ee. Les \'enonc\'es suivants sont \'equivalents :
\begin{itemize}
\item[1)] $\C$ est \'equivalente, en tant que cat\'egorie extriangul\'ee, \`a une sous-cat\'egorie pleine stable par extensions d'une cat\'egorie triangul\'ee alg\'ebrique.
\item[2)] $\C$ est \'equivalente, en tant que cat\'egorie extriangul\'ee, \`a $\B/(\mathcal P_0)$ pour une cat\'egorie exacte de Quillen $\B$ et une classe $\mathcal P_0$ d'objets projectifs-injectifs.
\item[3)] $\C$ est \'equivalente, en tant que cat\'egorie extriangul\'ee, \`a $H^0(\A)$ pour une {\em dg-cat\'egorie exacte} $\A$.
\end{itemize}
Si l'une des conditions \'equivalentes ci-dessus est satisfaite, alors $\C$ est appel\'ee une {\em cat\'egorie extriangul\'ee alg\'ebrique}.
\end{Definition-Proposition}

\`A partir de la D\'efinition-Proposition, il est clair que la classe des cat\'egories extriangul\'ees alg\'ebriques est stable sous~:
\begin{itemize}
\item[a)] passage aux sous-cat\'egories stables par extensions ;
\item[b)] quotients par des id\'eaux engendr\'es par les identit\'es d'objets projectifs-injectifs ;
\item[c)] passage aux sous-structures extriangul\'ees, cf.~Corollaire~\ref{cor:exactsubstructure}.
\end{itemize}
Nous nous attendons \'egalement \`a ce qu'elle soit stable lors du passage aux localisations au sens de Nakaoka--Ogawa--Sakai \cite{NakaokaOgawaSakai22}.
Il est \'egalement clair que si une cat\'egorie extriangul\'ee alg\'ebrique est triangul\'ee, alors elle est aussi alg\'ebrique en tant que cat\'egorie triangul\'ee.

\`A partir de la partie 1) de la D\'efinition-Proposition, il est clair qu'une cat\'egorie extriangul\'ee alg\'ebrique $\C$ poss\`ede un $\delta$-foncteur bivariant non born\'e au sens de~\cite[Definition 4.5]{GorskyNakaokaPalu21}.
En particulier, cela fournit un bifoncteur additif $\mathbb E^{-1}$ qui est crucial dans l'\'etude
de l'ensemble partiellement ordonn\'e des {\em paires de $s$-torsion} dans une cat\'egorie 
extriangul\'ee dans~\cite{AdachiEnomotoTsukamoto23}.

\subsection{Classes d'exemples de cat\'egories extriangul\'ees et de dg-cat\'egories exactes}
L'exemple suivant est tir\'e de ~\cite{Jin20}. Soit $k$ un corps.
Une dg-alg\`ebre $A$ sur $k$ est {\em propre} si $\sum_{i\in \mathbb Z}\dim_{k}{H^i(A)<\infty}$. 
Elle est {\em de Gorenstein} si la sous-cat\'egorie \'epaisse $\per(A)$ de la cat\'egorie d\'eriv\'ee $\D(A)$ 
engendr\'ee par $A$ coïncide avec la sous-cat\'egorie \'epaisse $\thick(DA)$ engendr\'ee par $DA$, où $D=\Hom_k(-,k)$ est le $k\mbox{-dual}$.
Soit $A$ une dg-alg\`ebre de Gorenstein propre et connective.
Un dg-module sur $A$ est {\em \`a valeurs parfaites} si son complexe sous-jacent d'espaces
vectoriels est parfait. Nous notons $\pvd(A)$ la cat\'egorie triangul\'ee des dg-modules 
\`a valeurs parfaites.
Un dg-module $M$ sur $A$ dans $\pvd(A)$ est {\em {de Cohen--Macaulay}} si $H^i(M)=0$ et $\Hom_{\D(A)}(M,\Sigma^i A)=0$ pour $i>0$.
Soit $\CM A$ la sous-cat\'egorie de $\pvd(A)$ constitu\'ee des dg-modules de Cohen--Macaulay sur $A$.
D'apr\`es \cite[Th\'eor\`eme 2.4]{Jin20}, la cat\'egorie $\CM A$ est une cat\'egorie extriangul\'ee de Frobenius $\Ext\mbox{-}$finie avec $\proj(\CM A)=\add A$. 
D'apr\`es la Proposition~\ref{prop:dgsingularitycategory}, la cat\'egorie stable $\CM(A)/[\proj(\CM A)]$ est \'equivalente \`a la cat\'egorie des singularit\'es $\pvd(A)/\per(A)$.
Nous verrons dans l'Exemple~\ref{exm:CMdgmodule} que si $\CM_{dg} A$ d\'esigne l'enrichissement diff\'erentiel gradu\'e de $\CM A$ h\'erit\'e
de $\pvd_{dg}(A)$,
alors l'inclusion de $\CM_{dg}(A)$ dans $\pvd_{dg}(A)$ induit une \'equivalence entre la cat\'egorie d\'eriv\'ee de $\CM_{dg}(A)$ et $\pvd_{dg}(A)$. De plus, cette \'equivalence induit une \'equivalence entre la cat\'egorie des singularit\'es de $\CM_{dg}(A)$ et celle de $A$, de mani\`ere parfaitement analogue \`a la situation d'une alg\`ebre d'Iwanaga--Gorenstein de dimension finie concentr\'ee en degr\'e $0$.

On peut facilement g\'en\'eraliser les r\'esultats de l'Exemple~\ref{exm:CMdgmodule} \`a la situation de r\'eduction bousculante (silting) au sens d'Iyama--Yang~\cite{IyamaYang18}~:
Soit $\T$ une cat\'egorie triangul\'ee alg\'ebrique et $\P$ une sous-cat\'egorie pr\'ebousculante de 
$\T$, c'est-\`a-dire une sous-cat\'egorie pleine telle que $\Hom_{\T}(\P,\Sigma^{i}\P)=0$ pour $i>0$. 
Soit $\mathcal S$ la sous-cat\'egorie \'epaisse de $\T$ engendr\'ee par $\P$.
Soient $\V$ et $\W$ les sous-cat\'egories pleines suivantes de $\T$~:
\[
\V=\{X|\Hom_{\T}(X,\Sigma^{i}\P)=0 \text{ pour $i>0$}\} \quad \W=\{X|\Hom_{\T}(\P,\Sigma^{i}X)=0 \text{ pour $i>0$}\}.
\]
Posons $\Z=\V\cap \W$.
Alors $\Z$ est stable par extensions dans $\T$ et donc est canoniquement extriangul\'ee. Notons que les objets dans $\P$ sont projectifs-injectifs dans $\Z$.
Supposons que $\Z$ est une cat\'egorie extriangul\'ee de Frobenius dont les objets projectifs-injectifs sont les objets dans $\P$, c'est-\`a-dire que $(\Z,\Z)$ est une paire de 
$\P$-mutation~\cite[Definition 2.5]{IyamaYoshino08}.
Par exemple, cela est vrai lorsque $\P$ est  finie \`a gauche\footnote{les restrictions \`a $\P$
des foncteurs $\Hom(?,V)$, $V\in \V$, sont de type finie, par exemple si l'inclusion de $\P$ dans
$\V$ admet un adjoint \`a gauche} dans $\V$
et finie \`a droite dans $\W$.
Dans ce cas, le quotient additif $\Z/[\P]$ est canoniquement une cat\'egorie triangul\'ee.
Supposons en outre que $\Z$ engendre $\T$ en tant que cat\'egorie triangul\'ee. Cela est vrai si et seulement si pour chaque objet $X\in \T$, on a $\Sigma^{-l}X\in \V$ et $\Sigma^{l}X\in \W$ lorsque $l\gg 0$.
Si nous remplaçons les cat\'egories ci-dessus par leurs dg-enrichissements, alors la \mbox{$\tau_{\leq 0}$-troncation} de $\Z$ porte une structure dg exacte canonique dont la dg-cat\'egorie d\'eriv\'ee born\'ee est $\T$ et dont la dg-cat\'egorie des singularit\'es est $\T/\mathcal S$.
Des consid\'erations similaires s'appliquent au cas sp\'ecial de~\cite[1.2]{IyamaYang20} où la paire de torsion $\mathcal S=\X\perp \Y$ est une t-structure.

Les dg-cat\'egories exactes apparaissent \'egalement naturellement dans le contexte de la cat\'egorification des alg\`ebres amass\'ees \`a coefficients. Ces alg\`ebres amass\'ees sont associ\'ees \`a des {\em carquois glac\'es}, qui sont des carquois munis d'un sous-carquois gel\'e (pas n\'ecessairement plein). Pour r\'ealiser la cat\'egorification, on munit un tel
carquois glac\'e d'un potentiel non d\'eg\'en\'er\'e. 
Dans sa th\`ese \cite{Wu21, Wu23a, Wu23b}, Yilin Wu a construit une cat\'egorie extriangul\'ee de Frobenius, la {\em cat\'egorie de Higgs}, pour chaque carquois glac\'e \`a potentiel Jacobi-fini. Sa construction g\'en\'eralise les cat\'egories de modules sur les alg\`ebres pr\'eprojectives utilis\'ees avec beaucoup de succ\`es par Geiss--Leclerc--Schröer \cite{GeissLeclercSchroeer11b,GeissLeclercSchroeer08b,GeissLeclercSchroeer06,GeissLeclercSchroeer13}. La cat\'egorie de Higgs est une sous-cat\'egorie stable par extensions d'une cat\'egorie triangul\'ee alg\'ebrique canonique (la cat\'egorie amass\'ee relative) et poss\`ede donc une dg-enrichissement canonique. Pr\'ecisons cela davantage (nous utilisons les notations de \cite{Wu23a}): soit $(Q,F,W)$ un carquois glac\'e \`a potentiel Jacobi-fini. Notons $\Gamma_{rel}$ la dg-alg\`ebre de Ginzburg 
relative $\Gamma_{rel}(Q,F,W)$. Soit $e=\sum_{i\in F}e_i$ l'idempotent associ\'e \`a tous les sommets gel\'es. Soit $\pvd_{e}(\Gamma_{rel})$ la sous-cat\'egorie pleine des dg-modules sur $\Gamma_{rel}$ dont la restriction aux sommets gel\'es est acyclique. 
Alors la cat\'egorie amass\'ee relative $\C=\C(Q,F,W)$ associ\'ee \`a $(Q,F,W)$ est d\'efinie comme le quotient de Verdier 
\[
\per(\Gamma_{rel})/\pvd_{e}(\Gamma_{rel}).
\]
Soit $\mathcal P=\add(e\Gamma_{rel})$. Le {\em domaine fondamental relatif} $\mathcal F^{rel}_{\Gamma_{rel}}=\mathcal F^{rel}$ associ\'e \`a $(Q,F,W)$ est d\'efini comme la sous-cat\'egorie suivante de $\per(\Gamma_{rel})$
\[
\mathcal F^{rel}{\coloneqq}\{\Cone(X_1\xrightarrow{f}X_0) | X_i\in\add(\Gamma_{rel}) \text{ et } \Hom(f,I) \text{ est surjective}, \forall I\in \mathcal P\}.
\]
\'Etant donn\'e que $\add{\Gamma_{rel}}$ est pr\'ebousculant dans $\per\Gamma_{rel}$, 
la sous-cat\'egorie pleine 
\[
{\Z'=\add\Gamma_{rel}\ast\add\Sigma\Gamma_{rel}}
\]
 est stable par extensions dans $\per\Gamma_{rel}$ et h\'erite donc d'une structure extriangul\'ee canonique. 
Ensuite, $\F^{rel}$ peut \'egalement être d\'ecrit comme la sous-cat\'egorie stable par extensions de $\Z'$ et form\'ee des objets $X$ qui sont relativement $\mathbb E$-projectifs par rapport \`a $\P$, c'est-\`a-dire les objets $X$ tels que $\Ext^1_{\T}(X,\P)=0$. Soit $\pi^{rel}:\per(\Gamma_{rel})\rightarrow \C$ le foncteur quotient canonique. La {\em cat\'egorie de Higgs} $\mathcal H$ est l'image de $\mathcal F^{rel}$ dans $\C$ par le foncteur quotient $\pi^{rel}$. D'apr\`es \cite[Proposition 5.39, Theorem 5.46]{Wu23a}, la cat\'egorie de Higgs $\mathcal H$ est stable par extensions dans $\C$ et, dot\'ee de la structure extriangul\'ee h\'erit\'ee, devient une cat\'egorie extriangul\'ee de Frobenius dont la sous-cat\'egorie des projectifs-injectifs est $\P$.
Nous remplaçons les cat\'egories pr\'ec\'edentes par leurs dg-enrichissements. Le dg-foncteur canonique $\tau_{\leq 0}\F^{rel}\rightarrow \tau_{\leq 0}\mathcal H$ est une quasi-\'equivalence et la structure dg exacte sur $\tau_{\leq 0}\F^{rel}$ est une sous-structure de celle sur $\tau_{\leq 0}\mathcal H$. Par le Th\'eor\`eme~\ref{intro:maximalstructure}, la dg-cat\'egorie sous-jacente $\tau_{\leq 0}\F^{rel}$ poss\`ede la plus grande structure exacte. Si nous prenons la sous-structure exacte de la plus grande structure exacte qui rend les objets de $\P$ projectifs-injectifs, nous retrouvons la dg-cat\'egorie exacte $\tau_{\leq 0}\mathcal H$. Nous verrons \`a la Section~\ref{subsection:higgscategories} que la cat\'egorie d\'eriv\'ee born\'ee de $\tau_{\leq 0}\F^{rel}$ est $\per\Gamma_{rel}$. La cat\'egorie d\'eriv\'ee born\'ee de 
$\tau_{\leq 0}\mathcal H$ est $\C$ et la cat\'egorie des singularit\'es de $\tau_{\leq 0}\mathcal H$ est la cat\'egorie amass\'ee associ\'ee au carquois \`a potentiel $(\overline Q,\overline W)$, où $\overline{Q}$ est le carquois obtenu \`a partir de $Q$ en supprimant tous les sommets gel\'es et toutes les fl\`eches incidentes aux sommets gel\'es, et $\overline W$ est le potentiel sur $\overline Q$ obtenu en supprimant tous les cycles passant par les sommets de $F$ dans $W$.

\subsection{$\infty$-cat\'egories stables et $\infty$-cat\'egories exactes}
La notion d'ensemble simplicial a \'et\'e introduite dans~\cite{EilenbergZilber50} comme un affinement de la notion de complexe simplicial.
Boardman--Vogt~\cite[Definition 4.8]{BoardmanVogt73} ont utilis\'e la {\em condition de Kan restreinte} pour d\'efinir la notion de {\em complexe de Kan faible} (= {\em quasicat\'egorie} = {\em $\infty$-cat\'egorie}) dans la cat\'egorie des ensembles simpliciaux.
Chaque petite cat\'egorie peut être vue comme une $\infty$-cat\'egorie via le foncteur {\em nerf} $N$.
\`A chaque ensemble simplicial $X$, on peut associer sa {\em cat\'egorie homotopique} $h(X)$.
Elle a une description simple lorsque $X$ est une $\infty$-cat\'egorie, cf.~\cite{BoardmanVogt73}.
La th\'eorie des quasicat\'egories a \'et\'e consid\'erablement d\'evelopp\'ee par Joyal et Lurie, entre autres, cf.~\cite{Lurie09,LurieHA,Joyal08,Cisinski19}.
Lurie \cite{Lurie06}, cf.~aussi~\cite{LurieHA} et \cite{ToenVezzosi04a}, a introduit la notion de {\em $\infty$-cat\'egorie stable}, bas\'ee sur les techniques \'etudi\'ees dans son livre~\cite{Lurie09}, et a montr\'e que la cat\'egorie homotopique d'une $\infty$-cat\'egorie stable est canoniquement triangul\'ee.
Les cat\'egories triangul\'ees ainsi obtenues sont appel\'ees {\em topologiques}.
Le foncteur {\em dg-nerf} $N_{dg}$ a \'et\'e introduit dans \cite{LurieHA} comme un outil cl\'e reliant les dg-cat\'egories aux $\infty$-cat\'egories.
Il a montr\'e qu'il existe une adjonction de Quillen entre la {\em structure de mod\`ele de Joyal} pour les $\infty$-cat\'egories et la {\em structure de mod\`ele de Dwyer--Kan} pour les dg-cat\'egories.
Faonte a d\'emontr\'e dans son article \cite{Faonte17b} que le dg-nerf d'une dg-cat\'egorie pr\'e-triangul\'ee est une $\infty$-cat\'egorie stable, où il a g\'en\'eralis\'e le foncteur dg-nerf au foncteur 
nerf simplicial reliant les $A_{\infty}$-cat\'egories aux $\infty$-cat\'egories.
Barwick a introduit la notion de $\infty$-cat\'egorie exacte dans~\cite{Barwick15} en tant que g\'en\'eralisation homotopique de la notion de cat\'egorie exacte de Quillen.
La construction $Q$ de Quillen a \'egalement \'et\'e g\'en\'eralis\'ee \`a ce cadre, cf.~\cite{Barwick13}.
Dans~\cite{NakaokaPalu20}, il a \'et\'e d\'emontr\'e que la cat\'egorie homotopique d'une $\infty$-cat\'egorie exacte $\E$ porte une structure extriangul\'ee canonique.
Les cat\'egories extriangul\'ees obtenues ainsi sont appel\'ees {\em topologiques}.
Klemenc~\cite{Klemenc22} a construit, pour chaque petite $\infty$-cat\'egorie exacte $\E$, l'{\em enveloppe stable} $\D^b_{\infty}(\E)$ avec un foncteur pleinement exact et pleinement fid\`ele $\eta_{\E}:\E\rightarrow \D^b_{\infty}(\E)$ dont l'image essentielle est stable par extensions.
\'Etant donn\'e que la cat\'egorie homotopique de $\D^b_{\infty}(\E)$ est triangul\'ee (d'apr\`es le th\'eor\`eme de Lurie), cela donne une autre preuve du r\'esultat de Nakaoka--Palu.
Par rapport aux dg-cat\'egories, les espaces de morphismes dans les quasi-cat\'egories ne sont pas faciles \`a manipuler.
C'est l'une de nos motivations pour introduire la notion de {\em dg-cat\'egorie exacte}.

\subsection{Dg-cat\'egories exactes}
Soit $\A$ une dg-cat\'egorie. Pour simplifier, nous supposons que $\A$ est connective, a des complexes Hom cofibrants et que $Z^0(\A)$ est additive. 
Nous utilisons les notations introduites dans la section~\ref{subsection:notations}.

La cat\'egorie {\em carr\'ee} $\mathrm{Sq}$ est la cat\'egorie des chemins du carquois 
\[
\begin{tikzcd}
00\ar[r,"f"]\ar[d,"g"swap] & 01\ar[d,"j"] \\
10\ar[r,"k"swap] & 11
\end{tikzcd}
\]
avec la relation de commutation $jf\sim kg$.
Nous identifions les objets $X\in\rep(k\mathrm{Sq},\A)$ avec les carr\'es commutatifs
\begin{equation}
\begin{tikzcd}\label{D(Sq)1}
X_{00}\ar[r,"f"]\ar[d,"g"swap]&X_{01}\ar[d,"j"]\\
X_{10}\ar[r,"k"swap]&X_{11}
\end{tikzcd}
\end{equation}
dans $\C(\A)$ où $jf=kg$. 
\begin{Definition}[D\'efinition~\ref{maindef}]
Un objet $X\in \rep(k\Sq,\A)$ est un {\em{{carr\'e cart\'esien homotopique}}} par rapport \`a $\A$ si l'application canonique $X_{00}\rightarrow \Sigma^{-1}\mathrm{Cone}((-j,k))$ induit un isomorphisme   
\[
\tau_{\leq 0}\RHom_\A(A^{\wedge},X_{00})\rightarrow \tau_{\leq 0}\RHom_\A(A^{\wedge}, \Sigma^{-1}\mathrm{Cone}((-j,k)))
\]
dans $\D(k)$ pour chaque $A$ dans $\A$.
\end{Definition}
De mani\`ere duale, on peut d\'efinir la notion de {\em carr\'e cocart\'esien homotopique} et ensuite d\'efinir la notion de {\em carr\'e bicart\'esien homotopique}.
Nous prenons souvent une r\'esolution cofibrante $\overline{\Sq}\rightarrow \Sq$ de $\Sq$, cf.~Section~\ref{res}, et via l'\'equivalence de cat\'egories $\rep(k\overline{\Sq},\A)\rightarrow\rep(k\Sq,\A)$ nous identifions les objets des deux côt\'es. 
Un avantage est que nous pouvons supposer que tous les termes d'un objet $X\in \rep(k\overline{\Sq},\A)$ sont des modules dg repr\'esentables (et pas seulement quasi-repr\'esentables).

Soit $\I$ la dg-cat\'egorie des chemins du carquois \`a relations suivant
\begin{equation}\label{quiv1:3term}
\begin{tikzcd}
0\ar[r,"f",""{swap, name=1}]&1\ar[r,"g",""{swap,name=2}]&2\ar[r,from=1,to=2,dashed,no head,swap,bend right =8ex]
\end{tikzcd}
\end{equation}
où $f$ et $g$ sont des morphismes ferm\'es de degr\'e 0 et $gf=0$.
La cat\'egorie $\mathcal H_{3t}(\A)$ des complexes homotopiques \`a 3 termes (=3-term h-complexes) sur $\A$ est d\'efinie comme la cat\'egorie de 0-homologie $H^0(\Fun_{\infty}(\I,\A))$ de la dg-cat\'egorie $\Fun_{\infty}(\I,\A)$ des $A_{\infty}$-foncteurs 
(pas des $\infty$-foncteurs !) de $\I$ \`a $\A$, cf.~D\'efinition~\ref{def:3termhomotopy}.
Soit $\J$ la dg-cat\'egorie des chemin du carquois diff\'erentiel gradu\'e suivant 
\begin{equation}\label{quiv1:3termcofibrant}
\begin{tikzcd}
0\ar[r,"f"]\ar[rr,"h"swap,bend right=6ex]&1\ar[r,"g"]&2
\end{tikzcd}
\end{equation}
où $f$ et $g$ sont des morphismes ferm\'es de degr\'e 0 et $h$ est de degr\'e $-1$ tel que $gf=-d(h)$.
Alors les complexes homotopiques \`a 3 termes sur $\A$ peuvent être identifi\'es avec
les dg-foncteurs de $\J$ \`a $\A$.
Nous identifions donc les objets dans $\mathcal H_{3t}(\A)$ avec les diagrammes dans $\A$
\begin{equation}\label{intro1:F}
\begin{tikzcd}
&A_0\ar[r,"f"]\ar[rr,bend right = 8ex,"h"swap]&A_1\ar[r,"j"]&A_2
\end{tikzcd}
\end{equation}
où $|f|=|j|=0$, $|h|=-1$ et $d(f)=0$, $d(j)=0$ et $d(h)=-jf$.
La cat\'egorie $\mathcal H_{3t}(\A)$ est \'equivalente \`a la sous-cat\'egorie pleine de $\rep(\Sq,\A)$ constitu\'ee des objets $X$ tels que $X_{10}$ est acyclique.
Un h-complexe \`a 3 termes~\ref{intro1:F} est {\em homotopiquement exact (resp. \`a gauche, \`a droite)} si l'objet correspondant dans $\rep(k\Sq, \A)$ est homotopiquement bicart\'esien (resp.~cart\'esien, cocart\'esien).
\begin{Lemme}[Lemme~\ref{lem:3termhcomplex}]
Un h-complexe \`a 3 termes~\ref{intro1:F} est homotopiquement exact \`a gauche si et seulement si les conditions suivantes sont satisfaites :
pour chaque $A\in\A$ et $n\leq 0$ et chaque paire de morphismes $(v,w)\in Z^{n}\A(A,A_1)\times \A^{n-1}(A,A_2)$ tels que $d(w)=-jv$, il existe un morphisme $u\in Z^n \A(A,A_0)$, unique \`a un 
cobord pr\`es, tel qu'il existe une paire de morphismes $(v',w')\in \A^{n-1}(A,A_1)\times \A^{n-2}(A,A_2)$ satisfaisant $v-fu=d(v')$, $w-hu=-d(w')-jv'$.  
\end{Lemme}
Il est clair par ce lemme que l'on peut faire de la chasse aux diagrammes en utilisant des suites 
homotopiquement exactes.
Nous utilisons cette description pour prouver certains lemmes sur les diagrammes au Chapitre~\ref{sec:diagramlemmas}. 
Il convient de noter que bien que parfois les r\'esultats homologiques dans les cat\'egories triangul\'ees soient applicables, le point subtil r\'eside dans la v\'erification de la compatibilit\'e des homotopies, cf.~Proposition~\ref{push}.

\begin{Definition} [D\'efinition~\ref{exactdgstructure}]
Une {\em structure exacte} sur $\A$ est une classe $\mathcal{S}\subseteq \mathcal{H}_{3t}(\A)$ stable par isomorphismes, constitu\'ee de suites homotopiquement exactes (appel\'ees {\em conflations})
\[
\begin{tikzcd}
A\ar[r, tail,"i"]\ar[rr,bend right=8ex,"h"swap]&B\ar[r, two heads, "p"]&C\\
\end{tikzcd} 
\]
où $i$ est appel\'e une {\em inflation} et $p$ est appel\'e une {\em d\'eflation}, telles que les axiomes suivants soient satisfaits :
\begin{itemize}
\item[Ex0] $\Id_{0}$ est une d\'eflation.
\item[{Ex}1] Les compositions de d\'eflations sont des d\'eflations.
\item[{Ex}2] \'Etant donn\'e une d\'eflation $p:B\rightarrow C$ et toute application $c: C'\rightarrow C$ dans $Z^0(\A)$, l'objet
\[
\begin{tikzcd}
&C'\ar[d,"c"]\\
B\ar[r,"p"swap]&C
\end{tikzcd}
\]
admet un pullback homotopique
\[
\begin{tikzcd} 
{B'}\ar[r,"{p'}"]\ar[d,"{b}"swap]\ar[rd,"s"blue,blue]&{C'}\ar[d,"{c}"]\\
{B}\ar[r,"{p}"swap]&{C}
\end{tikzcd}
\]
et ${p'}$ est \'egalement une d\'eflation. 
\item[$\Ex2^{op}$] \'Etant donn\'ee une inflation $i: A\rightarrow B$ et toute application $a:A\rightarrow A'$ dans $Z^0(\A)$, l'objet
\[
\begin{tikzcd}
A\ar[r,"i"]\ar[d,"a"swap]&B\\
A'&
\end{tikzcd}
\]
admet un pushout homotopique
\[
\begin{tikzcd}
{A}\ar[r,"{i}"]\ar[d,"{a}"swap]\ar[rd,"s"blue,blue]&{B}\ar[d,"{j}"]\\
{A'}\ar[r,"{i'}"swap]&{B'}
\end{tikzcd}
\]
et ${i'}$ est \'egalement une inflation.
\end{itemize}
Nous appelons $(\A,\mathcal {S})$ ou simplement $\A$ une {\em dg-cat\'egorie exacte}.
\end{Definition}

\begin{Theorem} [Th\'eor\`eme~\ref{fun}]
Soit $\A$ une petite dg-cat\'egorie exacte.
Si $\B$ est une petite dg-cat\'egorie connective, alors $\rep_{dg}(\B,\A)$ est canoniquement une dg-cat\'egorie exacte.
\end{Theorem}

\begin{Theorem} [Th\'eor\`eme~\ref{main}]\label{intro1:main}
Soit $\A$ une dg-cat\'egorie exacte.
Il existe un morphisme exact universel $F:\A\rightarrow \D^b_{dg}(\A)$ dans $\Hqe$ de $\A$ vers une dg-cat\'egorie pr\'etriangul\'ee $\D^b_{dg}(\A)$. Si de plus $\A$ est connective, ce morphisme v\'erifie les propri\'et\'es suivantes :
\begin{itemize}
\item[1)] Il induit une quasi-\'equivalence de $\tau_{\leq 0}{\A}$ vers $\tau_{\leq 0}\D'$ pour une 
sous-dg-cat\'egorie $\D'$ de $\D^b_{dg}(\A)$ stable par extensions.
\item[2)] Il induit une bijection naturelle $\mathbb E(C,A)\xrightarrow{\sim} \Ext^1_{\D^b(\A)}(FC,FA)$ 
pour chaque paire d'objets $C,A$ dans $H^0(\A)$ où $\D^b(\A)=H^0(\D^b_{dg}(\A))$.
\end{itemize}
Nous appelons $\D^b_{dg}(\A)$ la {\em dg-cat\'egorie d\'eriv\'ee born\'ee} de $\A$.
\end{Theorem}
Nous construisons $\D^b_{dg}(\A)$ comme le dg-quotient de l'enveloppe pr\'e-triangul\'ee $\pretr(\A)$ par une sous-dg-cat\'egorie pleine $\N$, qui g\'en\'eralise la dg-cat\'egorie des complexes acycliques d'une cat\'egorie exacte de Quillen. Nous en d\'eduisons que si $\A$ est une cat\'egorie exacte de Quillen, alors $\D^b_{dg}(\A)$ est quasi-\'equivalente au dg-enrichissement canonique de la cat\'egorie d\'eriv\'ee born\'ee de $\A$, d'où le nom et la notation.

Le plongement universel $\A\rightarrow \D^b_{dg}(\A)$ permet de d\'emontrer des lemmes diagrammatiques pour les dg-cat\'egories exactes qui peuvent être difficile \`a prouver directement. 
Par exemple, soit $\J$ la dg-cat\'egorie de (\ref{quiv1:3termcofibrant}) et soit $F: \J\rightarrow \rep_{dg}(\J,\A)$ un dg-foncteur tel que $F(0)$ et $F(2)$ sont toutes les deux des conflations dans $\A$. Soit 
$F_i: \J\rightarrow \A$ la composition de $F$ avec $\rep_{dg}(\J,\A)\rightarrow \A$, induite par l'inclusion $i:k\rightarrow \J$ pour $i=0$, $1$ et $2$. Si tous les $F_i$ sont des conflations dans $\A$, alors $F(1)$ est \'egalement une conflation dans $\A$. Ce `lemme diagrammatique' est l'ingr\'edient principal de la d\'emonstration du Lemme~\ref{lem:internalhom}.

Le foncteur dg-nerf $N_{dg}$ envoie chaque petite dg-cat\'egorie exacte sur une $\infty$-cat\'egorie exacte. De plus, nous avons le th\'eor\`eme suivant.
\begin{Theorem}[Th\'eor\`eme~\ref{nerve}]
\label{intro1:nerve} Soit $\A$ une dg-cat\'egorie telle que $H^0(\A)$ est une cat\'egorie additive. Alors il existe une bijection entre la classe des structures exactes sur la dg-cat\'egorie $\A$ et la classe des structures exactes sur la $\infty$-cat\'egorie $N_{dg}(\A)$. De plus, si $\A$ est une dg-cat\'egorie pr\'e-triangul\'ee munie de la structure exacte maximale (donn\'ee par toutes les suites exactes homotopiques), alors son dg-nerf est une $\infty$-cat\'egorie stable munie de la structure exacte maximale (où chaque morphisme est \`a la fois une inflation et une d\'eflation).
\end{Theorem}

Cela \'etend en particulier un th\'eor\`eme de Faonte \cite{Faonte17b}, qui a montr\'e que si une dg-cat\'egorie est pr\'e-triangul\'ee, alors son dg-nerf est une $\infty$-cat\'egorie stable.

Soit $(\C,\mathbb E,\mathfrak s)$ une petite cat\'egorie extriangul\'ee. Pour $n\geq 0$, Gorsky--Nakaoka--Palu~\cite{GorskyNakaokaPalu21} ont d\'efini le bimodule des {\em extensions sup\'erieures} $\mathbb E^{n}(?,-)$ sur $\C$ comme la puissance tensorielle $n$-i\`eme sur $\C$ de $\mathbb E$. Un {\em $\delta$-foncteur} est un triplet $(T,\epsilon,\eta)$ où $T=(T^i)_{i\geq 0}$ est une s\'equence de $\C$-$\C$-bimodules et $\epsilon$ et $\eta$ sont des collections de morphismes de connexion satisfaisant certaines conditions, cf.~D\'efinition~\ref{def:deltafunctor}. Il est clair \`a partir de la d\'efinition que $(\mathbb E^n)_{n\geq 0}$ avec les morphismes de connexion associ\'es est un $\delta$-foncteur. En utilisant un r\'esultat g\'en\'eral concernant les $\delta$-foncteurs avec des bimodules $T^i$ faiblement effaçables pour $i>0$, cf.~Corollaire~\ref{cor:effaceablebimodule}, nous avons le r\'esultat suivant.
\begin{proposition}[Proposition~\ref{higher}]\label{intro:higher}
Soit $F:\A\rightarrow \D^{b}_{dg}(\A)$ le morphisme exact universel de $\A$ vers une cat\'egorie pr\'etriangul\'ee. Alors nous avons un isomorphisme canonique de $\delta$-foncteurs $\alpha:\mathbb E^n(?,-)\xrightarrow{\sim} \Ext^n_{\D^b(\A)}(?,-)$.
\end{proposition}
Un r\'esultat similaire pour les $\infty$-cat\'egories exactes est \'egalement valable. En cons\'equence, pour une dg-cat\'egorie exacte $\A$, l'enveloppe stable de $N_{dg}(\A)$ est canoniquement \'equivalente \`a $N_{dg}(\D^b_{dg}(\A))$.

Pour une cat\'egorie extriangul\'ee $\C$ avec une sous-cat\'egorie $\P$ constitu\'ee d'objets projectifs-injectifs (non n\'ecessairement tous) dans $\C$, Nakaoka--Palu ont montr\'e que le quotient additif $\C/\P$ a la structure d'une cat\'egorie extriangul\'ee, induite \`a partir de celle de $\C$, cf.~\cite[Proposition 3.30]{NakaokaPalu19}. Le th\'eor\`eme suivant am\'eliore leur r\'esultat.
\begin{Theorem}[Th\'eor\`eme~\ref{quot}]\label{intro:quot}
Soit $\A$ une dg-cat\'egorie exacte et $\P$ une sous-cat\'egorie pleine de $\A$ constitu\'ee d'objets projectifs-injectifs dans $\A$. Soit $\mathcal S_{dg}$ l'enrichissment diff\'erentiel gradu\'e canonique de $\D^b(\A)/\tr(\P)$. Alors le dg-quotient $\A/\P$ poss\`ede une structure dg exacte canonique induite par celle de $\A$ et sa dg-cat\'egorie d\'eriv\'ee est quasi-\'equivalente \`a $\mathcal S_{dg}$.
\end{Theorem}

Soient $\A$ et $\A'$ des dg-cat\'egories exactes connectives et $\C$ une dg-cat\'egorie exacte arbitraire. Un morphisme $\mu: \A\otimes \A'\rightarrow \C$ dans $\Hqe$ est {\em biexact} si, pour tous les objets $A\in \A$ et $A'\in \A'$, les morphismes induits
\[
\mu_{A,-}:\A'\rightarrow \C,\;\;\mu_{-,A'}:\A\rightarrow \C
\]
sont tous les deux des morphismes exacts. 
Soit $\A\boxtimes \A'$ la $\tau_{\leq 0 }$-truncation de la fermeture par extensions de l'image quasi-essentielle du morphisme dans $\Hqe$
\[
F: \A\otimes \A'\rightarrow \D^b_{dg}(\A)\otimes\D^b_{dg}(\A')\rightarrow \pretr(\D^b_{dg}(\A)\otimes\D^b_{dg}(\A')).
\]
\begin{proposition}[Proposition~\ref{prop:universalbilinear}]\label{intro:universalbilinear}
La morphisme naturel $\A\otimes \A'\rightarrow \A\boxtimes \A'$ est le morphisme biexact universel dans $\Hqe$ de $\A\otimes \A'$ vers une dg-cat\'egorie exacte.
\end{proposition}

Soient $\B$ et $\C$ des dg-cat\'egories exactes connectives.
Soit $\rep_{dg}^{ex}(\B,\C)$ la sous-dg-cat\'egorie pleine $\rep_{dg}(\B,\C)$ constitu\'ee des morphismes exacts.
Il est d\'emontr\'e au Lemme~\ref{lem:internalhom} que $\rep_{dg}^{ex}(\B,\C)$ est stable par extensions dans $\rep_{dg}(\B,\C)$ et h\'erite donc d'une structure dg exacte canonique.
\begin{Corollaire}[Corollaire~\ref{cor:exactadjunction}]\label{intro:exactadjunction}
Soient $\A$ et $\B$ de petites dg-cat\'egories exactes connectives. Soit $\C$ une petite dg-cat\'egorie exacte.
Alors nous avons la bijection naturelle suivante
\[
\Hqe_{ex}(\A\boxtimes \B,\C)\iso \Hqe_{ex}(\A,\rep_{dg}^{ex}(\B,\C)).
\]
\end{Corollaire}

Soit $(\C,\mathbb E,\mathfrak s)$ une petite cat\'egorie extriangul\'ee.
Le {\em d\'efaut} d'une $\mathfrak s$-conflation 
\[
A\xrightarrow{x}B\xrightarrow{y}C
\] 
est d\'efini comme le conoyau du morphisme de $\C$-modules \`a droite $\C(?,B)\xrightarrow{(?,y)}\C(?,C)$.
La sous-cat\'egorie pleine de la cat\'egorie des modules $\Mod\C$ dont les objets sont les d\'efauts est not\'ee 
$\Def\mathbb E$. Elle a \'et\'e introduite dans cette g\'en\'eralit\'e par Ogawa dans \cite{Ogawa21} 
et \'etudi\'ee plus en d\'etail par Enomoto \cite{Enomoto21}, qui a montr\'e que $\Def \mathbb E$ est
une sous-cat\'egorie de Serre de la cat\'egorie ab\'elienne des $\C$-modules coh\'erents et que l'application
qui associe \`a un sous-bifoncteur ferm\'e de $\mathbb E$ sa cat\'egorie de d\'efauts est un isomorphisme de
posets sur le poset des sous-cat\'egories de Serre de $\Def \mathbb E$.
Si $\A$ est une dg-cat\'egorie exacte connective, alors $\Def H^0(\A)$ est \'equivalente
au cœur d'une $t$-structure sur la cat\'egorie $\N$ des complexes `acycliques', cf.
la remarque apr\`es le Th\'eor\`eme~\ref{intro1:main}.

\begin{Theorem}[Th\'eor\`eme~\ref{bijectionstructures}]\label{intro:bijectionstructures}
Soit $(\A,\mathcal S)$ une dg-cat\'egorie exacte et $(H^0(\A),\mathbb E,\mathfrak s)$ la cat\'egorie extriangul\'ee alg\'ebrique associ\'ee.
Les posets suivants (où les relations d'ordre sont donn\'ees par l'inclusion) sont canoniquement isomorphes :
\begin{itemize}
\item[(1)]Le poset des sous-structures exactes de $(\A,\mathcal S)$.
\item[(2)]Le poset des sous-bifoncteurs ferm\'es de $\mathbb E$.
\item[(3)]Le poset des sous-cat\'egories de Serre de $\Def\mathbb E$.
\end{itemize}
\end{Theorem}
Ici, l'isomorphisme entre les posets (2) et (3) d\'ecoule du th\'eor\`eme d'Enomoto \cite{Enomoto21}.

Dans~\cite{Rump11}, Rump a d\'emontr\'e que toute cat\'egorie additive poss\`ede une plus grande structure exacte, ce qui g\'en\'eralise les r\'esultats de~\cite{SiegWegner11,Crivei12}. 
Le th\'eor\`eme suivant g\'en\'eralise le r\'esultat de Rump aux dg-cat\'egories additives.
\begin{Theorem}[Th\'eor\`eme~\ref{thm:maximalexactdgstructure}]\label{intro:maximalstructure}
Chaque dg-cat\'egorie additive $\A$ admet une plus grande structure exacte  (c'est-\`a-dire maximale unique).
\end{Theorem}
Il n'est pas clair comment formuler, et encore moins comment d\'emontrer, un \'enonc\'e 
analogue pour les structures extriangul\'ees.
\newpage

\section{Introduction (English version)} 

\subsection{Overview} 
In this thesis, we introduce and study the notion of {\em exact dg category}.
We illustrate the position of this class of categories among the other types of 
categories that we will consider in this thesis in the following diagram
\[
\begin{tikzcd}
&\{\text{pretri. dg cat.}\}\ar[rd,red]\ar[rr,"N_{dg}"{description,near start},hook]\ar[dd,hook]&&\{\text{stable $\infty$-cat.}\}\ar[dd,hook]\ar[ld,red]\\
&&\{\text{tri. cat.}\}\ar[dd,hook]&\\
\{\text{Quillen ex. cat.}\}\ar[r,hook]\ar[rrd,hook]&\{\text{ex. dg cat.}\}\ar[rr,hook,"N_{dg}"{description,near start}]\ar[rd,red]&&\{\text{Barwick ex. $\infty$-cat.}\}\ar[ld,red]\\
&&\{\text{extri. cat.}\}&
\end{tikzcd}
\]
Here, the black arrows denote inclusions of classes and the red arrows send a dg category (respectively an
$\infty$-category)  $\A$ to $H^0(\A)$ (respectively $h\A$). As we see in the diagram, the notion of an exact dg category is
\begin{itemize}
\item[1)] a dg version of Barwick's \cite{Barwick15} notion of exact $\infty$-category;
\item[2)] a simultaneous generalization of the notions of exact category in the sense of Quillen \cite{Quillen73} and of 
pretriangulated dg category in the sense of Bondal--Kapranov \cite{BondalKapranov90} and
\item[3)] a dg enhancement of the notion of extriangulated category in the sense of Nakaoka--Palu 
\cite{NakaokaPalu19} in analogy with Bondal--Kapranov's dg enhancement of the notion of triangulated category.
\end{itemize}
We also have characterizations of some of the subclasses occuring in the diagram:
\begin{itemize}
\item[1)] an extriangulated category is a Quillen exact category if and only if any inflation is monomorphic and any deflation is epimorphic, cf.~\cite[Corollary 3.18]{NakaokaPalu19};
\item[2)] an extriangulated category is a triangulated category if and only it is a Frobenius extriangulated category whose
projective-injectives are the zero objects, cf.~\cite[Corollary 7.6]{NakaokaPalu19};
\item[3)] an exact dg category is quasi-equivalent to a Quillen exact category if and only if the underlying dg category has its
homology concentrated in degree zero;
\item[4)] a connective exact dg category is quasi-equivalent to the $\tau_{\leq 0}$-truncation of a pretriangulated dg category 
if and only if its zeroth homology category is canonically triangulated, cf.~Example~\ref{exm:exactdgpretriangulated}.
\end{itemize}
More details on these classes and the historical developments are given in the following subsections. 
Let us emphasize that our notion of exact dg category is completely different from Positselski's notion
of exact DG-category \cite{Positselski21}. For example, an exact structure in our sense can be
transported along a quasi-equivalence, cf.  Remark~\ref{truncationexactdgstructure} b), which is not at all the case for exact structures in the sense of Positselski. His principal motivation is to axiomatise
situations where we have a {\em pretriangulated} dg category $\A$ whose category $Z^0(\A)$
is moreover endowed with a Quillen exact structure (for example the category of all 
complexes with components in a given Quillen exact category). By contrast, we aim at axiomatising
a class of {\em not necessarily pretriangulated} dg categories endowed with additional
structure.

\subsection{Exact categories}
The notion of exact category was introduced by Quillen in \cite{Quillen73} in order to define
his higher algebraic $K$-theory (idempotent complete exact categories were introduced 15 years earlier
by Heller \cite{Heller58}).
Loosely speaking, an exact category is an additive category $\E$ together with a class of kernel-cokernel pairs called conflations
\[
\begin{tikzcd}
L\ar[r,tail,"i"]&M\ar[r,two heads,"p"]&N
\end{tikzcd}
\]
where $i$ is called an inflation and $p$ is called a deflation, which satisfies certain axioms. 
The list of axioms has been reduced to ``the" minimum by Keller in \cite[Appendix A.1]{Keller90}.  
Notably, Quillen's obscure axiom is redundant and can be deduced from the other axioms, cf.~also in \cite[p.525, Corollary]{Yoneda60}. 
For a comprehensive treatment of Quillen's exact category, we refer to B\"uhler's survey \cite{Buhler10}. For a study of higher extension groups in an exact category, we refer to the lecture notes~\cite{FrerickSieg10}.
Quillen stated and Thomason--Trobaugh \cite{ThomasonTrobaugh90} proved 
that for an essentially small exact category $\E$, there is an exact equivalence $G:\E\rightarrow \C$ onto an extension-closed full subcategory $\C$ of an abelian category such that the functor $G$ is fully exact, i.e.~a sequence in $\E$ is a conflation if and only if its image in $\C$ is a short exact sequence, cf.~\cite[A.2 Proposition]{Keller90}.
Quillen used exact categories as a framework for his remarkable $Q$-construction, which is related to the bicategory of spans introduced by B\'enabou, cf.~\cite[2.6]{Benabou67}. The category $Q(\E)$ of an exact category consists of objects in $\E$ and the morphisms from $M$ to $M'$ are given by equivalence class of roofs
\[
\begin{tikzcd}
M&N\ar[l,two heads,"p"swap]\ar[r,tail,"i"]&M'
\end{tikzcd}
\]
where $p$ is a deflation and $i$ is an inflation and two roofs are considered equivalent if there is an isomorphism between the roofs which restricts to the identities of $M$ and of $M'$. 
The compositions in $Q(\E)$ are given by pullbacks.
In fact, Quillen's list of axioms is exactly what one needs when forming the $Q$-construction.
The generalization from abelian categories to exact categories allows for more examples, for example the category of free modules over a ring, the category of vector bundles over a scheme, the categories of locally free modules over a generalized preprojective algebra \cite{GeissLeclercSchroer17}, \ldots\ .

\subsection{From triangulated categories to dg categories}
On the other hand, in order to prove and formulate his extensions of Serre's duality theorem \cite{Grothendieck60}, 
Grothendieck invented the notion of derived categories and, 
together with his student Verdier, introduced the notion of triangulated category \cite{Verdier96}. 
At almost the same time, Puppe also found the axioms (except the octahedral axiom) for
triangulated categories \cite{Puppe62}. 
For a long time, triangulated categories were perceived to be too coarse to allow the development of
a substantial theory, principally due to the non-functoriality of the mapping cone construction. 
This view has changed drastically \cite{Neeman01,Neeman22} 
but still many operations on derived categories cannot be performed in the framework of triangulated categories, 
notably tensor products of triangulated categories and functor categories with target triangulated categories. 
A viable solution to this problem is provided by the theory of dg categories. 
The notion of dg category already appeared in Kelly's work \cite{Kelly65}.
In the nineties, Bondal--Kapranov \cite{BondalKapranov90} observed that 
the morphism spaces of the triangulated categories appearing naturally in algebra 
are the zeroth cohomology groups of certain complexes and 
it is a principle (popularized by Grothendieck) in homological algebra to consider the complex itself 
rather than its cohomology. 
Therefore, they proposed to enhance triangulated categories using dg categories, 
requiring that the  triangulated structure should be compatible with the  dg structure.
Formalising this idea, they defined the property of a dg category to be pretriangulated and 
observed that each dg category can freely generate a pretriangulated dg category. 
Soon after, Keller published his seminal paper \cite{Keller94} 
discussing homotopy categories and unbounded derived categories of dg categories. 
Drinfeld \cite{Drinfeld04} gave an elegant construction of the dg quotient, whose existence
goes back to Keller's \cite{Keller99}, and which enhances Verdier's triangle quotient. 

\subsection{From extriangulated categories to exact dg categories} 
The notion of an extriangulated category was introduced by Nakaoka--Palu in~\cite{NakaokaPalu19}. 
It is a simultaneous generalization of Quillen's notion of exact category and 
Grothendieck--Verdier's notion of a triangulated category. 
Their aim was to give a convenient setup for writing down proofs 
which apply to both exact categories and triangulated categories, and more generally to
extension-closed subcategories of triangulated categories. 
The theory of extriangulated categories has developed a lot since their introduction in 2019 
and many notions and constructions have been generalized to this setting (or the more general 
setting of $n$-exangulated categories \cite{HerschendLiuNakaoka21}). Such notions and
constructions include
\begin{itemize}
\item  cotorsion pairs and their hearts and cohearts, and silting and $n\mbox{-}$cluster tilting subcategories, cf.~\cite{NakaokaPalu19, LiuNakaoka19,ChangZhouZhu19, AdachiTsukamoto22,HuertaMendozaSaenzSantiago22,AdachiTsukamoto23,ZhaoZhuZhuang21, ZhuZhuang20,LiuZhouZhouZhu21};
\item  mutations and cluster-tilting theory, cf.~\cite{ ChangZhouZhu19, Pressland23, ZhouZhu18, LiuZhou20, LiuZhou20a, LiuZhou21, JorgensenShah22,GorskyNakaokaPalu23};
\item idempotent completions, cf.~\cite{WangWeiZhangZhao23, Dixy22,KlapprothMsapatoShah22};
\item  Auslander--Reiten theory, cf.~\cite{IyamaNakaokaPalu18, TanWangZhao23};
\item Auslander's formula for the functor category of an extriangulated category, cf.~\cite{Ogawa21};
\item substructures of an extriangulated category, cf.~\cite{HerschendLiuNakaoka21,Enomoto21, Sakai23};
\item recollement of extriangulated categories, cf.~\cite{WangWeiZhang22, HeHuZhou22, HuZHou21, MaZhaoZhuang23};
\item Grothendieck groups of extriangulated categories, cf.~\cite{ZhuZhuang21, Haugland21, EnomotoSaito22};
\item Hall algebras of extriangulated categories, cf.~\cite{WangWeiZhang22a,FangGorsky22};
\item  localizations of extriangulated categories, cf.~\cite{NakaokaOgawaSakai22, Ogawa22, HeHeZhou22};
\item \ldots
\end{itemize}
 The study of the class of substructures (=closed subbifunctors) of an exact category already appears in \cite{ButlerHorrocks61}, cf.~also~\cite{AuslanderSolberg93a, DraxlerReitenSmaloSolberg99, BrustleHassounLangfordRoy20, FangGorsky22, BaillargeonBrustleGorskyHassoun22}. 
 Notice that the class of triangulated substructures of a triangulated category only consists of the trivial one.
  It would be interesting to find a common generalization of one-sided exact categories and suspended categories, cf.~\cite{KellerVossieck87,BazzoniCrivei13}.

 In analogy with the notion of algebraic triangulated category \cite{Keller06d}, we define an {\em algebraic extriangulated category} 
 as follows.
 \begin{definition-proposition}[Definition-Proposition~\ref{algebraic}]\label{def:algebraic}Let $\C$ be an extriangulated category. The following statements are equivalent:
\begin{itemize}
\item[1)] $\C$ is equivalent, as an extriangulated category, to a full extension-closed subcategory of an algebraic 
triangulated category.
\item[2)] $\C$ is equivalent, as an extriangulated category, to $\B/(\mathcal P_0)$ for a Quillen exact category $\B$ and a class $\mathcal P_0$ of projective-injective objects.
\item[3)] $\C$ is equivalent, as an extriangulated category, to $H^0(\A)$ for an {\em exact dg category} $\A$.
\end{itemize}
If one of the above equivalent conditions holds, then $\C$ is called an {\em algebraic extriangulated category}.
\end{definition-proposition}
From the Definition-Proposition, it is clear that the class of algebraic extriangulated categories is stable under
passage to
\begin{itemize}
\item[a)] extension-closed subcategories;
\item[b)] quotients by ideals generated by identities of projective-injective objects;
\item[c)] extriangulated substructures, cf.~Corollary~\ref{cor:exactsubstructure}.
\end{itemize}
We also expect that it is stable under passage to localizations in the sense of
Nakaoka--Ogawa--Sakai \cite{NakaokaOgawaSakai22}.
It is also clear that if an algebraic extriangulated category is triangulated, then it is algebraic as a triangulated category.

From part 1) of the Definition-Proposition, it is clear that an algebraic extriangulated category $\C$ has an unbounded bivariant $\delta$-functor in the sense of~\cite[Definition 4.5]{GorskyNakaokaPalu21}. 
In particular, this provides an additive bifunctor $\mathbb E^{-1}$ which is crucial in the study of the poset of 
{\em $s$-torsion pairs} in an extriangulated category in~\cite{AdachiEnomotoTsukamoto23}.

\subsection{Classes of examples of extriangulated and exact dg categories}
The following example is taken from ~\cite{Jin20}.
Let $k$ be a field. 
A dg $k$-algebra $A$ is {\em proper} if $\sum_{i\in \mathbb Z}\dim_{k}{H^i(A)<\infty}$. 
It is {\em Gorenstein} if the thick subcategory $\per(A)$ of the derived category $\D(A)$ 
generated by $A$ coincides with the thick subcategory $\thick(DA)$ generated by $DA$, where $D=\Hom_k(-,k)$ is the $k\mbox{-dual}$.
Let $A$ be a connective proper Gorenstein dg algebra.
A dg $A\mbox{-}$module is {\em perfectly valued} if its total cohomology is finite-dimensional and we denote by $\pvd(A)$ the triangulated category of perfectly valued dg $A\mbox{-}$modules.
A dg $A\mbox{-}$module $M$ in $\pvd(A)$ is {\em {Cohen--Macaulay}} if $H^i(M)=0$ and $\Hom_{\D(A)}(M,\Sigma^i A)=0$ for $i>0$.
Let $\CM A$ be the subcategory of $\pvd(A)$ consisting of Cohen-Macaulay dg $A\mbox{-}$modules.
By \cite[Theorem 2.4]{Jin20}, the category $\CM A$ is an $\Ext\mbox{-}$finite Frobenius extriangulated category with $\proj(\CM A)=\add A$. 
By Proposition~\ref{prop:dgsingularitycategory}, the stable category $\CM(A)/[\proj(\CM A)]$ is equivalent to the singularity category
$\pvd(A)/\per(A)$.
We will see in Example~\ref{exm:CMdgmodule} that if $\CM_{dg} A$ denotes the dg enhancement of $\CM A$ inherited from $\pvd_{dg}(A)$,
then the inclusion of $\CM_{dg}(A)$ into $\pvd_{dg}(A)$ induces an equivalence from the derived category of 
$\CM_{dg}(A)$ onto $\pvd_{dg}(A)$. Moreover, this equivalence induces an equivalence from the singularity
category of $\CM_{dg}(A)$ onto that of $A$ in complete analogy with the situation for a finite-dimensional
Iwanaga--Gorenstein algebra concentrated in degree $0$.

One can easily generalize the results of Example~\ref{exm:CMdgmodule} to the setup of silting reduction in the sense of 
Iyama--Yang~\cite{IyamaYang18} (details will appear elsewhere): 
Let $\T$ be an algebraic triangulated category and $\P$ a presilting subcategory of $\T$, i.e.~a full subcategory such that $\Hom_{\T}(\P,\Sigma^{i}\P)=0$ for $i>0$. Let $\mathcal S$ be the thick subcategory of $\T$ generated by $\P$.
Let $\V$ and $\W$ be the following full subcategories of $\T$ respectively 
\[
\V=\{X|\Hom_{\T}(X,\Sigma^{i}\P)=0 \text{ for $i>0$}\} \quad \W=\{X|\Hom_{\T}(\P,\Sigma^{i}X)=0 \text{ for $i>0$}\}.
\]
Put $\Z=\V\cap \W$.
Then $\Z$ is extension-closed in $\T$ and hence is canonically extriangulated. Notice that objects in $\P$ are projective-injective in $\Z$.
Assume that $\Z$ is a Frobenius extriangulated category whose projective-injective objects are the objects in $\P$, i.e.~$(\Z,\Z)$ is a $\P$-mutation pair~\cite[Definition 2.5]{IyamaYoshino08}.
For example, this holds when $\P$ is covariantly finite in $\V$ and contravariantly finite in $\W$. 
In this case, the additive quotient $\Z/[\P]$ is canonically a triangulated category.
Assume further that the $\Z$ generates $\T$ as a triangulated category. 
 This holds if and only if for each object $X\in \T$, we have $\Sigma^{-l}X\in \V$ and $\Sigma^{l}X\in \W$ when $l\gg 0$. 
If we replace the above categories with their dg enhancements, then the \mbox{$\tau_{\leq 0}$-truncation} of $\Z$ carries a canonical exact dg structure whose bounded dg derived category is $\T$ and whose dg singularity category is $\T/\mathcal S$.
Similar  considerations apply to the special case of~\cite[1.2]{IyamaYang20} where the torsion pair $\mathcal S=\X\perp \Y$ is a 
t-structure, cf.~Example~\ref{exm:Iyama--Yang}.

Exact dg categories also appear naturally in the additive categorification of cluster algebras with coefficients.
Such cluster algebras are associated with {\em ice quivers}, i.e.~quivers endowed with a (not necessarily full)
{\em frozen} subquiver. For the categorification, one endows the ice quiver with a non degenerate potential.
In his thesis \cite{Wu21, Wu23a, Wu23b}, Yilin Wu has constructed a Frobenius extriangulated
category, the {\em Higgs category}, for each Jacobi-finite ice quiver with potential.
His construction generalizes the categories of modules over preprojective algebras 
used with a great success by Geiss--Leclerc--Schr\"oer \cite{GeissLeclercSchroeer11b,GeissLeclercSchroeer08b,GeissLeclercSchroeer06,GeissLeclercSchroeer13}.
The Higgs category is an extension closed subcategory of a canonical ambient algebraic triangulated
category (the relative cluster category) and thus carries a canonical dg enhancement. Let us make
this more precise (we use the notations of \cite{Wu23a}): Let $(Q,F,W)$ be a Jacobi-finite ice quiver with potential. 
Denote by $\Gamma_{rel}$ the relative Ginzburg dg algebra $\Gamma_{rel}(Q,F,W)$. 
Let $e=\sum_{i\in F}e_i$ be the idempotent associated with all frozen vertices. 
Let $\pvd_{e}(\Gamma_{rel})$ be the full subcategory of $\pvd(\Gamma_{rel})$ of the dg $\Gamma_{rel}$-modules whose restriction to frozen vertices is acyclic.
Then the relative cluster category $\C=\C(Q,F,W)$ associated to $(Q,F,W)$ is defined as the Verdier quotient of triangulated categories
\[
\per(\Gamma_{rel})/\pvd_{e}(\Gamma_{rel}).
\]
Let $\mathcal P=\add(e\Gamma_{rel})$.
The {\em relative fundamental domain $\mathcal F^{rel}_{\Gamma_{rel}}=\mathcal F^{rel}$} associated to $(Q,F,W)$ is defined as the following subcategory of $\per(\Gamma_{rel})$
\[
\mathcal F^{rel}{\coloneqq}\{\Cone(X_1\xrightarrow{f}X_0) | X_i\in\add(\Gamma_{rel}) \text{ and } \Hom(f,I) \text{ is surjective}, \forall I\in \mathcal P\}.
\]
Since $\add{\Gamma_{rel}}$ is presilting in $\per\Gamma_{rel}$, the full subcategory $\Z'=\add\Gamma_{rel}\ast\add\Sigma\Gamma_{rel}$ is extension-closed in $\per\Gamma_{rel}$ and thus inherits a canonical extriangulated structure.
Then $\F^{rel}$ can also be described as the extension-closed subcategory of $\Z'$ consisting of objects $X$ which are relatively $\mathbb E$-projective with respect to $\P$, i.e.~the objects $X$ such that $\Ext^1_{\T}(X,\P)=0$.
Let $\pi^{rel}:\per(\Gamma_{rel})\rightarrow \C$ be the canonical quotient functor. 
The {\em Higgs category} $\mathcal H$ is the image of $\mathcal F^{rel}$ in $\C$ under the quotient functor $\pi^{rel}$. 
By \cite[Proposition 5.39, Theorem 5.46]{Wu23a}, the Higgs category $\mathcal H$ is extension-closed in $\C$ and,
endowed with the inherited extriangulated structure, becomes a Frobenius extriangulated category whose subcategory
of projective-injectives is $\P$.
We replace the above categories by their dg enhancements. 
The canonical dg functor $\tau_{\leq 0}\F^{rel}\rightarrow \tau_{\leq 0}\mathcal H$ is a quasi-equivalence and the exact dg structure on $\tau_{\leq 0}\F^{rel}$ is a substructure of that on $\tau_{\leq 0}\mathcal H$.
By Theorem~\ref{intro:maximalstructure}, the underlying dg category $\tau_{\leq 0}\F^{rel}$ has the greatest exact structure.
If we take the exact substructure of the greatest exact structure which makes the objects in $\P$ projective-injective, then we recover the exact dg category $\tau_{\leq 0}\mathcal H$.
We will see in Section~\ref{subsection:higgscategories} that the bounded derived category of $\tau_{\leq 0}\F^{rel}$ is $\per\Gamma_{rel}$. 
The bounded derived category of $\tau_{\leq 0}\mathcal H$ is $\C$ and the singularity category of $\tau_{\leq 0}\mathcal H$ is the cluster category associated with the quiver with potential $(\overline Q,\overline W)$ where $\overline{Q}$ is the quiver obtained from $Q$ by deleting all frozen vertices and all arrows incident with frozen vertices and $\overline W$ is the potential on $\overline Q$ obtained by deleting all cycles passing through vertices of $F$ in $W$.

\subsection{Stable $\infty$-categories and exact $\infty$-categories}
The notion of simplicial set was introduced in~\cite{EilenbergZilber50} as a refinement of the notion of simplicial complex.
Boardman--Vogt~\cite[Definition 4.8]{BoardmanVogt73} used the {\em restricted Kan condition} to define the notion of {\em weak Kan complex} (= {\em quasicategory} = {\em $\infty$-category}) in the category of simplicial sets.
Each small category can be seen as an $\infty$-category via the {\em nerve} functor $N$.
To each simplicial set $X$, one can associate its {\em homotopy category} $h(X)$. 
It has a simple description when $X$ is an $\infty$-category, cf.~\cite{BoardmanVogt73}.
The theory of quasicategories has been developed substantially by Joyal and Lurie among others, cf.~\cite{Lurie09,LurieHA,Joyal08,Cisinski19}.
Lurie (\cite{Lurie06}, cf.~also~\cite{LurieHA} and \cite{ToenVezzosi04a}) introduced the notion of {\em stable $\infty$-category}, based on the techniques studied in his book~\cite{Lurie09}, and showed that the homotopy category of a stable $\infty$-category is canonically triangulated.
The triangulated categories thus obtained are called {\em topological}.
The {\em dg nerve} functor $N_{dg}$ was introduced in \cite{LurieHA} as a key tool linking dg categories with $\infty$-categories.
He showed that there is a Quillen adjunction between the {\em Joyal model structure} for $\infty$-categories and the {Dwyer--Kan model structure} for dg categories.
Faonte showed that the dg nerve of a pretriangulated dg category is a stable $\infty$-category in
his paper \cite{Faonte17b}, where he generalized the dg nerve functor to the simplicial nerve functor 
linking $A_{\infty}$-categories to $\infty$-categories.
Barwick introduced the notion of exact $\infty$-category in~\cite{Barwick15} as a homotopy-theoretic
generalization of Quillen's notion of exact category.
Quillen's $Q$-construction has also been generalized to this setting, cf.~\cite{Barwick13}.
In~\cite{NakaokaPalu20}, it was shown that the homotopy category of an exact $\infty$-category $\E$ carries a canonical extriangulated structure. 
The extriangulated categories obtained are called {\em topological}.
Klemenc~\cite{Klemenc22} constructed, for each small exact $\infty$-category $\E$, the {\em stable hull} $\D^b_{\infty}(\E)$ together with a fully exact fully faithful functor $\eta_{\E}:\E\rightarrow \D^b_{\infty}(\E)$ whose essential image is closed under extensions.
Since the homotopy category of  $ \D^b_{\infty}(\E)$ is triangulated (by Lurie's theorem), this gives another
proof of the result by Nakaoka--Palu.
Compared with dg categories, the morphism spaces in quasicategories are not easy to handle.
This is one of our motivations to introduce the notion of {\em exact dg category}.

\subsection{Exact dg categories}
Let $\A$ be a dg category. For simplicity, we assume that $\A$ is connective and has cofibrant Hom complexes and that $Z^0(\A)$ is additive. We use the notations introduced in Section~\ref{subsection:notations}.

The {\em square} category $\mathrm{Sq}$ is the path category of the quiver 
\[
\begin{tikzcd}
00\ar[r,"f"]\ar[d,"g"swap] & 01\ar[d,"j"] \\
10\ar[r,"k"swap] & 11
\end{tikzcd}
\]
with the commutativity relation $jf\sim kg$.
We identify objects $X\in\rep(k\mathrm{Sq},\A)$ with commutative squares
\begin{equation}
\begin{tikzcd}\label{D(Sq)}
X_{00}\ar[r,"f"]\ar[d,"g"swap]&X_{01}\ar[d,"j"]\\
X_{10}\ar[r,"k"swap]&X_{11}
\end{tikzcd}
\end{equation}
in $\C(\A)$ where $jf=kg$. 
\begin{definition}[Definition~\ref{maindef}]
An object $X\in \rep(k\Sq,\A)$ is a \text{\em{{homotopy cartesian square}}} with respect to $\A$ if the canonical map $X_{00}\rightarrow \Sigma^{-1}\mathrm{Cone}((-j,k))$ induces an isomorphism   
\[
\tau_{\leq 0}\RHom_\A(A^{\wedge},X_{00})\rightarrow \tau_{\leq 0}\RHom_\A(A^{\wedge}, \Sigma^{-1}\mathrm{Cone}((-j,k)))
\]
in $\D(k)$ for each $A$ in $\A$.
\end{definition}
Dually one can define the notion of {\em homotopy cocartesian square} and then define the notion of {\em homotopy bicartesian square}.
We often take a cofibrant resolution $\overline{\Sq}\rightarrow \Sq$ of $\Sq$, cf.~Section~\ref{res}, and via the equivalence of categories $\rep(k\overline{\Sq},\A)\rightarrow\rep(k\Sq,\A)$ we identify objects on both sides. 
One advantage is that we may assume that 
all terms in an object $X\in \rep(k\overline{\Sq},\A)$ are representable (and not just quasi-representable) dg modules.

Let $\I$ be the dg $k$-path category of the following quiver with relations
\begin{equation}\label{quiv:3term}
\begin{tikzcd}
0\ar[r,"f",""{swap, name=1}]&1\ar[r,"g",""{swap,name=2}]&2\ar[r,from=1,to=2,dashed,no head,swap,bend right =8ex]
\end{tikzcd}
\end{equation}
where $f$ and $g$ are closed of degree 0 and $gf=0$.
The category $\mathcal H_{3t}(\A)$ of 3-term homotopy complexes (=3-term h-complexes) over $\A$ is defined to be the 0-th homology category $H^0(\Fun_{\infty}(\I,\A)))$ of the dg category $\Fun_{\infty}(\I,\A)$ of $A_{\infty}$-functors 
(not $\infty$-functors!) from $\I$ to $\A$, cf.~Definition~\ref{def:3termhomotopy}.
Let $\J$ be the dg $k$-path category of the following graded quiver with relations
\begin{equation}\label{quiv:3termcofibrant}
\begin{tikzcd}
0\ar[r,"f"]\ar[rr,"h"swap,bend right=6ex]&1\ar[r,"g"]&2
\end{tikzcd}
\end{equation}
where $f$ and $g$ are closed of degree 0 and $h$ is of degree $-1$ such that $gf=-d(h)$.
Then 3-term homotopy complexes over $\A$ can be identified with dg functors from $\J$ to $\A$.
So we identify objects in $\mathcal H_{3t}(\A)$ with diagrams in $\A$
\begin{equation}\label{intro:F}
\begin{tikzcd}
&A_0\ar[r,"f"]\ar[rr,bend right = 8ex,"h"swap]&A_1\ar[r,"j"]&A_2
\end{tikzcd}
\end{equation}
where $|f|=|j|=0$, $|h|=-1$ and $d(f)=0$, $d(j)=0$ and $d(h)=-jf$.
The category $\mathcal H_{3t}(\A)$ is equivalent to the full subcategory of $\rep(\Sq,\A)$ consisting of objects $X$ such $X_{10}$ is acyclic.
A 3-term h-complex~\ref{intro:F} is {\em homotopy (resp.~left, right ) short exact} if the corresponding object in $\rep(k\Sq, \A)$ is homotopy bicartesian (resp.~cartesian, cocartesian).
\begin{lemma}[Lemma~\ref{lem:3termhcomplex}]
A 3-term h-complex~\ref{intro:F} is homotopy left short exact if and only if the following holds:
for each $A\in\A$ and $n\leq 0$ and each pair of morphisms $(v,w)\in Z^{n}\A(A,A_1)\times \A^{n-1}(A,A_2)$ such that $d(w)=-jv$, there exists a morphism $u\in Z^n \A(A,A_0)$, unique up to a coboundary, such that there exists a pair of morphisms $(v',w')\in \A^{n-1}(A,A_1)\times \A^{n-2}(A,A_2)$ satisfying $v-fu=d(v')$, $w-hu=-d(w')-jv'$.  
\end{lemma}
It is clear from this lemma that one can do diagram chasing using homotopy short exact sequences. 
We use this description to prove some diagram lemmas in Chapter~\ref{sec:diagramlemmas}. 
It is worth noting that while sometimes homological results in triangulated categories are applicable, the subtle point lies in checking that the homotopies are compatible, cf.~Proposition~\ref{push}.

\begin{definition}[Definition~\ref{exactdgstructure}]\label{intro:exactdgstructure}
An {\em exact structure} on $\A$ is a class $\mathcal{S}\subseteq \mathcal H_{3t}(\A)$ stable under isomorphisms, consisting of homotopy short exact sequences (called {\em conflations})
\[
\begin{tikzcd}
A\ar[r, tail,"i"]\ar[rr,bend right=8ex,"h"swap]&B\ar[r,two heads, "p"]&C\\
\end{tikzcd} 
\]
where $i$ is called an {\em inflation} and $p$ is called a {\em deflation}, such that the following axioms are satisfied
\begin{itemize}
\item[Ex0] $\Id_{0}$ is a deflation.
\item[{Ex}1] Compositions of deflations are deflations.
\item[{Ex}2] Given a deflation $p:B\rightarrow C$ and any map $c: C'\rightarrow C$ in $Z^0(\A)$, the object
\[
\begin{tikzcd}
&C'\ar[d,"c"]\\
B\ar[r,"p"swap]&C
\end{tikzcd}
\]
admits a homotopy pullback 
\[
\begin{tikzcd} 
{B'}\ar[r,"{p'}"]\ar[d,"{b}"swap]\ar[rd,"s"blue,blue]&{C'}\ar[d,"{c}"]\\
{B}\ar[r,"{p}"swap]&{C}
\end{tikzcd}
\]
and ${p'}$ is also a deflation. 
\item[$\Ex2^{op}$] Given an inflation $i: A\rightarrow B$ and any map $a:A\rightarrow A'$ in $Z^0(\A)$, the object
\[
\begin{tikzcd}
A\ar[r,"i"]\ar[d,"a"swap]&B\\
A'&
\end{tikzcd}
\]
admits a homotopy pushout 
\[
\begin{tikzcd}
{A}\ar[r,"{i}"]\ar[d,"{a}"swap]\ar[rd,"s"blue,blue]&{B}\ar[d,"{j}"]\\
{A'}\ar[r,"{i'}"swap]&{B'}
\end{tikzcd}
\]
and ${i'}$ is also an inflation.
\end{itemize}
We call $(\A,\mathcal {S})$ or simply $\A$ an {\em exact dg category}.
\end{definition} 

Let $(\A,\mathcal S)$ and $(\A',\mathcal S')$ be exact dg categories.
A morphism $F:\A\rightarrow \A'$ in $\Hqe$ is {\em exact} if the induced functor $\mathcal H_{3t}(\A)\rightarrow \mathcal H_{3t}(\A')$ sends the objects in $\mathcal S$ to the objects in $\mathcal S'$.
We denote by $\Hqe_{ex}(\A,\A')$ the subset of $\Hqe(\A,\A')$ consisiting of the exact morphisms.

\begin{theorem}[Theorem~\ref{fun}]\label{intro:fun}
Let $\A$ be a small exact dg category. 
If $\B$ is a small connective dg category, then $\rep_{dg}(\B,\A)$ is canonically an exact dg category. 
\end{theorem}

\begin{theorem}[Theorem~\ref{main}]\label{intro:main}
Let $\A$ be an exact dg category. 
There exists a universal exact morphism $F:\A\rightarrow \D^b_{dg}(\A)$ in $\Hqe$ from $\A$ to a pretriangulated dg category $\D^b_{dg}(\A)$. If moreover $\A$ is connective, this morphism satisfies:
\begin{itemize}
\item[1)]It induces a quasi-equivalence from $\tau_{\leq 0}{\A}$ to $\tau_{\leq 0}\D'$ for an extension-closed dg subcategory $\D'$ of $\D^b_{dg}(\A)$.
\item[2)]It induces a natural bijection $\mathbb E(C,A)\xrightarrow{\sim} \Ext^1_{\D^b(\A)}(FC,FA)$ for each pair of objects $C,A$ in $H^0(\A)$ where $\D^b(\A)=H^0(\D^b_{dg}(\A))$.
\end{itemize}
We call $\D^b_{dg}(\A)$ the {\em bounded dg derived category} of $\A$. 
\end{theorem}
We construct $\D^b_{dg}(\A)$ as the dg quotient
of the pretriangulated hull $\pretr(\A)$ by a full dg subcategory $\N$, which generalizes the dg category of acyclic
complexes over a Quillen exact category. We deduce that if $\A$ is a Quillen exact category, then $\D^b_{dg}(\A)$ is 
quasi-equivalent to the canonical dg enhancement of the bounded derived category of $\A$, whence the name
and the notation.

The universal embedding $\A\rightarrow \D^b_{dg}(\A)$ allows us to prove diagram lemmas for exact dg categories, 
which may be hard to prove directly. 
For example, let $\J$ be the dg category of (\ref{quiv:3termcofibrant}) and 
$F:\J\rightarrow \rep_{dg}(\J,\A)$ a dg functor such that $F(0)$ and $F(2)$ are both conflations in $\A$. 
Let $F_i:\J\rightarrow \A$ be the composition of $F$ with $\rep_{dg}(\J,\A)\rightarrow \A$ which is induced by the inclusion $i:k\rightarrow \J$ for $i=0$, $1$ and $2$.
If all $F_i$ are conflations in $\A$, then $F(1)$ is also an conflation in $\A$. This `diagram lemma' is the
main ingredient of the proof of Lemma~\ref{lem:internalhom}.

The dg nerve functor $N_{dg}$ sends each small exact dg category to an exact $\infty$-category. 
Moreover, we have the following theorem.
\begin{theorem}[Theorem~\ref{nerve}]
\label{intro:nerve}Let $\A$ be a dg category such that $H^0(\A)$ is an additive category. 
Then there is a bijection between the class of exact structures on the dg category $\A$ and the class of 
exact structures on the $\infty$-category $N_{dg}(\A)$.
Moreover, if $\A$ is a pretriangulated dg category endowed with the maximal exact structure 
(given by all homotopy short exact sequences), then its dg nerve is a stable $\infty$-category
endowed with the maximal exact structure (where each morphism is both an inflation and a deflation).
\end{theorem}

This extends in particular a theorem by Faonte \cite{Faonte17b}, who showed that 
if a dg category is pretriangulated, then its dg nerve is a stable $\infty$-category.

Let $(\C,\mathbb E,\mathfrak s)$ be a small extriangulated category.
For $n\geq 0$, Gorsky--Nakaoka--Palu~\cite{GorskyNakaokaPalu21} defined the {\em higher 
extension} bimodule $\mathbb E^{n}(?,-)$ on $\C$ as the $n$th tensor power over $\C$ of $\mathbb E$.
A {\em $\delta$-functor} is a triple $(T,\epsilon,\eta)$ where $T=(T^i)_{i\geq 0}$ be a sequence of $\C$-$\C$-bimodules
and $\epsilon$ and $\eta$ are collections of connecting morphisms satisfying certain conditions, cf.~Definition~\ref{def:deltafunctor}.
It is clear from the definition that $(\mathbb E^n)_{n\geq 0}$ together with the associated connecting morphisms is a $\delta$-functor.
Using a general result concerning $\delta$-functors with weakly effaceable bimodules $T^i$ for $i>0$, cf.~Corollary~\ref{cor:effaceablebimodule}, we have the following result.
\begin{proposition}[Proposition~\ref{higher}]\label{intro:higher}
Let $F:\A\rightarrow \D^{b}_{dg}(\A)$ be the universal embedding of $\A$ 
into a pretriangulated category. 
Then we have a canonical isomorphism of $\delta$-functors $\alpha:\mathbb E^n(?,-)\xrightarrow{\sim} \Ext^n_{\D^b(\A)}(?,-)$.\end{proposition}
A similar result for exact $\infty$-categories also holds. As a consequence, for an exact dg category $\A$, the stable hull of $N_{dg}(\A)$ is canonically equivalent to $N_{dg}(\D^b_{dg}(\A))$.

For an extriangulated category $\C$ with a subcategory $\P$ consisting of (not necessarily all) projective-injective objects in $\C$, Nakaoka--Palu showed that the ideal quotient $\C/\P$ has the structure of an extriangulated category, induced from that of $\C$, cf.~\cite[Proposition 3.30]{NakaokaPalu19}. The following theorem enhances their result.
\begin{theorem}[Theorem~\ref{quot}]\label{intro:quot}
Let $\A$ be a connective exact dg category and $\P$ a full dg subcategory of $\A$ consisting of projective-injective objects in $\A$.
Let $\mathcal S_{dg}$ be the canonical dg enhancement of $\D^b(\A)/\tr(\P)$.
Then the dg quotient $\A/\P$ carries a canonical exact dg structure induced from that of $\A$ and its dg derived category is quasi-equivalent to $\mathcal S_{dg}$.
\end{theorem}

Let $\A$ and $\A'$ be connective exact dg categories and $\C$ an arbitrary exact dg category.
A morphism $\mu: \A\otimes \A'\rightarrow \C$ in $\Hqe$ is {\em biexact}, if for all objects $A\in \A$ and 
$A'\in \A'$, the induced morphisms 
\[
\mu_{A,-}:\A'\rightarrow \C,\;\;\mu_{-,A'}:\A\rightarrow \C
\]
are both exact morphisms. Let $\A\boxtimes \A'$ be the $\tau_{\leq 0 }$-truncation of the extension closure of
the quasi-essential image of the morphism in $\Hqe$
\[
F: \A\otimes \A'\rightarrow \D^b_{dg}(\A)\otimes\D^b_{dg}(\A')\rightarrow \pretr(\D^b_{dg}(\A)\otimes\D^b_{dg}(\A'))
\]

\begin{proposition}[Proposition~\ref{prop:universalbilinear}]\label{intro:universalbilinear}
 The natural morphism $\A\otimes \A'\rightarrow \A\boxtimes \A'$ is the universal biexact morphism in $\Hqe$ from $\A\otimes \A'$ to an exact dg category.
\end{proposition}

Let $\B$ and $\C$ be connective exact dg categories.
Let $\rep_{dg}^{ex}(\B,\C)$ be the full dg subcategory $\rep_{dg}(\B,\C)$ consisting of exact morphisms.
It is shown in Lemma~\ref{lem:internalhom} that $\rep_{dg}^{ex}(\B,\C)$ is stable under extensions in $\rep_{dg}(\B,\C)$ and thus inherits a canonical exact dg structure. 
\begin{corollary}[Corollary~\ref{cor:exactadjunction}]\label{intro:exactadjunction}
Let $\A$ and $\B$ be small connective exact dg categories. Let $\C$ be a small exact dg category. 
Then we have the following natural bijection
\[
\Hqe_{ex}(\A\boxtimes \B,\C)\iso \Hqe_{ex}(\A,\rep_{dg}^{ex}(\B,\C)).
\]
\end{corollary}

Let $(\C,\mathbb E,\mathfrak s)$ be a small extriangulated category.
The {\em defect} of an $\mathfrak s$-conflation 
\[
A\xrightarrow{x}B\xrightarrow{y}C
\] 
is defined to be the cokernel of the morphism of right $\C$-modules $\C(?,B)\xrightarrow{(?,y)}\C(?,C)$.
The full subcategory of the module category $\Mod\C$ whose objects are the defects is denoted by 
$\Def \mathbb E$. It was introduced in this generality by Ogawa in \cite{Ogawa21} 
and further studied by Enomoto \cite{Enomoto21}, who showed that $\Def \mathbb E$ is
a Serre subcategory of the abelian category of coherent $\C$-modules and that the map
taking a closed subbifunctor of $\mathbb E$ to its category of defects is a poset
isomorphism onto the poset of Serre subcategories of $\Def \mathbb E$.
If $\A$ is a connective exact dg category, then $\Def H^0(\A)$ is equivalent
to the heart of a $t$-structure on the category $\N$ of `acyclic complexes', cf.
the remark after Theorem~\ref{intro:main}.

\begin{theorem}[Theorem~\ref{bijectionstructures}]\label{intro:bijectionstructures}
Let $(\A,\mathcal S)$ be an exact dg category and $(H^0(\A),\mathbb E,\mathfrak s)$ the associated algebraic extriangulated category.
The following posets (where the order relations are given by inclusion) are canonically isomorphic:
\begin{itemize}
\item[(1)]The poset of exact substructures of $(\A,\mathcal S)$.
\item[(2)]The poset of closed subbifunctors of $\mathbb E$.
\item[(3)]The poset of Serre subcategories of $\Def \mathbb E$.
\end{itemize}
\end{theorem}
Here, the isomorphism between the posets in (2) and (3) follows from Enomoto's theorem \cite{Enomoto21}.

In~\cite{Rump11}, Rump proved that every additive category has the greatest exact structure, which generalized the results in~\cite{SiegWegner11,Crivei12}. 
The following theorem generalizes Rump's result to additive dg categories.
\begin{theorem}[Theorem~\ref{thm:maximalexactdgstructure}]\label{intro:maximalstructure}
Each additive dg category $\A$ admits a greatest (i.e.~unique maximal) exact dg structure.
\end{theorem}
It is not clear how to formulate, let alone prove, an analogous statement for extriangulated
structures. 
\newpage
\section{Homotopy diagrams}
	\subsection{Notations and terminology}\label{subsection:notations}
	In this subsection, we collect basic notations and terminology needed in this paper. The standard references for dg categories are \cite{Keller94, Keller06d, Drinfeld04,Toen11, BondalKapranov90}.
	
	Throughout we fix a commutative ring $k$. 
	We write $\otimes$ for the tensor product over $k$.  
	We denote by $\C_{dg}(k)$ the dg category of complexes of $k$-modules and by $\C(k)$ the category of complexes of $k$-modules. 
	
	Let $\A$ be a dg $k$-category.
	For two objects $A_1$ and $A_2$, the Hom complex is denoted by $\Hom_{\A}(A_1,A_2)$ or 
	$\A(A_1,A_2)$.	An element $f$ of $\A(A_1,A_2)^{p}$ will be called a {\em homogeneous} morphism of degree $p$ with the notation $|f|=p$. 
	A homogeneous morphism $f:A_1\rightarrow A_2$ is {\em closed} if we have $d(f)=0$.
	
                 Let $M$ be complex of $k$-modules. We put
	\[
	\begin{tikzcd}
	\tau_{\leq 0}M=(\cdots\ar[r]&M^{-2}\ar[r]&M^{-1}\ar[r]&Z^0M\ar[r]&0\ar[r]&\cdots).
	\end{tikzcd}
	\] 
                 For a dg category $\A$, denote by $\tau_{\leq 0}\A$ the dg category with the same objects as $\A$ 
                 and whose morphism complexes are given by
                 \[
                 (\tau_{\leq 0} \A)(A_1, A_2) = \tau_{\leq 0}(\A(A_1, A_2)).
                  \]
                 The composition is naturally induced by that of $\A$. 

 A dg category $\A$ is {\em connective} if for each pair of objects $A_1, A_2\in \A$, the 
 complex $\Hom_{\A}(A_1, A_2)$  has cohomology concentrated in non-positive degrees.
It is called {\em strictly connective} if the components of the complex $\Hom_{\A}(A_1, A_2)$ vanish
in all positive degrees.

	We denote by $Z^0(\A)$ the category with the same objects as $\A$ and whose morphism
spaces are defined by
	\[
	(Z^0\A)(A_1,A_2)=Z^0(\A(A_1,A_2)),
	\]
 where $Z^0$ is the kernel of $d:\A(A_1,A_2)^0\rightarrow \A(A_1,A_2)^1$.
 Similarly, we denote by $H^0(\A)$ the category with the same objects as $\A$ and whose 
 morphism spaces are given by 
 \[
 (H^0\A)(A_1,A_2)=H^0(\A(A_1,A_2)),
 \]
 where $H^0$ denotes the zeroth cohomology of the complex.
 
	The {\em oppositie dg category} $\A^{op}$ has the same objects as $\A$ and its morphism spaces are defined by
	\[
	\A^{op}(X,Y)=\A(Y,X);
	\] 
                 the composition of $f\in\A^{op}(Y,X)^{op}$ with $g\in \A^{op}(Z,Y)^{q}$ is given by $(-1)^{pq}gf$.
	By a dg $\A$-module $M$, we mean a right dg $\A$-module, i.e. a dg functor $\A^{op}\rightarrow \C_{dg}(k)$. 
	We denote by $\C_{dg}(\A)$ the dg category of right dg $\A$-modules and by $\C(\A)$ the category of dg 
	$\A$-modules. Note that $Z^0(\C_{dg}(\A))=\C(\A)$. 
	The category up to homotopy of dg $\A$-modules is 
	\[
	\mathcal H(\A)=H^0(\C_{dg}(\A)).
	\]

An object $A$ in $\A$ is {\em contractible} if it is a zero object in $H^0(\A)$, or equivalently if we have $\Id_{M}=d(h)$ for some morphism $h:M\rightarrow M$ of degree $-1$.
A dg category $\A$ has {\em contractible pre-envelopes} provided for each object $A$ in $\A$, there exists a closed morphism $f:A\rightarrow A'$ of degree $0$ such that $A'$ is contractible and for each morphism $h:A\rightarrow B$ in $\A$, there exists some $h':A'\rightarrow B$ satisfying $h'\circ f=h$ (hence $d(h')\circ f=d(h)$). 
Dually one has the notion of a dg category having {\em contractible pre-covers}.
A dg $\A$-module is {\em contractible} if it is contractible in the dg category $\C_{dg}(\A)$.

 For each object $A\in\A$, we have the right dg module {\em represented by} $A$
 \[
 A^{\wedge}=\Hom_{\A}(-,A).
 \]
 A dg $\A$-module $M$ is {\em representable} if it is isomorphic to $A^\wedge$ for some
 $A\in\A$ and {\em quasi-representable} if it is quasi-isomorphic to $A^\wedge$ for some $A \in \A$.
 
For a dg $\A$-module $M$, put $IM=\Cone(\Id_M)$ and $PM=\Cone(\Id_{\Sigma^{-1}M})$. 
Explicitly, we have $IM=\Sigma M\oplus M$ and $PM=M\oplus \Sigma^{-1}M$ as graded $\A$-modules. 
We have the natural inclusion $i=[0,1]^{\intercal}:M\rightarrow IM$ and projection $p=[1,0]:PM\rightarrow M$.
 
 A dg functor $F:\A\rightarrow \B$ is a {\em quasi-equivalence} if
 \begin{itemize}
 \item[a)] it is {\em quasi-fully faithful}, i.e.~for all objects $A, A'\in A$, the morphism
  \[
 F_{A,A'}:\A(A,A')\rightarrow \B(FA,FA').    
 \] 
 is a quasi-isomorphism and
 \item[b)] the induced functor  $H^0(F):H^0(\A)\rightarrow H^0(\B)$ is an equivalence of categories.
 \end{itemize}
 
 The {\em Yoneda dg functor}
 \[
 \A\rightarrow \C_{dg}(\A),\;\; A\mapsto A^{\wedge}
 \]
 is fully faithful.
 Let $\A'$ be the full dg subcategory of $\C_{dg}(\A)$ consisting of objects $A^{\wedge}$, $\Cone(\Id_{A^{\wedge}})$ and $\Cone(-\Id_{\Sigma^{-1}A^{\wedge}})$ for each $A\in\A$.
 The dg category $\A'$ has contractible pre-covers and contractible pre-envelopes and the inclusion $\A\rightarrow \A'$ is a quasi-equivalence. 
 Thus, up to quasi-equivalence, we may assume $\A$ has contractible pre-covers and contractible pre-envelopes. 
 
  The category $\dgcat$ of small dg categories admits the Dwyer-Kan model structure (cf.~\cite{Tabuada05}), 
 whose weak equivalences are the quasi-equivalences. Its homotopy category is denoted by $\Hqe$. 
 There exists a cofibrant replacement functor $Q$ on $\dgcat$ such that for any $\A\in\dgcat$, the natural dg functor $Q(\A)\rightarrow \A$ is the identity on the set of objects. 
 In particular, the Hom-complexes of $Q(\A)$ are cofibrant over $k$.
 The {\em tensor product} $\A\otimes \B$ of two dg categories $\A$ and $\B$ has the class of objects  $\obj(\A)\times \obj(\B)$ and the morphism spaces
 \[
 (\A\otimes\B)((A,B),(A',B'))=\A(A,A')\otimes \B(B,B')
 \]
 with the natural compositions and units. 
 This defines a symmetric monoidal structure $-\otimes -$ on $\dgcat$ which is closed.
 For $\A, \B\in\dgcat$, put $\A\otimes^{\mathbb L}\B=\A\otimes Q(\B)$. 
 This extends to a bifunctor $-\otimes^{\mathbb L}-:\dgcat\times \dgcat\rightarrow \dgcat$ and then passes through the homotopy categories
 \[
 -\otimes^{\mathbb L}-:\Hqe\times \Hqe\rightarrow \Hqe.
 \]
 
 Let $\Cat$ be the category of small categories. 
 Let $[\Cat]$ be the category with objects small categories and with morphisms isomorphism classes of functors.
 In particular the isomorphism class of an equivalence of categories is an isomorphism in $[\Cat]$.
 So the functor $H^0:\dgcat\rightarrow \Cat$ induces a functor
 \[
 H^0:\Hqe\rightarrow [\Cat].
 \]
 
 For a category $\C$, we denote by $\Iso(\C)$ its isomorphism classes of objects.
 
 For a dg category $\A$, we denote by $\D(\A)$ its {\em derived category}, a triangulated category. 
 By definition $\D(\A)$ is the localization of $\C(\A)$ at the class of {\em quasi-isomorphisms}, i.e.~morphisms of dg $\A$-modules which induce quasi-isomorphisms of complexes when evaluated at any object in $\A$.
 Let $\pi:\C(\A)\rightarrow \D(\A)$ be the quotient functor. 
 For a morphism $j:M\rightarrow N$ in $\C(\A)$, we denote by $\overline{\jmath}=\pi(j):M\rightarrow N$ the corresponding morphism in $\D(\A)$.
 A dg functor $F:\A\rightarrow \B$ induces a triangle functor $F_*:\D(\A)\rightarrow \D(\B)$ which is an equivalence of triangulated categories if $F$ is a quasi-equivalence.
 The {\em dg derived category} $\D_{dg}(\A)$ of $\A$ is defined to be the full dg subcategory of $\C_{dg}(\A)$ consisting of cofibrant dg $\A$-modules in the {\em projective model structure} of $\C(\A)$ (cf.~\cite[Theorem 3.2]{Keller06d}).
 The canonical functor $H^0(\D_{dg}(\A))\rightarrow \D(\A)$ is an equivalence of triangulated categories.
 
 For dg categories $\A$ and $\B$, we define $\rep(\B,\A)$ to be the full subcategory of $\D(\A\otimes^{\mathbb L} \B^{op})$ whose objects are the dg bimodules $X$ such that $X(-,B)$ is quasi-representable for each object $B$ of $\B$. 
 We define the canonical dg enhancement $\rep_{dg}(\B,\A)$ to be the full dg subcategory of $\D_{dg}(\A\otimes^{\mathbb L}\B^{op})$ whose objects are those of $\rep(\B,\A)$.
 
 A dg category $\A$ is {\em pretriangulated} if the canonical inclusion $H^0(\A)\rightarrow \D(\A)$ is a triangulated subcategory.
Note that $\rep_{dg}(\B,\A)$ is pretriangulated if $\A$ is pretriangulated.
For a dg category $\A$, its {\em pretriangulated hull} $\pretr(\A)$ (\cite{BondalKapranov90, Drinfeld04, BondalLarsenLunts04}) is defined as follows. 
The objects of $\pretr(\A)$ are ``one-sided twisted complexes'', i.e.~formal expressions $(\oplus_{i=1}^{n}A_i[r_i],q)$, where $A_i\in \A$, $r_i\in\mathbb Z$, $n\geq 0$, $q=(q_{ij})$, $q_{ij}\in \Hom_{\A}^{r_i-r_j+1}(A_j,A_i)$, $q_{ij}=0$ for $i\geq j$ and $dq+q^2=0$. 
Here we adopt the convention $(dq)_{ij}=(-1)^{r_i}d_{\A}(q_{ij})$ and $(q^2)_{ij}=\sum_{k}q_{ik}\circ q_{kj}$.
If $A$ and $A'$ are two objects in $\pretr(\A)$ with $A=(\oplus_{i=1}^{n}A_i[r_i],q)$ and $A'=(\oplus_{i'=1}^{n'}A_{i'}'[r_{i'}'],q')$, the complex $\Hom_{\pretr(\A)}(A,A')$ has as the degree $m$ component the space of matrices $f=(f_{ij})$, $f_{ij}\in\Hom_{\A}^{m+r_{i}'-r_j}(A_{j}, A_{i}')$. 
The differential $d$ carries a morphism $f=(f_{ij})$ of degree $m$ to $df=((df)_{ij})$ where 
\[
(df)_{ij}=(-1)^{r_i'}d_{\A}(f_{ij})+\sum_{k}q'_{ik}\circ f_{kj}-(-1)^{m} \sum_{k} f_{ik}\circ q_{kj}.
\]
The composition map 
\[
\Hom_{\pretr(\A)}(A',A'')\otimes \Hom_{\pretr(\A)}(A,A')\rightarrow \Hom_{\pretr(\A)}(A,A'')
\]
 is the matrix multiplication: $f'\otimes f\mapsto f''$ where $f''_{ij}=\sum_{k}f'_{ik}\circ f_{kj}$.
 We denote the triangulated category $H^0(\pretr(\A))$ by $\tr(\A)$.
The Yoneda dg functor $\A\rightarrow \C_{dg}(\A)$ extends to a fully faithful dg functor $\pretr(\A)\rightarrow \D_{dg}(\A)$.
It induces a fully faithful triangle functor $\tr(\A)\rightarrow \D(\A)$ whose essential image is the triangulated subcategory of $\D(\A)$ generated by the representable dg modules.

A {\em quiver} (or {\em directed graph}) $Q$ is a quadruple $(Q_0,Q_1,s,t)$ where $Q_0$ is a set whose elements are called {\em objects} or {\em vertices} of $Q$, $Q_1$ is a set of {\em arrows} $f$ (or {\em edges}), and $s,t$ are maps $Q_1\rightarrow Q_0$ where $s(f)$ is the {\em source} of the arrow $f$ and $t(f)$ is the {\em target} of $f$. 
A morphism between quivers is defined in the obvious way. 
We denote the category of quivers by $\mathrm{Quiv}$.

Each small category $\C$ has an {\em underlying quiver} $F(\mathcal C)$ and this extends to the forgetful functor $F:\Cat\rightarrow\Quiv$. 
This functor admits a left adjoint $P:\Quiv\rightarrow \Cat$ which associates to a quiver $Q$ the {\em path category} $P(Q)$ of $Q$. 
The path category $P(Q)$ has the same objects as $Q$ and its morphisms $x_1\rightarrow x_n$ are finite paths
\[
\begin{tikzcd}
x_1\ar[r,"f_1"]&x_2\ar[r]&\cdots\ar[r,"f_{n-1}"]&x_n
\end{tikzcd}
\]
consisting of $n\geq 1$ objects $x_1,\cdots,x_n$ of $Q$ which are connected by arrows $f_i:x_i\rightarrow x_{i+1}$ of $Q$. 
The composite of two paths is defined by concatenation: $\mathrm{concat}(f,g)=g\circ f$.

Let $Q$ be a quiver and $R$ a function which assigns to each pair of objects $x,y$ of $Q$ a binary relation $R_{x,y}$ on the set of finite paths from $x$ to $y$. 
Then the path category of the quiver $Q$ with relations $R$ is defined to be the quotient category $P(Q)/\overline{R}$, where $\overline{R}$ is the smallest family of equivalence relations $\overline{R}_{x,y}$, $x,y\in Q_0$, containing $R$ and stable under pre- and postcomposition with morphisms. 
Often we only write down nontrivial relations in $R$.

For a small category $\I$, we have the $k$-category $k\I$ whose set of objects is the same as that of $\I$ and for each pair of objects $x,y$ in $\C$, the space $k\I(i,j)$ is the free $k$-module generated by the set $\I(i,j)$. 
In particular, if we view $k\I$ as a dg category concentrated in degree 0, then it is {\em $k$-cofibrant}, i.e. the Hom complexes are $k$-cofibrant. 
Let $\B$ be a $k$-cofibrant dg category. 
We have a quasi-equivalence $\A\otimes ^{\mathbb L}\B\iso\A\otimes \B$.
A dg functor $F:\B\rightarrow\A$ gives rise to a $\B$-$\A$-bimodule $_{F}\A_{\A}$ defined by
\[
(A,B)\mapsto \Hom_{\A}(A,FB).
\]
The assignment $F\mapsto  {_{F}\A_{\A}}$ passes to the following map
\[
\Hqe(\B,\A)\rightarrow \Iso(\rep(\B,\A))
\]
which is a bijection, cf.~\cite{Toen07}.
It is shown in loc.~cit~that $\rep_{dg}(\B,\A)$ is the internal Hom of the monoidal category $(\Hqe,-\otimes^{\mathbb L}-)$, i.e.~ for small dg categories $\A$, $\B$ and $\C$, we have 
\[
\Hqe(\A\otimes^{\mathbb L}\B,\C)\iso\Hqe(\A,\rep_{dg}(\B,\C)).
\]
\subsection{Preliminaries}\label{Preliminaries}	
	Throughout we fix a dg $k$-category $\A$. 
	Consider the derived category $\D(\A)$ of the dg category $\A$. 
	Our first aim is to formulate proper definitions of homotopy (co)cartesian squares in the dg category $\A$ (not  $\D(\A)$).

\begin{notation}
 For a small category $\mathcal I$, put
\[
\D(\I)=\D(k(\I^{op})\otimes \A),\;\; \C(\I)=\C(k(\I^{op})\otimes \A), \;\;\rep(\I)=\rep(k(\I),\A)
\]
and 
\[
\rep_{dg}(\I)=\rep_{dg}(k(\I),\A).
\]
\end{notation}
We are mainly concerned with the case when $\I$ is the category with one object and one morphism or one of the following categories:

The {\em square} category $\mathrm{Sq}$ is the path category of the quiver 
\[
\begin{tikzcd}
00\ar[r,"f"]\ar[d,"g"swap] & 01\ar[d,"j"] \\
10\ar[r,"k"swap] & 11
\end{tikzcd}
\]
with the commutativity relation $jf\sim kg$;

the {\em cospan} category $\mathrm{Cosp}$ is the path category of the quiver
\[
\begin{tikzcd}
&01\ar[d]\\
10\ar[r]&11;
\end{tikzcd}
\]

the {\em span} category $\mathrm{Sp}$ is the path category of the quiver
\[
\begin{tikzcd}
00\ar[r]\ar[d]&01\\
10&
\end{tikzcd};
\]

the {\em composition} category $\mathrm{Com}$ is the path category of the quiver
\[
\begin{tikzcd}
0\ar[r]&1\ar[r]&2;
\end{tikzcd}
\]

the {\em morphism} category $\mathrm{Mor}$ is the path category of the quiver
\[
\begin{tikzcd}
0\ar[r]&1.
\end{tikzcd}
\]

We identify objects $X\in\D(\mathrm{Sq})$ with commutative squares
\begin{equation}
\begin{tikzcd}\label{D(Sq)}
X_{00}\ar[r,"f"]\ar[d,"g"swap]&X_{01}\ar[d,"j"]\\
X_{10}\ar[r,"k"swap]&X_{11}
\end{tikzcd}
\end{equation}
in $\C(\A)$. 

In general we identify an object $X\in \D(\I)$ with an $\I$-shaped diagram $X$ in $\C(\A)$, i.e.~a functor 
\[
X:\I\rightarrow \C(\A),\;\; i\mapsto X_i.
\] 
 Let $Q:\C(\A)\rightarrow \D(\A)$ be the canonical quotient functor. 
The map sending $X$ to $Q\circ X$ defines the canonical {\em diagram functor} $\Dia:\D(\I)\rightarrow\Fun(\I,\D(\A))$ which is {\em conservative}, i.e. detects isomorphisms. Note that it is not full in general. 

We have the following diagram with the obvious functors $i:\Cosp\rightarrow\Sq$ and $s:\Mor\rightarrow\Cosp$. 
\[
\begin{tikzcd}
0\ar[d,""{name=1}]&\;\ar[d,""{name=4}, white]&01\ar[d,""{name=2}]&00\ar[rd,phantom,"="]\ar[r,"f"]\ar[d,"g"{name=3,swap}] & 01\ar[d]\ar[d,"j"] \\
1&10\ar[r]&11&10\ar[r,"k"swap] & 11\ar[r, from=2, to=3,red,"i"]\ar[r,from=1,to=4,red,"s"]
\end{tikzcd}\;.
\]

Write $L:\D(\mathrm{Sq})\rightarrow \D(\mathrm{Cosp})$ for the restriction functor along $i$, $U$ for its left adjoint and $R$ for its right adjoint.

Write $L':\D(\mathrm{Cosp})\rightarrow \D(\mathrm{Mor})$ for the restriction functor along $s$, $U'$ for its left adjoint and $R'$ for its right adjoint. 

We describe these functors at the level of objects. The functor $U'$ sends an object $Y\xrightarrow{j} Z$ to 
\[
\begin{tikzcd}
&Y\ar[d,"j"]\\
0\ar[r]&Z
\end{tikzcd}\;.
\]
The functor $R'$ sends $Y\xrightarrow{j} Z$ to 
\[
\begin{tikzcd}
&Y\ar[d,"j"]\\
Z\ar[r,equal]&Z
\end{tikzcd}\;.
\]
The functor $U$ sends 
\[
\begin{tikzcd}
&Y\ar[d,"j"]\\
W\ar[r,"k"swap]&Z
\end{tikzcd}
\] 
to 
\[
\begin{tikzcd}
0\ar[r]\ar[d]&Y\ar[d,"j"]\\
W\ar[r,"k"swap]&Z
\end{tikzcd}\;. 
\]
Let us describe the functor $R$. 
Let $S$ be the object
\[
\begin{tikzcd}  & B\ar[d,"{j}"] \\ 
                       C\ar[r,"k"swap] & D
                       \end{tikzcd}
\]
in $\D(\mathrm{Cosp})$. 
Then the functor $R$ takes $S$ to the dg $k\Sq^{op}\otimes \A$-module
\[ 
\RHom_{k\mathrm{Cosp^{op}}}(M,S)
\]
where $M$ is the dg $k\Cosp^{op}\otimes k(\Sq)$-module $\Hom_{k(\Sq)^{op}}(i(-),-)$.
We will construct a quasi-isomorphism $M'\rightarrow M$ of bimodules such that $M'(-,j)$ is cofibrant for each $j\in \Sq$.
The dg $k\Cosp^{op}$-module $M(-,00)$ is given by the diagram
\[
\begin{tikzcd}
&k\ar[d,"\Id"]\\
k\ar[r,"\Id"swap]&k
\end{tikzcd}\;.
\]
It fits into a short exact sequence in $\C(k\mathrm{Cosp}^{op})$
\[
0\rightarrow 11^{\wedge}\rightarrow 01^{\wedge}\oplus 10^{\wedge}\rightarrow M(-,00)\rightarrow 0.
\]
So it is quasi-isomorphic to the cokernel $N$ of the graded-injective map 
\[
11^{\wedge}\rightarrow \Cone(\Id_{11^{\wedge}})\oplus 10^{\wedge}\oplus 01^{\wedge}.
\]
Let $M'$ be the dg $k\Cosp^{op}\otimes k(\Sq)$-module given by the following diagram in $\C(k\Cosp^{op})$
\[
\begin{tikzcd}
N&10^{\wedge}\ar[l]\\
01^{\wedge}\oplus \Cone(\Id_{11^{\wedge}})\ar[u]&11^{\wedge}\ar[l]\ar[u]\mathrlap{\;.}
\end{tikzcd}
\]
Then we get the desired quasi-isomorphism $M'\rightarrow M$.
So the functor $R$ takes $S$ to $\Hom_{k\Cosp^{op}}(M',S)$ which is given by the following diagram in $\C(\A)$
\[
\begin{tikzcd}
K\ar[r]\ar[d]&B\ar[d,"{j}"]\\
C\oplus \Cone(-\Id_{\Sigma^{-1}D})\ar[r,"(k\ p)"swap]&D
\end{tikzcd}
\]
such that $p:\Cone(-\Id_{\Sigma^{-1}D})\rightarrow D$ is the canonical projection map and that the diagram fits into the following graded split short exact sequence in $\C(\A)$
\[
0\rightarrow K\rightarrow \Cone(-\Id_{\Sigma^{-1}D})\oplus B\oplus C\xrightarrow{[p\;\;j\;\;k]} D\rightarrow 0.
\]

In summary, we have the following diagram of triangle functors 
\begin{equation}
\begin{tikzcd}\label{adjointtriple}
\D(\mathrm{Mor})\ar[shift left =2ex, rr,"U'"]\ar[shift right=2ex,rr,"R'"]
&&\D(\mathrm{Cosp})\ar[ll,"L'"swap]\ar[shift left=2ex, rr,"U"]\ar[shift right=2ex, rr,"R"]
&&\D(\mathrm{Sq}).\ar[ll,"L"swap]
\end{tikzcd}
\end{equation}

Consider the category $\C(\Sq)$ of dg modules over $k(\mathrm{Sq})^{op}\otimes \A$. The representable dg modules are 
\[
\begin{tikzcd}A^{\land}\ar[r,equal]\ar[d,equal]&A^{\land}\ar[d,equal]\\A^{\land}\ar[r,equal]&A^{\land}\end{tikzcd},
\begin{tikzcd}0\ar[r]\ar[d]&A^{\land}\ar[d,equal]\\0\ar[r]&A^{\land}\end{tikzcd},
\begin{tikzcd}0\ar[r]\ar[d]&0\ar[d]\\A^{\land}\ar[r,equal]&A^{\land}\end{tikzcd},
\begin{tikzcd}0\ar[r]\ar[d]&0\ar[d]\\0\ar[r]&A^{\land}\end{tikzcd}
\]
where $A\in \A$.

Since the cofibrant objects are direct summands, as graded modules, of sums of modules which are shifts of representable modules, we have that cofibrant objects in $\C(\Sq)$ are of the form
\[
\begin{tikzcd}
X_{00}\ar[r,"f"]\ar[d,"g"swap]&X_{01}\ar[d,"j"]\\
X_{10}\ar[r,"k"swap]&X_{11}
\end{tikzcd} 
\]
where $f,g,j,k$ are graded-split injections and all $X_i$ are cofibrant dg $\A$-modules.

\begin{lemma}[\cite{KellerNicolas13,BeligiannisReiten07}] \label{lemma:tstructure}
 Let $\A$ be a dg category.
There is a canonical t-structure 
\[
(\D(\A)^{\leq 0},\D(\A)^{\geq 0})
\]
on $\D(\A)$ such that $\D(\A)^{\geq 0}$ is formed by those dg modules whose cohomology is concentrated in non-negative degrees. 
\end{lemma}
Note that we have $A^{\wedge}\in\D(\A)^{\leq 0}$ for each $A\in\A$.

Let $\B$ be a connective dg $k$-category. 
Put $\T=\D(\A\otimes \B^{op})$. 
Let $(\T^{\leq 0},\T^{\geq 0})$ be the canonical t-structure on $\T$. 

Let $X$ be a dg $\B-\A$-bimodule. It is clear that $X\in \T^{\geq 0}$ if and only if $X(-,B)\in \D(\A)^{\geq 0}$ for each $B\in\mathcal B$.
\begin{lemma}\label{cha}
We have $X\in \T^{\leq 0}$ if $X(-,B)\in \D(\A)^{\leq 0}$ for each $B\in\B$.
\end{lemma}
\begin{proof}
Suppose $X(-,B)\in \D(\A)^{\leq 0}$ for each $B\in\B$. Let $Y\in \T^{\geq 1}$. 

Then we have $\Hom_{\D(\A)}(X(-,B),Y(-,B'))=0$ for $B,B'\in\B$. 

So we have $\tau_{\leq 0}\RHom_{\A}(X,Y)=0$ and we have
\[
\begin{aligned}\RHom_{\B^{op}\otimes \A}(X,Y)=&\RHom_{\B^{op}\otimes \B}(\B,\RHom_{\A}(X,Y))\\
=&\RHom_{\B^{op}\otimes \B}(\B,\tau_{\leq 0}\RHom_{\A}(X,Y))\\
=&0
\end{aligned}
\]
Hence $X\in\T^{\leq 0}$.
\end{proof}
In particular, for a small category $\mathcal I$, the category $\rep(\I)$ is a full subcategory of $\D(\I)^{\leq 0}$.

\begin{definition}[\cite{Bondarko10,Pauksztello08}]
 Let $\T$ be a triangulated category.
 A {\em co-t-structure} on $\T$ is a pair of full additive subcategories $(\T^{w\leq 0},\T^{w\geq 0})$ such that the following properties hold:
 \begin{itemize}   
\item[(i)] $\T^{w\leq 0}$ and $\T^{w\geq 0}$ are stable under forming direct summands;
\item[(ii)] $\Sigma^{-1}\T^{w\geq 0}\subset \T^{w\geq 0}$; $\Sigma \T^{w\leq 0}\subset \T^{w\leq 0}$;
\item[(iii)] $\Hom_{\T}(\Sigma^{-1}\T^{w\geq 0}, \T^{w\leq 0})=0$;
\item[(iv)] For any object $X$ of $\T$ there exists a triangle
\begin{align}
\Sigma^{-1} B\rightarrow X\rightarrow A\rightarrow B \label{weight} \tag{2.1}
\end{align}
such that $A\in \T^{w\leq 0}$ and $B\in \T^{w\geq 0}$.

The triangle (\ref{weight}) is called a {\em weight decomposition} of $X$.
\end{itemize}
\end{definition}
Let $(\T^{w\leq 0},\T^{w\geq 0})$ be a co-t-structure on a triangulated category $\T$.
\begin{definition}[\cite{Bondarko10, Pauksztello08}]
\begin{itemize}
\item The full subcategory $\mathcal H{\coloneqq}\T^{w\leq 0}\cap \T^{w\geq 0}$ is called the {\em coheart} of the co-t-structure.
\item $\T^{w\geq l}{\coloneqq}\Sigma^{-l}\T^{w\geq 0}$ and $\T^{w\leq l}{\coloneqq}\Sigma^{-l}\T^{w\leq 0}$.
\item For any $i,j\in\mathbb Z$, $i\leq j$, we define $\T^{[i,j]}{\coloneqq}\T^{w\leq j}\bigcap \T^{w\geq i}$.
\item The co-t-structure $(\T^{w\leq 0}, \T^{w\geq 0})$ is called {\em bounded} if 
\[
\T=\bigcup_{l\in \mathbb Z}\T^{w\leq l}=\bigcup_{l\in\mathbb Z}\T^{w\geq l}.
\]
\end{itemize} 
\end{definition}
\begin{example}[\cite{Bondarko10, Pauksztello08}]
Let $\B$ be an additive category. 
We denote by $\mathcal H(\B)^{\leq 0}$ (resp.~$\mathcal H(\B)^{\geq 0}$) the full subcategory consisting of complexes that are homotopy equivalent to complexes concentrated in degrees $\leq 0$ (resp.~$\geq 0$). 
The pair of subcategories $(\mathcal H(\B)^{\leq 0},\mathcal H(\B)^{\geq 0})$ is a co-t-structure on $\mathcal H(\B)$.
Its coheart is the Karoubi-closure of $\B$ in $\mathcal H(\B)$, i.e.~the full subcategory consisting of the objects $X\in \mathcal H(\B)$ which are retracts of objects in $\B$. 
For example, suppose $p: Y^0\rightarrow Y^1$ is an epimorphism in $\B$ with a left inverse $s:Y^1\rightarrow Y^0$. 
Then the complex
 \[
 \cdots\rightarrow0\rightarrow 0\rightarrow Y^0\xrightarrow{p} Y^1\rightarrow 0\rightarrow \cdots
 \]
where $Y^0$ is in degree 0, is a retract of the stalk complex $Y^0$ and hence belongs to the coheart. 
Each such complex belongs to $\B$ if and only if $\B$ is weakly idempotent complete.
\end{example}
Let $\A$ be a dg $k$-category. 
Recall that we denote by $\pretr(\A)$ its pretriangulated hull and by $\tr(\A)$ the triangulated category $H^0(\pretr(\A))$ and that the Yoneda embedding induces a fully faithful triangle functor $\tr(\A)\rightarrow \D(\A)$.

\begin{proposition}$($\cite[Proposition 6.2.1]{Bondarko10}$)$ \label{cot}
Suppose $\A$ is a strictly connective dg category. 
Put $\T=\tr(\A)$. 
Let $\T^{\leq 0}$ (resp.~$\T^{\geq 0}$) be the full subcategory of $\T$ which consists of objects isomorphic to those of the form $A=(\oplus_{i=1}^{n} A_i[r_i],q)$ for $n\geq 0$ and $r_i\geq 0$ (resp. $r_i\leq 0$) for each $i$. 
Then $(\T^{\leq 0}, \T^{\geq 0})$ is a co-t-structure on $\T$, the {\em canonical co-t-structure}. 
\end{proposition}

\subsection{Homotopy cartesian squares for triangulated categories}
Let $\T$ be a triangulated category. Consider a commutative square in $\T$
\begin{equation}\label{triangulatedpullback}
\begin{tikzcd}
Y \arrow[r,"{f}"] \arrow[d,"g" swap] & Z \ar[d,"g'"] \\
Y' \arrow[r,"f'"swap] & Z'.
\end{tikzcd}
\end{equation}
It is {\em homotopy cartesian} (\cite[Definition 1.4.1]{Neeman99}) if there is a (distinguished) triangle 
\[
\begin{tikzcd}
Y\ar[r,"{(g,-f)^{\intercal}}"] &Y'\oplus Z\ar[r,"{(f'\ g')}"]&Z'\ar[r,"{\partial}"] &\Sigma Y
\end{tikzcd}
\]
for some $\partial:Z'\rightarrow \Sigma Y$. 
For a general commutative square as above, let $u=(g, -f)^{\intercal}$. 
We complete it to a triangle
\[
\begin{tikzcd}
Y\ar[r,"u"]&Y'\oplus Z\ar[r,"v"]&W\ar[r,"w"]&\Sigma Y.
\end{tikzcd}
\]
Then we have a map $t:W\rightarrow Z'$ such that $(f'\ \ g')=t\circ v$. 
However this map is not canonical and we do not know if this map is an isomorphism. 
Therefore we cannot determine whether the commutative square is homotopy cartesian from this map.

It turns out that we have a satisfactory definition of homotopy cartesian squares provided $\T$ is algebraic, i.e.~$\T$ is equivalent, as a triangulated category, to $H^0(\A)$ for a pretriangulated dg category $\A$. 
In this case, the adjunction $(L,R)$ in the diagram \ref{adjointtriple} restricts to an adjuntion 
\[
\begin{tikzcd}
\rep(k(\mathrm{Sq}),\A)\ar[r,shift left=1ex,"L"]& \rep(k(\mathrm{Cosp}),\A)\ar[l, shift left=1ex,"R"].
\end{tikzcd}
\]
\begin{definition}Let $\A$ be a pretriangulated dg category. 
An object $X\in \rep(k(\mathrm{Sq}),\A)$ is {\em homotopy cartesian} if the adjunction morphism $X\rightarrow RLX$ is invertible.
\end{definition}
Recall that we identify $X$ with the diagram \ref{D(Sq)}
\[
\begin{tikzcd}
X_{00}\ar[r,"f"]\ar[d,"g"swap]&X_{01}\ar[d,"j"]\\
X_{10}\ar[r,"k"swap]&X_{11}
\end{tikzcd}\;.
 \]
The object $X$ is homotopy cartesian if and only if the canonical morphism
\[
X_{00}\xrightarrow{(f,g,0)^{\intercal}}\Sigma^{-1}\mathrm{Cone}((-j,k))
\] 
is an isomorphism in $\D(\A)$ (cf.~the proof of Lemma \ref{adj}). 
Note that this implies $\Dia(X)\in \Fun(\mathrm{Sq}, H^0(\A))$ is a homotopy cartesian square in $\T=H^0(\A)$. 

Conversely, let $Y\in \Fun(\Sq, H^0(\A))$ be a square in $H^0(\A)$. If it is a homotopy cartesian square, then it is isomorphic to $\Dia(X)$ for some homotopy cartesian object $X\in \rep(\Sq,\A)$. 
Indeed, by Lemma \ref{epi} below, the object $\Res(Y)\in \Fun(\Cosp, H^0(\A))$ is isomorphic to $\Dia(V)$ for some $V\in \rep(\Cosp, \A)$.
Then the object $X\coloneqq R(V)\in \rep(\Sq,\A)$ is homotopy cartesian and satisfies $\Dia(X)\iso Y$.

Recall that for a small category $\I$, the canonical diagram functor 
\[
\Dia:\D(\I)\rightarrow \mathrm{Fun}(\I,\D({\A}))
\]
sends an object $X:\I\rightarrow \C(\A)$ to the corresponding diagram $\Dia(X)=\pi \circ X$ where $\pi:\C(\A)\rightarrow \D(\A)$ is the canonical quotient functor. 

Recall that, for a category $\C$, we denote by $\Iso(\C)$ the class of isomorphism classes of objects in $\mathcal{C}$. 
The following lemma is well-known. 
Here we include a proof for the convenience of the reader.
\begin{lemma}\label{epi}
Let $\I$ be the path category $P(Q)$ of a quiver $Q$ (without relations).
The functor $\Dia:\D(\I)\rightarrow \Fun(\I,\D(\A))$ is an epivalence, i.e.~it is full and dense and detects isomorphisms. 
In particular, it induces a bijection 
\[
\Iso(\D(I))\xrightarrow{\sim}\Iso(\mathrm{Fun}(\I,\D(\A))).
\]
\end{lemma}

\begin{proof}
The functor $\Dia$ clearly detects isomorphisms.
Let us first show that it is dense.

Let $S: \I\rightarrow \D(\A)$ be an object in $\Fun(\I, \D(\A))$.
For each object $i\in \I$, we take a cofibrant resolution $\theta_i: X_i\rightarrow S(i)$.
Then for each arrow $\alpha:i\rightarrow j$ in $Q$, we have a morphism $X_{\alpha}:X_i\rightarrow X_{j}$ in $\C(\A)$ such that $\theta_{j}\circ \pi(X_{\alpha})\circ \theta_{i}^{-1}=S(\alpha)$ where $\pi:\C(\A)\rightarrow \D(\A)$ is the quotient functor.
This extends to a functor $X:\I=kQ\rightarrow \C(\A)$ such that $\Dia(X)=\pi \circ X\iso S$. 
This shows that the functor $\Dia$ is dense.

Let us now show that it is full.

Let $\C$ be an additive category with coproducts and $\I$ a small category.
For each object $i\in \I$, the map sending a functor $F:\I \rightarrow \C$ to $F(i)$ determines the {\em evaluation functor at $i$}
\[
i^{*}:\Fun(\I,\C)\rightarrow \C,\;\; F\mapsto F(i).
\] 
It admits a left adjoint 
\[
i_{!}:\C\rightarrow \Fun(\I,\C)
\]
sending an object $X$ to the functor $j\mapsto \coprod_{p:i\rightarrow j} X$, where $p$ runs through morphisms from $i$ to $j$ in $\I$.

Suppose $\I$ is the path category $P(Q)$ of a quiver $Q$.
Then for each $F\in\Fun(\I,\C)$, we have the following sequence
\[
0\rightarrow \coprod_{\alpha: i\rightarrow j \text{ in $Q$}}j_{!}i^{*}F\rightarrow \coprod_{i\in Q}i_{!}i^{*}F\rightarrow F\rightarrow 0.
\]
which splits when evaluated at each object $i\in \I$.

Now suppose $\C=\C(\A)$.
The above sequence is a short exact sequence in $\Fun(P(Q),\C(\A))\iso\C(\A\otimes k(P(Q))^{op})$.
Thus it induces a triangle in $\D(\A\otimes k(P(Q))^{op})$
\begin{equation}\label{triangle}
\bigoplus_{\alpha:i\rightarrow j \text{ in $Q$}}j_!i^*F\rightarrow \bigoplus_{i\in Q}i_{!}i^{*}F\rightarrow F\rightarrow \Sigma\bigoplus_{\alpha:i\rightarrow \text{ in $Q$}}j_!i^*F.
\end{equation}
Notice that the adjunction $(i_{!},i^{*})$ is a Quillen adjunction with respect to the projective model structures on both categories.
Thus we have
\[
\RHom(i_{!}X,F)\iso\RHom(X,i^*F).
\]

Let $G$ be another object in $\Fun(\I,\C(\A))$.
If we apply $\RHom(-, G)$ to the triangle \ref{triangle}, we otain the following triangle in $\D(k)$
\[
\RHom(F,G)\rightarrow \prod_{i\in Q}\RHom(F(i),G(i))\xrightarrow{\phi} \prod_{\alpha:i\rightarrow j\text{ in $Q$}}\RHom(F(i),G(j))\rightarrow \Sigma\RHom(F,G).
\]
By direct inspection, we see that the kernel of $H^0(\phi)$ is the space of morphisms of diagrams from $\Dia(F)$ to $\Dia(G)$.
Thus by applying $H^0$ to the triangle we see that the map
\[
\Hom_{\D(\A\otimes k(P(Q))^{op})}(F,G)\rightarrow \Hom_{\Fun(P(Q),\D(\A))}(\Dia(F), \Dia(G))
\]
 is a surjection.
\end{proof}
Let $\A$ be a pretriangulated dg category. 
In summary, we have the following diagram of functors
\[
\begin{tikzcd}
\rep(k(\Sq),\A)\ar[r,"\Dia"]\ar[d,"L"swap,shift right=1ex]&\Fun(\Sq, H^0(\A))\ar[d,"\Res"]\\
\rep(k(\Cosp),\A)\ar[r,"\Dia"swap]\ar[u,"R"swap, shift right=1ex]&\Fun(\Cosp,H^0(\A))
\end{tikzcd}
\]
where \begin{itemize}
\item $\Res\circ\Dia=\Dia\circ L$ and $\Res\circ\Dia\circ R\iso \Dia$; 
\item by Lemma \ref{epi}, the bottom horizontal functor induces a bijection between the isomorphism classes of objects; 
\item the functor $R$ is fully faithful and by definition, an object $X$ in $\rep(\Sq,\A)$ is homotopy cartesian if and only if it is in the essential image of $R$; 
\item for each homotopy cartesian square $X$ in $H^0(\A)$, there exists a non-unique homotopy bicartesian object in $\A$ whose image under $\Dia$ is isomorphic to $X$;
\item the right vertical functor does not admit a right adjoint because maps in the triangulated category $H^0(\A)$ do not have kernels in general.
\end{itemize}
\subsection{Homotopy pullbacks/pushouts}
In this subsection, we extend the definitions of homotopy (co)cartesian squares to the general case when $\A$ is not necessarily pretriangulated. 
\begin{definition}\label{maindef}
An object $X\in \rep(\mathrm{Sq})$ is a \text{\em{{homotopy cartesian square}}} with respect to $\A$ if the canonical map $X_{00}\rightarrow \Sigma^{-1}\mathrm{Cone}((-j,k))$ induces an isomorphism   
\[
\tau_{\leq 0}\RHom(A^{\wedge},X_{00})\rightarrow \tau_{\leq 0}\RHom(A^{\wedge}, \Sigma^{-1}\mathrm{Cone}((-j,k)))
\]
in $\D(k)$ for each $A$ in $\A$.
\end{definition}
\begin{remark}\label{subcategory}
Let $\A'$ be a full dg subcategory of $\A$. 
Let  $F:\rep(\Sq,\A')\rightarrow \rep(\Sq,\A)$ be the inclusion functor induced by the inclusion dg functor $\A'\rightarrow \A$.
Let $X$ be an object in $\rep(\Sq,\A')$. 
If $F(X)$ is homotopy cartesian with respect to $\A$, then $X$ is homotopy cartesian with respect to $\A'$. 
\end{remark}
In the rest of the chapter, we will only use this relative version of homotopy (co)cartesian squares.

For an object $X\in \rep(\mathrm{Sq})$, by Lemma \ref{cha} the unit morphism $X\rightarrow RL X$ induces a canonical morphism $X\rightarrow \tau_{\leq 0} RLX$ in $\D(\mathrm{Sq})$.

\begin{lemma} \label{adj}Let $X$ be an object in $\rep(\mathrm{Sq})$. If the canonical map $X\rightarrow \tau_{\leq 0} RLX$ is an isomorphism, then $X$ is homotopy cartesian. If $\A$ is a connective dg category, then the converse also holds.
\end{lemma}
\begin{proof}Let $X$ be the object
\[ \begin{tikzcd}A\ar[r,"f"]\ar[d,"g"swap,""{name=2}]&B\ar[d,"j"]\\
C\ar[r,"k"swap]&D
\end{tikzcd}
\]
in $\rep(\mathrm{Sq})$. Let $p:\Cone(-\Id_{\Sigma^{-1}D})\rightarrow D$ be the canonical projection. 

Then the canonical map $X\rightarrow RLX$ can be described by the following commutative diagram in $\C(\A)$:
\[
\begin{tikzcd}
A\ar[r,"\begin{bmatrix}f\\g\end{bmatrix}"]\ar[d,"u"swap]&B\oplus C\ar[d,"\begin{bmatrix}0\ 1\ 0\\0\ 0\ 1\end{bmatrix}"]\ar[r,"{[}-j{,}k{]}"]&D\ar[d,equal]\\
K\ar[r, tail]&\Cone(-\Id_{\Sigma^{-1}D})\oplus B\oplus C\ar[r,"{[}p{,}-j{,}k{]}"swap, two heads]&D
\end{tikzcd}
\]
where the second row is a short exact sequence in $\C(\A)$. 
By Lemma \ref{cha}, we have $(\tau_{\leq 0} RLX)_{00}\iso \tau_{\leq 0}K$.

We have a canonical commutative diagram in $\D(\A)$
\[
\begin{tikzcd}
A\ar[r]\ar[d,"u"swap]&\Sigma^{-1}\Cone(-j,k)\ar[d,"\sim"]\\
K\ar[r,"\sim"]&\Sigma^{-1}\Cone(p,-j,k)
\end{tikzcd}.
\]
The map $u$ induces a canonical map $A\rightarrow \tau_{\leq 0}K$ in $\D(\A)$. 
It is an isomorphism in $\D(\A)$ if and only if the canonical map $X\rightarrow \tau_{\leq 0}RLX$ is an isomorphism in $\D(\Sq)$.

The object $X$ is a homotopy cartesian square if and only if $u:A\rightarrow K$ induces an isomorphism in $\D(k)$ 
\[
\tau_{\leq 0}\RHom(A'^{\wedge},A)\iso \tau_{\leq 0}\RHom(A'^{\wedge}, K) 
\]
for each $A'$ in $\A$. 

Note that we have
\[
\tau_{\leq 0}\RHom(A'^{\wedge},\tau_{\leq 0}K) \iso \tau_{\leq 0}\RHom(A'^{\wedge}, K).
\]
So if the canonical map $A\rightarrow \tau_{\leq 0}K$ is an isomorphism in $\D(\A)$, then $X$ is homotopy cartesian. 

When the dg category $\A$ is connective, we have 
\[
\tau_{\leq 0}\RHom(A'^{\wedge},A)\iso \RHom(A'^{\wedge},A)
\]
and
\[
\tau_{\leq 0}\RHom(A'^{\wedge},\tau_{\leq 0}K)\iso \RHom(A'^{\wedge},\tau_{\leq 0}K).
\]
So the converse also holds.

\end{proof}

\begin{example}
Each graded-split short exact sequence  $0\rightarrow A\xrightarrow{f} B\xrightarrow{j} C\rightarrow 0$ in $\mathcal A$ gives rise to a homotopy cartesian square 
\[
\begin{tikzcd}
A^{\wedge}\arrow[r,"f^{\wedge}"]\arrow[d]&B^{\wedge}\arrow[d,"j^{\wedge}"]\\
0\arrow[r]&C^{\wedge}
\end{tikzcd}.
\]
\end{example}
\begin{example}\label{ordinary}
Let $\A$ be an additive category which we consider as a dg category concentrated in degree zero. 
Then the square
\[
 \begin{tikzcd}
 A^{\wedge}\arrow[r,"f^{\wedge}"]\arrow[d]&B^{\wedge}\arrow[d,"j^{\wedge}"]\\0\arrow[r]&C^{\wedge}
 \end{tikzcd}
 \]
  with $A,B,C\in \A$ is homotopy cartesian if and only if in the sequence 
  \[
  \begin{tikzcd}0\arrow[r]&A\arrow[r,"f"]&B\arrow[r,"j"]&C\end{tikzcd},
  \]
 $f$ is a kernel of $j$. Indeed we identify $\pretr(\A)$ with $\C^{b}_{dg}(\A)$. Then the object $\Sigma^{-1}\mathrm{Cone}(j)$ is identified with the complex $0\rightarrow B\rightarrow C\rightarrow 0$ where $B$ is in degree $0$. 
  
  For each $A'\in \A$, the space $\Hom(A'^{\wedge},\Sigma^{-1}\mathrm{Cone}(g))$ is then identified with the kernel of the map $\Hom_{\A}(A',B)\rightarrow \Hom_{\A}(A',C)$.
\end{example}

The following lemma was obtained independently by Genovese, Lowen and Van den Bergh, c.f. \cite[Lemma 2.5.5]{GenoveseLowenVandenBergh22}. 
\begin{lemma}\label{trun}
Let $\B$ be a connective dg category. Then the natural functor 
\[
F: \rep(\B,\tau_{\leq 0}\A)\rightarrow\rep(\B,\A)
\]
is an equivalence of categories. 
Let $H:\D(\A\otimes^{\mathbb L}\B^{op})\rightarrow \D(\tau_{\leq 0}\A\otimes^{\mathbb L} \B^{op})$ be the restriction functor. 
Then a quasi-inverse of $F$ is given by $G=\tau_{\leq 0}\circ H$. 
Moreover, for any $X, Y\in \rep(\B,\tau_{\leq 0}\A)$, we have $\tau_{\leq 0}\RHom(X, Y)\xrightarrow{\sim}\tau_{\leq 0}\RHom(FX,FY)$. 
\end{lemma} 
\begin{proof}
Let $\dgcat_{\leq 0}$ be the category of small strictly connective dg categories. 
Then we have the following adjunction
\[
\begin{tikzcd}
\dgcat_{\leq 0}\ar[r,"i", shift left =0.8ex]&\dgcat\ar[l,"\tau_{\leq 0}", shift left=0.8ex]
\end{tikzcd}
\]
where the left adjoint $i:\dgcat_{\leq 0}\rightarrow \dgcat$ is the inclusion functor and the right adjoint $\tau_{\leq 0}:\dgcat\rightarrow \dgcat_{\leq 0}$ sends a small dg category $\A$ to $\tau_{\leq 0}\A$.

Let $\Hqe_{\leq 0}$ be the localization of $\dgcat_{\leq 0}$ at the class of quasi-equivalences of strictly connective dg categories.
The functors $i$ and $\tau_{\leq 0}$ clearly preserve quasi-equivalences. 
Hence by the universal properties of localizations of categories, the above adjunction induces the following adjunction
\[
\begin{tikzcd}
\Hqe_{\leq 0}\ar[r,"i", shift left =0.8ex]&\Hqe\ar[l,"\tau_{\leq 0}", shift left=0.8ex]
\end{tikzcd}
\]
where $i$ is fully faithful.

We claim that the canonical morphism in $\Hqe_{\leq 0}$
\[
\tau_{\leq 0}\rep_{dg}(\B,\tau_{\leq 0}\A)\rightarrow \tau_{\leq 0}\rep_{dg}(\B,\A)
\]
is an isomorphism.

Let $\C$ be an arbitrary small strictly connective dg category.
Recall that the tensor product of strictly connective dg categories remains strictly connective and 
that each connective dg category admits a cofibrant replacement which is strictly connective.
Recall that $\rep_{dg}(\B,\A)$ is the internal Hom of the monoidal category $(\Hqe,-\otimes^{\mathbb L}-)$.
Then we have
\[
\begin{aligned}
\Hom_{\Hqe_{\leq 0}}(\C,\tau_{\leq 0}\rep_{dg}(i(\B),i(\tau_{\leq 0}\A)))
&=\Hom_{\Hqe}(i(\C),\rep_{dg}(i(\B),i(\tau_{\leq 0}\A)))\\
&=\Hom_{\Hqe}(i(\C)\otimes^{\mathbb L}i(\B),i(\tau_{\leq 0}\A))\\
&=\Hom_{\Hqe_{\leq 0}}(\C\otimes^{\mathbb L}\B,\tau_{\leq 0}\A)\\
&=\Hom_{\Hqe}(i(\C)\otimes^{\mathbb L}i(\B),\A)\\
&=\Hom_{\Hqe}(i(\C),\rep_{dg}(i(\B),\A))\\
&=\Hom_{\Hqe_{\leq 0}}(\C,\tau_{\leq 0}\rep_{dg}(i(\B),\A)).
\end{aligned}
\]
\end{proof}
The following lemma allows us to reduce, in many cases, to the connective case when considering homotopy (co)cartesian squares.
\begin{lemma}\label{truncationhomotopycartesian}
Consider the equivalence of categories $F:\rep(\Sq,\tau_{\leq 0}\A)\iso \rep(\Sq,\A)$ as in Lemma \ref{trun}.
An object $X$ in $\rep(\mathrm{Sq},\tau_{\leq 0}\A)$ is homotopy cartesian
 if and only if its image under $F$ in $\rep(\mathrm{Sq},\A)$ is homotopy cartesian.
\end{lemma}
\begin{proof} Let $X$ be the commutative square of quasi-representable dg $\tau_{\leq 0}\A$-modules given by the diagram \ref{D(Sq)}.
Consider the following diagram 
\[
\begin{tikzcd}
X_{00}\ar[d,"{[}f{,}g{,}0{]}^{\intercal}"swap]\ar[rd,"{[}f{,}g{]}^{\intercal}"]&&&\\
\Sigma^{-1}\mathrm{Cone}(\psi)\ar[r]&X_{10}\oplus X_{01}\ar[r,"\psi={[}-j{,}k{]}"]& X_{11}\ar[r]&\mathrm{Cone}(\psi)
\end{tikzcd}\;.
\] 
By applying $\Hom_{\D(\tau_{\leq 0}\A)}(A^{\wedge},-)$ to the above triangle for each $A\in \tau_{\leq 0}\A$, we get the associated long exact sequence 
\[
\begin{tikzcd}
&&\Hom(A^{\wedge},X_{00})\ar[d]&&\\
\dots\ar[r]&\Hom(A^{\wedge},\Sigma^{-1}X_{11})\ar[r]&\Hom(A^{\wedge},\Sigma^{-1}\Cone(\psi)) \ar[r]&\Hom(A^{\wedge},X_{10}\oplus X_{01})\ar[r]&\dots
\end{tikzcd}
\]
We have $\Hom(A^{\wedge}, \Sigma^{-n}X_{ij})\iso\Hom(FA^{\wedge},\Sigma^{-n}FX_{ij})$ for $n\geq 0$ where $F:\D(\tau_{\leq 0}\A)\rightarrow \D(\A)$ is the induction functor.
By the Five-Lemma, it is then clear that an object $X$ in $\rep(\mathrm{Sq},\tau_{\leq 0}\A)$ is homotopy cartesian if and only if its image under $F$ in $\rep(\mathrm{Sq},\A)$ is homotopy cartesian.
\end{proof}
\begin{definition}\label{homotopypullbackdef}
An object $S$ in $\rep(\mathrm{Cosp})$ is said to admit a {\em homotopy pullback} if there is a homotopy cartesian square $X$ in $\rep(\mathrm{Sq})$ with an isomorphism $\phi:LX\rightarrow S$ in $\D(\mathrm{Cosp})$. 
In this case, the pair $(X,\phi)$ is called a homotopy pullback of $S$. 
We will also abuse the notation and say that $X$ is a {\em homotopy pullback} of $S$.
\end{definition}

\begin{remark}The isomorphism $\varphi:LX\rightarrow S$ in the definition of a homotopy pullback is in general not the identity. 
For example, let $\A$ be the dg category of two-term complexes of finitely generated projective modules over a finite dimensional algebra $A$. 
The cospan
\[
\begin{tikzcd}
&A\ar[d]\\
0\ar[r]&\Cone(\Id_A)
\end{tikzcd}
\]
where $A$ denotes the corresponding stalk complex, does not admit a homotopy pullback $(X,\phi)$ with $\phi$ being the identity. 
It admits a homotopy pullback with $X$ given by
\[
\begin{tikzcd}
A\ar[r,equal]\ar[d]&A\ar[d]\\
0\ar[r]&0
\end{tikzcd}
\]
and $\phi^{-1}$ given by the obvious morphism 
\[
\begin{tikzcd}
         &A\ar[d]                                  &\ar[d, white,""{name=4}]&  &\ar[d,white,""{name=3}]&           &A\ar[d]\ar[r,blue,from=4,to=3]\\
0\ar[r]&\Cone({\Id_A})       &\                                     &  &\                                    &0\ar[r] &0
\end{tikzcd}.
\]  
\end{remark}

\begin{proposition}\label{pullbackunique}
The functor $L:\D(\mathrm{Sq})\rightarrow \D(\mathrm{Cosp})$ restricted to
 the full subcategory consisting of homotopy cartesian squares is fully faithful. 
In particular, the homotopy pullback of an object $S$ in $\rep(k\mathrm{Cosp},\A)$, 
if it exists, is unique up to a unique isomorphism. 
\end{proposition}
\begin{proof}
We first show the case when $\A$ is connective. 

Let $X$ and $X'$ be two homotopy cartesian squares. 
We first show fullness.
Let $f:LX\rightarrow LX'$ be a morphism in $\rep(\Cosp,\A)$.

From the adjunction 
\begin{tikzcd}
\D({\mathrm{Sq}})\arrow[shift left=1ex, r, "L"]&\D(\mathrm{Cosp})\arrow[shift left=1ex, l,"R"] 
\end{tikzcd},
we have a morphism $g:X\rightarrow RLX'$.

Since the object $X'$ is homotopy cartesian, by Lemma \ref{adj}, 
the map $X\rightarrow RLX'$ induces a morphism $h:X\rightarrow \tau_{\leq 0}RLX'\xleftarrow{\sim} X'$, as the following diagram shows 
\[
\begin{tikzcd}
&X'\ar[d,dashed,"\sim"]\ar[rd]&\\
X\ar[ru,"h"red,red]\ar[r]&\tau_{\leq 0}RLX'\ar[r]&RLX'.
\end{tikzcd}
\]
By applying the functor $L$ to the above diagram, we see that $L(h)=f$.
This shows the fullness.

Next we show faithfulness.
Let $h:X\rightarrow X'$ be a morphism between homotopy cartesian squares.
We have the following diagram in $\D(\Sq)$
\[
\begin{tikzcd}
X\ar[r]\ar[d,"h"swap]&\tau_{\leq 0}RLX\ar[d]\ar[r]&RLX\ar[d,"RL(h)"]\\
X'\ar[r,"\sim"]&\tau_{\leq 0}RLX'\ar[r]&RLX'\\
\end{tikzcd}
\] 
Therefore if the morphism $L(h)$ is zero, then $h$ is also zero.

For the general case, consider the following diagram of functors
\[
\begin{tikzcd}
\rep(k\Sq,\tau_{\leq 0}\A)\arrow[rd, hook]\arrow[shift right=1ex,ddd,"L"swap]\arrow[shift left=1ex,rrr,"F"]&&&\rep(k\Sq,\A)\arrow[shift left=1ex,lll,"G"]\arrow[shift right=1ex, ddd, "L'"]\arrow[ld,hook]\\
&\D( \tau_{\leq 0}\A\otimes k(\Sq^{op}))\arrow[r,  shift left=1ex, "F"] \arrow[shift right=1ex,d,"L"swap]  &\D(\A\otimes k(\Sq^{op}))\arrow[l, shift left=1ex, "H"]\arrow[shift right=1ex,d,"L'"swap]&\\
&\D(\tau_{\leq 0}\A\otimes k(\Cosp^{op}))\arrow[shift left=1ex, r, "F'"]\arrow[shift right=1ex,u,"R"swap]&\D(\A\otimes k(\Cosp^{op}))\arrow[shift left=1ex, l, "H'"]\arrow[shift right=1ex,u,"R'"swap]&\\
\rep(k\Cosp,\tau_{\leq 0}\A)\arrow[ru,hook]\arrow[shift left=1ex,rrr,"F'"]&&&\rep(k\Cosp,\A)\arrow[shift left=1ex, lll,"G'"]\arrow[lu,hook]
\end{tikzcd}
\]

Let $X$ and $X'$ be two objects in $\rep(k\Sq,\A)$ which are homotopy cartesian. 

We have a canonical isomorphism of triangle functors $L\circ H\iso H'\circ L'$. 
Since the restriction functor $L$ commutes with the truncation functors $\tau_{\leq 0}$ on $\D(\tau_{\leq 0}\A\otimes k(\Sq^{op}))$ and on $\D(\tau_{\leq 0}\A\otimes k\Cosp^{op})$, and since the functor $G$ (resp.~$G'$) is a composition of $\tau_{\leq 0}$ and $H$ (resp.~$H'$), we have a natural isomorphism of functors $G'\circ L'
\xrightarrow{\sim}L\circ G:\rep(\Sq,\A)\rightarrow \rep(k\Cosp,\tau_{\leq 0}\A)$. 

Let $f:L'X\rightarrow L'X'$ be a morphism in $\rep(k\Cosp,\A)$. 
The map $G'(f):G'L'X\rightarrow G'L'S$ then yields a canonical morphism $g:LGX \xrightarrow{\sim} LGX'$. 
From the connective case, there is a morphism $h':GX\rightarrow GX'$ such that $g=L(h')$.
Since $G$ is an equivalence of categories, there exists a morphism $h:X\rightarrow X'$ such that $h'=G(h)$.
Since $G'$ is also an equivalence of categories, we see that $f=L'(h)$.

Similarly, one shows that if a morphism $h:X\rightarrow X'$ is such that $L'(h)=0$, then we have $h=0$.
\end{proof}
\begin{remark}[Universal property of homotopy pullbacks]
\label{pullbackuniversal}
Similarly, one can show that given a homotopy cartesian square $X'$ and an object $X$ in $\rep(\Sq)$, if we have a map $\varphi: LX\rightarrow LX'$ in $\rep(\mathrm{Cosp})$, then we have a canonical map $\theta:X\rightarrow X'$ in $\rep(\mathrm{Sq})$ such that $L(\theta)=\varphi$. 
\end{remark}
Consider an object $X\in \D(\mathrm{Sq})$ as follows
\[
 \begin{tikzcd}
 A\ar[r,"f"]\ar[d,"g"swap]&B\ar[d,"j"]\\
 N\ar[r,"k"swap]&C 
 \end{tikzcd}
 \]
where $N$ is acyclic.

\begin{lemma} 
We have an isomorphism in $\D(\mathrm{Sq})$ from $X$ to 
\[
\begin{tikzcd}
A\ar[r,"\begin{bmatrix}f\\-g\end{bmatrix}"]\ar[d]&B\oplus N\ar[d,"{[}j{,}k{]}"]\\
0\ar[r]&C 
\end{tikzcd}
\] 
which restricts to $\Id_{A}$ and $\Id_{C}$. 
\end{lemma}
\begin{proof}
Indeed we can write an isomorphism in $\D(\Sq)$ explicitly as follows:
\[
\begin{tikzcd}
A\ar[r,"f"]\ar[d,"g"swap,""{name=2}]\ar[bend left=10ex,rr,blue,"\begin{bmatrix}1\\-g\end{bmatrix}"]&B\ar[bend left=10ex, rr,blue,"\begin{bmatrix}1\\0\end{bmatrix}"]\ar[d,"j"]& A\oplus N\ar[d,"{[}0\text{,}-1{]}"]\ar[r,"\begin{bmatrix}f\ \ 0\\g\ \ 1\end{bmatrix}"]&B\oplus N\ar[d,"{[}j\text{,}-k{]}"]&A\ar[d]\ar[ll,bend right=10ex,red,"\begin{bmatrix}1\\0\end{bmatrix}"swap]\ar[r,"\begin{bmatrix}f\\g\end{bmatrix}"]&B\oplus N\ar[ll,bend left=6ex, red,"\ \ \ \ \ \Id"]\ar[d,"{[}j\text{,}-k{]}"]\\
N\ar[r,"k"swap]\ar[bend right=6ex,rr,blue,"\Id"swap]&C\ar[bend right=6ex,rr,blue,"\Id"swap]&N\ar[r,"k"swap]&C&0\ar[ll,bend left=6ex, red,"\mathrm{0}"]\ar[r]&C\ar[ll,bend left=6ex,red,"\Id"] 
\end{tikzcd}.
\]
\end{proof}
\begin{remark}
Therefore if an object $X$ in $\rep(\mathrm{Sq},\A)$ is such that $N$ is acyclic, we will assume that $N$ is actually the zero dg module.
\end{remark}

Recall that we have the pair of duality functors \[\begin{tikzcd}
\rep(\I,\A)\arrow[shift left=1ex,rr,"\mathrm{RHom}(- \text{,} \mathcal{A})"]&&\rep(\I^{op},\A^{op})^{op}\arrow[shift left=1ex, ll,"\RHom(-\text{,}\A)"]
\end{tikzcd}\]
Then we can form the dual definitions of a homotopy cocartesian square with respect to $\A$ and of a homotopy pushout of an object in $\D(\mathrm{Sp})$.
\begin{definition}
An object $X\in \rep(\mathrm{Sq})$ given by the diagram \ref{D(Sq)} is a \text{\em{{homotopy}}} \text{\em{{cocartesian square with respect to $\A$}}} if the canonical map $\mathrm{Cone}((f,g)^{\intercal})\rightarrow X_{11}$ induces an isomorphism in $\D(k) $  
\[
\tau_{\leq 0}\RHom(X_{11},A^{\wedge})\rightarrow \tau_{\leq 0}\RHom(\mathrm{Cone}((f,g)^{\intercal}),A^{\wedge})
\]
for each $A$ in $\A$.
\end{definition}
Consider the adjunction \begin{tikzcd}[column sep=huge]
 \D(\mathrm{Sp})
  \arrow[shift left=1ex,"L"]{r}[name=D]{} &
 \D(\mathrm{Sq})
 \arrow[shift left=1ex,"R"]{l}[name=U]{}
 \end{tikzcd}
induced by the inclusion $\mathrm{Sp}\rightarrow\mathrm{Sq}$. 
\begin{definition}
An object $S$ in $\rep(\mathrm{Sp})$ is said to admit a {\em homotopy pushout} if there is a homotopy cartesian square $X$ in $\rep(\mathrm{Sq},\A)$ with an isomorphism $\phi:S\rightarrow RX$ in $\D(\mathrm{Sp})$. In this case, the pair $(X,\phi)$ is called a {\em homotopy pushout of} $S$. We will also abuse the notation and say that $X$ is a {\em homotopy pushout} of $S$.
\end{definition}

Let us give another way of presenting the duality between homotopy pushouts and homotopy pullbacks.
First we observe that $k\Sq$ is quasi-equivalent to the cofibrant dg $k$-path category 
$\Lambda=\Lambda(\Delta^1 \times \Delta^1)$ given by the differential graded quiver
\[
\begin{tikzcd}
0\ar[r,"f"]\ar[rd,"l"{yshift=-3pt}]\ar[rd,bend right=5ex,"h_1"{yshift=6pt,xshift=2pt,red},red,swap]\ar[rd,bend left=5ex,"h_2"{yshift=-5pt, red},red]\ar[d,"g"swap]&1\ar[d,"j"]\\
2\ar[r,"k"swap]&3
\end{tikzcd}
\]
where $|f|=|j|=|l|=|g|=|k|=0$, $|h_1|=|h_2|=-1$, $d(h_1)=l-kg$ and $d(h_2)=l-jf$, cf.~the discussions after 
Remark \ref{homotopylimit}.
Then objects in $\rep(\Lambda,\A)$ are isomorphic to objects whose terms are representable,  
e.g.~one can argue as in Lemma \ref{Comrepresentable}.
Such objects canonically give objects in $\rep(\Lambda,\A^{op})$ by simply reversing the arrows and
composing with the obvious isomorphism $\Lambda \iso \Lambda^{op}$.
In such a way, one defines an object in $\rep(\Lambda,\A)$ to be a homotopy pushout if and only if the corresponding object in $\rep(\Lambda,\A^{op})$ is a homotopy pullback.

\begin{example}\label{example:2-term}
Let $A$ be a finite-dimensional $k$-algebra. 
Let $\mathcal H^{[-1,0]}(\proj A)$ be the full subcategory of $\T=\mathcal H^b(\proj A)$ consisting of two-term complexes $P^{-1} \to P^0$ of finitely generated projective $A$-modules. 
Let $\A=\C^{b}_{dg}(\proj A)$ be the canonical enhancement of
${\mathcal H}^b(\proj A)$ and $\A'$ its full dg subcategory on the objects $P^{-1} \to P^0$.
For an $A$-module $M$, we denote by $M^*$ its $A$-dual.
One can show that there exists a homotopy bicartesian square in $\A'$ which is not homotopy bicartesian in 
$\A$ if and only if there exists an $A$-module $M$ satisfying the following properties
\begin{itemize}
\item[(1)] $\pd_A M= 1$, and
\item[(2)] $\pd_{A^{op}} M^*=1$, and
\item[(3)] $M$ is reflexive, i.e. the canonical map $M\rightarrow M^{**}$ is an isomorphism.
\end{itemize}
Below, we will show by examples that there is no redundance in these conditions.
Suppose we are given such an $A$-module $M$. 
Take a minimal projective presentation of $M$
\[
0\rightarrow P_1\rightarrow P_0\rightarrow M\rightarrow 0.
\]
Then we have a left exact sequence of left $A$-modules:
\[
0\rightarrow M^*\rightarrow P_0^*\rightarrow P_1^*.
\]
Take a minimal projective presentation of $N=M^*$
\[
0\rightarrow Q_1\rightarrow Q_0\rightarrow N\rightarrow 0.
\]
Then we have a left exact sequence of right $A$-modules:
\[
0\rightarrow M\rightarrow Q_0^*\rightarrow Q_1^*.
\]
Then the following square
\[
\begin{tikzcd}
(P_1\rightarrow P_0)\ar[d]\ar[r]&(0\rightarrow Q_0^*)\ar[d]\\
0\ar[r]&(0\rightarrow Q_1^*)
\end{tikzcd}
\]
is homotopy bicartesian in $\A'$, but not in $\A$.

Note that conditions (1) and (3) together do not imply condition (2), as the following example shows.
Let $A_1$ be the algebra given by the following quiver the relations
\[
\begin{tikzpicture}
 \node(a) at (-1.5,0){
  $1$};
 \node(b) at (1.5,0){
  $2$};
 \node(c) at (0,-1.5){
  $3\mathrlap{.}$};
 
 \draw[->,transform canvas={yshift=0.2cm}](a)--(b) node [pos=0.5,above] {$ a_{1} $};
 \draw[->,transform canvas={yshift=-0.2cm}](a)--(b) node [pos=0.5,below] {$ a_{2} $};
 \draw[->,transform canvas={yshift=-0.1cm}](b)--(c) node [pos=0.3,below] {$ b $};
 \draw[->,transform canvas={yshift=-0.1cm}](c)--(a) node [pos=0.6,below] {$ c $};
 \draw[dashed,blue] ([shift=(-57:0.65)]-1.5,0) arc (-55:-350:0.7cm);
 \draw[dashed,blue] ([shift=(-160:0.8)]1.5,0) arc (-168:-135:0.8cm);
 \draw[dashed,blue] ([shift=(45:0.6)]0,-1.5) arc (45:135:0.6cm);
\end{tikzpicture}
\]
Let $L$ be following $A_1$-module 
\[
\begin{tikzcd}
k\ar[rd,"0"swap]&&k\ar[ll,shift right=0.8ex,"\Id"swap]\ar[ll,shift left=0.8ex,"0"]\\
&0\mathrlap{.}\ar[ru,"0"swap]&
\end{tikzcd}
\]
It admits a projective presentation 
\[
0\rightarrow P_1\rightarrow P_2\rightarrow L\rightarrow 0
\]
and its $A_1$-dual is the simple module $S_3'$ at the vertex 3.
The $A_1^{op}$-module $S_3'$ admits a projective resolution
\[
0\rightarrow P_2'\rightarrow P_1'\rightarrow P_3'\rightarrow S_3'\rightarrow 0.
\]
Thus $L^{**}$ is isomorphic to $L$, both being the kernel of $P_3\rightarrow P_1$.
Therefore the $A_1$-module $L$ satisfies properties (1) and (3) but not property (2). 
Also, the $A_1^{op}$-module $L^*$ satisfies properties (2) and (3) but not property (1).

Note also that conditions (1) and (2) together do not imply condition (3), as the following example shows. 
Let $A_2$ be the algebra given by the following quiver with relations
\[
\begin{tikzcd}
&1&\\
2\ar[r,""{name=6,near end,swap},""{name=7,near end}]&3\ar[u,""{name=1,swap},""{name=8,near start}]\ar[r,""{name=4},""{name=5,near start,swap}]&4.\ar[lu,""{name=2},""{name=3}]\ar[r,from=3,to=4,dashed,no head, bend right=8ex,blue]\ar[r,from=5,to=6,dashed,no head,bend left=8ex,blue]\ar[r,from=7,to=8,dashed,no head,bend left=8ex,blue]
\end{tikzcd}
\] 
Consider the simple module $S_3$ at vertex 3.
It has a projective resolution
\[
0\rightarrow P_2\rightarrow P_3\rightarrow S_3\rightarrow 0
\]
Its $A_2$-dual is the kernel of $P_3'\rightarrow P_2'$ and is isomorphic to $S_1'\oplus S_4'$.
The module $S_1'=P_1'$ is projective and $S_4'$ admits a projective resolution
\[
0\rightarrow P_1'\rightarrow P_4'\rightarrow S_4'\rightarrow 0
\]
So the $A_2$-bidual of $S_3$ contains $P_1$ as a direct summand and hence $S_3$ is not reflexive.

We give examples of algebras $A$ admitting such modules.
Note that such an algebra can neither be hereditary, nor local nor self-injective.
Consider the algebras given by the quivers with relations
\[
\begin{tikzcd}
 &2\ar[rd,""{name=b,swap,near start},""{name=b',swap, near end},"b",""{name=b'',swap}]&\\
1\ar[ru,""{name= a,swap,near end},"a"]&&3,\ar[ll,""{name=c, swap,near start},"c"]\arrow[from = a, to = b,bend right=20, dashed, no head]\ar[from=b'',to=c,bend right=20,dashed,no head]
\end{tikzcd}
\begin{tikzcd}
&&5&\\
&&4\ar[u,""{name=4}]&\\
1&3\ar[l]&2\ar[u,""{name=3,near end},""{name=2,swap,near start}]\ar[l]&6.\ar[l,""{near end,swap,name=1}]\ar[r,from=1,to=2,dashed,no head,bend right=12ex]\ar[r,from=3,to=4,dashed,no head,bend left=12ex]
\end{tikzcd}
\]
\[
\begin{tikzcd}
1\ar[r,""{name=1,swap, near end},"a"]&2\ar[r,""{name=2,swap,near start},""{name=3,swap, near end},"b"]&3\ar[r,""{name=4,swap,near start},"c"]&4\ar[r,bend right=12ex, from =1,to=2,dashed,no head]\ar[r,bend right=12ex,from=3,to=4,dashed,no head],
\end{tikzcd}
\begin{tikzcd}
&&4\ar[d,""{name=5,swap,near end}]&\\
1&3\ar[l,""{name=1,near start}]&2\ar[l,""{name=2,near end},""{name=3,near start},""{name=6,near start, swap}]&5,\ar[l,""{name=4,near end}]\ar[r,from=1,to=2,dashed,no head,bend right=8ex]\ar[r,from =3,to=4,dashed,no head,bend right=8ex]\ar[r,from=5,to=6,dashed,no head,bend right =8ex]
\end{tikzcd}
\begin{tikzcd}
2\ar[r,""{name=2,near start,swap},""{name=3,near end, swap}]&3\ar[d,""{name=4,near start, swap}]\\
1\ar[u,""{name=1,swap,near end}]&4.\ar[l]\ar[r,from=1,to=2,dashed,no head,bend right=8ex]\ar[r,from=3,to=4,no head,dashed, bend right=8ex]
\end{tikzcd}
\]
Then the simple modules $S_2$ at the vertices 2 satisfy the above properties.
\end{example}

\subsection{An $A_{\infty}$-theoretic description} 
Our aim is to give a description of the homotopy
category of homotopy $3$-term complexes using the $A_\infty$-formalism. 
We first recall some basic facts and notation for $A_\infty$-categories. Our standard references for $A_\infty$-categories are \cite{Keller01,Keller02, Keller06c,Lefevre03}. We follow the sign conventions of \cite{GetzlerJones90}. Note that we keep the assumption that $k$ is a commutative ring.

\begin{definition}An $A_{\infty}$-category $\A$ is the datum of 
\begin{itemize}
\item a class of objects $\obj(\A)$,
\item for all $A, B\in \A$, a $\mathbb Z$-graded $k$-module $\Hom_{\A}(A,B)$, often denoted by $(A,B)$,
\item for all $n\geq 1$ and all $A_0$, \ldots, $A_n$, a graded map
\[
m_n:(A_{n-1},A_n)\otimes(A_{n-2},A_{n-1})\otimes \cdots\otimes(A_0,A_1)\rightarrow (A_0,A_n)
\]
of degree $2-n$
\end{itemize}
such that for each $n\geq 1$ and all $A_0, \ldots, A_n\in\A$, we have the identity
\[
\sum (-1)^{r+st}m_u(\boldsymbol{1}^{\otimes r}\otimes m_s\otimes \boldsymbol{1}^{\otimes t})=0
\]
of maps
\[
(A_{n-1},A_n)\otimes(A_{n-2},A_{n-1})\otimes \cdots\otimes(A_0,A_1)\rightarrow (A_0,A_n)
\]
where the sum runs over all decompositions $n=r+s+t$ and we put $u=r+1+t$.
\end{definition}
\begin{remark}The composition $m_2$ induces an associative composition 
\[
\mu: H^0(\A)(B,C)\otimes H^0(\A)(A,B)\rightarrow H^0(\A)(A,C)
\]
for each triple of objects $(A,B,C)$ in $\A$.
\end{remark}
\begin{definition}Let $\A$ be an $A_{\infty}$-category and $A\in\A$. A morphism $e_{A}\in \Hom^{0}_{\A}(A,A)$ is a {\em strict identity} if we have 
\[
m_2(f,e_A)=f \mbox{ and } m_2(e_A,g)=g
\]
whenever these make sense and 
\[
m_n(\ldots,e_A,\ldots)=0
\]
for all $n\neq 2$.
In particular, $e_A$ is a cycle of the complex $(\Hom_{\A}(A,A),m_1)$. 
Clearly, if $e_A$ exists, it is unique. 
The $A_{\infty}$-category $\A$ {\em has strict identities} if there is a strict identity $e_A$ for each object $A\in \A$.
\end{definition}
\begin{definition}Let $\A$ and $\B$ be two $A_{\infty}$-categories. An $A_{\infty}$-functor $F:\A\rightarrow \B$ is the datum of
\begin{itemize}
\item a map $F:\obj (\A)\rightarrow \obj (\B)$,
\item for all $n\geq 1$ and all $A_0$, \ldots, $A_n\in\A$ a graded map
\[
F_n: (A_{n-1},A_n)\otimes (A_{n-2}, A_{n-1})\otimes \cdots \otimes (A_0,A_1)\rightarrow \Hom_{\B}(FA_0,FA_n)
\]
of degree $1-n$
\end{itemize}
such that 
\[
\sum (-1)^{r+st}F_u(\boldsymbol{1}^{\otimes r}\otimes m_s\otimes \boldsymbol1^{\otimes t})=\sum (-1)^sm_r(F_{i_1}\otimes F_{i_2}\otimes \cdots \otimes F_{i_r})
\]
where the first sum runs over all decompositions $n=r+s+t$, we put $u=r+1+t$, and the second sum runs over all $1\leq r\leq n$ and all decompositions $n=i_1+\cdots+i_r$; the sign on the right hand side is given by
\[
s=(r-1)(i_1-1)+(r-2)(i_2-1)+\cdots+2(i_{r-2}-1)+(i_{r-1}-1).
\]
The $A_{\infty}$-functor $F$ is {\em strict} if we have $F_n=0$ for all $n\geq 2$. Clearly the above datum with $F_n=0$ for $n\geq 2$ defines a strict $A_{\infty}$-functor if and only if for each $k\geq 1$, we have
\[
F_1m_k=m_k(F_1\otimes \cdots\otimes F_1).
\]
\end{definition}
\begin{definition}Let $\A$ and $\B$ be two $A_{\infty}$-categories with strict identities. An $A_{\infty}$-functor $F$ is {\em strictly unital} if 
\[
F_1(e_A)=e_{FA}
\]
and
\[
F_i(\cdots\otimes e_A\otimes\cdots)=0
\]
for all $i\geq 2$.
\end{definition}
The following simple example, due to M. Kontsevich, shows that already between dg categories, $A_{\infty}$-functors are quite interesting.
\begin{example}[\cite{Kontsevich98,Drinfeld04}]\label{homotopyequivalence}
Let $ I_{2}$ be the $k$-category generated by the category $J_2$ with $2$ objects and precisely one morphism with any given source and target
\[
\begin{tikzcd}
1\ar[r,bend left=4ex,"a"]&2.\ar[l,bend left=4ex,"b"]
\end{tikzcd}
\]
We consider it as a dg category concentrated in degree 0. Let $\A$ be a dg category and $X,Y$ two objects in $\A$. 

We will show that $X$ and $Y$ are homotopy equivalent if and only if there exists a strictly unital $A_{\infty}$-functor from $I_2$ to $\A$ such that $1$ is sent to $X$ and $2$ is sent to $Y$.

Let $\mathcal K$ be the Kontsevich dg category defined as the dg $k$-path category given by 
\[
\begin{tikzcd}
1\ar[r,bend left=4ex,"a"]\ar[r,bend left=12ex,"r"]\ar[loop left,looseness=10,"h_1"] &2\ar[l,bend left=4ex,"b"]\ar[loop right, looseness=10, "h_2"]
\end{tikzcd}
\]
where $|a|=|b|=0$, $|h_1|=-1=|h_2|$, $|r|=-2$ and $d(a)=d(b)=0$, $d(h_1)=1-ba$, $d(h_2)=1-ab$, $d(r)=-ah_1+h_2a$.
By \cite[Corollary 3.7.3]{Drinfeld04}, $\mathcal K$ is a cofibrant resolution of $I_2$. 
In particular we have 
\begin{align}
\Ext^{n}_{\mathcal K}(i,j)=0\label{Kont}
\end{align}
for $i,j\in\{1,2\}$ and $n\neq 0$.

Suppose $X$ and $Y$ are homotopy equivalent. 
Then we have a homotopy equivalence $f: X\rightarrow Y$ and a homotopy inverse $g:Y\rightarrow X$. 
By definition there exist $\tilde{h_1}:X\rightarrow X$ and $\tilde{h_2}: Y\rightarrow Y$ of degree $-1$ such that $d(\tilde{h_1})=1-gf$ and $d(\tilde{h_2})=1-fg$.
Set ${h_1}=\tilde{h_1}+g\tilde{h_2}f-gf\tilde{h_1}$ and ${h_2}=\tilde{h_2}$. Let $r=\tilde{h_2}\tilde{h_2}f-\tilde{h_2}f\tilde{h_1}$.
Then $d({h_1})=d(\tilde{h_1})$ and $-f{h_1}+{h_2}f=d(r)$. 
Thus $(f,g,h_1,h_2,r)$ defines a dg functor $G: \mathcal K \rightarrow \A$.

Notice that $m_i$ vanishes in $I_2$ for $i\neq 2$ and in $\A$ for $i\geq 3$. 
When $i=2$, $m_2$ produces either identities or zeros in $I_2$. 
Thus for an $A_{\infty}$-functor $F:I_2\rightarrow \A$, the conditions become
\begin{align}
d(F_1(a))&=0, d(F_1(b))=0\nonumber\\
1=F_1(b)F_1(a)+d(F_2(b,a)), 1&=F_1(a)F_1(b)+d(F_2(a,b))\nonumber\\
\sum (-1)^{r}F_{n-1}(\boldsymbol1^{\otimes r}\otimes m_2\otimes \boldsymbol1^{\otimes t})&=d\circ(F_n)+\sum(-1)^{i_1-1}m_2(F_{i_1}\otimes F_{i_2}) \label{A3}
\end{align}
where $n\geq 3$, $r+t=n-2$ and $i_1+i_2=n$. 
When acting on $(a_1,\cdots, a_n)$, the conditions become trivial for $n\geq 3$ if some $a_i$ is an identity. 
When each $a_i$ is non identity, $m_2$ in each summand of the left hand side will produce an identity and hence becomes zero by $F_{n-1}$.
Thus we only need to take care of $F_n(a_1,a_2,\cdots, a_n)$ with $a_i\in\{a, b\}$ and $a_i\neq a_{i+1}$.

We define an $A_{\infty}$-functor $F:I_2\rightarrow \A$ as follows: the value of $F_n$ on a sequence $(a_1,a_2,\cdots, a_n)$, where $a_i\in\{a,b\}$ and $a_i\neq a_{i+1}$, is the image under the dg functor $G$ of an element $\tilde{F_n}(a_1,a_2,\cdots,a_n)$, which is constructed by induction, in $\mathcal K(i,j)$ with $(i,j)$ depending on the sequence $(a_1,a_2,\cdots,a_n)$. 
More precisely
\begin{align*}
\tilde{F_1}(a)&=a,\tilde{F_1}(b)=b\\
\tilde{F_2}(a,b)&=h_2, \tilde{F_2}(b,a)=h_1.\\
&\tilde{F_3}(a,b,a)=r.
\end{align*}

By induction we have 
\begin{align}
\sum(-1)^{i_1}d\circ m_2(\tilde{F_{i_1}}\otimes \tilde{F_{i_2}})(a_1,a_2,\cdots,a_n)=0.\label{Ainf}
\end{align}
where $(i_1,i_2)$ runs through the pair such that $i_1+i_2=n$. Here $(\tilde{F_{i_1}}\otimes \tilde{F_{i_2}})(a_1,a_2,\cdots,a_n)$ is understood as $\tilde{F_{i_1}}(a_1,\cdots,a_{i_1})\otimes \tilde{F_{i_2}}(a_{i_1+1},\cdots,a_n)$.
 
By \ref{Kont}, there exists an element $\tilde{F_n}(a_1,a_2,\cdots,a_n)\in \mathcal K(i,j)^{1-n}$ whose coboundary is the cocycle 
\[
\sum(-1)^{i_1}\tilde{F_{i_1}}(a_1,\cdots,a_{i_1})\circ\tilde{F_{i_2}}(a_{i_1+1},\cdots,a_n)
\]
which appears on the left hand side of \ref{Ainf}. 

For example we may put
\[
\tilde{F_3}(b,a,b)=-brb+bh_2^2+h_1^2b-h_1bh_2.
\]
\end{example}
\begin{remark}
The Kontsevich dg category $\mathcal K$ is quasi-equivalent to the dg $k$-path category $\C$ given by the differential graded quiver
\[
\begin{tikzcd}
1\ar[r,bend left=4ex,"a"]\ar[loop left,looseness=10,"h_1"] &2\ar[l,bend left=4ex,"b"]\ar[loop right, looseness=10, "h_2"]\ar[l,bend right=12ex,"b'"swap]
\end{tikzcd}
\]
where $|a|=|b|=|b'|=0$, $|h_1|=-1=|h_2|$, and $d(a)=d(b)=d(b')=0$, $d(h_1)=1-ba$, $d(h_2)=1-ab'$.
A quasi-equivalence $F:\mathcal K\rightarrow \C$ is given as follows: on the objects we have $F(1)=1$ and $F(2)=2$ and on the arrows we have $F(a)=a$, $F(b)=b'$, $F(h_2)=h_2$, 
\begin{align*}
F(h_1) &=h_1+b'(h_2a-ah_1)+(1-b'a)(bh_2-h_1b')a \\
 F(r)     &=h_2h_2a-h_2a(h_1+(bh_2-h_1b')a).
 \end{align*}
We call this dg category the {\em Cisinski dg category} because it is the image of Cisinski's interval $J$ (cf.~\cite[Definition 3.3.3]{Cisinski19}) under the functor $\Lambda$ (cf.~\ref{Lambdadgnerve}).

The datum of a dg functor $F:\mathcal K\rightarrow \A$ is equivalent to a morphism $\alpha:X_1\rightarrow X_2$ closed of degree $0$ in $\A$ together with a contracting homotopy for $\Cone(\alpha^{\wedge}:X_1^{\wedge}\rightarrow X_2^{\wedge})$, while the datum of a dg functor $G:\C\rightarrow \A$ is equivalent a morphism $\alpha:X_1\rightarrow X_2$ closed of degree $0$ in $\A$ together with null-homotopies of the maps $[0,\;1]^{\intercal}: X_2^{\wedge}\rightarrow \Cone(\alpha^{\wedge})$ and $[1,\;0]:\Cone(\alpha^{\wedge})\rightarrow \Sigma X_1^{\wedge}$.
\end{remark}
Let $\B$ be a small $A_{\infty}$-category and $\A$ an $A_{\infty}$-category. 
It is known, cf.~\cite{Kontsevich98}, that there is a natural $A_{\infty}$-category $\Fun_{\infty}(\B,\A)$ of strictly unital $A_{\infty}$-functors between $\B$ and $\A$. 
It was constructed in \cite[8.1.3]{Lefevre03} using twists of $A_{\infty}$-structures. 
Keller further developed an idea in \cite{Lyubashenko03} and interpreted it in \cite[5.7]{Keller06c} as an internal $\Hom$-object in the tensor category of cocomplete augmented dg cocategories.
 Note that when the target $\A$ is a dg category, the $A_\infty$-category $\Fun_{\infty}(\B,\A)$ is again a dg category. 
 The reason is that we use the $m_j^{\A}$, $j\geq i$, to define the map $m_i$ of the $A_{\infty}$-category 
$\Fun_{\infty}(\B,\A)$ and when $\A$ is a dg category, then the $m_{j}^{\A}$ vanish for all $j\geq 3$.
\begin{example}\label{Mor(A)}
The dg category $\Fun_{\infty}(k\Mor,\A)$ is isomorphic to the {\em dg morphism category} $\Mor(\A)$ introduced in \cite[2.9]{Drinfeld04}. More precisely, $\Mor(\A)$ has as objects the closed morphisms $f:A\rightarrow B$ of degree 0 in $\A$. The morphism complex $\Mor(\A)((A,B,f),(A',B',f'))$ has as the degree $m$ component the space of matrices
$\begin{pmatrix}
j&0\\
h&l
\end{pmatrix}$
where the morphisms $j:A\rightarrow A'$, $l:B\rightarrow B'$ are of degree $m$ and the morphism $h:A\rightarrow B'$ is of degree $m-1$. The differential $d$ carries a morphism 
$\begin{pmatrix}
j&0\\
h&l
\end{pmatrix}$
to
$\begin{pmatrix}
-d(j)&0\\
d(h)+f'j-(-1)^{m}lf&d(l)
\end{pmatrix}$.
The composition map is the multiplication of matrices.
\end{example}
Now let $\A$ be a dg category and $\B$ be the dg $k$-path category of the following graded quiver with relations
\begin{equation}\label{3term}
\begin{tikzcd}
0\ar[r,"a"]&1\ar[r,"b"]&2
\end{tikzcd}
\end{equation}
where $|a|=|b|=0$, $d(a)=0$, $d(b)=0$ and $ba=0$.

Our aim is to describe the dg category $\Fun_{\infty}(\B,\A)$ following the construction in \cite{Lefevre03}. 

By definition, it has strictly unital $A_{\infty}$-functors $F:\B\rightarrow\A$ as objects. 
A strictly unital $A_{\infty}$-functor $F:\A\rightarrow \B$ is given by the datum of
\begin{itemize}
\item objects $F(i)=A_i$ in $\A$ with $i=0,1,2$;
\item graded maps 
\[
F_1:\Hom_{\B}(i,j)\rightarrow \Hom_{\A}(A_i,A_j)
\] 
of degree $0$ where $i\leq j$ and a graded map
\[
F_2:\Hom_{\B}(1,2)\otimes\Hom_{\B}(0,1)\rightarrow \Hom_{\A}(A_0,A_2)
\]
of degree $-1$ such that $F_1$ sends identities in $\B$ to identities in $\A$ and 
\[
F_1\circ d_{\B}=d_{\A}\circ F_1
\]
and
\[
F_1m_2-F_2(\boldsymbol1\otimes d_{\B})-F_2(d_{\B}\otimes \boldsymbol1)=m_2(F_1\otimes F_1)+d_{\A}F_2.
\]
\end{itemize}

Let $f=F_1(a)$, $j=F_1(b)$ and $h=F_2(b,a)$. Then $(f,j,h)$ defines an $A_{\infty}$-functor $\B\rightarrow \A$ if and only if $|f|=|j|=0$, $|h|=-1$ and $d(f)=0$, $d(j)=0$ and $d(h)=-jf$. 

Thus we identify a strictly unital $A_{\infty}$-functor $F:\B\rightarrow\A$ with the corresponding diagram
\begin{equation}\label{F}
\begin{tikzcd}
&A_0\ar[r,"f"]\ar[rr,bend right = 8ex,"h"swap]&A_1\ar[r,"j"]&A_2
\end{tikzcd}
\end{equation}
where $|f|=|j|=0$, $|h|=-1$ and $d(f)=0$, $d(j)=0$ and $d(h)=-jf$.

Let $F$ and $F'$ be two strictly unital $A_{\infty}$-functors from $\B$ to $\A$. We identify them with their corresponding diagrams. Let $e_i$ be the unit of the object $i$ in $\B$. An element $H$ of degree $n$ in the graded space $\Hom(F,F')$ of strictly unital morphisms is given by graded morphisms 
\[
H_0: {k}\cdot e_i\rightarrow \Hom_{\A}(A_i,A'_i)
\]
of degree $n$ where $i=0,1,2$, graded morphisms
\[
H_1:\Hom_{\B}(i,j)\rightarrow \Hom_{\A}(A_{i},A'_{j})
\] 
of degree $n-1$ where $i< j$ and graded morphisms
\[
H_2:\Hom_{\B}(1,2)\otimes\Hom_{\B}(0,1)\rightarrow \Hom_{\A}(A_0,A_2')
\]
of degree $n-2$.

Let $h_i=H_0(e_i)$, $s_1=H_1(a)$, $s_2=H_1(b)$ and $t=H_2(b,a)$. We identify a strictly unital morphism $H$ of degree $n$ with the following diagram
\begin{equation}\label{mordegreen}
\begin{tikzcd}
&A_0\ar[r,"f"]\ar[d,"h_0"swap]\ar[rd,"s_1"red,swap ,red]\ar[rrd,"t"blue,blue]\ar[rr,bend left = 8ex,"h"]&A_1\ar[d,"h_1"swap]\ar[rd,"s_2"red,red]\ar[r,"j"]&A_2\ar[d,"h_2"]\\
&A_0'\ar[r,"f'"swap]\ar[rr,bend right = 8ex,"h'"swap]&A_1'\ar[r,"j'"swap]&A_2'
\end{tikzcd}
\end{equation}
where $|h_i|=n$ for $i=0,1,2$, $|s_1|=|s_2|=n-1$, $|t|=n-2$.

Let $H\in \Hom^n(F,F')$. Denote by $H'{\coloneqq}d(H)\in \Hom^{n+1}(F,F')$. It is given as follows 
\begin{align*}
H'_{0} & =d_{\A}\circ H_0, \\
H_1' & =d_{\A}\circ H_1+(-1)^{n+1}m_2^{\A}(F_1'\otimes H_0)+(-1)^{n}m_2^{\A}(H_0\otimes F_1), \\
H_2' & =d_{\A}(H_2)+(-1)^{n}m_2^{\A}(F_1'\otimes H_1) +m_2^{\A}(F_2'\otimes H_0)\\
        & \quad\quad +(-1)^nm_{2}^{\A}(H_1\otimes F_1)+(-1)^{n+1}m_2^{\A}(H_0\otimes F_2).
\end{align*}
Via the identification with diagrams, we see that the corresponding diagram for $d(H)$ is
\[
 \begin{tikzcd}
&A_0\ar[r,"f"]\ar[d,"h_0'"swap]\ar[rd,"s_1'"red,swap ,red]\ar[rrd,"t'"blue,blue]\ar[rr,bend left = 8ex,"h"]&A_1\ar[d,"h_1'"swap,near end]\ar[rd,"s_2'"red,red]\ar[r,"j"]&A_2\ar[d,"h_2'"]\\
&A_0'\ar[r,"f'"swap]\ar[rr,bend right = 8ex,"h'"swap]&A_1'\ar[r,"j'"swap]&A_2'
\end{tikzcd}
 \]
where 
\[
h_i'=d(h_i)
\]
 for $i=0,1,2$, and
 \[
 s_1'=d(s_1)+(-1)^{n+1}(f'\circ h_0)+(-1)^{n}h_1\circ f,
 \]
 \[
 s_2'=d(s_2)+(-1)^{n+1}(j'\circ h_1)+(-1)^{n}h_2\circ j
 \]
 and 
 \[
 t'=d(t)+(-1)^{n}(j'\circ s_1)+h'\circ h_0+(-1)^{n}(s_2\circ f)+(-1)^{n+1}(h_2\circ h).\]
 Clearly $d^2(H)=0$.
 
 We now describe the compositions. 
 
 Let $F$, $F'$ and $F''$ be $A_{\infty}$-functors from $\B$ to $\A$. Let $H'\in \Hom^r(F,F')$ and $H''\in \Hom^s(F',F'')$ be two morphisms in the dg category $\Fun_{\infty}(\B,\A)$. 
 
Assume the corresponding diagram is given by
\[
\begin{tikzcd}
&A_0\ar[r,"f"]\ar[d,"h_0'"swap]\ar[rd,"s_1'"red,swap ,red]\ar[rrd,"t'"blue,blue]\ar[rr,bend left = 8ex,"h"]&A_1\ar[d,"h_1'"swap,near end]\ar[rd,"s_2'"red,red]\ar[r,"j"]&A_2\ar[d,"h_2'"]\\
&A_0'\ar[r,"f'"]\ar[d,"h_0''"swap]\ar[rrd,"t''"blue,blue]\ar[rd,"s_1''"red,swap,red] &A_1'\ar[rd,"s_2''"red, red]\ar[d,"h_1''"swap, near end]\ar[r,"j'"]&A_2'\ar[d,"h_2''"]\\
&A_0''\ar[r,"f''"swap]\ar[rr,bend right = 8ex,"h''"swap]&A_1''\ar[r,"j''"swap]&A_2''
\end{tikzcd}
\]
where, in order to keep the diagram readable, we omit the arrow $h':A_0'\rightarrow A_2'$.

 We denote by $\tilde{H}\in \Hom^{r+s}(F,F'')$ the composition of $H$ and $H'$. 
 Then the corresponding diagram is given as follows:
  \[
\begin{tikzcd}
&A_0\ar[r,"f"]\ar[d,"\tilde{h_0}"swap]\ar[rd,"\tilde{s_1}"red,swap ,red]\ar[rrd,"\tilde{t}"blue,blue]\ar[rr,bend left = 8ex,"h"]&A_1\ar[d,"\tilde{h_1}"swap,near end]\ar[rd,"\tilde{s_2}"red,red]\ar[r,"j"]&A_2\ar[d,"\tilde{h_2}"]\\
&A_0''\ar[r,"f''"swap]\ar[rr,bend right = 8ex,"h''"swap]&A_1''\ar[r,"j''"swap]&A_2''
\end{tikzcd}
\]
 
 where 
 \[
 \tilde{h_i}=h_i''h_i'
 \]
 for $i=1,2,3$, and
 \[
 \tilde{s_1}=h_1''s_1'+(-1)^{r}s_1''h_0', 
 \]
 \[
 \tilde{s_2}=h_2''s_2'+(-1)^{r}s_2''h_1'
 \]
 and 
 \[
 \tilde{t}=h_2''t'+t''h_0'+(-1)^{r+1}s_2''s_1'.
 \]
 Note that in the expression of $\tilde{t}$, we only have length $2$ paths with suitable signs. 
 This is due to the vanishing of $m_j^{\A}$ for $j\geq 3$.

\begin{definition}\label{def:3termhomotopy}
Let $\A$ be a dg category. 
The {\em homotopy category $\mathcal H_{3t}(\A)$ of 3-term homotopy complexes (= 3-term $h$-complexes, or 3-term complexes up to homotopy) over $\A$} is defined to be $H^0\Fun_{\infty}(\B,\A)$ where $\B$ is the dg category \ref{3term}.
More precisely, its objects are given by diagrams $X$ in $\A$ of the form
\[
\begin{tikzcd}
&A_0\ar[r,"f"]\ar[rr,bend right = 8ex,"h"swap]&A_1\ar[r,"j"]&A_2
\end{tikzcd}
\]
where $|f|=|j|=0$, $|h|=-1$ and $d(f)=0$, $d(j)=0$ and $d(h)=-jf$.

Suppose $X'$ is given by the diagram with objects and morphisms with superscript a prime symbol. 

A morphism from $X$ to $X'$ is given by a homotopy equivalence class of 6-tuples $(h_i,s_j,t)$, $i=0,1,2$, $j=1,2$ where 
\[
h_i:A_i\rightarrow A_i', |h_i|=0, d(h_i)=0,
\]
\[
s_j:A_{j-1}\rightarrow A_{j}', |s_j|=-1, d(s_1)=f'h_0-h_1f,
\]
\begin{equation}
 d(s_2)=j'h_1-h_2j, \label{s2} %\tag{}
\end{equation} 
and 
\begin{equation}
t:A_0\rightarrow A_2', |t|=-2, d(t)=h_2\circ h-h'\circ h_0-s_2\circ f-j'\circ s_1.\label{tt} 
\end{equation}
Two 6-tuples $(h_i,s_j,t)$ and $(h_i',s_j',t')$ are {\em homotopy equivalent} if there exists a 6-tuple $(h''_i,s_j'',t'')$ such that 
\[
|h_i''|=-1, h_i-h_i'=d(h_i''),
\]
\[
|s_j''|=-2, s_1-s_1'=d(s_1'')+f'\circ h_0''-h_1''\circ f,
\]
\[
s_2-s_2'=d(s_2'')+j'\circ h_1''-h_2''\circ j,
\]
and 
\[
|t''|=-3, t-t'=d(t'')-j'\circ s_1''+h'\circ h_0''-s_2''\circ f+h_2''\circ h.
\]
The composition of two morphisms given by homotopy equivalence classes of 6-tuples $(h_i',s_j',t'):F\rightarrow F'$ and $(h_i'',s_j'',t''):F'\rightarrow F'$, is the morphism given by the 6-tuple $(\tilde{h}_i,\tilde{s}_j,\tilde{t}):F\rightarrow F''$ where 
\begin{align*}
\tilde{h}_i  &=h_i''h_i' ,\\
\tilde{s}_1 &=h_1''s_1'+s_1''h_0', \\
\tilde{s}_2 &=h_2''s_2'+s_2''h_1' \\ 
\tilde{t}      &=h_2''t'+t''h_0'-s_2''s_1'.
\end{align*}
The identity morphism for $X$ is given by the homotopy equivalence class of the 6-tuple $(h_i,s_j,t)$ where $h_i=\Id$, $s_j=0$ and $t=0$ for $i=0,1,2$ and $j=1,2$. 
\end{definition}
From the definition, we see that $\mathcal H_{3t}(\A)$ is isomorphic to $\mathcal H_{3t}(\tau_{\leq 0}(\A))$. 
Thus we may assume $\A$ is strictly connective when needed.

Clearly if a 6-tuple $(h_i,s_j,t)$ gives rise to an isomorphism in $\mathcal H_{3t}(\A)$, then $h_i$ is a homotopy equivalence for each $i=0,1,2$. 
We will prove that the converse is also true.

Let $\A'$ be the dg category obtained from $\A$ by adding a zero object. 
The canonical dg functor $\A\rightarrow \A'$ is fully faithful and hence the induced dg functor $\Fun_{\infty}(\B,\A)\rightarrow \Fun_{\infty}(\B,\A')$ is also fully faithful.
Therefore we may assume that the dg category $\A$ contains a zero object.

Let $F$ be an object in $\Fun_{\infty}(\B,\A)$ which corresponds to the diagram \ref{F}. We have the following diagrams
\[
\begin{tikzcd}
0\ar[r,"0"]\ar[rr, bend left=8ex,"0"]&A_1\ar[d,"1"red,red]\ar[r,"0"]&0\\
A_0\ar[r,"f"]\ar[d,"1"red,red]\ar[rr,bend right=8ex,"0"]&A_1\ar[r,"0"]&0\\
A_0\ar[r,"0"]\ar[rr,bend right=8ex,"0"swap]&0\ar[r,"0"]&0
\end{tikzcd}
\]
where we omit the zero maps. 
This gives rise to a graded-split short exact sequence in $\Fun_{\infty}(\B,\A)$: 
\[
\begin{tikzcd}
0\ar[r]& A_1^{1}\ar[r] &U\ar[r] &A_0^{0}\ar[r]&0
\end{tikzcd}
\]
where we use the superscript $i$ to indicate the position of the object in the diagram, for example $A_1^1$ stands for the diagram 
\[
\begin{tikzcd}
0\ar[r,"0"]\ar[rr,bend right=8ex,"0"swap]& A_1\ar[r,"0"]&0.
\end{tikzcd}
\]
Similarly we have a graded-split short exact sequence in $\Fun_{\infty}(\B,\A)$:
\[
\begin{tikzcd}
0\ar[r]&A_2^{2}\ar[r]&F\ar[r]&U\ar[r]&0.
\end{tikzcd}
\]
So $F$ can be obtained by graded-split extensions from $A_i^i$, $i=0,1,2$. 
\begin{proposition}[\cite{Lefevre03}, {Proposition 8.2.2.3}] \label{termwiseequivalence}
Let $\theta:F\rightarrow F'$ be a morphism in $\mathcal H_{3t}(\A)$ given by the homotopy equivalence class of a 6-tuple $(h_i,s_j,t)$ as above. 
If $h_i$ is a homotopy equivalence in $\A$ for each $i=0,1,2$, then $\theta$ is an isomorphism. 
\end{proposition}
\begin{proof}
We use the notations in the above discussion. 
For each $F''$ in $\mathcal H_{3t}(\A)$, we have a morphism of graded-split short exact sequences of complexes over $k$
\[
\begin{tikzcd}
0\ar[r]&(F'',A_2^2)\ar[d,"\alpha"swap]\ar[r]&(F'',F)\ar[d,"\beta"]\ar[r]&(F'',U)\ar[d,"\gamma"]\ar[r]&0\\
0\ar[r]&(F'',A_2'^2)\ar[r]&(F'',F')\ar[r]&(F'',U')\ar[r]&0.
\end{tikzcd}
\] 
We want to show that the morphisms $\alpha$ and $\gamma$ are quasi-isomorphisms of complexes. 
It is enough to show the following claim.

Claim: for each homotopy equivalence $A\rightarrow A'$ in $\A$, and for each $i=0$, $1$, $2$, the induced morphism
\[
(F'',A^i)\rightarrow (F'',A'^i)
\]
is a quasi-isomorphism of complexes.

To see this, we apply $\Hom(-,A^i)$ and $\Hom(-,A'^i)$ to the associated graded-split short exact sequences for $F''$ and get morphisms of graded-split short exact sequences of complexes over $k$
\[
\begin{tikzcd}
0\ar[r]&((A_0'')^{0},A^i)\ar[d,""swap,blue]\ar[r]&(U'',A^i)\ar[d,""]\ar[r]&((A_1'')^{1},A^i)\ar[d,"",blue]\ar[r]&0\\
0\ar[r]&((A_0'')^{0},A'^i)\ar[r]&(U'',A'^i)\ar[r]&((A_1'')^{1},A'^i)\ar[r]&0
\end{tikzcd}
\] 
and
\[
\begin{tikzcd}
0\ar[r]&(U'',A^i)\ar[d,""swap]\ar[r]&(F'',A^i)\ar[d,""]\ar[r]&((A_2'')^{2},A^i)\ar[d,"",blue]\ar[r]&0\\
0\ar[r]&(U'',A'^i)\ar[r]&(F'',A'^i)\ar[r]&((A_2'')^{2},A'^i)\ar[r]&0.
\end{tikzcd}
\] 
Now we observe that for any pair of objects $B$ and $C$ in $\A$, the complex $\Hom_{\Fun_{\infty}(\B,\A)}(B^i,C^j)$ is isomorphic to a shift of $\Hom_{\A}(B,C)$ if $i\leq j$ and is zero if $j>i$.
Since $A\rightarrow A'$ is a homotopy equivalence in $\A$, the above maps in blue are quasi-isomorphisms of complexes.
This implies that the morphism $(F'',A^i)\rightarrow (F'',A'^i)$ is a quasi-isomorphism.
This proves the claim and hence the proposition.
\end{proof}
Denote by $\tilde{\mathrm{Com}}$ the dg path $k$-category of the quiver 
\begin{equation}\label{3term'}
\begin{tikzcd}
0\ar[rr,bend left=6ex,"c"]\ar[r,"a"swap]&1\ar[r,"b"swap]&2 
\end{tikzcd}
\end{equation}
where ${|}a{|}=0$, ${|}b{|}=0$, ${|}c{|}=-1$, and $d(c)=ba$.

We identify objects in $\D(\A\otimes \tilde{\Com}^{op})$ with sequences
\[
\begin{tikzcd}
X\ar[r,"f"]\ar[rr,"h"swap, bend right=6ex]&Y\ar[r,"j"]&Z
\end{tikzcd}
\]
in $\C_{dg}(\A)$ where $f$ and $j$ are closed morphisms of degree 0 and $h$ is a morphism of degree $-1$ such that $d(h)=jf$. 

\begin{lemma}\label{Comrepresentable}
Each object in $\rep(\tilde{\Com},\A)$ is isomorphic to an object of the form
\[
\begin{tikzcd}
A^{\wedge}\ar[r,"f^{\wedge}"]\ar[rr, bend right =6ex,"h^{\wedge}"swap]&B^{\wedge}\ar[r,"j^{\wedge}"]&C^{\wedge}
\end{tikzcd}
\]
where $A,B,C$ are objects in $\A$, the maps $f$ and $j$ are of degree $0$, the map $h$ is of degree $-1$ and $d(h)=jf$.
\end{lemma}
\begin{proof}
We observe that $\tilde{\Com}$ is a cofibrant dg category. 
Recall that for dg categories $\A$ and $\B$ where $\B$ is cofibrant, we have a canonical bijection \cite[Theorem 4.2]{Toen07}
\[
\Hom_{\Hqe}(\B,\A)=\Hom_{\dgcat}(\B,\A)/\mathrm{htp} \iso \Iso(\rep(\B,\A))
\]
where the homotopy class of a dg functor $F:\B\rightarrow \A$ is sent to the isomorphism class of the dg bimodule $\Hom_{\A}(-,F(-))$.
Thus each object in $\rep(\tilde{\Com},\A)$ is isomorphic to an object given by a dg functor $\tilde{\Com}\rightarrow \A$.
\end{proof}

The following theorem is claimed by Kontsevich, modulo some results by To\"en.
We give a direct proof in the special case where $\B$ is the dg category \ref{3term}.

\begin{theorem}[\cite{CanonacoOrnaghiStellari18} cf.~also \cite{Faonte17a}]
%\cite{Toen07} %\cite{Drinfeld04}] 
We have a canonical isomorphism 
\[
\Fun_{\infty}(-,-)\rightarrow \rep_{dg}(-,-)
\]
of bifunctors $\Hqe_{k-\mathrm{cf}}^{op}\times \Hqe\rightarrow \Hqe$, where $\Hqe_{k-\mathrm{cf}}$ is the
full subcategory of $\Hqe$ whose objects are the dg categories whose morphism
complexes are cofibrant dg $k$-modules.
\end{theorem}
\begin{corollary}\label{strictify}
Let $\A'$ be a full dg subcategory of $\A$ and $\B$ be a small dg category. 
Let $F:\rep(\B,\A')\rightarrow \rep(\B,\A)$ be the functor induced by the inclusion $\A'\rightarrow \A$.
Then an object $X\in \rep(\B,\A)$ lies in the essential image of $F$ if and only if for each object $B$ in $\B$, $X(-,B)$ is quasi-isomorphic to a dg $\A$-module represented by an object in $\A'$.
\end{corollary}
\begin{proof}[Proof of the Corollary]
Let $\A''$ be the closure of $\A'$ under homotopy equivalences in $\A$. The inclusion of 
$\A'$ into $\A''$ is a quasi-equivalence and therefore induces a quasi-equivalence in
$\rep_{dg}(\B,?)$. By the above theorem, it also induces a quasi-equivalence in
$\Fun_\infty(\B,?)$. Thus, we may assume that $\A'$ is stable under homotopy equivalences
in $\A$. Now the inclusion $\Fun_\infty(\B,\A') \to \Fun_\infty(\B,\A)$ is clearly an isomorphism
onto the full subcategory of $A_\infty$-functors whose values on objects lie in $\A'$. 
Thanks to the theorem, this implies the statement.
\end{proof}
\begin{lemma}\label{dgfunctorcofibrant}
Let $\A$ be a small dg category and $\C$ be a dg category given by a differential graded quiver $Q$ without relations and whose differential $d$ is such that $Q$ admits a finite filtration
\[
F_0\subset F_1\subset F_2\subset\ldots\subset F_N=Q
\]
such that all $F_p$ have the same objects as $Q$, and the bimodule of arrows of $F_0$ vanishes and $d(F_p)$ is contained in $kF_{p-1}$. 
Then for each dg functor $F:\C\rightarrow \A$, the associated dg $\C$-$\A$-bimodule $_{F}\A_{\A}=\Hom_{\A}(-,F(-))$ admits a cofibrant resolution 
\[
 \begin{tikzcd}
0\ar[r]& \ker(p) \ar[r,"u"] &\bigoplus_{i\in Q}\Hom_{\A}(-,F(i))\otimes \Hom_{\C}(i,-)\ar[r,"p"]                          & \Hom_{\A}(-,F(-))\ar[r]& 0.
 \end{tikzcd}
 \]
\end{lemma}
\begin{proof}
 Let $F:\C\rightarrow \A$ be a dg functor.
 Then we have a short exact sequence of dg $\A\otimes \C^{op}$-modules
 \[
 \begin{tikzcd}
0\ar[r]& \ker(p) \ar[r,"u"] &\bigoplus_{i\in Q}\Hom_{\A}(-,F(i))\otimes \Hom_{\C}(i,-)\ar[r,"p"]                          & \Hom_{\A}(-,F(-))\ar[r]& 0.
 \end{tikzcd}
 \]
 where the dg module $\ker(p)$ is isomorphic to 
 \[
 \bigoplus_{\alpha:i\rightarrow j\text{ in $Q$}}\Hom_{\A}(-,F(i))\otimes\Hom_{\C}(j,-)\otimes k\alpha
 \]
  as a graded module.
 It has a finite increasing filtration 
 \[
 0\subset M_0\subset M_{1}\subset \ldots\subset M_{N}=\ker(p)
 \]
 where 
 \[
 M_{k}=\bigoplus_{\alpha:i\rightarrow j\text{ in $F_k$}}\Hom_{\A}(-,F(i))\otimes\Hom_{\C}(j,-)\otimes k\alpha.
 \] 
 The subquotient $M_{k+1}/M_{k}$ is isomorphic to the shift of the representable dg bimodule 
 \[
 \bigoplus_{\alpha:i\rightarrow j\text{ in $F_{k+1}\backslash F_{k}$}}\Sigma^{|\alpha|}\Hom_{\A}(-,F(i))\otimes\Hom_{\C}(j,-).
 \]
 Thus $\Cone(u)$ is a cofibrant resolution of the dg bimodule $_{F}\A_{\A}$ associated with the dg functor $F:\C\rightarrow \A$.
 \end{proof}
\begin{proposition}Let $\B$ be the dg category defined by \ref{3term}. 
We have a canonical isomorphism $\mu:\Fun_{\infty}(\B,\A)\rightarrow \rep_{dg}(\B,\A)$ in $\mathrm{Hqe}$.
\end{proposition}
\begin{proof}
Let $\B'$ be the dg category $\tilde{\mathrm{Com}}$ defined by \ref{3term'}. 
We have a quasi equivalence $\B'\rightarrow \B$. 

We define a dg functor $\tilde{\mu}: \Fun_{\infty}(\B,\A)\rightarrow \rep_{dg}(\B',\A)$ as follows:

Let $F$ be an object in $\Fun_{\infty}(\B,\A)$ with the corresponding diagram \ref{F}. 
Let $M$ be the dg $\B'^{op}\otimes \A$ module which corresponds to the diagram 
\[
\begin{tikzcd}
A_0^{\wedge}\ar[r,"f^{\wedge}"]\ar[rr,bend right=8ex,"-h^{\wedge}"swap]&A_1^{\wedge}\ar[r,"j^{\wedge}"]&A_2^{\wedge}.
\end{tikzcd}
\]

Denote by $X=2^{\wedge}\otimes A_0^{\wedge}$, $Y=1^{\wedge}\otimes A_0^{\wedge}\oplus 2^{\wedge}\otimes A_1^{\wedge}\oplus 2^{\wedge}\otimes \Cone(\Id_{A_0^{\wedge}})$ and $Z=0^{\wedge} \otimes A_0^{\wedge}\oplus 1^{\wedge}\otimes A_1^{\wedge}\oplus 2^{\wedge}\otimes A_2^{\wedge}$ the cofibrant dg $\B'^{op}\otimes \A$-modules. 
By Lemma~\ref{dgfunctorcofibrant}, we have an exact sequence of dg $\B'^{op}\otimes \A$-modules:
\[
\begin{tikzcd}
0\ar[r]& X\ar[r,"\alpha"]&Y\ar[r,"\beta"]&Z\ar[r,"\gamma"]&M\ar[r]&0
\end{tikzcd}
\]
where 
\[
\alpha^{\intercal}=
\begin{bmatrix}
b^{\wedge}\otimes \Id& \Id\otimes f^{\wedge}&1\otimes {[}0,1{]}^{\intercal}
\end{bmatrix},
\]
\[
\beta=
\begin{bmatrix}
a^{\wedge}\otimes \Id&0&-(ba)^{\wedge} \otimes {[}0,1{]}-c^{\wedge}\otimes {[}s^{-1},0{]} \\
-\Id\otimes f^{\wedge}&b^{\wedge}\otimes \Id&0\\
0&-\Id\otimes j^{\wedge}&\Id\otimes {[}-h^{\wedge}{,}(jf)^{\wedge}{]}
\end{bmatrix}
\]
and $\gamma$ is the obvious morphism.

The object $F$ is sent by $\tilde{\mu}$ to the totalization of the sequence
\[
\begin{tikzcd}
X\ar[r,"\alpha"]&Y\ar[r,"\beta"]&Z.
\end{tikzcd}
\]
So $\tilde{\mu}(F)$ is quasi-isomorphic to $M$.

Let $H:F\rightarrow F'$ be a morphism of degree $n$ in $\Fun_{\infty}(\B,\A)$ given by the diagram \ref{mordegreen}.

Recall that for a dg $\A$-module $M$, $s:M\rightarrow \Sigma M$ is the graded map of degree $-1$ which is $\Id: {M^n}\rightarrow (\Sigma M)^{n-1}=M^{n}$ on each degree $n$. 
Clearly $d(s)=0$.

We have the following diagram in $\C_{dg}(\B'^{op}\otimes\A)$
\[
\begin{tikzcd}
0\ar[r]& X\ar[r,"\alpha"]\ar[d,"u"]&Y\ar[r,"\beta"]\ar[d,"v"]\ar[rd,dashed,"r"]&Z\ar[r]\ar[d,"w"]&0\\
0\ar[r]& X'\ar[r,"\alpha'"swap]&Y'\ar[r,"\beta'"swap]&Z'\ar[r]&0
\end{tikzcd}
\]
where 
\[
u=\Id\otimes h_0^{\wedge},
\]
\[
v=(-1)^{n}
\begin{bmatrix}
\Id\otimes h_0^{\wedge}&0&0\\
0&\Id\otimes h_1^{\wedge}&\Id\otimes{[}s_1^{\wedge}\circ s^{-1},(f'h_0-h_1f)^{\wedge}{]}\\
0&0&\Id\otimes \begin{bmatrix}(-1)^{n}s\circ h_0^{\wedge}\circ s^{-1}\ \ \ 0\\0\ \ \ \ \ \ \ \ \ \ \ \ \ \ \ \ \ \ \ \ \ \ \ h_0^{\wedge}\end{bmatrix}
\end{bmatrix},
\]
\[
w=
\begin{bmatrix}
\Id\otimes h_0^{\wedge}& 0&0\\
0&\Id\otimes h_1^{\wedge}&0\\
0&0&\Id\otimes h_2^{\wedge}
\end{bmatrix},
\]
\[
r=
\begin{bmatrix}
0&0&0\\
\Id\otimes s_1^{\wedge}&0&-b^{\wedge}\otimes {[}0,s_1^{\wedge}{]}\\
0&\Id\otimes s_2^{\wedge}& -\Id\otimes {[}t^{\wedge},(s_2\circ f)^{\wedge}{]}
\end{bmatrix}.
\]

The morphism $H:F\rightarrow F'$ is sent by the dg functor $\tilde{\mu}$ to the morphism
\[
\begin{bmatrix}
u&0&0\\
0&v&0\\
0&r&w
\end{bmatrix}.
\]
We have that 
\[
\alpha'u=(-1)^n v\alpha,
\]
\[
d(u)=\Id\otimes (d(h_0))^{\wedge},
\]
\[
d(v)=(-1)^{n}
\begin{bmatrix}
\Id\otimes d(h_0)^{\wedge}&0&0\\
0&\Id\otimes d(h_1)^{\wedge}&\Id\otimes {[}s_1'^{\wedge} \circ s^{-1},(f'd(h_0)-d(h_1)f)^{\wedge}{]}\\
0&0&\Id\otimes\begin{bmatrix}(-1)^{n+1}s\circ (d(h_0))^{\wedge}\circ s^{-1}\ \ 0\\0\ \ \ \ \ \ \ \ \ \ \ \ \ \ \ \ \ \ \ \ \ \ \ \ \ \ \ \ \ d(h_0)\end{bmatrix}
\end{bmatrix},
\]
\[
d(w)=
\begin{bmatrix}
\Id\otimes d(h_0)^{\wedge}&0&0\\
0&d(h_1)^{\wedge}&0\\
0&0&\Id\otimes d(h_1)^{\wedge}
\end{bmatrix},
\]
\[
d(r)+\beta'v-(-1)^n w\beta=
\begin{bmatrix}
0&0&0\\
\Id\otimes s_1'^{\wedge}&0&-b^{\wedge}\otimes{[}0,s_1'^{\wedge}{]}\\
0&\Id\otimes s_2'^{\wedge}&-\Id\otimes {[}t'^{\wedge},(s_2'\circ f)^{\wedge}{]}
\end{bmatrix}
\]
where 
\[
s_1'=d(s_1)+(-1)^{n+1}f'h_0+(-1)^nh_1f,
\]
\[
s_2'=d(s_2)+(-1)^{n+1}j'h_1+(-1)^nh_2j,
\]
\[
t'=d(t)+(-1)^nj's_1+h'h_0+(-1)^ns_2f+(-1)^{n+1}h_2h.
\]
So the functor $\tilde{\mu}:\Fun_{\infty}(\B,\A)\rightarrow \rep_{dg}(\B',\A)$ is indeed a dg functor.

By Lemma \ref{Comrepresentable}, the functor $\tilde{\mu}$ is quasi-dense. 

To show the quasi-full faithfulness, we may assume that the dg category $\A$ contains a zero object.
We adopt the notations used in the discussions before Proposition \ref{termwiseequivalence}. 
In particular, the object $U\in \Fun_{\infty}(\B,\A)$ corresponds to the following diagram
\[
\begin{tikzcd}
A_0\ar[r,"f"]\ar[rr,bend right=8ex,"0"swap]&A_1\ar[r,"0"]&0.
\end{tikzcd}
\]
And we have graded-split short exact sequences in $\Fun_{\infty}(\B,\A)$:
\[
\begin{tikzcd}
0\ar[r]&A_1^1\ar[r]&U\ar[r]&A_0^0\ar[r]&0,
\end{tikzcd}
\]
\[
\begin{tikzcd}
0\ar[r]&A_2^2\ar[r] &F\ar[r]&U\ar[r]&0.
\end{tikzcd}
\]
Thus it suffices to show that $\tilde{\mu}$ induces quasi-isomorphisms between objects of the form $A^i$ where $A\in \A$ and $i=0,1,2$.

This can be checked using the fact that $\tilde{\mu}(F)$ is a cofibrant resolution of $M$.
\end{proof}

Denote by $\tilde{\mathrm{Sq}}$ the dg quotient of $k\Sq$ by its full dg subcategory consisting of the object 10, i.e.~the dg path $k$-category of the quiver
\begin{equation}
\begin{tikzcd}
00\ar[rd,phantom,"="]\ar[r,"f"]\ar[d,"g"swap]&01\ar[d,"j"]\\
10\ar[loop left,"h"]\ar[r,"l"swap]&11
\end{tikzcd}
\end{equation}
where ${|}f{|}=0={|}j{|}={|}g{|}={|}l{|}$, ${|}h{|}=-1$, $d(h)=\Id_{10}$, and $jf=lg$.

We identify objects in $\D(\A\otimes\tilde{\Sq}^{op})$ with diagrams
\[
\begin{tikzcd}
X\ar[rd,phantom,"="]\ar[r,"f"]\ar[d,"g"swap]&Y\ar[d,"j"]\\
N\ar[loop left,"h"]\ar[r,"l"swap]&Z
\end{tikzcd}
\]
in $\C_{dg}(\A)$ where $f,g,j,l$ are closed morphisms of degree 0 such that $jf=lg$ and $h$ is a morphism of degree $-1$ such that $d(h)=\Id_{N}$.

Let $F:\tilde{\mathrm{Com}}\rightarrow \tilde{\mathrm{Sq}}$ be the dg functor such that $F(0)=00$, $F(1)=01$, $F(2)=11$ and $F(a)=f$, $F(b)=j$ and $F(c)=lhg$.
Let $G:k\mathrm{Sq}\rightarrow \tilde{\mathrm{Sq}}$ be the obvious dg quotient functor.
We have the following cospan of dg functors
\[
\begin{tikzcd}
&\tilde{\mathrm{Com}}\ar[d,"F"]\\
k\mathrm{Sq}\ar[r,"G"swap]&\tilde{\mathrm{Sq}}
\end{tikzcd}\;.
\]
By \cite[Proposition 1.6.3, Proposition 4.6]{Drinfeld04}, the dg functor $\Id\otimes G^{op}:\A\otimes k\Sq^{op}\rightarrow \A\otimes\tilde{\Sq}^{op}$ is a dg quotient of $\A\otimes k\Sq^{op}$ modulo $\A\otimes \{01\}$ and hence a localization functor, i.e.~the restriction functor along $\Id\otimes G^{op}$ is fully faithful.

Let $\mathcal S$ be the right orthogonal of the objects $(A, 01)^{\wedge}$ in $\D(\Sq)$, where $A$ runs through objects in $\A$. 
More precisely, $\mathcal S$ is the full triangulated subcategory of $\D(\mathrm{Sq})$ consisting of squares 
\[
\begin{tikzcd}
X\ar[r]\ar[d]&Y\ar[d]\\
N\ar[r]&Z
\end{tikzcd}
\] 
where $N$ is an acyclic dg $\A$-module.
\begin{lemma}\label{equivalences}
$F$ is a quasi-equivalence and $G$ induces an equivalence $\mathcal S\xrightarrow{\sim}\D(\A\otimes\tilde{\Sq}^{op})$. 
\end{lemma}
\begin{proof}
It is enough to show that $F$ induces a quasi-isomorphism of complexes
\[
\Hom_{\tilde{\mathrm{Com}}}(0,2)\rightarrow \Hom_{\tilde{\mathrm{Sq}}}(00,11).
\]

On the left hand side, the complex is given by
\[
\cdots\rightarrow 0\rightarrow k\cdot c\xrightarrow{\sim} k\cdot(ba)\rightarrow 0\rightarrow \cdots
\]
which is contractible. On the right hand side, the complex is given by
\[
\cdots\rightarrow k\cdot(lh^3g)\xrightarrow{\sim} k\cdot(lh^2g)\xrightarrow{0} k\cdot(lhg)\xrightarrow{\sim} k\cdot(lg)\rightarrow 0\rightarrow\cdots
\]
which is also contractible. Thus $F$ is a quasi-equivalence.
\end{proof}
As a consequence of Lemma \ref{equivalences}, we have a fully faithful functor
\[
\begin{tikzcd}
\D(\A\otimes \tilde{\mathrm{Com}}^{op})\ar[r] &\D(\mathrm{Sq})
\end{tikzcd}
\]
which sends an object
\begin{equation}
\label{sequence}\tag{$\bigstar\bigstar$}
\begin{tikzcd}
&X\ar[r,"f"]\ar[rr,bend right = 6ex,"h"swap]&Y\ar[r,"j"]&Z
\end{tikzcd}
\end{equation}
to
\begin{equation}\label{square}\tag{$\bigstar\bigstar\bigstar$}
\begin{tikzcd}
&X\ar[r,"f"]\ar[d,"i"swap]&Y\ar[d,"j"]\\
&IX\ar[r,"{[}h{,}d(h){]}"swap]&Z
\end{tikzcd}\;.
\end{equation}
The following is a direct consequence of Lemma \ref{equivalences} and Lemma \ref{Comrepresentable}.
\begin{corollary}\label{rep}
Assume the dg category $\A$ contains a contractible object, e.g.~$Z^0(\A)$ is an additive category.
Each object
\[
\begin{tikzcd}
X\ar[r,""]\ar[d]&Y\ar[d,""]\\
N\ar[r]&Z
\end{tikzcd}
\]
in $\rep(\mathrm{Sq})$, where $N$ is acyclic, is isomorphic in $\D(\mathrm{Sq})$ to an object of the form
\[
\begin{tikzcd}
A^{\wedge}\ar[r,"f^{\wedge}"]\ar[d]&B^{\wedge}\ar[d,"j^{\wedge}"]\\
IA^{\wedge}\ar[r] &C^{\wedge}
\end{tikzcd}
\]
\end{corollary}
\begin{remark}
If $X$ and $Z$ are already representable dg modules, 
then the isomorphism can be chosen such that when restricted to both ends of the squares, the morphisms are identities in $\D(\A)$.
\end{remark}

Recall that the homotopy category $\mathcal H_{3t}(\A)$ of 3-term h-complexes over $\A$ is defined to be $H^0(\Fun_{\infty}(\B,\A))$ where $\B$ is the dg category defined by \ref{3term}. 
Then we have a fully faithful functor
\[
F: \mathcal H_{3t}(\A)\iso \rep(\tilde{\Com},\A)\rightarrow \rep(\Sq)
\]
sending a 3-term h-complex over $\A$
\begin{equation}\label{3t}
\begin{tikzcd}
A_0\ar[r,"f"]\ar[rr,bend right = 8ex,"h"swap]&A_1\ar[r,"j"]&A_2
\end{tikzcd}
\end{equation}
where $|f|=|j|=0$, $|h|=-1$ and $d(f)=0$, $d(j)=0$ and $d(h)=-jf$, to the following object in $\rep(\Sq)$
\begin{equation}\label{3tsquare}
\begin{tikzcd}
A_0^{\wedge}\ar[r,"f^{\wedge}"]\ar[d,"i"swap]                                              &A_1^{\wedge}\ar[d,"j^{\wedge}"]\\
IA_0^{\wedge}\;\;\ar[r,"{[}-h^{\wedge}{,}-d(h^{\wedge}){]}"swap]                  &\;\;A_2^{\wedge}\mathrlap{.}
\end{tikzcd}
\end{equation}

\begin{definition}\label{homotopyses}
A 3-term h-complex over $\A$ given by \ref{3t} is a {\em homotopy (resp.~left, resp.~right) short exact sequence} if the corresponding object \ref{3tsquare} in $\rep(\Sq)$ is homotopy bicartesian (resp.~cartesian, resp.~cocartesian).
\end{definition}
Let $\C_{dg}^{\leq 0}(k)$ be the full dg subcategory of $\C_{dg}(k)$ consisting of complexes concentrated in non-positive degrees.
Consider a sequence in $\C_{dg}^{\leq 0}(k)$
\[
\begin{tikzcd}
U\ar[r,"f"]\ar[rr,bend right = 8ex,"h"swap]&V\ar[r,"j"]&W.
\end{tikzcd}
\]
where $|f|=|j|=0$, $|h|=-1$ and $d(f)=0$, $d(j)=0$ and $d(h)=-jf$. 
It is a {\em homotopy left exact sequence} and we say that it is a {\em homotopy kernel} of $j:V\rightarrow W$, provided that the canonical morphism 
\[
U\xrightarrow{[f\;,h]^{\intercal}} \tau_{\leq 0}\Sigma^{-1}\Cone (j)
\] 
is a quasi-isomorphism of complexes, or equivalently, for each $n\leq 0$ and each pair of elements $(v,w)\in Z^{n}V\times W^{n-1}$ such that $d(w)=-j(v)$, there exists an element $u\in Z^n U$, unique up to a coboundary, such that there exists a pair of elements $(v',w')\in V^{n-1}\times W^{n-2}$ satisfying $v-f(u)=d(v')$, $w-h(u)=-d(w')-j(v')$. 
Then we have 
\begin{lemma} \label{lem:3termhcomplex}
A 3-term h-complex given by \ref{3t} is a homotopy left (resp.~right) exact sequence if and only if for each object $A$ in $\A$, the corresponding 3-term h-complex over $\C_{dg}^{\leq 0}(k)$
\[
\begin{tikzcd}
\tau_{\leq 0}(A,A_0)\ar[r,"\tau_{\leq 0}{(}A{,}f{)}"]\ar[rr,bend right = 8ex,"\tau_{\leq 0}{(}A{,}h{)}"swap]&\tau_{\leq 0}(A,A_1)\ar[r,"\tau_{\leq 0}{(}A{,}j{)}"]&\tau_{\leq 0}(A,A_2),
\end{tikzcd}
\]
resp.
\[
\begin{tikzcd}
\tau_{\leq 0}(A_2,A)\ar[r,"\tau_{\leq 0}{(}j{,}A{)}"]\ar[rr,bend right = 8ex,"\tau_{\leq 0}{(}h{,}A{)}"swap]&\tau_{\leq 0}(A_1,A)\ar[r,"\tau_{\leq 0}{(}f{,}A{)}"]&\tau_{\leq 0}(A_0,A)
\end{tikzcd}
\]
is a homotopy left exact sequence.
It is a homotopy short exact sequence if and only if both of the above sequences are homotopy left exact.
\end{lemma}

Recall that for a dg category $\A$, we denote by $\pretr(\A)$ its pretriangulated hull. 
 We observe that the dg category $\Fun_{\infty}(\B,\A)$ can be seen a non-full dg subcategory of $\pretr(\A)$ via the dg functor $\mathrm{Tot}_{\A}:\Fun_{\infty}(\B,\A)\rightarrow \pretr(\A)$. On objects $\mathrm{Tot}_{\A}$ sends an $A_{\infty}$-functor $F:\mathcal B\rightarrow \A$ which corresponds to the diagram \ref{F}, to the formal sum $(\Sigma^2A_0\oplus \Sigma A_1\oplus A_0, q)$ where $q$ is the matrix 
 \[
 \begin{bmatrix}
 0&0&0\\
 -f&0&0\\
 h&-j&0
 \end{bmatrix}.
 \]
 Let $d(q)$ be the matrix with $d(q)_{ij}=d(q_{ij})$. 
 Clearly $d(q)+q^2=0$. 
 The dg functor $\mathrm{Tot}_{\A}$ sends a morphism which corresponds to the diagram \ref{mordegreen}, to the matrix 
 \[
 \begin{bmatrix}
 h_0&0&0\\
 (-1)^{n-1}s_1&(-1)^n h_1&0\\
 t&s_2&h_2
 \end{bmatrix}.
 \] 
Since the dg functor $\mathrm{Tot}_{\A}$ is not full,  
the induced functor $H^0(\mathrm{Tot}_{\A}):\mathcal H_{3t}(\A)\rightarrow \tr(\A)$ is neither full nor faithful in general.

For two objects $F$ and $F'$ in $\Fun_{\infty}(\B,\A)$, we use $(F,F')$ to indicate the morphism complex in $\Fun_{\infty}(\B,\A)$ and use $\A(F,F')$ to indicate the morphism complex in $\pretr(\A)$ of the objects $\mathrm{Tot}_{\A}(F)$ and $\mathrm{Tot}_{\A}(F')$. 
In the following, we will describe the functor $\mathrm{Tot}_{\A}:\Fun_{\infty}(\B,\A)\rightarrow \pretr(\A)$ in more detail.
We assume the dg category $\A$ contains a zero object.
Note that $\mathrm{Tot}_{\A}(A_0^0)=\Sigma^2 A_0^{\wedge}$ and $\mathrm{Tot}_{\A}(A_1^1)=\Sigma A_1^{\wedge}$.

Let $V'$ and $W$ be the following objects in $\Fun_{\infty}(\B,\A)$
\[
\begin{tikzcd}
A_1'\ar[r,"j'"swap]\ar[rr,bend left=8ex,"0"]&A_2'\ar[r,"0"swap]&0,
\end{tikzcd}
\]
\[
\begin{tikzcd}
0\ar[r,"0"swap]\ar[rr,bend left=8ex,"0"]&A_0\ar[r,"f"swap]&A_1.
\end{tikzcd}
\]

Let $M$ be the complex 
\[
(\Sigma^{2}\A(A_2',A_0)\oplus \Sigma \A(A_1',A_0)\oplus \Sigma \A(A_2',A_1),d)
\]
 where $d(h_1,s_1,s_2)=(d(h_1), d(s_1)+(-1)^{n}(h_1\circ j), d(s_2)+(-1)^{n+1}f\circ h_1)$ where $h_1:A_2'\rightarrow A_0$ is of degree $n+2$ and $s_1$ and $s_2$ are of degree $n+1$.
  
We have the following commutative diagrams in $\C(k)$, which relates the complex $(F,F')$ with $\A(F',F)$
\[
\begin{tikzcd}
0\ar[r]&\A(A_1', \Sigma^{-2}A_1)\ar[r,"\alpha"]& \Sigma^2(V', W)\ar[r,"\beta"] & M\ar[r]&0
\end{tikzcd}
\]
\[
\begin{tikzcd}
0\ar[r]& (F',A_2^{2}) \ar[r,tail]\ar[d, equal]   & (F',F)\ar[r,two heads] \ar[d, tail]                    & (F', U)\ar[r, two heads]\ar[d,tail]                              &0\\
0\ar[r]& \A(F', A_2^{2}) \ar[r,tail]                  & \A(F',F)\ar[r,two heads]\ar[d,two heads]       & \A(F', U)\ar[r]\ar[d, two heads,"\gamma"]               &0\\
         &                                                        &M\ar[r,equal]                                                   &M                                                                            &
\end{tikzcd}
\]
where $\alpha(t)=(0,0,t)$, $\beta(h_i,s_j,t)=(h_1,s_1,s_2)$ and the action of $\gamma$ on the degree $n$ component is given as follows
\[
\gamma(\begin{bmatrix}h_0'&s_1&h_1\\(-1)^{n+1}s_1'&(-1)^n h_1'&s_2\\t&s_2&h_2'\end{bmatrix})=\begin{bmatrix}h_1&s_1&(-1)^{n+1}s_2\end{bmatrix}.
\]

Let $\mathrm{Filt}_3(\pretr(\A))$ be the dg category with objects the twisted objects $(Y=\oplus_{i\in I}\Sigma^{n_i}Y_i^{\wedge},d)$ with a 3-step filtration $I_0\subset I_1\subset I_2=I$ such that $\oplus_{i\in I_0}\Sigma^{n_i}Y_i^{\wedge}\subseteq \oplus_{i\in I_1}\Sigma^{n_i}Y_i^{\wedge}\subseteq \oplus_{i\in I_2}\Sigma^{n_i}Y_i^{\wedge}$ are dg submodules.
Then we have a localization dg functor $\mathrm{Filt}_3(\pretr\A)\rightarrow \pretr(\A)$
 which sends an object $(Y,I_0\subseteq I_1\subseteq I_2)$ to $Y$.
 It admits a fully faithful left adjoint $\pretr(\A)\rightarrow \mathrm{Filt}_3(\pretr(\A))$ which sends a dg module $Y=(\oplus_{i\in I}\Sigma^{n_i}Y_i^{\wedge},d)$ to $(Y,\emptyset\subseteq \emptyset\subseteq I)$.
 Then the totalization dg functor $\mathrm{Tot}_{\A}:\Fun_{\infty}(\B,\A)\rightarrow \pretr(\A)$ is the composition of a fully faithful funtor $i:\Fun_{\infty}(\B,\A)\rightarrow \mathrm{Filt}_3(\A)$ and the localization dg functor $\mathrm{Filt}_3(\pretr(\A))\rightarrow \pretr(\A)$.
 Indeed, for each object $F\in \Fun_{\infty}(\B,\A)$, the dg module $\mathrm{Tot}_{\A}(F)$ has a natural filtration given by 
 \[
 i(F)=(\mathrm{Tot}_3(\A), A_2^{\wedge}\subseteq \Cone(j^{\wedge}:A_1^{\wedge}\rightarrow A_2^{\wedge})\subseteq \mathrm{Tot}_{\A}(F)).
 \]

\newpage
\section{Basic diagram lemmas}\label{sec:diagramlemmas}
\subsection{Homotopy kernels and cokernels}
Recall the morphism category $\Mor$ defined in subsection \ref{Preliminaries}
\[
\begin{tikzcd}
0\ar[r]&1
\end{tikzcd}
\]
and the dg category $\B$ defined in \ref{3term}
\[
\begin{tikzcd}
0\ar[r,"f"] &1\ar[r,"j"]&2.
\end{tikzcd}
\]
where $|f|=|j|=0$, and $d(f)=d(j)=0$, $jf=0$.
 Let $i:k\Mor\rightarrow \B$ be the inclusion dg functor which sends object $0$ to $0$ and object $1$ to $1$.
 Let $i':k\Mor\rightarrow \B$ be the inclusion dg functor which sends object $0$ to $1$ and object $1$ to $2$.
 Let $\Res:\Fun_{\infty}(\B,\A)\rightarrow\Fun_{\infty}(k\Mor,\A)\iso\Mor(\A)$ (resp.~$\Res'$) be the restriction functor along $i$ (resp.~$i'$).  
\begin{definition}\label{homotopykernel} 
By {\em homotopy kernel} of an object $j:B\rightarrow C$ in $\Mor(\A)$ (cf.~Example \ref{Mor(A)}), we mean a homotopy left exact squence $X$ (cf.~Definition \ref{homotopyses}) 
\[
\begin{tikzcd}
A\ar[rr,bend right=8ex,"h"swap]\ar[r,"f"]&B\ar[r,"j"]&C.
\end{tikzcd}
\]
Sometimes we say $f:A\rightarrow B$ or $A$ is the homotopy kernel of $j$. 
 
Dually we define the {\em homotopy cokernel} of an object in $\Mor(\A)$. 
\end{definition}
\begin{remark}
 Assume the dg category $\A$ contains a zero object.
An object $f:A\rightarrow B$ in $\Mor(\A)$ is a {\em homotopy monomorphism} if its homotopy kernel is zero, or equivalently, for each $A'\in\A$, the morphism $\Hom_{\A}(A',A)\rightarrow \Hom_{\A}(A',B)$ induces an isomorphism in each negative degree cohomology and an injection in degree zero cohomology. Dually we have the notion of {\em homotopy epimorphism}.

 A homotopy kernel need not be the homotopy kernel of its homotopy cokernel (if they exist).
  Also, a homotopy kernel need not be a homotopy monomorphism.

An object $f:A\rightarrow B$ in $\Mor(\A)$ is a homotopy monomorphism if and only if the object in $\rep(\Sq)$
\[
\begin{tikzcd}
A^{\wedge}\ar[r,equal]\ar[d,equal]&A^{\wedge}\ar[d,"f^{\wedge}"]\\
A^{\wedge}\ar[r,"f^{\wedge}"swap]&B^{\wedge}
\end{tikzcd}
\]
is a homotopy cartesian square.
\end{remark}
 In the above definition, we have that $\Res'(X)=j$. 
 Compare Definition~\ref{homotopykernel} with Definition \ref{homotopypullbackdef}.
The two definitions are compatible with each other, due to the following observation:
\begin{lemma}\label{Mor(A)and3term}
Let $j:B\rightarrow C$ be an object in $\Mor(\A)$. 
Consider a morphism $\theta$ from $j':B'\rightarrow C'$ to $j$ in $H^0(\Mor(\A))$ as follows
\[
\begin{tikzcd}
B'\ar[r,"j'"]\ar[d,"h_1"swap]\ar[rd,"s_2"]&C'\ar[d,"h_2"]\\
B\ar[r,"j"swap]&C\mathrlap{.}
\end{tikzcd}
\]
Suppose we are given an object
\[
\begin{tikzcd}
A'\ar[r,"f'"]\ar[rr,bend right=8ex,"h'"swap]&B'\ar[r,"j'"]&C'
\end{tikzcd}
\]
in $\mathcal H_{3t}(\A)$.
Then we have the following morphism $\mu$ in $\mathcal H_{3t}(\A)$ which restricts to the morphism $\theta$
\[
\begin{tikzcd}
A'\ar[r,"f'"]\ar[rd,"0"{red,swap},red]\ar[rrd,"0"blue,blue]\ar[d,equal]\ar[rr,"h'",bend left=8ex]&B'\ar[r,"j'"]\ar[d,"h_1"swap]\ar[rd,"s_2"red,red]&C'\ar[d,"h_2"]\\
A'\ar[r,"h_1f'"swap]\ar[rr,"h_2h'-s_2f'"swap,bend right=8ex]&B\ar[r,"j"swap]&C\mathrlap{.}
\end{tikzcd}
\]
The morphism $\mu$ is an isomorphism if and only if the morphism $\theta$ is an isomorphism, cf.~Proposition \ref{termwiseequivalence}.
\end{lemma}
Recall that $\mathcal H_{3t}$ is defined as $H^0(\Fun_{\infty}(\B,\A))$ where $\B$ is the dg category defined by \ref{3term}.
Let us recall the universal property of homotopy left exact sequences. 
Suppose we have 3-term h-complexes $X_1$ and $X_2$ with $X_2$ being homotopy left exact.
Then each morphism $\theta:\Res'(X_1)\rightarrow \Res'(X_2)$ in $H^0(\Mor(\A))$ extends uniquely to a morphism $\mu:X_1\rightarrow X_2$ in $\mathcal H_{3t}(\A)$ such that $\theta=\Res'(\mu)$. 
Indeed, we have the following
\begin{lemma}\label{strictmorphismlift}
Suppose we have 3-term h-complexes $X_i$ in $\A$, $i=1$, $2$, of the form
\[
\begin{tikzcd}
A_i\ar[rr,bend right=8ex,"h_i"swap]\ar[r,"f_i"]&B_i\ar[r,"j_i"]&C_i
\end{tikzcd}
\]
where $X_2$ is homotopy left exact. Then any morphism in $Z^0(\Mor(\A))$
\[
\begin{tikzcd}
B_1\ar[r,"j_1"]\ar[d,"b"swap]\ar[rd,"s_2"red,red]&C_1\ar[d,"h_3"]\\
B_2\ar[r,"j_2"swap]&C_2
\end{tikzcd}
\]
extends to a morphism in $Z^0(\Fun_{\infty}(\B,\A))$ 
\[
\begin{tikzcd}
A_1\ar[rd,"s_1"{red,swap},red]\ar[rrd,"t"{blue},blue]\ar[d,"a"swap]\ar[rr,bend left=8ex,"h_1"]\ar[r,"f_1"]&B_1\ar[r,"j_1"]\ar[d,"b"swap]\ar[rd,"s_2"{red},red]&C_1\ar[d,"h_3"]\\
A_2\ar[rr,bend right=8ex,"h_2"swap]\ar[r,"f_2"swap]&B_2\ar[r,"j_2"swap]&C_2\mathrlap{.}
\end{tikzcd}
\]
\end{lemma}
\begin{proof}
Put $u=h_2f_1$ and $v=h_3h_1-s_2f_1$. Then the pair $(u,v)$ satisfies $d(v)=-h_3j_1f_1-(j_2h_2-h_3j_1)f_1=-j_2h_2f_1=-j_2u$.

Since $X_2$ is homotopy left exact, we have a closed morphism $a:A_1\rightarrow A_2$ of degree $0$ such that there exists a pair $(-s_1,-t)$ where $s_1$ is a morphism from $A_1$ to $B_2$ of degree $-1$, $t$ is a morphism from $A_1$ to $C_2$ of degree $-2$ satisfying $u-f_2a=d(-s_1)$ and $v-h_2a=-d(-t)-j(-s_1)$. So the 6-tuple $(a,b,h_3,s_1,s_2,t)$ is a morphism from $X_1$ to $X_2$ in $Z^0(\Fun_{\infty}(\B,\A))$.
\end{proof}
Let $f:X\rightarrow Y$ be an object in $\Mor(\A)$ and 
\[
\begin{tikzcd}
&X\ar[d,"f"]\\
0\ar[r]&Y
\end{tikzcd}
\]
the associated object $S$ in the category $\rep(\Cosp)$.
Recall that we have a fully faithful functor $\mathcal H_{3t}(\A)\hookrightarrow \rep(\Sq,\A)$.
Then a homotopy kernel of the object $f$ in $\mathcal H_{3t}(\A)$ can be identified with a homotopy pullback of the cospan $S$.
Hence by Proposition \ref{pullbackunique}, the homotopy kernel of an object in $\Mor(\A)$ is unique up to a unique isomorphism if it exists in $\mathcal H_{3t}(\A)$.
\subsection{Restriction of morphisms in $\mathcal H_{3t}(\A)$} \label{res}
Let $\overline{\Sq}$ be the dg $k$-path category of the following graded quiver
\[
\begin{tikzcd}
00\ar[r,"f"]\ar[d,"g"swap]\ar[rd,"h"]&01\ar[d,"j"]\\
10\ar[r,"k"swap]&11
\end{tikzcd}
\]
where $f$, $g$, $j$ and $k$ are closed morphisms of degree $0$ and $h$ is a morphism
of degree $-1$ such that $d(h)=kg-jf$.
The obvious dg functor $\overline{\Sq}\rightarrow \Sq$ is a quasi-equivalance and induces an equivalence of categories $\rep(\Sq,\A)\rightarrow\rep(\overline{\Sq},\A)$.
We will transport the definitions and results in the previous sections to objects in $\rep(\overline{\Sq},\A)$. 
One reason for working with this category is that objects in $\rep(\overline{\Sq},\A)$ are isomorphic to objects given by dg functors $F:\overline{\Sq}\rightarrow \A$, because the dg category $\overline{\Sq}$ is cofibrant.

If we take a span (resp.~cospan) of representable dg modules, if its homotopy pushout (resp.~pullback) exists, we can assume the homotopy pushout (resp.~pullback) restricts to the original span (resp.~cospan), and is
not merely isomorphic to it. 

Let $f:A\rightarrow B$ and $f':A'\rightarrow B'$ be two objects in $\Mor(\A)$. 
Then a morphism $\eta:f\rightarrow f'$ in $H^0(\Mor(\A))$ gives rise to a class of objects in $\rep(\overline{\Sq},\A)$
\[
\begin{tikzcd}
A\ar[r,"f"]\ar[d,"g"swap]\ar[rd,"h"]&B\ar[d,"j"]\\
A'\ar[r,"f'"swap]&B'
\end{tikzcd}
\] 
which are isomorphic to each other: two homotopic morphisms from $f$ to $f'$ give rise to homotopic dg functors $\overline{\Sq}\rightarrow \A$ in the sense of \cite[Remark 2.0.12]{Tabuada10}. 

\begin{lemma}\label{squareepivalence}
We have a quasi-equivalence
\[
\rep_{dg}(\overline{\Sq},\A)\iso\Mor(\Mor(\A)).
\]
In particular, we have an epivalence
\[
\rep(\overline{\Sq},\A)\rightarrow\Fun(\Mor,H^0(\Mor(\A))).
\]
\end{lemma}
\begin{proof}
Recall that we have quasi-equivalences $k\overline{\Sq}\iso k\Mor\otimes k\Mor$ and $\Mor(\A)\iso \rep_{dg}(k\Mor,\A)$.
Then for any small dg category $\B$, we have canonical bijections
\begin{align*}
\Hqe(\B,\rep_{dg}(\overline{\Sq},\A))&\iso\Hqe(\overline{\Sq}\otimes \B,\A)\\
&\iso\Hqe(k\Mor\otimes k\Mor\otimes \B,\A)\\
&\iso\Hqe(\B,\rep_{dg}(k\Mor,\Mor(\A)).
\end{align*}
Therefore we have a quasi-equivalence  
\[
\rep_{dg}(\overline{\Sq},\A)\iso \rep_{dg}(k\Mor,\Mor(\A)).
\]
Thus we have an epivalence
\[
\rep(\overline{\Sq},\A)\iso \rep(k\Mor,\Mor(\A))\xrightarrow{\Dia} \Fun(\Mor,H^0(\Mor(\A))).
\]
\end{proof}

Recall the restriction functors $\Res,\Res':\Fun_{\infty}(\B,\A)\rightarrow \Mor(\A)$ defined before Definition \ref{homotopykernel}.
Suppose we have a morphism $\alpha:X\rightarrow X'$ in $\mathcal H_{3t}(\A)$ of the form
\begin{equation}\label{morphism3term}
\begin{tikzcd}
&A_0\ar[r,"f"]\ar[d,"h_0"swap]\ar[rd,"s_1"red,swap ,red]\ar[rrd,"t"blue,blue]\ar[rr,bend left = 8ex,"h"]&A_1\ar[d,"h_1"swap]\ar[rd,"s_2"red,red]\ar[r,"j"]&A_2\ar[d,"h_2"]\\
&A_0'\ar[r,"f'"swap]\ar[rr,bend right = 8ex,"h'"swap]&A_1'\ar[r,"j'"swap]&A_2'
\end{tikzcd}
\end{equation}
Then by the above discussion, the two morphisms $\Res(\alpha)$, $\Res'(\alpha)$ give rise to two objects in $\rep(\overline{\Sq},\A)$.

\begin{definition}As above, let $\alpha:X\rightarrow X'$ be a morphism in $\mathcal H_{3t}(\A)$. 
We call $\Res(\alpha)\in \rep(\overline {\Sq},\A)$ the {\em restriction} of $\alpha:X\rightarrow X'$ to $f:A\rightarrow B$. 
Similarly, we call $\Res'(\alpha)$ the {\em restriction} of $\alpha:X\rightarrow X'$ to $j:B\rightarrow C$.
\end{definition}

\begin{proposition}\label{push}
Suppose we have objects $X$ and $X'$ in $\mathcal H_{3t}(\A)$ of the following form
\[
\begin{tikzcd}
A_0\ar[rr,bend right=8ex,"h"swap]\ar[r,"f"]&A_1\ar[r,"j"]&A_2.
\end{tikzcd}
\]
where for $X'$ we add to each term a prime symbol superscript. 
Let $\alpha:X\rightarrow X'$ be a morphism in $\mathcal H_{3t}(\A)$ as in \ref{morphism3term}.
\begin{itemize}
\item[1)] Suppose $h_2:A_2\rightarrow A_2'$ is a homotopy equivalence.
 Let $X''\in \rep(\overline{\Sq},\A)$ be the restriction of $\alpha$ to $f:A_0\rightarrow A_1$. 
 Then the following statements hold:
\begin{itemize}
\item[(a)] If $X'$ is homotopy left exact, then $X$ is a homotopy left exact if and only if $X''$ is homotopy cartesian.
\item[(b)] If $X''$ is homotopy cocartesian, then $X$ is homotopy right exact if and only if so is $X'$.
\end{itemize}
\item[2)]Suppose $h_0:A_0\rightarrow A_0'$ is a homotopy equivalence. 
Let $X'''$ be the restriction of $\alpha$ to $j:A_1\rightarrow A_2$. Then the following statements hold:
\begin{itemize}
\item[(a)] If $X$ is homotopy right exact, then $X'$ is homotopy right exact if and only if $X'''$ is homotopy cocartesian.
\item[(b)] If $X'''$ is homotopy cartesian, then $X$ is homotopy left exact if and only if so is $X'$.
\end{itemize}
\end{itemize}
\end{proposition}
\begin{proof}
We prove 1) and then 2) follows by duality. 
For simplicity, we omit the symbol $^{\wedge}$ of the representable dg modules $A^{\wedge}$. 
Since $h_2$ is a homotopy equivalence, we may assume $h_2$ is the identity of ${A_2}$.
The morphism $\alpha$ is given by the following diagram in $\C(\A)$:
\[
\begin{tikzcd}
A_0\ar[rddd,"\begin{bmatrix}s_1\\t\end{bmatrix}"{swap,red},red]\ar[rd,"u=\begin{bmatrix}f\\h\end{bmatrix}"]\ar[dd,"h_0"swap]&&&\\
&V=\Sigma^{-1}\Cone(j)\ar[r,"{[}-1{,}0{]}"]\ar[dd,"\begin{bmatrix}h_1\;\;0\\-s_2\;\; 1\end{bmatrix}"]&A_1\ar[r,"j"]     \ar[dd,"h_1"] \ar[rdd,"s_2"red,red] &A_2\ar[dd,equal]\\ 
A_0'\ar[rd,"u'=\begin{bmatrix}f'\\h'\end{bmatrix}"swap]&&&\\
&V'=\Sigma^{-1}\Cone(j')\ar[r,"{[}-1{,}0{]}"swap]&A_1'\ar[r,"j'"swap]&A_2
\end{tikzcd}
\]
where diagonal maps are the homotopies making the diagram commutative in $\mathcal H(\A)$.
Then $X''$ is given by
\[\begin{tikzcd}
A_0\ar[r,"f"]\ar[d,"h_0"swap]\ar[rd,"s_1"]&A_1\ar[d,"h_1"]\\
 A_0'\ar[r,"f'"swap]&A_1'
\end{tikzcd}\;.
\] 

Put $U=\Cone(f)$, $U'=\Cone(f')$ and $V_3=\Sigma^{-1}\Cone(h_1,f')$. 
We have a canonical morphism $r=(h,j):U\rightarrow A_2$ and similarly a morphism $r':U'\rightarrow A_2'$.
From the object $X$, we have the following diagram in $\D(\A)$ 
\[
\begin{tikzcd}
A_0\ar[r,"f"]\ar[d,"u"swap]&A_1\ar[r,""]\ar[d,equal]&U\ar[r]\ar[d,"r"]&\Sigma A_0\ar[d]\\ 
V\ar[r]&A_1\ar[r,"j"]&A_2\ar[r]\ar[d]&\Sigma V\ar[d]\\
 & &C(r)\ar[r,equal]&C(r)
\end{tikzcd}.
\]

Put $Y=\Sigma^{-1}\Cone(h_1)$ and $W=\Cone(u')$. 
Similarly we have the following diagrams in $\D(\A)$
\[
\begin{tikzcd}
V\ar[r,"{[}-1{,}0{]}"]\ar[d,"\begin{bmatrix}h_1\;\;0\\-s_2\;\;1\end{bmatrix}"swap]&A_1\ar[r,"j"]\ar[d,"h_1"]&A_2\ar[d,equal]\\ 
V'\ar[r,"{[}-1{,}0{]}"]\ar[d,"\begin{bmatrix}0\;\;0\\-1\;\;0\end{bmatrix}"swap]&A_1'\ar[d,"\begin{bmatrix}0\\1\end{bmatrix}"]\ar[r,"j'"]&A_2\\ 
\Sigma Y \ar[r,equal]\ar[d,"-\begin{bmatrix}1\;\;0\\s_2\;\;j'\end{bmatrix}"swap]& \Sigma Y\ar[d,"{[}1{,}0{]}"]&\\
\Sigma V\ar[r,"{[}-1{,}0{]}"swap]&\Sigma A_1&
\end{tikzcd},
\begin{tikzcd}
Y\ar[r,"{[}-1{,}0{]}"]\ar[d,"\begin{bmatrix}1\;\;0\\0\;\;0\\0\;\;1\end{bmatrix}"swap]&A_1\ar[r,"h_1"]\ar[d,"\begin{bmatrix}1\\0\end{bmatrix}"]&A_1'\ar[d,equal]\\
 V_3\ar[r,"\begin{bmatrix}-1\;\;0\;\;0\\0\;\;-1\;\;0\end{bmatrix}"]\ar[d,"{[}0{,}-1{,}0{]}"swap]&A_1\oplus A_0'\ar[d,"{[}0{,}1{]}"]\ar[r,"{[}h_1{,}f'{]}"]&A_1'\\ 
 A_0' \ar[d,"\begin{bmatrix}0\\-f'\end{bmatrix}"swap]\ar[r,equal]& A_0'\ar[d,"0"]&\\
 \Sigma Y\ar[r,"{[}-1{,}0{]}"swap]&\Sigma A_1&
 \end{tikzcd},
 \]
 \[
\begin{tikzcd}
Y\ar[r,"\begin{bmatrix}1\;\;0\\0\;\;0\\0\;\;1\end{bmatrix}"]\ar[d,equal]&V_3\ar[rr,"{[}0{,}-1{,}0{]}"]\ar[d,dashed,blue,"v=\begin{bmatrix}1\;\;\;\;0\;\;\;\;0\\s_2\;\;-h'\;\;j'\end{bmatrix}"]&&A_0'\ar[d,"u'=\begin{bmatrix}f'\\h'\end{bmatrix}"]\\
 Y\ar[r,"\begin{bmatrix}1\;\;0\\s_2\;\;j'\end{bmatrix}"swap]&V\ar[rr,"\begin{bmatrix}h_1\;\;0\\-s_2\;\;1\end{bmatrix}"swap]\ar[d,dashed,blue]&&V'\ar[d,"\begin{bmatrix}0\;\;0\\1\;\;0\\0\;\;1\end{bmatrix}"]\\ 
 & W\ar[rr,equal]&&W
 \end{tikzcd}.
 \]
We have a morphism from the sequence in blue to the mapping triangle of $v$ as follows
\[
\begin{tikzcd}
V_3\ar[r,"v"]\ar[d,equal]&V\ar[r,"\begin{bmatrix}0\;\;0\\h_1\;\;0\\-s_2\;\;1\end{bmatrix}"]\ar[d,equal]&W\ar[r,"\begin{bmatrix}0\;\;0\;\;0\\1\;\;0\;\;0\\0\;\;1\;\;0\end{bmatrix}"]\ar[d,"\theta"]&\Sigma V_3\ar[d,equal]\\
V_3\ar[r,"v"swap]&V\ar[r,"{[}0{,}1{]^{\intercal}}"swap]&\Cone(v)\ar[r,"{[}1{,}0{]}"swap]&\Sigma V_3\mathrlap{.}
\end{tikzcd}
\]
By direct inspection, we see that $\Cone(v)$ is isomorphic to the mapping cone of 
\[
\Cone(\Id_{\Sigma^{-1}A_1})\rightarrow W.
\]
The map $\theta$ is the canonical inclusion from $W$ to this cone.
One checks directly that the rightmost square commutes and the middle square commutes up to homotopy.

Now we see that the composition of the canonical maps $A_0\xrightarrow{[f,h_0,s_1]^{\intercal}} V_3$ and $v:V_3\rightarrow V$ is homotopic to the canonical map $A_0\rightarrow V$. 

(a) Since $X'$ is homotopy left exact, the map $u':A_0'\rightarrow V'$ induces a quasi-isomorphism 
\[
\tau_{\leq 0}\RHom(A,\Sigma^{-i}A_0')\rightarrow\tau_{\leq 0 }\RHom(A,\Sigma^{-i}V')
\]
for each $A'\in \A$.
Thus the induced map
\[ 
\tau_{\leq 0}\Hom(A',V_3)\rightarrow \Hom(A', V)
\]
is a quasi-isomorphism for each $A\in \A$. 

So $X$ is homotopy left exact if and only if $X''$ is homotopy cartesian.

(b) Put $U_3=\Cone((f,h_0)^{\intercal})$.
Consider the following diagrams
\[
\begin{tikzcd}
&A_1\ar[r,equal]\ar[d]&A_1\ar[d]\\
A_0\ar[r]\ar[d,equal]&A_0'\oplus A_1\ar[r]\ar[d]&U_3\ar[d]\\ 
A_0\ar[r]&A_0'\ar[r]&U 
\end{tikzcd},
\begin{tikzcd}
A_1\ar[r]\ar[d,equal]&U_3\ar[r]\ar[d]&U\ar[d]\\ 
A_1\ar[r]&A_1'\ar[d]\ar[r]&U'\ar[d]\\
 & Z\ar[r,equal]& Z
 \end{tikzcd},
\begin{tikzcd}
U\ar[r]\ar[d]&A_2\ar[d,equal]\\
U'\ar[r]&A_2\end{tikzcd}.
\]

Since $X''$ is homotopy cocartesian, the canonical map $U_3\rightarrow A_1'$ induces a quasi-isomorphism 
\[
\tau_{\leq 0}\RHom(A_1',A')\rightarrow \tau_{\leq 0}\RHom(U_3, A')
\]
for each $A'\in\A$. This implies that the canonical map $U\rightarrow U'$ induces a quasi-isomorphism 
\[
\tau_{\leq 0}\RHom(U', A')\rightarrow \tau_{\leq 0}\RHom(U, A')
\]
for each $A'\in\A$. From the rightmost square, we infer that $X$ is homotopy right exact if and only if so is $X'$.
\end{proof}

\begin{corollary}\label{Comp}
Suppose we have the following diagram in $\C_{dg}(\A)$ with lower square being homotopy cartesian and the upper square an object in $\rep(\overline{\Sq},\A)$. 
Then the outer square (i.e.~the square on the right hand side) is homotopy cartesian if and only if so is the upper square. 
\[
\begin{tikzcd}
X_1\arrow[r,"f"]\arrow[d,"g"swap,""{name=1} ]\ar[rd,"h"blue,blue]&X_2\arrow[d,"j",""{name=2}]\\
X_3\ar[d,"g'"swap,""{name=3}]\ar[r,"f'"]\ar[rd,"h'"blue,blue]&X_4\ar[d,"j'",""{name=4}]\\
X_5\ar[r,"k'"swap]&X_6
\arrow[phantom,from=1,to=2,""]\arrow[phantom,from=3,to=4,""]
\end{tikzcd}\;\;,\;\;
\begin{tikzcd}
X_1\ar[r,"f"]\ar[d,"g'g"swap]\ar[rd,"h'g+j'h"{near start,blue},blue]&X_2\ar[d,"j'j"]\\
X_5\ar[r,"k'"swap]&X_6
\end{tikzcd}
\]
\end{corollary}
\begin{proof}
We may assume that the dg category $\A$ admits a zero object and direct sums.
Note that direct sums of homotopy cartesian squares remain homotopy cartesian.
Note also that squares of the form
\[
\begin{tikzcd}
0\ar[r,"0"]\ar[d,"0"swap]\ar[rd,"0"blue,blue]&A\ar[d,equal]\\
0\ar[r,"0"swap]&A
\end{tikzcd}
\]
are trivially homotopy cartesian.

The statement follows by applying Proposition \ref{push} to one of the following two diagrams
\[
\begin{tikzcd}X_1\ar[rd,blue,"s"]\ar[rr,bend left=8ex,blue,"h'g+j'h"blue]\ar[r,"u"]\ar[d,"g"swap]&X_2\oplus X_5\ar[r,"v"]\ar[d,"r"]&X_6\ar[d,equal] &&X_1\ar[r]\ar[d,equal]&X_2\oplus X_3\ar[d]\ar[r] &X_4\ar[d]\\
X_3\ar[rr,bend right=8ex,"h'"{blue,swap},blue]\ar[r,"w"swap]&X_4\oplus X_5 \ar[r,"z"swap] &X_6&&X_1\ar[r]&X_2\oplus X_5\ar[r]&X_6
\end{tikzcd}
\]
where $r=\begin{bmatrix}j\;\;0\\0\;\;\Id\end{bmatrix}$, $s=\begin{bmatrix}h\\0\end{bmatrix}$, $u=\begin{bmatrix}f\\g'g\end{bmatrix}$, $v={[}j'j{,}-k'{]}$, $w=\begin{bmatrix}f'\\g'\end{bmatrix}$ and $z={[}j'{,}-k'{]}$.
\end{proof}
Suppose we have a homotopy pullback square $X_2$
\[
\begin{tikzcd}
B\ar[r,"b"]\ar[d,"e"swap]\ar[rd,"s_2"]&C\ar[d,"f"]\\
E\ar[r,"h"swap]&F
\end{tikzcd}
\]
and an object $X_3$ in $\rep(\overline{\Sq},\A)$
\[
\begin{tikzcd}
A\ar[r,"c"]\ar[d,"d"swap]\ar[rd,"s_3"]&C\ar[d,"f"]\\
D\ar[r,"hg"swap]&F
\end{tikzcd}
\]
where $g$ is a morphism from $D$ to $E$.
Then the triple $(c,gd,s_3)$ satisfies $d(s_3)=hgd-fc=-[-h,\;f]\begin{bmatrix}gd\\c\end{bmatrix}$ and hence we have a morphism $a:A\rightarrow B$ such that there exists a triple $(s,s_1,t)$ such that 
\[
\begin{bmatrix}gd\\c\end{bmatrix}-\begin{bmatrix}e\\b\end{bmatrix}a=\begin{bmatrix}d(s_1)\\d(s)\end{bmatrix}\]
 and 
 \[
 s_3-s_2a=-d(t)-[-h,\;f]\begin{bmatrix}s_1\\s\end{bmatrix}.
 \]
 So we have the following diagram
\[
\begin{tikzcd}
A\ar[d,"d"swap]\ar[rr,bend left=12ex,"s",dashed]\ar[r,dashed,"a"]\ar[rr,bend left=6ex,"c"]\ar[rd,"s_1"{red,swap},red,dashed]\ar[rrd,"t"{blue},blue,dashed,bend left=2ex]&B\ar[r,"b"]\ar[d,"e"swap]\ar[rd,"s_2"red,red,]&C\ar[d,"f"]\\
D\ar[r,"g"swap]&E\ar[r,"h"swap]&F
\end{tikzcd}
\]
where the left hand square $X_1$ is an object in $\rep(\overline{\Sq},\A)$ and $(-s,0,-t)$ is a homotopy between $X_3$ and the composition of $X_1$ and $X_2$, all viewed as morphisms in $Z^0(\Mor(\A))$.
So we have
\begin{corollary}\label{pastinglaw:second}
Keep the notations as above. 
If we have homotopy pullback squares $X_2$ and $X_3$, then there exists a pair $(a,s_1)$ such that $X_1$ is also homotopy pullback and that the composition of $X_1$ and $X_2$ is homotopic to $X_3$ viewed as morphisms in $Z^0(\Mor(\A))$.
\end{corollary}
\subsection{Homotopy pushouts/pullbacks of homotopy exact sequences}
Let $X$ be a homotopy right exact sequence over $\A$
\[
\begin{tikzcd}
A\ar[rr,bend right=8ex,"h"swap]\ar[r,"f"]&B\ar[r,"j"]&C.
\end{tikzcd}
\]
Let $\overline{\alpha}:A\rightarrow A'$ be a morphism in $H^0(\A)$.
Let $S\in \D(\mathrm{Sp})$ be the following span 
\[
\begin{tikzcd}
A\ar[r,"f"]\ar[d,"\alpha"swap]&B\\
A'& 
\end{tikzcd}.
\]
By Lemma \ref{epi}, it is determined by $\overline{\alpha}$ up to an isomorphism and this isomorphism restricts to the identity on each component in $H^0(\A)$.
Assume that the span $S$ admits a homotopy pushout.

\begin{proposition}\label{cons}
We have a homotopy right exact sequence $X'$ and a morphism $\mu:X\rightarrow X'$
\begin{equation}\label{const}
\begin{tikzcd}
&A\ar[r,"f"]\ar[d,"h_0=\alpha"swap]\ar[rd,"s_1"red,swap ,red]\ar[rrd,"t"blue,blue]\ar[rr,bend left = 8ex,"h"]&B\ar[d,"h_1"swap]\ar[rd,"s_2"red,red]\ar[r,"j"]&C\ar[d,equal]\\
&A'\ar[r,"f'"swap]\ar[rr,bend right = 8ex,"h'"swap]&B'\ar[r,"j'"swap]&C
\end{tikzcd}
\end{equation}
such that the restriction of $\mu$ to $f:A\rightarrow B$ is a homotopy pushout of $S$ (cf.~Subsection \ref{res}). 
Furthermore, if both $X$ and $X'$ are homotopy short exact, then the homotopy pushout of $S$ is homotopy bicartesian.
\end{proposition}
\begin{proof}
Let $X''$ be the pushout of $S$
\[
\begin{tikzcd}
&A\ar[r,"f"]\ar[d,"h_0=\alpha"swap]\ar[rd,"s_1"red,swap ,red]&B\ar[d,"h_1"]\\
&A'\ar[r,"f'"swap]&B'
\end{tikzcd}\;.
\]
 Consider the morphism 
 \[
 (h,-j,0):M=\Cone(A\xrightarrow{[f,\alpha]^{\intercal}} B\oplus A')\rightarrow C.
 \] 
 Since $X''$ is the homotopy pushout of $S$, we have a morphism $j':B'\rightarrow C$ such that there exists a graded morphism $(-t, s_2, h'):M\rightarrow C$ of degree $-1$ such that 
 \[
 d(-t,s_2,h')=(-j's_1+h, -j+j'h_1, -j'f').
 \] 
 So we have 
 \[
 (-d(t)-h'\alpha-s_2f, d(s_2), d(h'))=(-j's_1+h, -j+j'h_1, -j'f').
 \]
 Thus we have a morphism $\mu:X\rightarrow X'$ of the form \ref{const} where $X'$ is the 3-term h-complex on the second row. 

By Proposition \ref{push}, we deduce that $X'$ is homotopy right exact and that if both $X$ and $X'$ are homotopy short exact sequences, then $X''$ is homotopy bicartesian. 
\end{proof}

The following is a direct consequence of Lemma \ref{Mor(A)and3term} and Proposition \ref{cons}.
\begin{corollary}\label{cok}
Let $\A$ be an additive dg category. 
For a homotopy cocartesian square $X$ as follows
\[
 \begin{tikzcd}
 A\ar[r,"f"]\ar[d,"g"swap]\ar[rd,"s_1"blue,blue]&B\ar[d,"h"]\\
 C\ar[r,"k"swap]&D
 \end{tikzcd}\;,
 \]
the morphism $f$ has a homotopy cokernel if and only if $k$ has a homotopy cokernel. 
If this is the case, there is a morphism from the homotopy cokernel of $f$ to that of $k$ such that when restricted to the third term, it restricts to the identity, and that when restricted to $f:A\rightarrow B$, it restricts to $X$.
Here we still denote by $X$ the corresponding morphism from $f$ to $k$ in $H^0(\Mor(\A))$ given by the homotopy cocartesian square $X$. 
\end{corollary}

\begin{lemma}\label{univ}
Suppose we are given a homotopy cartesian square $X\in \rep(\overline{\Sq},\A)$ of the form
\[
\begin{tikzcd}
A\ar[r,"f"]\ar[d,"g"swap,""{name=2}]\ar[rd,"s"blue,blue]&B\ar[d,"j"]\\
C\ar[r,"k"swap]&D 
\end{tikzcd} 
\]
where all terms are representable, together with a morphism $\alpha$ from $f':A'\rightarrow B$ to $k:C\rightarrow D$ in $H^0(\Mor(\A))$
\[
\begin{tikzcd}
A'\ar[r,"f'"]\ar[d,"g'"swap]\ar[rd,"s'"blue,blue]&B\ar[d,"j"]\\
C\ar[r,"k"swap]&D
\end{tikzcd}\;.
\] 
We still denote by $X$ the associated morphism from $f:A\rightarrow B$ to $k:C\rightarrow D$ given by the homotopy cartesian square $X$. Then there exists a morphism $\beta: f'\rightarrow f$ in $H^0(\Mor(\A))$ such that $X\circ\beta=\alpha$ and that $\beta$ restricts to $\Id_{B}$.
\end{lemma}
\begin{proof}
Consider the morphism 
\[
A'\xrightarrow{[f',-g',s']^{\intercal}} M= \Sigma^{-1}\Cone(B\oplus C\xrightarrow{[j,k]} D).
\]
 Since $X$ is homotopy cartesian, we have a morphism $r:A'\rightarrow A$ such that there exists a graded morphism $[u,v,w]^{\intercal}:A'\rightarrow M$ of degree $-1$ such that 
 \[
 d([u,v,w]^{\intercal})=[f,-g,s]^{\intercal}\circ r-[f',-g',s']^{\intercal}.
 \]
 So we have $d(u)=fr-f'$, $d(v)=g'-gr$ and $d(w)=-ju-kv+s'-sr$.
 
 We have the following diagram
\[
\begin{tikzcd}
A'\ar[rrd,bend left=8ex,"f'"]\ar[rdd,bend right=8ex,"g'"swap]\ar[rrdd,dashed,bend left=5ex,red,"s'"]\ar[rd,dashed,"r"swap]&&\\
&A\ar[r,"f"]\ar[d,"g"swap,""{name=2}]\ar[rd,"s"blue,blue]&B\ar[d,"j"]\\
&C\ar[r,"k"swap]&D\mathrlap{.} 
\end{tikzcd} 
\]
So we have a morphism $\beta:f'\rightarrow f$ in $H^0(\Mor(\A))$ as follows
\[
\begin{tikzcd}
A'\ar[r,"f'"]\ar[d,"r"swap]\ar[rd,"u"blue,blue]&B\ar[d,equal]\\
A\ar[r,"f"swap]&B
\end{tikzcd}\;
\] 
such that the composition of $\beta$ with $X$ is homotopic to $\alpha$.
\end{proof}
We also have the following direct consequence of Lemma \ref{Mor(A)and3term} and Proposition \ref{cons}.
\begin{lemma}\label{deflationcomposition}
Let $j:B\rightarrow C$ and $j':C\rightarrow D$ be two objects in $H^0(\Mor(\A))$.
Assume that they both admit homotopy kernels which are homotopy short exact sequences
\[
\begin{tikzcd}
A\ar[r,"f"]\ar[rr,"h"swap, bend right=8ex]&B\ar[r,"j"]&C
\end{tikzcd}\;\;,\;\;
\begin{tikzcd}
A'\ar[r,"f'"]\ar[rr,"h'"swap, bend right=8ex]&C\ar[r,"j'"]&D
\end{tikzcd}
\]
If the cospan $L$
\[
\begin{tikzcd}
&B\ar[d,"j"]\\
A'\ar[r,"f'"swap]&C
\end{tikzcd}
\]
admits a homotopy pullback, then we have the following diagram
\[
\begin{tikzcd}
A\ar[dd,bend right=12ex,"h''"{swap,blue},blue]\ar[r,equal]\ar[d,"u"swap]\ar[rd,"s'"{red},red]\ar[rdd,"t"{red},red]&A\ar[d,"f"swap]\ar[dd,bend left=8ex,"h"blue,blue, near end]&\\
E\ar[r,"w"swap,near start]\ar[d,"v"swap]\ar[rr,bend left=6ex,"h'''"{blue,near end},blue]\ar[rd,"s"{red,swap},red]&B\ar[r,"j'j"swap]\ar[d,"j"swap]&D\ar[d,equal]\\
A'\ar[r,"f'"swap]\ar[rr,"h'"{swap,blue},bend right=8ex,blue]&C\ar[r,"j'"swap]&D
\end{tikzcd}
\]
where
\begin{itemize}
\item[a)] The left bottom square is a homotopy pullback square.
\item[b)] The sequence on the leftmost column is a homotopy left exact sequence. 
\item[c)] The sequence in the middle row is a homotopy left exact sequence which is homotopy short exact if the homotopy pullback of the cospan $L$ is homotopy short exact.
\item[d)] The 6-tuple $(\Id_{A},w,f',s',s,t)$ is a morphism from the 3-term h-complex in the first column to that in the second column in $\mathcal H_{3t}(\A)$.
\item[e)] We have $h'''=j's+h'v$ and the 6-tuple $(v,j,\Id_{D},s,0,0)$ is a morphism from the 3-term h-complex in the second row to that in the bottom row in $\mathcal H_{3t}(\A)$.
\end{itemize}
\end{lemma}
Since the diagram functor $\Dia: H^0(\Mor(\A)) \to \Mor(H^0(\A))$ is an epivalence, 
an object $f: A \to B$ of $\Mor(\A)$ is a homotopy equivalence if and only if
this holds for any object $f': A' \to B'$ such that $\Dia(f)$ and $\Dia(f')$ are 
isomorphic in $\Mor(H^0(\A))$.
It is direct to verify the following lemma.
\begin{lemma}\label{homotopyequivalencestable} Consider an object
 \[
\begin{tikzcd}
A\ar[r,"f'"]\ar[d,"g"swap]\ar[rd,"h"]&B\ar[d,"j"]\\
A'\ar[r,"f'"swap]&B'
\end{tikzcd}\;\;.
\] 
of $\rep(\overline{\Sq},\A)$ which is homotopy cartesian.
If the morphism $j$ is a homotopy equivalence, then so is the morphism $g$.
\end{lemma}
\newpage
\section{Exact dg categories}

Let $\A$ be a dg category over $k$ which is {\em additive}, i.e.~$H^0{\A}$ is additive. 
For convenience, we will assume that $Z^0(\A)$ is an additive category.

For an object $f:A\rightarrow B$ in $\rep(\Mor)$, denote by $[f]$ its isomorphism class and by $\overline{f}$ the corresponding morphism in $H^0(\A)$, i.e.~ the corresponding object in $\Fun(\Mor,\D(\A))$. 

Recall that by Lemma \ref{epi}, the diagram functor $\Dia:\D(\Mor)\rightarrow \Mor(\D(\A))$ is an epivalence.
It restricts to an epivalence $H^0(\Mor(\A))\rightarrow \Mor(H^0(\A))$.
Therefore if a property of objects in $H^0(\Mor(\A))$ is stable under isomorphisms, we will also say that the corresponding objects in $\Mor(H^0(\A))$ have the same property.
We emphasize that if the object $f:A\rightarrow B$ is isomorphic to $f':A'\rightarrow B'$ and $g:B\rightarrow C$ is isomorphic to $g':B'\rightarrow C'$ in $H^0(\Mor(\A))$, it is not necessarily true that $gf$
is isomorphic to $g'f'$. 
But this will not cause confusion because in what follows, we consider objects in $H^0(\Mor(\A))$ up to isomorphism. 
For example, when we say ``compositions of deflations are deflations", 
it means that $gf$ is a deflation whenever $g$ and $f$ are deflations and then $g'f'$ is also a deflation, 
although $g'f'$ and $gf$ may not be isomorphic.
\begin{definition}\label{exactdgstructure}
An {\em exact structure} on $\A$ is a class $\mathcal{S}\subseteq \mathcal H_{3t}(\A)$ stable under isomorphisms, consisting of homotopy short exact sequences (called {\em conflations})
\[
\begin{tikzcd}
A\ar[r, tail,"i"]\ar[rr,bend right=8ex,"h"swap]&B\ar[r,two heads, "p"]&C\\
\end{tikzcd} 
\]
where $i$ is called an {\em inflation} and $p$ is called a {\em deflation}, such that the following axioms are satisfied
\begin{itemize}
\item[Ex0]$\Id_{0}$ is a deflation.
\item[{Ex}1]Compositions of deflations are deflations.
\item[{Ex}2]Given a deflation $p:B\rightarrow C$ and any map $c: C'\rightarrow C$ in $Z^0(\A)$, the object
\[
\begin{tikzcd}
&C'\ar[d,"c"]\\
B\ar[r,"p"swap]&C
\end{tikzcd}
\]
admits a homotopy pullback 
\[
\begin{tikzcd} 
{B'}\ar[r,"{p'}"]\ar[d,"{b}"swap]\ar[rd,"s"blue,blue]&{C'}\ar[d,"{c}"]\\
{B}\ar[r,"{p}"swap]&{C}
\end{tikzcd}
\]
and ${p'}$ is also a deflation.
\item[$\Ex2^{op}$]Given an inflation $i: A\rightarrow B$ and any map $a:A\rightarrow A'$ in $Z^0(\A)$, the object
\[
\begin{tikzcd}
A\ar[r,"i"]\ar[d,"a"swap]&B\\
A'&
\end{tikzcd}
\]
admits a homotopy pushout 
\[
\begin{tikzcd}
{A}\ar[r,"{i}"]\ar[d,"{a}"swap]\ar[rd,"s"blue,blue]&{B}\ar[d,"{j}"]\\
{A'}\ar[r,"{i'}"swap]&{B'}
\end{tikzcd}
\]
and ${i'}$ is also an inflation.
\end{itemize}
We call $(\A,\mathcal {S})$ or simply $\A$ an {\em exact dg category}.
\end{definition} 
\begin{definition}\label{def:exactmorphism}
Let $(\A,\mathcal S)$ and $(\A',\mathcal S')$ be exact dg categories.
A morphism $F:\A\rightarrow \A'$ in $\Hqe$ is {\em exact} if the induced functor $\mathcal H_{3t}(\A)\rightarrow \mathcal H_{3t}(\A')$ sends objects in $\mathcal S$ to objects in $\mathcal S'$.
We denote by $\Hqe_{ex}(\A,\A')$ the subset of $\Hqe(\A,\A')$ consisiting of exact morphisms.
\end{definition}
\begin{remark}
Note that the class $\mathcal S\subseteq \mathcal H_{3t}(\A)$ is stable under isomorphisms. 
Let $\overline{\imath}$ be a morphism in $H^0(\A)$. 
If a representative $i$ of $\overline{\imath}$ is an inflation, then any representative in 
$\overline{\imath}$ is also an inflation.
Thus, being an inflation is a property of the morphism $\overline{\imath}$.
We have an even stronger property: If we view $\ol{\imath}:A_0\rightarrow A_1$ as an 
object in $\Mor(H^0(\A))$, then being an inflation is stable under isomorphisms in $\Mor(H^0(\A))$, cf.~Lemma \ref{Mor(A)and3term}.

In the setting of Axiom ${\Ex}2$, since the homotopy pullback is unique up to a unique isomorphism 
in $\mathcal H_{3t}(\A)$, by Lemma \ref{epi}, the object ${j'}$ is unique up to a canonical isomorphism
in $H^0(\Mor(\A))$.
\end{remark}
\begin{remark}\label{truncationexactdgstructure}Let $\A$ and $\A'$ be additive dg categories.
\begin{itemize}
\item[a)] By Lemma \ref{truncationhomotopycartesian} and its dual, the exact dg structures on $\A$ are in bijection with the exact structures on $\tau_{\leq 0}\A$.
\item[b)] Let $F:\A\rightarrow \A'$ be a quasi-equivalence of dg categories.
Then $F$ induces an equivalence of categories $\mathcal H_{3t}(\A)\iso\mathcal H_{3t}(\A')$ which 
preserves and reflects the property of being homotopy short exact.
Thus, the quasi-equivalence $F$ induces a bijection between the class of exact dg structures 
on $\A$ and that on $\A'$. 

For two objects $C$ and $A$ in $\A$, consider the set of `equivalence classes of extensions'
$\mathbb E(C,A)$ (resp.~$\mathbb E'(FC,FA)$), to be defined in Definition \ref{equivalencerelation}, which is associated with the exact structure on $\A$ (resp.~$\A'$). 
The quasi-equivalence $F$ induces a bijection between $\mathbb E(C,A)$ and $\mathbb E'(C,A)$, cf.~Lemma \ref{quasiequivalencebifunctor}. 
\item[c)] Let $\A$ be an exact dg category and $\A\rightarrow \pretr(\A)$ be the inclusion of $\A$ into its pretriangulated hull.
Let $\A'$ be the additive closure of $\A$ in $\pretr(\A)$.
Since $H^0(\A)$ is an additive category, the inclusion of $\A\rightarrow \A'$ is a quasi-equivalence.
By b), we may replace $\A$ by $\A'$ and assume that the dg category $\A$ is such that $Z^0(\A)$ is an additive category.

\end{itemize}
\end{remark}
Note that an object $f:A\rightarrow B$ in $H^0(\Mor(\A))$ is a deflation (resp. an inflation) if and only if it admits a homotopy kernel (resp. homotopy cokernel) (cf.~Definition \ref{homotopykernel}) which is a conflation. Observe that homotopy equivalences are both inflations and deflations by axioms 
Ex0, {Ex}2 and {Ex}$2^{op}$.

\begin{example}
Let $\A$ be an additive category which we consider as a dg category concentrated in degree 0. 

Let $A$, $B$ and $C$ be objects in $\A$ and consider a 3-term h-complex $X$ in $\A$
\[
\begin{tikzcd}
A\ar[rr,"h"swap,bend right=8ex]\ar[r,"f"]&B\ar[r,"j"]&C.
\end{tikzcd} 
\]
Since $\A$ is concentrated in degree zero, the homotopy $h$ is zero. 
By Example \ref{ordinary}, the 3-term h-complex $X$ is a homotopy short exact sequence if and only if the corresponding sequence $0\rightarrow A\xrightarrow {f}B\xrightarrow{j}C\rightarrow 0$ is a kernel-cokernel pair in the additive category $\A$. 
Also, two kernel-cokernel pairs are isomorphic (resp. equivalent) if and only if the corresponding homotopy short exact sequences are isomorphic (resp. equivalent). 

By \cite[Appendix A]{Keller90}, the axioms of a Quillen exact structure on $\A$ correspond to the axioms of an exact dg structure on $\A$.
Therefore, to endow an additive category $\A$ with a Quillen exact structure is the same as to endow $\A$ with an exact dg structure.
\end{example}
\begin{example}
 Let $\A$ be a pretriangulated dg category.
 By Remark \ref{truncationexactdgstructure}, we may assume that the dg category $\A$ is strictly pretriangulated.
 Then each morphism $f:A\rightarrow B$ admits a homotopy cokernel
 \[
 \begin{tikzcd}
 A\ar[r,"f"] \ar[rr, bend right=8ex,"h"swap]&B\ar[r,"j"]&\Cone(f).
 \end{tikzcd}
 \]
 Dually, each morphism admits a homotopy kernel. 
 It is not hard to check that the class of all homotopy short exact sequences in $\A$ defines an exact structure on $\A$.
 If not stated otherwise, we always consider this maximal exact structure on a pretriangulated dg category.
\end{example}
\begin{example-definition}\label{exactdgextensionclosed}
A full dg subcategory $\A'$ of an exact dg category $\A$ is {\em extension-closed} provided that for each conflation in $\A$
\[
\begin{tikzcd}
A\ar[r,"f"]\ar[rr,bend right=8ex,"h"swap]&B\ar[r,"j"]&C,
\end{tikzcd}
\]
if $A$ and $C$ belong to $\A'$, then so does $B$.

Let $\A'$ be an extension-closed subcategory of an exact dg category $(\A,\mathcal S)$.
Let $X$ be a conflation in $\A$ with all terms in $\A'$.
Then $X$ can be seen as a 3-term h-complex in $\A'$.
By Remark \ref{subcategory}, the object $X$ is a homotopy short exact sequence in $\A'$.

We have an inclusion functor $\mathcal H_{3t}(\A')\hookrightarrow \mathcal H_{3t}(\A)$.
Let $\mathcal S'$ be the class of objects in $\mathcal H_{3t}(\A')$ whose image in $\mathcal H_{3t}(\A)$ belongs to $\mathcal S$. 
We show that the class $\mathcal S'$ defines an exact dg structure on $\A'$.
Axiom ${\Ex}0$ is direct to check. 
Axiom ${\Ex}1$ follows from Lemma \ref{deflationcomposition} and the fact that $\A'$ is extension-closed in $\A$.
Consider the span $S$ of a morphism in $\A'$ along an inflation in $\A'$.
It admits a homotopy pushout in $\A$, and by Proposition \ref{cons}, 
the homotopy pushout $X$ is an extension of objects in $\A'$ and hence remains an object in $\A'$. 
Again by Proposition \ref{cons}, the pushout of the inflation in $\A'$ is an inflation in $\A$ and 
the conflation that this inflation belongs to is in $\mathcal S'$.
Thus Axiom $\Ex2^{op}$ is satisfied.
Dually, one checks Axiom $\Ex2$.
\end{example-definition}
\begin{example}\label{exm:Iyama--Yang}
We consider a special case of~\cite[1.2]{IyamaYang20} where the torsion pair $\mathcal S=\X\perp\Y$ is a t-structure.
Let $\T$ be an algebraic triangulated category and $\mathcal S$ a thick subcategory with a t-structure $\mathcal S=\X\perp \Y$. 
Assume that $\T$ has torsion pairs $\T=\X\perp \X^{\perp}={^{\perp}\Y}\perp \Y$ where $(-)^{\perp}$ and $^{\perp}(-)$ both denote the Hom orthogonal.
Put $\Z=\X^{\perp}\cap {^{\perp}\Sigma \Y}$ and $\U=\T/\mathcal S$. 
So we have a sequence of functors $\Z\hookrightarrow \T\xrightarrow{\pi}\T/\mathcal S=\U$.
By~\cite[Theorem 1.1]{IyamaYang20}, we have that the above composition is an equivalence of $k$-categories $F:\Z\iso \U$.
Indeed, the triangulated structure on $\Z$ is canonical and does not depend on the categories $\mathcal X$ and $\mathcal Y$, once we fix a dg enhancement for $\T$.
We denote the dg enhancements of the above categories by adding a subscript $dg$ to the notations of the corresponding categories.  
In particular, we have that $\U_{dg}$ is the dg quotient $\T_{dg}/\mathcal S_{dg}$.
Claim: the composition of dg functors $\Z_{dg}\hookrightarrow \T_{dg}\rightarrow \U_{dg}$ induces a quasi-equivalence  $\tau_{\leq 0}\Z_{dg}\iso \tau_{\leq 0}\U_{dg}$.
Indeed, let $Z$ be an object in $\Z$. Then $\Sigma^{-1}Z$ is in the subcategory $\Sigma^{-1}\X^{\perp}\cap {^{\perp}\Y}\subseteq \Sigma^{-1}\X^{\perp}\cap {^{\perp}\Sigma \Y}$.
We take the triangle of $\Sigma^{-1}Z$ associated with the t-structure $\T=\X\perp\X^{\perp}$
\[
\begin{tikzcd}
V\rightarrow \Sigma^{-1}Z\rightarrow Z'\rightarrow \Sigma V.
\end{tikzcd}
\]
We see that $Z'$ is in $\Z$ and $F(Z')\iso \Sigma^{-1}F(Z)$.
For $i\geq 0$ and $W\in \Z$, we have 
\[
\Hom_{\T}(Z',\Sigma^{-i}W)\iso \Hom_{\T}(Z,\Sigma^{-i+1}W).
\]
Therefore by induction we have 
\[
\Hom_{\T}(Z,\Sigma^{-i+1}W)\iso \Hom_{\U}(F(Z),\Sigma^{-i+1}F(W))
\]
 for $\mbox{$i\geq 0$}$.
Since $\U_{dg}$ is pretriangulated, we have that $\Z_{dg}$ is an exact dg category with conflations given by all homotopy short exact sequences in $\Z_{dg}$.
Note that in general $\Z_{dg}$ is not a pretriangulated dg category. 
\end{example}
The following diagram lemma is useful. 
It is a direct consequence of Lemma \ref{univ} and Axiom $\Ex2$ (or $\Ex2^{op}$).
\begin{lemma}\label{fact}
Let $(\A,\mathcal S)$ be an exact dg category. 
Suppose we are given conflations $X$
\[
 \begin{tikzcd}
 A\ar[r,"f"]          \ar[rr, bend right=8ex,"h"swap]                  &B\ar[r,"j"]&C
\end{tikzcd}
\]
and $X'$
\[
 \begin{tikzcd}
 A'\ar[r,"f'"]          \ar[rr, bend right=8ex,"h'"swap]                  &B'\ar[r,"j'"]&C'
\end{tikzcd}.
\]
Let $\alpha:X\rightarrow X'$ be a morphism in $\mathcal H_{3t}(\A)$. 
Denote by $\overline a:A\rightarrow A'$ (resp.~$\overline c:C\rightarrow C'$) the restriction of $\alpha$ to $A$ (resp.~$C$). 
Then $\alpha$ can be written as a composition $X\xrightarrow{\beta} \tilde{X}\xrightarrow{\gamma} X'$ where $\tilde{X}$ is a conflation with ends $A'$ and $C$, and the morphism $\beta$ restricts to $\overline a:A\rightarrow A'$ and $\Id_{C}$, and the morphism $\gamma$ restricts to $\Id_{A'}$ and $\overline{c}:C\rightarrow C'$.\end{lemma}
\begin{remark}If we use the biadditive bifunctor $\mathbb E$ which is to be defined in \ref{bi}, then we have $[X']\in \mathbb E(C,A')$ and the morphism $\beta$ (resp. $\gamma$) shows that $[\tilde{X}]=\overline a_*[X]$ (resp. $[\tilde{X}]=\overline c^{*}[X']$).
\end{remark}
Combining Lemma \ref{fact} and Lemma \ref{homotopyequivalencestable}, we have
\begin{corollary}\label{middleterm}
Let $\A$ be an exact dg category.
If a morphism $\alpha:X\rightarrow X'$ of conflations restricts to homotopy equivalences on both ends, it restricts to a homotopy equivalence between the middle terms.
\end{corollary}
\begin{definition}\label{equivalencerelation}
Let $(\A,\mathcal S)$ be an exact dg category and $A,C$ two objects in $H^0\A$. 
Let $X_i$, $i=1,2$, be two conflations of the form
\begin{equation}\label{twoconflations}
\begin{tikzcd}
A\ar[r,"f_i"]\ar[rr,"h_i"swap,bend right=8ex]&B_i\ar[r,"j_i"]&C
\end{tikzcd}. 
\end{equation}
A morphism $\theta:X_1\rightarrow X_2$ in $\mathcal H_{3t}(\A)$ is called an {\em equivalence} if it restricts to $\Id_{A}$ and $\Id_{C}$. 
Note that by Corollary \ref{middleterm}, an equivalence of conflations is necessarily an isomorphism in $\mathcal H_{3t}(\A)$.
For a conflation, we denote by $[X]$ the equivalence class to which $X$ belongs.

We define $\mathbb E_{\mathcal S}(C,A)$ to be the set of equivalence classes of conflations 
\[
\begin{tikzcd}
A\ar[r,"f"]\ar[rr,"h"swap,bend right=8ex]&B\ar[r,"j"]&C
\end{tikzcd}
\]
with fixed ends $A$ and $C$.
When there is no risk of confusion, we will simply denote it by $\mathbb E(C,A)$.
\end{definition} 
\begin{lemma}\label{quasiequivalencebifunctor}
Let $F:\A\rightarrow \A'$ be a quasi-equivalence of dg categories. 
Suppose $\A$ is an exact dg category.
Equip $\A'$ with the induced exact structure from the quasi-equivalence $F$.
For two objects $C$, $A$ in $\A$, consider the sets $\mathbb E(C,A)$ (resp.~$\mathbb E'(FC,FA)$), which is associated with the exact structure on $\A$ (resp.~$\A'$). 
The quasi-equivalence $F$ induces a bijection 
\[
\mathbb E(C,A)\rightarrow \mathbb E'(FC,FA).
\]
\end{lemma}
\begin{proof}
It is direct to check that the above map is an injection.
We show that it is a surjection.
For an object $X\in\mathcal H_{3t}(\A)$, we denote by $FX$ the object in $\mathcal H_{3t}(\A')$ induced by $F$.

Let $[X']$ be an element in $\mathbb E'(FC,FA)$. 
Suppose $X'$ is of the form
\[
\begin{tikzcd}
FA\ar[r,"f'"]\ar[rr,"h'"swap,bend right=8ex]&B'\ar[r,"j'"]&FC\mathrlap{.}
\end{tikzcd}
\]
Since $F:\A\rightarrow\A'$ is a quasi-equivalence, we have a homotopy equivalence $FB\iso B'$ for some $B\in\A$.
Hence we have an isomorphism in $H^0(\Mor(\A'))$ of the form
\[
\begin{tikzcd}
FB\ar[d]\ar[rd,"s"red,red]\ar[r,"F{(}j{)}"]&FC\ar[d,equal]\\
B'\ar[r,"j'"swap]&FC\mathrlap{.}
\end{tikzcd}
\]
By Lemma \ref{Mor(A)and3term}, we may assume $B'=FB$ and $j'=F(j)$ for some $j:B\rightarrow C$ in $Z^0(\A)$.
The morphism $j:B\rightarrow C$ is a deflation. 
Suppose its homotopy kernel $Y$ is of the form
\[
\begin{tikzcd}
A'\ar[r,"\tilde{f}"]\ar[rr,"\tilde{h}"swap,bend right=8ex]&B\ar[r,"j"]&C\mathrlap{.}
\end{tikzcd}
\]
Then we have an isomorphism $\theta: X'\rightarrow FY$ in $\mathcal H_{3t}(\A)$ which restricts to the identity of $j':FB\rightarrow FC$ in $H^0(\Mor(\A'))$.
Let $\overline{a'}:FA'\rightarrow FA$ be the restriction of $\theta$ to $FA$.
Then we have a morphism $a:A\rightarrow A'$ such that $F(a)$ is a homotopy inverse of $a'$.
We have an isomorphism in $H^0(\Mor(\A'))$ of the form
\[
\begin{tikzcd}
A\ar[r,"f=\tilde{f}a"]\ar[rd,"0"red,red]\ar[d,"a"swap]&B\ar[d,equal]\\
A'\ar[r,"\tilde{f}"swap]&B
\end{tikzcd}
\]
By Lemma \ref{Mor(A)and3term}, we infer that $X'$ is equivalent to $FX$ where $X$ is of the form
\[
\begin{tikzcd}
A\ar[r,"f"]\ar[rr,"h"swap,bend right=8ex]&B\ar[r,"j"]&C\mathrlap{.}
\end{tikzcd}
\] 

\end{proof}
\begin{proposition}\label{property}
Suppose $(\A, \mathcal S)$ is an exact dg category. 
The following statements hold:
\begin{itemize}
\item[a)] The diagram
\[
\begin{tikzcd}
A\ar[r,"{[}1\ 0{]}^{\intercal}"]\ar[rr,"0"swap,bend right =8ex]&A\oplus B\ar[r,"{[}0\ 1{]}"]\ar[r]&B\\
\end{tikzcd}
\]
 is a conflation for $A,B\in \A$.
\item[b)]If a morphism $g:B\rightarrow C$ in $Z^0(\A)$ admits a homotopy kernel and for some $h:B'\rightarrow B$ the composition $gh:B'\rightarrow C$ is a deflation, then $g$ is a deflation.
\item[b)$^{\text{op}}$]If a morphism $f:A\rightarrow B$ admits a homotopy cokernel and for some $e:B\rightarrow B'$ the composition $ef:A\rightarrow B'$ is an inflation, then $f$ is an inflation.
\item[c)]\label{direct}The direct sum of two conflations is a conflation.
\item[d)]Axiom ${\Ex1}^{op}$ holds.
\end{itemize}
\end{proposition}
\begin{proof}Suppose we are in the situation of Axiom {Ex}2. 

\bigskip\noindent
{\em Claim 1:} $B\oplus C'\xrightarrow{[j,c]} C$ is a deflation.

We have a conflation $X\in\mathcal H_{3t}(\A)$ of the form
\[
\begin{tikzcd}
A\ar[r,"f"]\ar[rr,bend right=8ex,"h"swap]&B\ar[r,"j"]&C\mathrlap{.}\\
\end{tikzcd}
\]
By Axiom {Ex}2, the cospan
\[
\begin{tikzcd}
&C'\ar[d,"c"]\\
B\ar[r,"j"swap]&C
\end{tikzcd}
\]
admits a homotopy pullback $Y$:
\[
\begin{tikzcd} 
 {B'}\ar[r," {j'}"]\ar[d," {b}"swap]\ar[rd,"s_2"red,red]& {C'}\ar[d," {c}"]\\
 {B}\ar[r," {j}"swap]& {C}
\end{tikzcd}\;\;.  
\]
By the dual of Proposition \ref{cons}, we have a morphism $\mu:X'\rightarrow X$ in $\mathcal H_{3t}(\A)$
\[
\begin{tikzcd}
&A\ar[r,"f'"]\ar[d,equal]\ar[rd,"s_1"red,swap ,red]\ar[rrd,"t"blue,blue]\ar[rr,bend left = 8ex,"h'"]&B'\ar[d,"b"swap]\ar[rd,"s_2"red,red]\ar[r,"j'"]&C'\ar[d,"c"]\\
&A\ar[r,"f"swap]\ar[rr,bend right = 8ex,"h"swap]&B\ar[r,"j"swap]&C
\end{tikzcd}
\]
where $X'$ is the 3-term h-complex in the first row.

By Axiom {Ex}2, the morphism ${j'}$ is also a deflation. 
So $X'$ is a conflation and in particular a homotopy short exact sequence. 
Hence, by the dual of Proposition \ref{cons}, $Y$ is homotopy bicartesian.

Let $X''$ be the 3-term h-complex
\[
\begin{tikzcd}
 {B'}\ar[r,"\begin{bmatrix} {j'}\\ {b}\end{bmatrix}"]\ar[rr,"s_2"swap,bend right=8ex]& {C'}\oplus  {B}\ar[r,"{[} {c}{,}-{j}{]}"]&C\\
\end{tikzcd}\;\;. 
\]
Then $X''$ is a homotopy short exact sequence. 

\bigskip\noindent
{\em Claim 2:} $X''$ is a conflation.

Consider the conflation $X$ and the object $f':A\rightarrow B'$ in $H^0(\Mor(\A))$.
The span $S$
\[
\begin{tikzcd}
A\ar[r,"f"]\ar[d,"f'"swap]&B\\
{B'}&
\end{tikzcd}
\] 
admits a homotopy pushout $Z$
\[
\begin{tikzcd}
A\ar[rd,"s_3"{red},red]\ar[r,"f"]\ar[d,"f'"swap]& B\ar[d,"h_1=\begin{bmatrix}0\\\Id\end{bmatrix}"]\\
B'\ar[r,"f''=\begin{bmatrix}j'\\b\end{bmatrix}"swap]&C'\oplus B
\end{tikzcd}.
\] 
where $s_3=\begin{bmatrix}-h'\\s_1\end{bmatrix}$. 
By Axiom {Ex}$2^{op}$, the morphism $B'\xrightarrow{[b,j']^{\intercal}} B\oplus C'$ is an inflation. 
So by Proposition \ref{cons}, there exists a morphism $X\rightarrow \tilde{X''}$ in $\mathcal H_{3t}(\A)$
\[
\begin{tikzcd}
&A\ar[r,"f"]\ar[d,"f'"swap]\ar[rd,"s_3"red,swap ,red]\ar[rrd,"t"blue,blue]\ar[rr,bend left = 8ex,"h"]&B\ar[d,"h_1"swap]\ar[rd,"s_2"red,red]\ar[r,"j"]&C\ar[d,equal]\\
&B'\ar[r,"f''"swap]\ar[rr,bend right = 8ex,"h''"swap]&C'\oplus B\ar[r,"j''"swap]&C
\end{tikzcd}
\]
where $\tilde{X''}$ is the 3-term h-complex in the second row.
By the uniqueness of homotopy cokernels, 
the homotopy right exact sequence $\tilde{X''}$ is isomorphic to $X''$.
By Axiom Ex$2^{op}$, $\tilde{X''}$ is a conflation. 
This proves Claim 2.
Then the morphism $B\oplus C'\xrightarrow{[j,c]} C$ is a deflation.
This proves Claim 1. 
Note that during the proof we only use Axioms {Ex}2 and {Ex}$2^{op}$, 
so the dual of the claims also holds.

Now we are ready to prove $\mathrm{b})$.  
Suppose we have a homotopy left exact sequence 
\[
\begin{tikzcd}
K\ar[r,"f"] \ar[rr,"h"swap,bend right=8ex]&B\ar[r,"g"]&C
\end{tikzcd}
\]
From the cospan
\[
\begin{tikzcd}
&B\ar[d,"g"]\\
B'\ar[r,"gh"swap, two heads]&C
\end{tikzcd},
\]
we know that the morphism $B\oplus B'\xrightarrow{[g,gh]} C$ is a deflation by claim 1. 
Then the morphism
\[
{[}g\ 0{]}:B\oplus B'\xrightarrow{\begin{bmatrix}1&-h\\0&1\end{bmatrix}} B\oplus B'\xrightarrow{\begin{bmatrix}g,gh\end{bmatrix}} C
\]
is also a deflation by Axiom {Ex}1. 

The 3-term h-complex
\[
\begin{tikzcd}
K\oplus B'\ar[r,"u"]\ar[rr,"{[}h{,}\;0{]}"swap,bend right=8ex]&B\oplus B'\ar[r,"v"]&C\mathrlap{,}\\
\end{tikzcd}
\]
where $u=\begin{bmatrix}f&0\\0&\Id_{B'}\end{bmatrix}$ and $v={[}g{,}\;0{]}$, is homotopy left exact, as the direct sum of a homotopy left exact sequence and a trivial 3-term h-complex.
Therefore it is a conflation. 

We have the following morphism of homotopy left exact sequences 
\[
\begin{tikzcd}
K\oplus B'\ar[rrd,"0"{blue,near end},blue,bend right=2ex]\ar[rd,"0"{red,swap},red,bend right=2ex]\ar[r,"u"]\ar[d,"{[}1{,} 0{]}"swap]\ar[rr,bend left=8ex,"{[}h{,}\;0{]}"] &B\oplus B'\ar[rd,"0"red,red]\ar[d,"{[}1{,}0{]}"swap]\ar[r,"v"]&C\ar[d,equal]\\
K\ar[r,"f"swap] \ar[rr,"h"swap,bend right=8ex]&B\ar[r,"g"swap]&C
\end{tikzcd} 
\]
where the left square is homotopy bicartesian, we infer that the morphism $g:B\rightarrow C$ is a deflation by {Ex}$2^{op}$. This proves Statement b).

Since the composition 
\[
\Id_B:B\xrightarrow{\begin{bmatrix}0\\1\end{bmatrix}} A\oplus B\xrightarrow{[0,\;1]} B
\] 
is a deflation and the morphism $A\oplus B\xrightarrow{[0,\;1]} B$ admits a homotopy kernel,
it is a deflation by Statement b). 
So Statement a) follows.

Let $X_i$, $i=1$, $2$, be a conflation of the form
\[
\begin{tikzcd}
A_i\ar[r,"f_i"]\ar[rr,bend right=8ex,"h_i"swap]&B_i\ar[r,"j_i"]&C_i\mathrlap{\;.}
\end{tikzcd}
\]
Clearly the direct sum $X_1\oplus X_2$ is homotopy short exact and the morphism
\[
\begin{tikzcd}
B_1\oplus B_2\ar[r,"\begin{bmatrix}j_1\ 0\\0\ j_2\end{bmatrix}"]&C_1\oplus C_2 
\end{tikzcd} 
\]
can be written as a composition 
\[\begin{tikzcd}
B_1\oplus B_2\ar[r,"r=\begin{bmatrix}1\ 0\\0\ j_2\end{bmatrix}"]&B_1\oplus C_2\ar[r,"s=\begin{bmatrix}j_1\ 0\\0\ 1\end{bmatrix}"]&C_1\oplus C_2
\end{tikzcd}.
\] 

Observe that the morphism $B_1\oplus B_2\xrightarrow{r} B_1\oplus C_2$ is a homotopy pullback of the deflation $j_2:B_2\rightarrow C_2$ along $B_1\oplus C_2\rightarrow  C_2 $ and hence is a deflation. 
Similarly, the morphism $B_1\oplus C_2\xrightarrow{s} C_1\oplus C_2$ is also a deflation. 
So the composition $B_1\oplus B_2\rightarrow C_1\oplus C_2$ is a deflation. 
This proves Statement c).

It remains to prove Axiom ${\Ex1}^{op}$ and then one proves b)$^{op}$ dually. 

Let $f:A\rightarrow B$ be an inflation with the homotopy cokernel given by
\[
\begin{tikzcd}
 {A}\ar[r," {f}"]\ar[rr,bend right=8ex,"h"swap]& {B}\ar[r,"j"]&C
\end{tikzcd}
\] 
Let $f':B\rightarrow B'$ be another inflation. 
Our aim is to show that the morphism $f'f$ is also an inflation. 

By Axiom ${\Ex2}^{op}$, the span
\[
\begin{tikzcd}
B\ar[r,"f'"]\ar[d,"j"swap]&B'\\
C&
\end{tikzcd}
\]
admits a homotopy pushout 
\[
\begin{tikzcd}
 {B}\ar[r," {f'}"]\ar[d," {j}"swap]\ar[rd,"s"red,red]& {B'}\ar[d," {j'}"]\\
 {C}\ar[r," {f''}"swap]& {C'}
\end{tikzcd}
\]
which is homotopy bicartesian. 
Thus by the dual of Claim 1, 
the morphism $[ {j'}, {f''}]: {B'}\oplus  {C}\rightarrow  {C'}$ is a deflation. 
By the dual of Corollary \ref{cok}, the morphism ${j'}$ admits a homotopy kernel which is isomorphic to $f'f:A\rightarrow B'$.

We have the following commutative diagram in $\A$
\[
\begin{tikzcd}
 {B'}\oplus  {B}\ar[rd,"{[}0{,}-s{]}"red,red]\ar[r,"\begin{bmatrix}1\ 0\\0\  {j}\end{bmatrix}"]\ar[d,"{[}1{,} {f'}{]}"swap]& {B'}\oplus  {C}\ar[d,"{[} {j'}{,} {f''}{]}"]\\
 {B'}\ar[r," {j'}"swap]& {C'}
\end{tikzcd}.
\]

By Axiom ${\Ex1}$, the morphism $[ {j'}, {f''}]\begin{bmatrix}1&0\\0&{j}\end{bmatrix}$ is a deflation.
By the above diagram, the morphism ${j'}[1, {f'}]$ is homotopic to it and hence is also a deflation. 
So by Statement b), the morphism ${j'}$ is a deflation. 
Hence the morphism $f'f$ is an inflation.
This proves Statement d).
\end{proof}
\subsection{The biadditive bifunctor $\mathbb E$}\label{bi}

Suppose we are given an element $[X]\in \mathbb E(C,A)$, where $X$ is as follows
\begin{equation}\label{X}
\begin{tikzcd}
A\ar[r,"{f}"]       \ar[rr,bend right=8ex,"{h}"swap]      &B\ar[r,"{j}"]         &C
\end{tikzcd}\;\;,
\end{equation}
and a map $\overline{a}:A\rightarrow A'$ in $H^0(\A)$. 
Below we will prove that there exists a unique equivalence class of conflations $[X']\in \mathbb E(C,A')$, 
such that for each representative $X$ of $[X]$, 
there exists a representative $X'$ of $[X']$ with a morphism $\theta:X\rightarrow X'$ 
which restricts to the identity on $C$ and the morphism $\overline{a}:A\rightarrow A'$.  

By Axiom ${\Ex2}^{op}$, the existence of $[X']$ is clear from Proposition \ref{cons}. 

Let us clarify its uniqueness. 
Suppose we have morphisms $\theta_1:X\rightarrow X_1$ and $\theta_2:X\rightarrow X_2$ with the required property.
So the objects $X_1  $ and $X_2  $ are both conflations.
By Proposition \ref{cons}, we may assume that the restriction of $\theta_1$ to $f:A\rightarrow B$ is a homotopy cocartesian square. 
Let $S_i\in\rep(\overline{\Sq},\A)$, $i=1$, $2$, be the restriction of $\theta_i$ to $f:A\rightarrow B$. 
Assume that $X_i  $, $i=1,2$, is of the form
 \[
 \begin{tikzcd}
 A'  \ar[r,"f_i"]\ar[rr,bend right=8ex,"h_i"swap]&B_i\ar[r,"j_i"]&C
 \end{tikzcd}\;,
 \]
 and that $S_1$ is a homotopy pushout of the cospan $L$
 \[
 \begin{tikzcd} 
 A\ar[r,"f"]\ar[d,"a"swap]&B
 \\A'
 \end{tikzcd}.
 \] 
 By the universal property of homotopy pushouts, cf.~Remark \ref{pullbackuniversal}, there is a unique morphism $\mu$ in $\rep(\overline{\Sq},\A)$ from $S_1$ to $S_2$ which restricts to $\Id_{L}$.
 We still denote by $S_i$ the morphism in $H^0(\Mor(\A))$ from $f$ to $f_i$ given by the square $S_i$ for $i=1$, $2$.
 By Lemma \ref{squareepivalence}, the morphism $\mu:S_1\rightarrow S_2$ yields morphisms $\Id_{f}$, $S_1$, $S_2$ and a morphism $\alpha$ from $f_1:A'\rightarrow B_1$ to $f_2:A'\rightarrow B_2$ which restricts to $\Id_{A'}$ and such that we have $\alpha\circ S_1=S_2\circ \Id_{f}$ in $H^0(\Mor(\A))$. 
 By the universal property of homotopy pushouts, the morphism $\alpha$ induces a morphism $\theta_3:X_1\rightarrow X_2$ in $\mathcal H_{3t}(\A)$ such that we have $\theta_3\circ \theta_1=\theta_2$. 
 Then the morphism $\theta_3$ is an equivalence between $X_1$ and $X_2$.
 Since $\theta_3$ is necessarily an isomorphism by Corollary \ref{middleterm}, we have that the morphism $\alpha:S_1\rightarrow S_2$ is an isomorphism.
 So we have
 \begin{lemma}\label{bothconflation}
 Let $\A$ be an exact dg category. 
 Suppose we have a morphism $\alpha:X\rightarrow X'$ of conflations in $\mathcal H_{3t}(\A)$ given by the diagram \ref{morphism3term}.
 If the morphism $h_2:A_2\rightarrow A_2'$ is a homotopy equivalence, then the left square is homotopy bicartesian.
 \end{lemma}  
 
\begin{proposition}The operation $\overline{a}:A\rightarrow A'\mapsto \overline{a}_*:\mathbb E(C,A)\rightarrow \mathbb E(C,A')$, 
where an element $[X]$ is sent to $\overline{a}_*[X]{\coloneqq}[X']$ which is specified by the above mentioned property, 
together with its dual operation $\overline{c}:C'\rightarrow C\mapsto \overline{c}^{*}:\mathbb E(C,A)\rightarrow \mathbb E(C',A)$, 
makes $\mathbb E:H^{0}(\A)^{op}\times H^{0}(\A)\rightarrow  \Set$ into a bifunctor.
\end{proposition}

\begin{remark} \label{rk:extriangulated-structure}
In section~\ref{canonicalstructure}, we will prove that $H^0(\A)$ carries a canonical extriangulated structure in the
sense of Nakaoka--Palu \cite{NakaokaPalu19} whose extension bifunctor is $\mathbb{E}$.
\end{remark}

\begin{proof}Suppose we are given $[X]\in \mathbb E(C,A)$, $\overline{a}:A\rightarrow A'$ and $\overline{c}:C'\rightarrow C$, where $X$ is given by the diagram \ref{X}.

Put $\overline a_*[X]=[X_1]$, $\overline c^{*}[X_1]=[X_2]$ and $\overline c^*[X]=[X_3]$. 
Then the conflations $X_1$, $X_2$ and $X_3$ are of the following form respectively
 \[
 \begin{tikzcd}
 A'\ar[r,"f_1"]\ar[rr,bend right=8ex,"h_1"swap]&E_1\ar[r,"j_1"]&C &A'\ar[r,"f_2"]\ar[rr,bend right=8ex,"h_2"swap]&E_2\ar[r,"j_2"]&C'&A\ar[r,"f_3"]\ar[rr,bend right=8ex,"h_3"swap]&E_3\ar[r,"j_3"]&C'
 \end{tikzcd}.
 \]
 
By the definition of $[X_2]$, we have a homotopy cartesian square which gives rise to a morphism $\alpha$ in $H^0(\Mor(\A))$ from $j_2:E_2\rightarrow C'$ to $j_1:E_1\rightarrow C$. 
We also have a morphism $\beta$ from $j_3:E_3\rightarrow C'$ to $j_1:E_1\rightarrow C'$. 
By Lemma \ref{univ}, we have a morphism $\gamma$ from $j_3:E_3\rightarrow C'$ to $j_2:E_2\rightarrow C'$ which restricts to $\Id_{C'}$ and such that $\alpha\circ \gamma=\beta$. 
The associated diagram in $\D(\A)$ is as follows:
\[
\begin{tikzcd}[cramped,sep=small]
&A\ar[rr,"\overline{f_3}"]\ar[dd,dotted,bend left=4ex]\ar[ld,equal]&
&E_3\ar[rr,"\overline{\jmath_3}"]\ar[dd, dotted,bend left=4ex]\ar[ld]&
&C'\ar[dd,equal]\ar[ld,"\overline{c}"]\\
A\ar[rr,"\overline{f}"]\ar[dd,"\overline{a}"swap]&&B\ar[rr,"\overline{\jmath}"]\ar[dd]
&&C\ar[dd,equal]
&\\
&A'\ar[rr]\ar[ld,equal]&
&E_2\ar[rr]\ar[ld]&
&C'\ar[ld,"\overline{c}"]\\
A'\ar[rr,"\overline{f_1}"]&&E_1\ar[rr,"\overline{\jmath_1}"]
&&
C&
\end{tikzcd}
\]
By universal property of the homotopy pullbacks (cf.~Remark \ref{pullbackuniversal}), the morphism $\gamma$ induces a morphism from $X_3$ to $X_2$. 
It clearly restricts to $\overline a:A\rightarrow A'$. 
Therefore $\overline a_*[X_3]=[X_2]$.
\end{proof}

Our next aim is to show that $\mathbb E(C,A)$ is an abelian group for each pair $(C,A)$ of objects in $H^0(\A)$ such that $\mathbb E:H^0(\A)^{op}\times H^0(\A)\rightarrow \Ab$ is a biadditive bifunctor. 
We start by describing the addition operation on $\mathbb E(C,A)$.

Let $[X]$ and $[X']$ be two elements in $\mathbb E(C,A)$ given by the diagram \ref{twoconflations}. 
Put $\overline a=[1\ 1]:A\oplus A\rightarrow A$ and $\overline c=[1\ 1]^{\intercal}:C\rightarrow C\oplus C$. Recall from Proposition \ref{direct} that the direct sum $X\oplus X'$ is also a conflation. 
Then the sum $[X]+[X']$ is defined to be $\overline c^{*}\overline a_*[X\oplus X']$. 
\begin{lemma} The above addition operation is well-defined, associative and commutative and that the equivalence class of 
\[
\begin{tikzcd}
A\ar[r,"\begin{bmatrix}1\\0\end{bmatrix}"]\ar[rr,"0"swap,bend right=8ex]&A\oplus C\ar[r,"{[}0{,}1{]}"]&C
\end{tikzcd}
\]
is a zero element. 
\end{lemma}
\begin{proof}
It is direct to check the well-definedness and commutativity.
To show that the above conflation is a zero element, observe that we have the following morphism in $\mathcal H_{3t}(\A)$
\[
\begin{tikzcd}
A\ar[rrd,"0"{blue,near start},blue]\ar[d,"\begin{bmatrix}0\\1\end{bmatrix}"swap]\ar[rr,bend left=8ex,"h"]\ar[r,"f"]\ar[rd,"s"{red,swap},red,bend right=2ex]&B\ar[rd,"0"red,red,bend left=2ex]\ar[d,"{[}0{,}\;j{,}\;1{]^{\intercal}}"]\ar[r,"j"]&C\ar[d,"{[}1{,}\;1{]^{\intercal}}"]\\
A\oplus A\ar[rr,bend right=8ex,"h'"swap]\ar[r,"f'"swap]&A\oplus C\oplus B\ar[r,"j'"swap]&C\oplus C
\end{tikzcd}
\]
where $f'=\begin{bmatrix}1&0\\0&0\\0&f\end{bmatrix}$, $j'=\begin{bmatrix}0&1&0\\0&0&j\end{bmatrix}$, $h'=\begin{bmatrix}0&0\\0&h\end{bmatrix}$, $s=[0,\;h,\;0]^{\intercal}$.

To show the associativity, observe that we have the equality
\[
\begin{bmatrix}1&0\\0&1\\0&1\end{bmatrix}\circ\begin{bmatrix}1\\1\end{bmatrix}=\begin{bmatrix}1&0\\1&0\\0&1\end{bmatrix}\circ\begin{bmatrix}1\\1\end{bmatrix}:A\rightarrow A\oplus A\oplus A.
\]
\end{proof}
Conflations in the equivalence class of zero elements are called {\sl splitting}. 

\begin{proposition}\label{split}Let $X$ be a conflation as follows
\[
\begin{tikzcd}
A\ar[r,"f"]\ar[rr,bend right=8ex,"h"swap]&B\ar[r,"j"]&C
\end{tikzcd}.  
\]
The following statements are equivalent:
\begin{itemize}
\item[1)]$X$ is a splitting conflation.
\item[2)]$\overline f$ is a split monomorphism in $H^{0}(\A)$.
\item[2)$^{\text{op}}$]$\overline{\jmath}$ is a split epimorphism in $H^{0}(\A)$.
\end{itemize} 
\end{proposition}
\begin{proof}
The implication $1)\Rightarrow 2)$ follows from Corollary \ref{middleterm}. 

We show $2)\Rightarrow 1)$. 

Suppose $\overline f$ is a split monomorphism in $H^{0}(\A)$. 
Then we have an isomorphism from $A\rightarrow A\oplus C'$ to $f:A\rightarrow B$ in $H^0(\Mor(\A))$, which restricts to $\Id_{A}$. 
So their homotopy cokernels are isomorphic and hence $C'$ is homotopy equivalent to $C$. 
So we have an isomorphism in $H^0(\Mor(\A))$ from $f'=(1\ 0)^{\intercal}: A\rightarrow A\oplus C$ to $f:A\rightarrow B$ which restricts to $\Id_{A}$. 
Then it induces an isomorphism $\alpha$ in $\mathcal H_{3t}(\A)$ 
\[
\begin{tikzcd}
A\ar[d,"h_0"swap]\ar[r,"{[}1{,}0{]^{\intercal}}"]\ar[rr,bend left=8ex,"0"]&A\oplus C\ar[r,"{[}0\ 1{]}"]\ar[d,"h_1"]&C\ar[d,"h_2"]\\
A\ar[r,"f"swap]\ar[rr,bend right=8ex,"h"swap]&B\ar[r,"j"swap]&C
\end{tikzcd}
\]
where we have omitted the diagonal morphisms and $h_0$ is homotopic to $\Id_{A}$.
Let $c$ be a homotopy inverse of $h_2$. 
Then the morphism $c$ gives rise to a morphism $\beta$ 
\[
\begin{tikzcd}
A\ar[d,equal]\ar[r,"{[}1{,}0{]^{\intercal}}"]\ar[rr,bend left=8ex,"0"]&A\oplus C\ar[r,"{[}0\ 1{]}"]\ar[d,"\begin{bmatrix}1\;\;0\\0\;\;c
\end{bmatrix}"]&C\ar[d,"c"]\\
A\ar[r,"{[}1{,}0{]^{\intercal}}"swap]\ar[rr,bend right=8ex,"0"]&A\oplus C\ar[r,"{[}0\ 1{]}"swap]&C
\end{tikzcd}\;\;.
\]
So the composition $\alpha\circ \beta $ is an equivalence between $X$ and a splitting conflation and hence $X$ is a splitting conflation.
\end{proof}

\begin{corollary}\label{zero}Let $X$ be a conflation of the form
\[
 \begin{tikzcd}
 A\ar[r,"f"]\ar[rr,"h"swap,bend right=8ex]&B\ar[r,"j"]&C
 \end{tikzcd}.
 \]
\begin{itemize}
\item[1)] The element $\overline{\jmath}^*[X]$ is zero in $\mathbb E(B,A)$. 
\item[2)]For any zero morphism $0:W\rightarrow C$ in $H^0(\A)$, the corresponding element $0^*[X]$ is zero in $\mathbb E(W,A)$.
\end{itemize}
\end{corollary}

\begin{proposition}
Let $[X]$ be an element in $\mathbb E(C,A)$ where $X$ is of the form
\[
\begin{tikzcd}
A\ar[r,"f"]\ar[rr,"h"swap,bend right=8ex]&B\ar[r,"j"]&C
\end{tikzcd}.
\]
 Then $[-X]$, where $-X$ is of the form
\[
\begin{tikzcd}
A\ar[r,"-f"]\ar[rr,"-h"swap,bend right=8ex]&B\ar[r,"j"]&C
\end{tikzcd}
\]
is an inverse of $[X]$ under the addition defined above. So this addition makes $\mathbb E(C,A) $ into an abelian group. 
\end{proposition}
\begin{proof}It is direct to check that the element $[-X]$ is well-defined. 
The following diagram in $\A$
\[
\begin{tikzcd}
A\ar[rr,bend left=8ex,"h"]\ar[d,swap,"\begin{bmatrix}-1\\1\end{bmatrix}"]\ar[r,"f"]&B\ar[d,"\begin{bmatrix}1\\1\end{bmatrix}"]\ar[r,"j"]&C\ar[d,"\begin{bmatrix}1\\1\end{bmatrix}"]\\
A\oplus A\ar[rr,"h'"swap,bend right=8ex]\ar[r,"f'"swap]&B\oplus B\ar[r,"j'"swap]&C\oplus C
\end{tikzcd}
\]
where the diagonal morphisms are zero and where $f'=\begin{bmatrix}-f&0\\0&f\end{bmatrix}$, $j'=\begin{bmatrix}j&0\\0&j\end{bmatrix}$ and $h'=\begin{bmatrix}-h&0\\0&h\end{bmatrix}$, gives a morphism $\alpha:X\rightarrow -X\oplus X$ in $\mathcal H_{3t}(\A)$. 
We apply Lemma \ref{fact} to the morphism $\alpha$. 

Put $a=(1\ 1):A\oplus A\rightarrow A$, $a'=(-1\ 1)^{\intercal}:A\rightarrow A\oplus A$ and $c=(1\ 1)^{\intercal}:C\rightarrow C\oplus C$. 

Then we have $\alpha=\gamma\circ\beta$ where $\beta:X\rightarrow\tilde{X}$, $\gamma:\tilde{X}\rightarrow -X\oplus X$ and $[\tilde{X}]=\overline{a'}_{*}[X]=\overline c^*[-X\oplus X]$.

So we have
\[
[X]+[-X]=\overline a_{*}\overline c^*[-X\oplus X]=\overline a_{*}[\tilde{X}]=\overline a_{*}\overline {a'}_{*}[X]=0_*[X]=0.
\] 
This shows that $[-X]$ is an additive inverse of $[X]$.

\end{proof}

\begin{proposition} The bifunctor $\mathbb E$ is a biadditive bifunctor from $H^0(\A)^{op}\times H^0(\A)$ to the category $\Ab$ of abelian groups.
\end{proposition}
\begin{proof}Let $X_i$, $i=1$, $2$, be a conflation of the form 
\[
\begin{tikzcd}
A\ar[r,"f_i"]\ar[rr,"h_i"swap,bend right=8ex]&B_i\ar[r,"j_i"]&C
\end{tikzcd}
\]
and $\overline b:A\rightarrow A'$ a morphism in $H^0(\A)$. 

Put $a=(1\ 1):A\oplus A\rightarrow A$, $a'=(1\ 1):A'\oplus A'\rightarrow A'$, $c=(1\ 1)^{\intercal}:C\rightarrow C\oplus C$ and $d=\begin{bmatrix}b&0\\0&b\end{bmatrix}:A\oplus A\rightarrow A'\oplus A'$. 

Put $[\tilde{X_i}]=\overline b_*[X_i]$, $i=1,2$. 
Then we have $d_*[X_1\oplus X_2]=[\tilde{X_1}\oplus \tilde{X_2}]$. 

So we have
\[
\begin{aligned}
\overline b_*([X_1]+[X_2])&=\overline b_*\overline a_*\overline c^*[X_1\oplus X_2]=\overline a'_*\overline d_*\overline c^{*}[X_1\oplus X_2]\\
&=\overline a'_*\overline c^{*}\overline d_*[X_1\oplus X_2]=\overline a'_*\overline c^{*}[\tilde{X_1}\oplus \tilde{X_2}]\\
&=\tilde{X_1}+\tilde{X_2}=\overline b_*[X_1]+\overline b_*[X_2].
\end{aligned}
\] 
So the map $\overline{b}_*$ is a morphism of abelian groups.

Let $X$ be a conflation of the form
\[
\begin{tikzcd}
A\ar[r,"f"]\ar[rr,"h"swap,bend right=8ex]&B\ar[r,"j"]&C
\end{tikzcd}.
\]
Suppose we have a morphism $\overline b_i:A\rightarrow A'$ in $H^0(\A)$ for each $i=1$, $2$.

Put $d'=\begin{bmatrix}b_1&0\\0&b_2
\end{bmatrix}$ and $a''=[1,1]^{\intercal}:A\rightarrow A\oplus A$. 
Put $[Y_i]=(\overline b_i)_*[X]$ for $i=1$, $2$. 

We have a canonical morphism $\alpha:X \rightarrow X\oplus X$ which is componentwise given by the morphism $(1\ 1)^{\intercal}$. 

By Lemma \ref{fact}, the morphism $\alpha$ factorizes as $\alpha=\gamma\circ \beta$ where we have $\beta: X\rightarrow \tilde{X}$, $\gamma:\tilde{X}\rightarrow X\oplus X$ and $[\tilde{X}]=\overline a''_*[X]=\overline c^*[X\oplus X]$. 

Then we have
\[
\begin{aligned}
(\overline b_1+\overline b_2)_*[X]&=\overline a'_*\overline d'_*\overline a''_*[X]=\overline a'_*\overline d'_*\overline c^*[X\oplus X]\\
&=\overline a'_*\overline c^*\overline d'_{*}[X\oplus X]=\overline a'_*\overline c^*[Y_1\oplus Y_2]\\
&=Y_1+Y_2=(\overline b_1)_*[X]+(\overline b_2)_*[X].
\end{aligned}
\] 
Thus we have
$(\overline b_1+\overline b_2)_*=(\overline b_1)_*+(\overline b_2)_*$. 
This together with its dual, shows that the bifunctor $\mathbb E$ is biadditive.
\end{proof}
For two elements $\delta_i=[X_i]\in \mathbb E(C_i,A_i)$, $i=1$, $2$, let $\delta_1\oplus \delta_2$ be the element in 
\[
\mathbb E(C_1\oplus C_2, A_1\oplus A_2)\simeq \mathbb E(C_1,A_1)\oplus \mathbb E(C_2, A_2)\oplus \mathbb E(C_1, A_2)\oplus \mathbb E(C_2, A_1)
\]
corresponding to the element $(\delta_1,\delta_2,0,0)$.
\begin{lemma}\label{sum}$\delta_1\oplus \delta_2$ is the equivalence class of $X_1\oplus X_2$.
\end{lemma}
\begin{proof}Denote by $a_i:A_1\oplus A_2\rightarrow A_i$ the canonical projection for $i=1$, $2$. 
Denote by $c_i:C_i\rightarrow C_1\oplus C_2$ the canonical inclusion for $i=1$, $2$. 

It is enough to show that we have $c_j^*(a_i)_*[X_1\oplus X_2]=0$ for $i\neq j$ and $c_i^*(a_i)_*[X_1\oplus X_2]=[X_i]$.

We show the case when $i=1$. 
We have $(a_1)_*[X_1\oplus X_2]=[X_1\oplus Y]$ where $Y$ is the conflation
\[
\begin{tikzcd}
0\ar[r]\ar[rr,"0"swap,bend right=8ex]&C_2\ar[r,equal]&C_2
\end{tikzcd}.
\]
So we have $c_1^*(a_1)_*[X_1\oplus X_2]=[X_1]$. 

By Corollary \ref{zero}, we have $c_2^*(a_1)_*[X_1\oplus X_2]=0$.
\end{proof}

\subsection{DG nerve of an exact dg category}
In this subsection, we link our notion of exact dg category to Barwick's notion of exact
$\infty$-category \cite{Barwick15}. Standard references for $\infty$-categories are \cite{Lurie09,LurieHA,Cisinski19,Land21,Lurie23}. 

One of the aims of this section is to prove the following theorem. The construction of
the dg nerve will be recalled in \ref{Lambdadgnerve} and the notion of an exact structure on an $\infty$-category
in Definition~\ref{def:exactinfinity}. The notion of exact structure on a dg category is given in Definition~\ref{exactdgstructure}.

\begin{theorem}\label{nerve}Let $\A$ be a dg category such that $H^0(\A)$ is an additive category. 
Then there is a bijection between the class of exact structures on the dg category $\A$ and the class of 
exact structures on the $\infty$-category $N_{dg}(\A)$.
Moreover, if $\A$ is a pretriangulated dg category endowed with the maximal exact structure 
(given by all homotopy short exact sequences), then its dg nerve is a stable $\infty$-category
endowed with the maximal exact structure (where each morphism is both an inflation and a deflation).
\end{theorem}

This extends in particular a theorem by Faonte \cite{Faonte17b}, who showed that 
if a dg category is pretriangulated, then its dg nerve is a stable $\infty$-category. 

We begin by recalling some basic definitions concerning simplicial sets. We follow the treatment in \cite{Hovey99}.
The {\em simplicial} category $\Delta$ is the full subcategory of the category $\cat$ of small categories consisting of the totally ordered sets 
\[
[n]=\{0,1,\cdots,n\}
\]
for $n\geq 0$.
The simplicial category $\Delta$ is generated by the injective morphisms $d^i:[n-1]\rightarrow [n]$ for $n\geq 1$ and $0\leq i\leq n$, where the image of $d^i$ does not include $i$, and the surjective morphisms 
$s^i:[n]\rightarrow [n-1]$ for $0\leq i\leq n-1$, where $s^i$ identifies $i$ and $i+1$.
It is not hard to check that these satisfy the following relations called cosimplicial identities
\begin{itemize}
\item $d^jd^i=d^id^{j-1}$ for $i<j$,
\item \begin{equation*}
    s^jd^i=
    \begin{cases}
      d^is^{j-1}, & (i<j) \\
      \Id, & (i=j\  \text{or}\  j+1)\\
      d^{i-1}s^j, & (i>j+1)
    \end{cases}
  \end{equation*}
  \item $s^js^i=s^{i-1}s^j$ for $i>j$.
\end{itemize}
Moreover, one can show that the category $\Delta$ has a presentation given by these generators
and relations, cf.~\cite[II.2.2]{GabrielZisman67}.
For a category $\C$, the category of {\em simplicial objects} in $\C$ is the 
category $\Fun(\Delta^{op},\C)$ of $\C$-valued presheaves on $\Delta$.
The most important examples are the category $\sSet$ of simplicial sets and the
category $\sMod k$ of simplicial $k$-modules.

For $n\geq 0$, we have the {\em standard $n$-simplex} $\Delta^n$ defined as the functor 
represented by the object $[n]$. For a simplicial set $X$ and an integer $n\geq 0$, we write 
$X_n$ for the set $X([n])$ of {\em $n$-simplices of $X$}. Notice that by the Yoneda lemma, 
this set identifies with 
\[
\Hom_{\sSet}(\Delta^n,X).
\]
Thus, an $n$-simplex $x$ of $X$ can be viewed as a morphism $x:\Delta^n\rightarrow X$. 
For a morphism $d^i:[n-1]\rightarrow[n]$ in $\Delta$, by abuse of notation, we put $d_i=X(d^i):X_{n}\rightarrow X_{n-1}$, which is a {\em face map}. 
Similarly for a morphism $s^j:[n]\rightarrow [n-1]$, we put $s_j=X(s^j):X_{n-1}\rightarrow X_n$, which is a {\em degeneracy map}.
The set $\pi_0(X)$ of {\em connected components} of $X$ is defined as the coequalizer of the double arrow
\[
\begin{tikzcd}(X_1\ar[r,"d_1",shift left=0.8ex]\ar[r,"d_0"swap,shift right =0.8ex]&X_0).\end{tikzcd}
\]
This extends to a functor
\[
\pi_0:\sSet\rightarrow \Set,\;\; X\mapsto \pi_0(X).
\]
It admits a right adjoint 
\[
\const:\Set\rightarrow \sSet,\;\; A\mapsto \const(A)
\]
where $\const(A)$ is the constant functor $\Delta^{op}\rightarrow \Set$ with value $A$.
So we have the following adjunction
\[
\begin{tikzcd}
\sSet\ar[r,shift left=0.8ex,"\pi_0"]&\Set\ar[l,shift left=0.8ex,"\const"]
\end{tikzcd}
\]
where $\pi_0$ is the left adjoint.

The inclusion functor $\Delta\rightarrow \cat$ extends to a unique colimit-preserving functor
\[
h(-):\sSet\rightarrow \cat
\]
which associates to each simplicial set $X$ its {\em homotopy category} $hX$. It can be described 
as the path category of the quiver $(X_0,X_1,d_1,d_0)$ modulo the relations given 
by 
\[
d_1(h)\sim d_0(h)\circ d_2(h)
\]
 for all elements $h\in X_2$.
The {\em nerve functor} $N(-):\cat\rightarrow \sSet$ is right adjoint to $h(-)$. 
More explicitly, for a small category $\C$ we have 
\[
N(\C)_n=\Hom_{\cat}([n],\C)
\]
for each $n\geq 0$. For any $n\geq 0$, the {\em boundary} of the standard $n$-simplex is
\[
\begin{tikzcd}
\partial \Delta^n=\bigcup_{E} \Delta^{E}\subset\Delta^n   
\end{tikzcd}
\]
where $E$ runs through the proper full subcategories of $[n]$ and $\Delta^{E}$ is the nerve $N(E)$. 
Similarly, for integers $n\geq 1$ and $0\leq k\leq n$, the $k$-th {\em horn} of $\Delta^n$ is
\[
\Lambda^n_k=\bigcup_{k\in E} \Delta^{E}\subset \Delta^n
\]
where $E$ runs through proper full subcategories of $[n]$ which contain the object $k$. 
Obviously, we have $\Lambda^n_k\subsetneq \partial \Delta^n\subsetneq \Delta^n$.

A simplicial set $X$ is an {\em $\infty$-category} provided each map $\Lambda^n_k\rightarrow X$ can be lifted to a map $\Delta^n\rightarrow X$ along $\Lambda^n_k\rightarrow \Delta^n$ whenever $n\geq 2$ and $0<k<n$. 
It is a {\em Kan complex} if the same is true whenever $n\geq 2$ and $0\leq k\leq n$. 

For two simplicial sets $X$ and $Y$, their {\em product} $X\times Y$ is defined by $(X\times Y)_n=X_n\times Y_n$. 
The {\em mapping} simplicial set $\Map(X,Y)$ is defined by putting
\[
\Hom_{\sSet}(\Delta^n,\Map(X,Y))\iso\Hom_{\sSet}(\Delta^n\times X,Y)
\]
for each $n\geq 0$. 
The simplicial set $\Map(X,Y)$ is an $\infty$-category (resp.~Kan complex) whenever $Y$ is. 

Let $\E$ be an $\infty$-category. 
An element $x\in \E_0$ is called an {\em object} or a {\em vertex} of $\E$. 
An element $f\in \E_1$ is called a {\em morphism} or an {\em arrow} of $\E$.
 Put $x=d_1(f)$ and $y=d_0(f)$. 
We write $f$ as $f:x\rightarrow y$ and say that $f$ is a {\em morphism from $x$ to $y$}.
A morphism $f:x\rightarrow y$ is an {\em isomorphism} if it becomes an isomorphism in the homotopy category $h(\E)$ and in this case, we say that $x$ is {\em isomorphic} to $y$ as objects of the $\infty$-category $\E$.
Let $\sigma$ be a $2$-simplex in $\mathcal E$. Put $f=d_2(\sigma)$, $g=d_0(\sigma)$ and $h=d_1(\sigma)$. 
Then $\sigma$ corresponds to a diagram 
\[
\begin{tikzcd}
&y\ar[rd,"g"]&\\
x\ar[ru,"f"]\ar[rr,"h"swap]&&z
\end{tikzcd}
\]
which identifies $h$ with the {\em composition} $g\circ f$ and we shall write $g\circ f\sim h$.
The homotopy category $h\mathcal E$ can be described alternatively as follows (cf.~\cite{BoardmanVogt73}, also \cite{Cisinski19}): let us fix two objects $x$ and $y$ in $\E$. We define the following four relations on the set of morphisms of $\E$ from $x$ to $y$:
\begin{itemize}
\item[] $f\sim_1g$ if $f1_x\sim g$;
\item[] $f\sim_2g$ if $1_yf\sim g$;
\item[] $f\sim_3 g$ if $g1_x\sim f$;
\item[] $f\sim_4g$ if $1_yg\sim f$;
\end{itemize}
It is a standard result that the above four relations are equal and that they are equivalence relations on the set of morphisms from $x$ to $y$ in $\E$.

For two objects $x$ and $y$ in $\E$, put
\[
\Hom_{h'\E}(x,y)=\{\text{morphisms from }x\text{ to }y\}/\sim_1.
\]
Given a morphism $f:x\rightarrow y$ in $\E$, we write $[f]$ for its equivalence class in $\Hom_{h'\E}(x,y)$.
For any triple of objects $(x,y,z)$ in $\E$, the composition map is given by
\[
\begin{tikzcd}
\Hom_{h'\E}(y,z)\times\Hom_{h'\E}(x,y)\ar[r]& \Hom_{h'\E}(x,z)\\
({[}g{]},{[}f{]})\ar[r,mapsto]&{[}g{]}\circ {[}f{]}
\end{tikzcd}
\]
where $[g]\circ [f]=[h]$ whenever $gf\sim h$. 
It is standard that this map is well-defined and that the above produces a category $h'\E$.
By the universal property of $h\E$, there exists a functor $h\E\rightarrow h'\E$ which is the 
identity on the set of objects.
It is standard that this functor is fully faithful and thus we have a canonical isomorphism of categories $h\E\iso h'\E$.

Let $f:X\rightarrow Y$ be a morphism of simplicial sets. 
It is a {\em weak homotopy equivalence} provided the induced map $\pi_0\Map(Y,K)\iso\pi_0\Map(X,K)$ is a bijection for each Kan complex $K$. 
The category $\sSet$ carries the {\em Kan-Quillen model structure} whose weak equivalences are the weak homotopy equivalences and whose cofibrations are the injections of simplicial sets.
Fibrations in the Kan-Quillen model structure are called {\em Kan fibrations}.
The fibrant objects in the Kan-Quillen model structure are exactly the Kan complexes.

The triple $(\sSet,\times,\Delta^0)$ is a closed symmetric monoidal model category with the Kan-Quillen model structure, where the internal Hom is given by $\Map$.
A {\em simplicial category} is a category enriched over $\sSet$.
A simplicial category $\C$ is {\em fibrant} if for each pair $(X,Y)$ of objects in $\C$, the Hom simplicial set $\C(X,Y)$ is fibrant, i.e.~a Kan complex. 

Let $\C$ and $\D$ be two $\infty$-categories. A morphism $f:\C\rightarrow \D$ of $\infty$-categories is an {\em equivalence of $\infty$-categories} if there exists a morphism $g:\D\rightarrow \C$ such that $gf$ is isomorphic to $\mathrm{id_{\C}}$ in the $\infty$-category $\Map(\C,\C)$ and $fg$ is isomorphic to $\mathrm{id_{\D}}$ in $\Map(\D,\D)$. 
The induced functor $h(f):h(\C)\rightarrow h(\D)$ of an equivalence of $\infty$-categories $f:\C\rightarrow \D$ is in particular an equivalence of categories.
 
A morphism $f:X\rightarrow Y$ of simplicial sets is a {\em categorical equivalence} provided the induced morphism
\[
\Map(Y,\C)\rightarrow\Map(X,\C)
\]
is an equivalence of $\infty$-categories for each $\infty$-category $\C$.

$\sSet$ carries the {\em Joyal model structure} whose weak equivalences are the
categorical equivalences and whose cofibrations are the injections of simplicial sets.
The fibrant objects in the Joyal model structure are exactly the $\infty$-categories and 
the categorical equivalences between $\infty$-categories are exactly the equivalences of $\infty$-categories.
We will write $\sSet_{Quillen}$ or $\sSet_{Joyal}$ when we want to stress the model structure we are using.
Let $\C$ and $\D$ be two $\infty$-categories and $f:\C\rightarrow \D$ a morphism in $\mathrm{Ho}(\sSet_{Joyal})$. 
The morphism $f$ yields an isomorphism class of functors $h(\C)\rightarrow h(\D)$. 
In the special case of $\C=\Delta^0$, we obtain a canonical bijection
\[
\Hom_{\mathrm{Ho}(\sSet_{Joyal})}(\Delta^0, \D)\iso \Iso(h(\D)).
\]
The triple $(\sSet,\times,\Delta^0)$ is a closed symmetric monoidal model category with the Joyal model structure, where the internal Hom is given by $\Map$, cf.~\cite[Theorem 6.12]{Joyal08}.

Let $K$ be a Kan complex and $x, y\in K_0$ two 0-simplices. 
Define $x$ to be {\em homotopic} to $y$, written $x\sim y$, if and only if there is a 1-simplex $z\in K_1$ such that $d_1z=x$ and $d_0z=y$.
One checks that homotopy of vertices is an equivalence relation and that $\pi_0K=K_0/\sim$.
If $v$ is a vertex of $K$, then one defines $\pi_0(K,v)$ to be the pointed set $\pi_0X$ with basepoint the equivalence class $[v]$ of $v$.
Let $n> 0$ and $v\in K_0$.  
For a simplicial set $X$, we will denote the map $X\rightarrow \Delta^0\xrightarrow{v} K$ by $v$ and refer to it as the {\em constant map with value $v$}. 
For example when $X=\Delta^n$, we will still denote by $v$ the corresponding $n$-simplex in $K$. 
Let $F$ be the pullback of the following diagram
\[
\begin{tikzcd}
&\Map(\Delta^n,K)\ar[d]\\
\Delta^0\ar[r,"v"]&\Map(\partial\Delta^n,K).
\end{tikzcd}
\]
It is clear that $F$ is a Kan complex. The pointed set $\pi_0(F,v)$ carries a canonical
group structure, which is abelian  for $n\geq 2$. It is called the 
{\em nth homotopy group} $\pi_n(K,v)$ of $K$ at $v$.

By Whitehead's theorem, a morphism $f:K\rightarrow K'$ of Kan complexes is a weak homotopy equivalence if and only if it satisfies the following conditions:
\begin{itemize}
\item[$(a)$] The map of sets $\pi_0(f):\pi_0(K)\rightarrow \pi_0(K')$ is a bijection;
\item[$(b)$] For each vertex $x\in K$ with image $y=f(x)\in K'$ and each $n\geq 1$, the map of homotopy groups $\pi_n(f):\pi_n(K,x)\rightarrow\pi_n(K',y)$ is an isomorphism.
\end{itemize}

For each simplicial set $K$, we have the free simplicial $k$-module $FK$ generated by $K$ so that $(FK)_n$ is the free $k$-module generated by the set $K_n$ for each $n\geq 0$ and the structure maps are extended from those of $K$ by $k$-linearity. 
This determines the {\em free functor} $F:\sSet \rightarrow \sMod k$.
Let $U$ be the forgetful functor $\sMod k\rightarrow \sSet$.
For a simplicial $k$-module $X$, the {\em underlying simplicial set} of $X$ is $U(X)$.
Then we have the adjunction
\[
\begin{tikzcd}
\sSet\ar[r, shift left=1ex,"F"]&\sMod k\ar[l,shift left=1ex,"U"]
\end{tikzcd}
\]
where $F$ is the left adjoint.

In order to relate dg categories with $\infty$-categories, we need the {\em Dold-Kan correspondence} (\cite{Dold58}) as a key ingredient. Note that we are working with cochain complexes. 
We denote by $\C(k)^{\leq 0}$ the category of nonpositively graded cochain complexes. 

Let $X$ be a simplicial $k$-module.  
Put
\[
C^{-n}(X)=X_n
\]
 and 
\[
d^{-n}=\sum_{0\leq i\leq n}(-1)^id_i:C^{-n}(X)\rightarrow C^{-n+1}(X)
\]
for $n\geq 0$.
The {\em unnormalized cochain complex} $C^{*}(X)$ is given by
\[
\cdots\rightarrow C^{-n}(X)\xrightarrow{d^{-n}} C^{-n+1}(X)\rightarrow \cdots C^{-2}(X)\xrightarrow{d^{-2}} C^{-1}(X)\xrightarrow{d^{-1}} C^0(X)\rightarrow 0\rightarrow \cdots.
\]
For each $n\geq 0$, put 
\[
N^{-n}(X)=\bigcap_{1\leq i\leq n}\ker(d_i)\subseteq C^{-n}(X)=X_n.
\]
The map $d^{-n}$ carries $N^{-n}(X)$ into $N^{-n+1}(X)$ and therefore we obtain the {\em normalized cochain complex} $N^*(X)$
\[
\cdots\rightarrow N^{-n}(X)\xrightarrow{d^{-n}} N^{-n+1}(X)\rightarrow \cdots \rightarrow N^{-2}(X)\xrightarrow{d^{-2}} N^{-1}(X)\xrightarrow{d^{-1}} N^0(X)\rightarrow 0\rightarrow \cdots\ .
\]

The assignment $X\mapsto N^{*}(X)$ determines the {\em normalized cochain complex functor} $N:\sMod k\rightarrow \C(k)^{\leq 0}$.

\begin{theorem}[\cite{Dold58}]
The normalized cochain complex functor $N:\sMod k\rightarrow \C(k)^{\leq 0}$ is an equivalence of categories.
\end{theorem}
Its quasi-inverse is the  {\em Dold-Kan construction} $DK:\C(k)^{\leq 0}\rightarrow \sMod k$:
A concrete description of $DK$ can be found in \cite[Construction 1.2.3.5]{LurieHA}.
The construction of the functor $DK$ remains valid if we replace $\Mod k$ by any additive category and it is an equivalence of categories if the additive category has split retractions (i.e. it is weakly idempotent complete).
In our case, for a complex $V\in\C(k)^{\leq 0}$, we identify the $k$-module $DK(V)_n$ with 
\[
\Hom_{\sMod k}(N(F(\Delta^n)),V)
\]
where $F(\Delta^n)$ is the free simplicial $k$-module generated by $\Delta^n$.

The category $\C(k)^{\leq 0}$ carries the {\em projective model structure} (cf.~\cite{Quillen67} and \cite[Theorem 1.5]{GoerssSchemmerhorn07}) such that
\begin{itemize}
\item the weak equivalences are the quasi-isomorphisms of complexes;
\item the cofibrations are the monomorphisms with componentwise projective cokernel;
\item the fibrations are the morphisms which are epimorphisms in strictly negative degrees.
\end{itemize}
By the Dold-Kan correspondence, this induces a model structure on $\sMod k$ which is called the {\em projective model structure}.

Let $f:X\rightarrow Y$ be a morphism in $\C(k)^{\leq 0}$. 
We will need the following explicit factorization of $f$ into a trivial cofibration followed by a fibration.
For a $k$-module $M$ and any $n\leq -1$, denote by $D^n(M)$ the following complex
\[
\dots\rightarrow 0\rightarrow M^{n}\xrightarrow{\Id} M^{n+1}\rightarrow 0\rightarrow \dots
\]
where $M^n=M$ is in degree $n$ and $M^{n+1}=M$ is in degree $n+1$.
For each $n\leq -1$, let $u^n: Q^n \to Y^n$ be a surjective morphism with projective $Q^n$. 
Let $P(Y)$ be the following contractible complex
\[
\bigoplus_{n\leq -1} D^n(Q^{n}).
\]
For each $n\leq -1$, the surjective map $u^n:Q^n\rightarrow Y^n$ induces a map $s^n: D^n(Q^{n})\rightarrow Y$ which is surjective in degree $n$.
 Let $v:P(Y)\rightarrow Y$ be the morphism given by the $s^n$, $n\leq -1$.
 Clearly $v$ is surjective in each strictly negative degree.
Then we have the following factorization of $f$
\begin{equation}\label{factor}
X\xrightarrow{g=\begin{bmatrix}1\\0\end{bmatrix}} X\oplus P(Y)\xrightarrow{h=\begin{bmatrix}f\;\; v\end{bmatrix}} Y.
\end{equation}
Here the morphism $g$ is a trivial cofibration and $h$ is a fibration.
Note that for a complex $Q\in \C(k)^{\leq 0}$ which is componentwise projective, each null-homotopic map $Q\rightarrow Y$ factors through the map $v:P(Y)\rightarrow Y$. 

Let $f:X\rightarrow Y$ be a morphism of simplicial $k$-modules. 
By \cite[Proposition 1.3.2.11]{LurieHA}, the map $U(f)$ is a Kan fibration if and only if the map $N(f)$ is a fibration in the projective model structure of $\C(k)^{\leq 0}$.
In particular, the underlying simplicial set of a simplicial $k$-module is a Kan complex.

For a simplicial $k$-module $X$, the additive operation $X\times X\rightarrow X$, $(x,y)\in X_n\times X_n \mapsto x+y\in X_n$ induces a map $\pi_n(X,0)\times \pi_n(X,0)\rightarrow \pi_n(X,0)$ for each $n\geq 0$. 
This gives an abelian group structure on $\pi_n(X,0)$. 
The Eckmann-Hilton argument shows that it agrees with the usual group structure on $\pi_n(X,0)$ for $n>0$.
An element $[\alpha]\in\pi_n(X,0)$ is an equivalence class of maps $\alpha:\Delta^n\rightarrow X$ such that $d_i(\alpha)=0$ for $i=0,\dots,n$.
Such maps $\alpha:\Delta^n\rightarrow X$ are in bijection with elements in $Z^{-n}N^*(X)$.
Unwinding the definition of the equivalence relation, we see that two such maps $\alpha,\beta:\Delta^n\rightarrow X$ are equivalent if and only if the corresponding elements in $Z^{-n}N^*(X)$ are cohomologous. 
Therefore $\pi_n(X,0)$ can be identified with $H^{-n}N^*(X)$.
Fix a vertex $x\in X_0$. 
The morphism $\theta_x: X\rightarrow X$ of simplicial sets, defined by sending $y\in X_n$ to $x+y\in X_n$, induces an isomorphism of pointed sets $\pi_n(X,0)\iso\pi_n(X,x)$ for $n>0$. 

Let $f:X\rightarrow Y$ be a morphism of simplicial $k$-modules. 
By the above discussions, the map $U(f)$ between the underlying simplicial sets
is a weak homotopy equivalence if and only if $N(f):N^*(X)\rightarrow N^{*}(Y)$ is a weak equivalence in the projective model structure. 
Thus, the projective model structure on $\sMod k$ is the {\em right transferred model structure} 
(cf.~\cite[Section 3]{Crans95}) from $\sSet_{Quillen}$ along $U$.
In particular, $\sMod k$ is cofibrantly generated and the adjunction
\[
\begin{tikzcd}
\sSet\ar[r, shift left=1ex,"F"]&\sMod k\ar[l,shift left=1ex,"U"]
\end{tikzcd}
\] 
is a Quillen adjunction, cf.~ \cite[9.1]{DwyerHirschhornKan97}.

Let $\I$ be a small category and $F:\I\rightarrow \C(k)^{\leq 0}$ a functor.
Since $\C(k)^{\leq 0}\iso \sMod k$ is a simplicial model category (\cite[Theorem 4.13]{GoerssSchemmerhorn07}), one has the local definition of homotopy limits of $F$ (\cite[Definition 8.2]{Shulman06}).

In general, one also has a nice global definition of homotopy limits of $F$ since $\sMod k$ is a {\em combinatorial model category} in the sense of Jeff Smith (cf.~also \cite[A.2.6]{Lurie09}),
and the global and local definitions of homotopy limits coincide (\cite[Theorem 8.5]{Shulman06}). 

We recall the global definition of homotopy limits and will use it for computation.
We follow Lurie's treatment as in \cite[A.2.8]{Lurie09}.
Let $\I$ be a small category and $\const: \sMod k\rightarrow \Fun(\I,\sMod k)$ the functor sending an object $X\in \sMod k$ to the constant functor $\const(X):\I\rightarrow \sMod k$ with value $X$.
It has a right adjoint $\lim:\Fun(\I,\sMod k)\rightarrow \sMod k$ sending a functor 
$F:\I\rightarrow \sMod k$ to $\lim F$. Thus we have the following adjunction
\[
\begin{tikzcd}
\sMod k\ar[r,"\const",shift left=0.8ex]&\Fun(\I,\sMod k).\ar[l,"\lim",shift left=0.8ex]
\end{tikzcd}
\]
 Since $\sMod k$ is combinatorial, $\Fun(\I,\sMod k)$ carries the {\em injective model structure}, where
 \begin{itemize}
 \item A morphism $\alpha:F\rightarrow G$ is a weak equivalence if and only if for each $i\in\I$, the map $\alpha_i:F(i)\rightarrow G(i)$ is a weak equivalence in $\sMod k$.
 \item A morphism $\alpha:F\rightarrow G$ is a cofibration if and only if for each $i\in\I$, the map $\alpha_i:F(i)\rightarrow G(i)$ is a cofibration in $\sMod k$ (with the projective model structure).
 \end{itemize} 
It follows that $\const: \sMod k\rightarrow \Fun(\I,\sMod k)$ is a left Quillen functor with respect to the injective model structure on $\Fun(\I, \sMod k)$.
Let $F:\I\rightarrow \sMod k$ be an object in $\Fun(\I,\sMod k)$ and $A$ an object in $\sMod k$. 
Let $\alpha: \const(A)\rightarrow F$ be a natural transformation and $\beta:A\rightarrow \lim F$ the corresponding morphism in $\sMod k$. 
Then $\alpha$ {\em exhibits $A$ as a homotopy limit of $F$} if for some weak equivalence $F\rightarrow F'$ where $F'$ is fibrant in $\Fun(\I,\sMod k)$, the composite map $A\rightarrow \lim F\rightarrow \lim F'$ is a weak equivalence in $\sMod k$. 

Note that $\sSet_{Quillen}$ is also combinatorial. 
It is immediate that the right Quillen functor $U:\sMod k\rightarrow \sSet$ is compatible with homotopy limits.
\begin{example}\label{complexpullback}
Let $\I$ be the cospan category $\Cosp$ defined in subsection \ref{Preliminaries}.
Then an object $F$ in $\Fun(\I,\C(k)^{\leq 0})$
\[
\begin{tikzcd}
&Y\ar[d,"j"]\\
X\ar[r,"k"swap]& Z
\end{tikzcd}
\]
is fibrant in the injective model structure if and only if both $j$ and $k$ are fibrations in $\C(k)^{\leq 0}$.

Thus using the factorization in \ref{factor}, we obtain a fibrant replacement
\[
\begin{tikzcd}
&Y\oplus P(Z)\ar[d,"\begin{bmatrix}j\;\; v\end{bmatrix}"]\\
X\oplus P(Z)\ar[r,"\begin{bmatrix}k\;\; v\end{bmatrix}"swap]& Z.
\end{tikzcd}
\]
By diagram chase, we see that the pullback of this diagram is quasi-isomorphic to 
\[
\tau_{\leq 0}\Sigma^{-1}\Cone(X\oplus Y\xrightarrow{[k,\;j]} Z).
\]
Take a cofibrant replacement $r:Q\rightarrow \tau_{\leq 0}\Sigma^{-1}\Cone(X\oplus Y\xrightarrow{[k,j]} Z)$.
Then any null-homotopic map $Q\rightarrow Z$ factors through $v:P(Z)\rightarrow Z$.
Thus we obtain a commutative diagram in $\C(k)^{\leq 0}$
\[
\begin{tikzcd}
Q\ar[r]\ar[d]&Y\ar[d]\\
X\oplus P(Z)\ar[r]&Z\mathrlap{.}
\end{tikzcd}
\]
Let $F'$ be the cospan obtained from the above diagram by restriction.
Then we have a weak equivalence $F\rightarrow F'$ in $\Fun(\Cosp,\C(k)^{\leq 0})$ and the above diagram exhibits $Q$ as a homotopy pullback of $F'$.

Therefore, a commutative diagram 
\[
\begin{tikzcd}
U\ar[r]\ar[d]&Y\ar[d,"j"]\\
X\ar[r,"k"swap]&Z
\end{tikzcd}
\]
is homotopy pullback if and only if the canonical map 
\[
U\rightarrow \tau_{\leq 0}\Sigma^{-1}\Cone(X\oplus Y\xrightarrow{[k,\;j]} Z)
\]
is a quasi-isomorphism of connective complexes.
\end{example}
Let $\Cat_{\Delta}$ be the category of {simplicial categories}, i.e. categories which are enriched over $\sSet$. 
There is a Quillen equivalence
\[
\begin{tikzcd}
\sSet\ar[r,"\C(-)", shift left=0.8ex]&\Cat_{\Delta}\ar[l, shift left=0.8ex,"N"]
\end{tikzcd}
\]
relating Bergner's model structure on $\Cat_{\Delta}$ with $\sSet_{Joyal}$, cf.~\cite[1.1.5]{Lurie09} and \cite{DuggerSpivak11} for an explicit description of the mapping spaces in $\C(S)$ for a general simplicial set $S$. 
The functor $N:\Cat_{\Delta}\rightarrow \sSet$ is called the {\em simplicial nerve functor} or {\em homotopy coherent nerve functor}.
An ordinary category $\B$ can be considered as a simplicial category via the functor $\const:\Set\rightarrow \sSet$.
Then the simplicial nerve of $\B$ agrees with the ordinary nerve of $\B$. 
\begin{example}
The simplicial category $\C(\Delta^2)$ has the elements of $\{0,1,2\}$ as objects. 
The mapping spaces $\Hom_{\C(\Delta^2)}(0,1)$ and $\Hom_{\C(\Delta^2)}(1,2)$ are isomorphic to $\Delta^0$. 
The mapping space $\Hom_{\C(\Delta^2)}(0,2)$ is isomorphic to $\Delta^1$.
The composition map
\[
\Hom_{\C(\Delta^2)}(1,2)\times \Hom_{\C(\Delta^2)}(0,1)\rightarrow \Hom_{\C(\Delta^2)}(0,2)
\]
can be identified with the map $1:\Delta^0\rightarrow \Delta^1$. 
\end{example}

In analogy with the nerve functor $N$ taking a small category to an $\infty$-category, there
is a {\em dg nerve functor $N_{dg}$} which turns a small dg category into an $\infty$-category.
A construction is given by Lurie in \cite[Construction 1.3.1.6]{LurieHA} using an explicit description of
the $n$-simplices due to Hinich and Schechtman \cite[Appendix A2.2]{HinichSchechtman87}. 
The dg nerve is proved in \cite[Proposition 1.3.1.20]{LurieHA} to be a right Quillen functor 
in a Quillen adjunction given by a cosimplicial object in $\dgcat$. 
The cosimplicial object is made explicit in \cite[Section 6]{RiveraZeinalian18}: it takes 
the poset $[n]$ to the cobar-bar construction of the associated $k$-category $k[n]$.
To make the Quillen adjunction more precise, 
let us endow $\dgcat$ with the Dwyer-Kan model structure and 
$\sSet$ with the Joyal model structure. Then the Quillen adjunction is (cf.~\cite{LurieHA})
\begin{equation}\label{Lambdadgnerve}
\begin{tikzcd}
\sSet\ar[r,shift left=0.8ex,"\Lambda"]&\dgcat\ar[l,shift left=0.8ex,"N_{dg}"]
\end{tikzcd}
\end{equation}
where $N_{dg}$ is the right Quillen functor. 
In particular, as each small dg category is fibrant in the Dwyer-Kan model structure, the
dg nerve sends quasi-equivalences to equivalences of $\infty$-categories.

Let us describe the functor $\Lambda$ in more detail. 
For a standard $n$-simplex $\Delta^n$, the dg category
$\Lambda(\Delta^n)$ is defined as follows: let $Q^{(n)}$ be the 
graded quiver with vertex set $[n]$ and for each chain $I=\{i<p_1<\cdots<p_l<j\}$ with $l\geq 0$ an arrow 
\[
I:i\rightarrow j
\]
of degree $l$. Then $\Lambda(\Delta^n)$ is the dg $k$-category whose underlying graded category is the free graded $k$-category generated by $Q^{(n)}$ with differential given by
\[
d(I)=\sum_{1\leq m\leq l}(-1)^{l-m}(I \backslash \{p_m\}-\{p_m< \cdots<p_l<j\}\circ\{i<p_1<\cdots<p_m\})
\] 
on each chain $I=\{i<p_1<\cdots<p_l<j\}$.

For a morphism $\sigma:[m]\rightarrow [n]$ in $\Delta$, the dg functor $\Lambda(\Delta^{\sigma}):\Lambda(\Delta^m)\rightarrow \Lambda(\Delta^n)$ is given by
\begin{equation*}
    f_{\sigma}(I)=
    \begin{cases}
      \sigma(I) & \text{if }\sigma|_{I}\text{ is injective,} \\
      \Id_p & \text{if }I=\{i<j\}\text{ and }\sigma(i)=\sigma(j)=p,\\
      0 & \text{otherwise.}
    \end{cases}
  \end{equation*}
  
    We now construct a natural dg functor (cf.~\cite{Lefevre03})
  \[
 \theta_{m,n}: \Lambda(\Delta^m\times \Delta^n)\rightarrow \Lambda(\Delta^m)\otimes \Lambda (\Delta^n).
  \]
  
  A {\em chain} of a finite ordered set is a totally ordered finite subset. 
  Recall that we have the following presentation of $\Delta^m\times \Delta^n$ (\cite[Chapter 2, 5.5]{GabrielZisman67})
  \[
  \begin{tikzcd}
  \coprod_{1\leq i<j\leq \left(\begin{smallmatrix}m+n\\m\end{smallmatrix}\right)} \Delta^{n_{c{(i)}\cap c{(j)}}}\ar[r,shift left=0.8ex,"u"]\ar[r,shift right=0.8ex,"v"swap]&\coprod_{1\leq i\leq \left(\begin{smallmatrix}m+n\\n\end{smallmatrix}\right)}\Delta^{n_{c{(i)}}}\ar[r]&\Delta^m\times \Delta^n
  \end{tikzcd}
  \]
  where each $c(i)$ is a maximal chain of the ordered set $[m]\times [n]$ and $n_{c(i)}=m+n$ is the length of the chain $c(i)$ for each $i$.
  The maximal chain $c(i)$ corresponds to an $(m,n)$-shuffle, i.e.~a strictly increasing map of ordered sets $\sigma^{(i)}:[m+n]\rightarrow [m]\times [n]$. 
  The chain $c(i)\cap c(j)$ corresponds to the equalizer $\mathrm{Eq}_{i,j}$ of the maps $\sigma^{(i)}$ and $\sigma^{(j)}$. 
  The maps $u$ and $v$ then correspond to the inclusion of 
 $\mathrm{Eq}_{i,j}$ into $[m+n]$.
  
  For an $(m,n)$-shuffle $\sigma:[m+n]\rightarrow [m]\times [n]$, put $\sigma_{-}:[m+n]\rightarrow [m]$ and $\sigma_{+}:[m+n]\rightarrow [n]$ the components of $\sigma$, i.e.~$\sigma(j)=(\sigma_{-}(j),\sigma_{+}(j))$ for $0\leq j\leq m+n$.
  
  Now for each $1\leq i\leq \left(\begin{smallmatrix}m+n\\n\end{smallmatrix}\right)$ define 
 \[
 F_{c(i)}:\Lambda(\Delta^{n_{c(i)}})\rightarrow \Lambda (\Delta^m)\otimes \Lambda(\Delta^n)
  \]
  to be the dg functor given by $\Lambda(\sigma^{(i)}_{-})$ and $\Lambda(\sigma^{(i)}_{+})$.
  Then by definition $(F_{c(i)})_{1\leq i\leq \left(\begin{smallmatrix}m+n\\n\end{smallmatrix}\right)}$ induces a dg functor
  \[
  \theta_{m,n}:\Lambda(\Delta^m\times\Delta^n)\rightarrow \Lambda(\Delta^m)\otimes\Lambda(\Delta^n).
  \]
  It extends by colimit to
  \[
  \theta_{X,Y}:\Lambda(X\times Y)\rightarrow \Lambda(X)\otimes\Lambda(Y)
  \]
  for simplicial sets $X$ and $Y$.
  \begin{proposition}\label{prop:Lambdamonoidal}
  For any pair of simplicial sets $(X,Y)$, the above dg functor $\theta_{X,Y}$ is a quasi-equivalence.
  \end{proposition}
This can be proved using the fact that the coherent nerve commutes with products (up to categorical
weak equivalence), which is proved in \cite[Corollary 2.2.5.6]{Lurie09}. 
We give a direct proof after Lemma \ref{lem:dglocalizationtensor}.
The key ingredient is to express each $\infty$-category as the localization of the nerve of an ordinary category by a subcategory, cf.~\cite[Proposition 7.3.15]{Cisinski19}.

Since right adjoints of monoidal functors preserve internal Homs, we obtain the following corollary.

\begin{corollary}[\cite{HolsteinLazarev22}, Lemma 5.1]\label{dgnerve}
Let $\A$ be a small dg category and $X$ a simplicial set. 
We have an equivalence of $\infty$-categories
\[
\phi_{X,\A}: N_{dg}(\rep_{dg}(\Lambda(X),\A))\rightarrow \Map(X,N_{dg}(\A)).
\]
In particular, we have an equivalence of categories
\[
\delta_{X,\A}: \rep(\Lambda(X),\A)\rightarrow h(\Map(X,N_{dg}(\A))).
\]
\end{corollary}
\begin{proof}
Let $Y$ be a simplicial set. 
Put $\mathcal H=\mathrm{Ho}(\sSet_{Joyal})$.
Then by adjunction we have 
\[
\begin{aligned}
\Hom_{\mathcal H}(Y,\Map(X,N_{dg}(\A)))&\iso \Hom_{\mathcal H}(Y\times X,N_{dg}(\A))\\
&\iso \Hom_{\Hqe}(\Lambda(Y\times X),\A) \\
&\xleftarrow{\Hom(\theta_{X,Y},\A)} \Hom_{\Hqe}(\Lambda(X)\otimes \Lambda(Y),\A)\\
&\iso \Hom_{\Hqe}(\Lambda(Y),\rep_{dg}(\Lambda(X),\A))\\
&\iso \Hom_{\mathcal H}(Y,N_{dg}\rep_{dg}(\Lambda(X),\A)).
\end{aligned}
\]
Thus we have a natural isomorphism in $\mathrm{Ho}(\sSet_{Joyal})$
\[
N_{dg}(\rep_{dg}(\Lambda(X),\A))\iso \Map(X,N_{dg}(\A)).
\]
\end{proof}
We have the following diagram of functors
\[
\begin{tikzcd}
k\mbox{-}\cat\ar[d,shift left=0.8ex, "F"]\ar[r,hook]&\dgcat_k\ar[d,shift left=0.8ex,"N_{dg}"]\\
\cat\ar[r,"N"swap]\ar[u,"k(-)",shift left=0.8ex]&\infty\mbox{-}\cat\ar[u,shift left=0.8ex,"\Lambda"]
\end{tikzcd}
\]
where $k(-)$ is the $k$-linearization functor and $F$ is the forgetful functor.
For a $k$-category $\C$, we have an isomorphism of $\infty$-categories $N(F(\C))\iso N_{dg}(\C)$.
For a small category $\I$, we have a natural morphism $N(\I)\rightarrow N_{dg}(k(\I))$ which induces a canonical dg functor $\Lambda(N(\I))\rightarrow k(\I)$. 

\begin{lemma} \label{lem:Lambdaklinearization}
The canonical dg functor $\Lambda(N(\I))\rightarrow k\I$ is a quasi-equivalence for each small category $\I$.
\end{lemma}
\begin{proof}

We first consider the special case when $\I$ is the path category $P(Q)$ of a quiver $Q$. 
  Recall that $N(\I)$ is the colimit of the functor $\Delta(N(\I))\rightarrow \sSet$ sending $(\Delta^n,x:\Delta^n\rightarrow N(I))$ to $\Delta^n$ (cf.~\cite[Lemma 3.13]{Hovey99}).
 Notice that the full subcategory $\J$ of $\Delta(N(\I))$ consisting of $n$-simplices of $N(\I)$ 
 \[
 (i_0\xrightarrow{\alpha_0} i_1\xrightarrow{\alpha_1}\ldots\xrightarrow{\alpha_{n-1}} i_n)
 \]
  where $\alpha_j$ is an arrow in $Q$ for $0\leq j\leq n-1$, is a cofinal subcategory.
 Notice also that the morphisms $(\Delta^n,p)\rightarrow (\Delta^m,q)$ in $\J$ are induced by inclusions of paths, i.e.~$q$ is a part of $p$ as a path. Therefore $\Lambda(N(\I))$ is the colimit 
 \[
\colim_{p:\Delta^n\rightarrow N(\I)} \Lambda(\Delta^n)
 \]
 where $p$ runs through the paths in $Q$ and $n$ is the length of $p$.
 
 For any pair of objects $(i,j)$, let $\mathcal S_{i,j}$ be the set of paths from $i$ to $j$. Let $\Lambda'$ be the dg category whose set of objects is $Q_0$ and the Hom complex $\Hom_{\Lambda'}(i,j)$ is given by the complex
 \[
\bigoplus_{p\in\mathcal S_{i,j}}\Hom_{\Lambda(\Delta^{m})}(0,m)
 \]
 where $m$ is the length of the path $p$.
 The composition map
 \begin{align*}
 \Hom_{\Lambda(\Delta^m)}(0,m)\otimes\Hom_{\Lambda(\Delta^n)}(0,n)\rightarrow \Hom_{\Lambda(\Delta^{m+n})}(0,m+n)
 \end{align*}
 where $m$ is the length of a path $p:i\rightarrow j$ and $n$ is the length of a path $q:j\rightarrow k$, is induced by the composition in the dg category $\Lambda(\Delta^{m+n})$ through the morphisms $\sigma:\Delta^m\rightarrow \Delta^{m+n}$ and $\delta:\Delta^n\rightarrow \Delta^{m+n}$, where $\sigma(l)=l$ for $0\leq l\leq m$ and $\delta(s)=s+m$ for $0\leq s\leq n$.

 For each path 
 \[
 p:i=p_0\rightarrow p_1\rightarrow \ldots\rightarrow p_{n-1}\rightarrow j=p_n
 \]
  of length $n$, we have a canonical dg functor $\Lambda(\Delta^n)\rightarrow \Lambda'$ which sends the object $i$ to $p_i$ for each $0\leq i\leq n$ and such that for each pair of objects $(u<v)$ in $\Lambda(\Delta^n)$, the morphism of Hom complexes
  \[
  \Hom_{\Lambda(\Delta^n)}(u,v)\iso\Hom_{\Lambda(\Delta^{v-u})}(0,v-u)\rightarrow \Hom_{\Lambda'}(p_u,p_v),\\
\]
is the inclusion into the component given by the path 
\[
p_u\rightarrow p_{u+1}\rightarrow\ldots\rightarrow p_v.
\]  
These dg functors are compatible with morphisms in $\J$ .
Therefore we have a canonical dg functor $\Lambda(N(\I))\rightarrow \Lambda'$ which is full and faithful and which is the identity on objects.
 Note that the canonical dg functor $\Lambda(\Delta^n)\rightarrow k[n]$ is a quasi-equivalence.
 Therefore the complex $\Hom_{\Lambda(N(\I))}(i,j)$ is quasi-isomorphic to $k\mathcal S_{i,j}$ for each pair $(i,j)$ and $\Lambda(N(\I))$ is quasi-equivalent to $k\I$.
 
 Now we show the case when $\I$ is a general small category.
 We consider the adjunction
 \[
 \begin{tikzcd}
 \Quiv\ar[r,shift left =0.8ex,"P(-)"]&\cat\ar[l,shift left=0.8ex,"F(-)"]
 \end{tikzcd}
 \]
 where the right adjoint $F(-)$ is the forgetful functor.
 This defines a comonad (cf.~\cite[Chapter VI, Definition 1]{MacLane98}) $(C=PF, \varepsilon, P\eta F)$ in the category $\cat$, where $\eta:\Id\rightarrow FP$ is the unit of the adjunction and $\varepsilon:PF\rightarrow \Id$ is the counit.
 By \cite[8.6.4]{Weibel94}, for each small category $\I$ we have a simplicial object $C(\I)$ in the category $\cat$
 where $C(\I)_n=(C)^{n+1}(\I)$ together with a morphism $C(\I)_0=PF(\I)\rightarrow \I$ given by the counit, which serves as an augmentation.
 
 We apply the $k$-linearization functor $k(-)$ and the functor $\Lambda\circ N(-)$ to the augmented simplicial object $C(\I)$.
 By the above, we have that the canonical dg functor $\Lambda N(C(\I)_n)\rightarrow k(C(\I)_n)$ is a quasi-equivalence for each $n\geq 0$. 
 Therefore the canonical dg functor $\Lambda N(\I)\rightarrow k(\I)$ is also a quasi-equivalence.
\end{proof}
Let $\C$ be a simplicial set and $\W\subset \C$ a simplicial subset. Given an $\infty$-category $\D$, we write $\Map_{\W}(\C,\D)$ for the full subcategory of $\Map(\C,\D)$ which consists of functors $f:\C\rightarrow \D$ such that, for any map $u:x\rightarrow y$ in $\W$, the induced map $f(u):f(x)\rightarrow f(y)$ is invertible in $\D$. In other words, there is a pull-back square of the following form, induced by the opertation of restriction along $\W\subset \C$
\[
\begin{tikzcd}
\Map_{\W}(\C,\D)\ar[r]\ar[d]&\Map(\C,\D)\ar[d]\\
\Map_{\W}(\W,\D)\ar[r]&\Map(\W,\D).
\end{tikzcd}
\]
\begin{definition}[\cite{Cisinski19}, Definition 7.1.2.]
A {\em localization} of $\C$ by $\W$ is a functor $\gamma:\C\rightarrow \W^{-1}\C$ such that 
\begin{itemize}
\item[(i)] $\W^{-1}\C$ is an $\infty$-category;
\item[(ii)] the functor $\gamma$ sends the morphisms of $\W$ to invertible morphisms in $\W^{-1}\C$;
\item[(iii)] for any $\infty$-category $\D$, composing with $\gamma$ induces an equivalences of $\infty$-categories of the form
\begin{equation}\label{mor:localization}
\Map(\W^{-1}\C,\D)\rightarrow \Map_{\W}(\C,\D).
\end{equation}
\end{itemize}
\end{definition}
By \cite[Proposition 7.1.3]{Cisinski19}, the localization of $\C$ by $\W$ always exists and is essentially unique.
By taking the isomorphism classes of objects of the associated homotopy categories of the $\infty$-categories in \ref{mor:localization}, we see that a localization $\gamma:\C\rightarrow \W^{-1}\C$ induces a bijection
\[
\mathcal H(\C,\D)\rightarrow \mathcal H_{\W}(\C,\D)
\]
where $\mathcal H=\mathrm{Ho}(\sSet_{Joyal})$ and $\mathcal H_{\W}(\C,\D)$ is the subset of the set of morphisms $\C\rightarrow \D$ in $\mathcal H$ whose induced functor $h(\C)\rightarrow h(\D)$ send morphisms in $\W$ to isomorphisms in $h(\D)$. 
\begin{definition}[\cite{Toen07,Toen11}]
Let $\A$ be a dg category and $S$ a subset of morphisms in $H^0(\A)$. 
A morphism $\theta:\A\rightarrow L_{S}\A$ in $\Hqe$ is a {\em localization of $\A$ along $S$} if for any $\A'\in\Hqe$, the induced morphism
\[
\theta^*:\Hqe(L_{S}\A,\A')\rightarrow \Hqe(\A,\A')
\]
is a bijection onto the subset $\Hqe_S(\A,\A')$ whose elements are the morphisms $\A\rightarrow \A'$ whose induced functor $H^0(\A)\rightarrow H^0(\A')$ sends all morphisms in $S$ to isomorphisms in $H^0(\A')$.
\end{definition}
For any dg category $\A$ and any set of morphisms $S$ in $H^0(\A)$, a localization $\A\rightarrow L_S(\A)$ exists in $\Hqe$, cf.~\cite[4.3, Proposition 3]{Toen11}. 
\begin{lemma}\label{lem:localizationinftydg}
Let $\C$ be a simplicial set and $\W\subset \C$ a simplicial subset.
Let $S\subseteq H^0(\Lambda(\C))\iso h(N_{dg}\Lambda(\C))$ be the set of morphisms given by the morphisms in $\W$ via the canonical functor $\C\rightarrow N_{dg}\Lambda(\C)$.
Then we have a canonical quasi-equivalence of dg categories
\[
\Lambda(\W^{-1}\C)\iso L_{S}\Lambda(\C).
\]
\end{lemma}
\begin{proof}
Let $\B$ be any dg category. Put $\mathcal H=\mathrm{Ho}(\sSet_{Joyal})$.
Then we have
\[
\begin{aligned}
\Hqe(\Lambda(\W^{-1}\C),\B)\iso& \mathcal H(\W^{-1}\C,N_{dg}(\B))\\
\iso &\mathcal H_{\W}(\C,N_{dg}(\B))\\
\iso &\Hqe_{S}(\Lambda(\C),\B).
\end{aligned}
\]
\end{proof}
\begin{lemma}[\cite{Toen11}, Exercise 25]\label{lem:dglocalizationtensor}
Let $\A_1$ and $\A_2$ be cofibrant dg categories and $S_1\subset H^0(\A_1)$ and $S_2\subset H^0(\A_2)$ be two sets of morphisms which contain all the identities. Then we have a canonical quasi-equivalence of dg categories
\[
L_{S_1\otimes S_2}(\A_1\otimes \A_2)\iso L_{S_1}\A_1\otimes L_{S_2}\A_2.
\]
\end{lemma}
\begin{proof}
The subset $S_1\otimes S_2\subset H^0(\A_1\otimes \A_2)$ consists of homotopy equivalence classes of $f\otimes g$ where $\overline{f}\in S_1$ and $\overline {g}\in S_2$. This is independent of the choice of the representatives $f$ and $g$.
Let $\B$ be any dg category. 
We have 
\[
\begin{aligned}
\Hqe(L_{S_1\otimes S_2}(\A_1\otimes\A_2),\B)\iso& \Hqe_{S_1\otimes S_2}(\A_1\otimes \A_2,\B)\\
\iso &\Hqe_{S_1}(\A_1,\rep_{dg}(\A_2,\B)_{S_2})\\
\iso &\Hqe(L_{S_1}\A_1,\rep_{dg}(L_{S_2}\A_2,\B))\\
\iso &\Hqe(L_{S_1}\A_1\otimes L_{S_2}\A_2,\B)
\end{aligned}
\]
where $\rep_{dg}(\A_2,\B)_{S_2}$ is the full dg subcategory of $\rep_{dg}(\A_2,\B)$ consisting of morphisms $\A_2\rightarrow \B$ in $\Hqe$ sending morphisms in $S_2$ to isomorphisms in $H^0(\B)$. 
\end{proof}
Similarly, for $\infty$-categories $\C_1$ and $\C_2$, and simplicial subsets $\W_1\subset \C_1$ and $\W_2\subset \C_2$ both of which contain all the identities, we have a canonical equivalence of $\infty$-categories
\[
(\W_1\times \W_2)^{-1}(\C_1\times \C_2)\iso \W_1^{-1}\C_1\times \W_2^{-1}\C_2.
\]
\begin{proof}[Proof of Proposition \ref{prop:Lambdamonoidal}]
We may suppose that $X$ and $Y$ are $\infty$-categories.
By \cite[Proposition 7.3.15]{Cisinski19}, the $\infty$-category $X$ (resp.~$Y$) is the localization of the nerve $N(\I)$ (resp.~$N(\J)$) of an ordinary category $\I$ (resp.~$\J$) by a certain simplicial subset $\W\subset N(\I)$ (resp.~$\W'\subset N(\J)$).

Let $S\subset H^0(\Lambda(N(\I)))$ (resp.~$S'\subset H^0(\Lambda(N(\J)))$) be the set of morphisms given by the morphisms in $\W$ (resp.~$\W'$).
Then we have 
\[
\begin{aligned}
\Lambda(X\times Y)\iso & \Lambda(\W^{-1}N(\I)\times W'^{-1}N(\J))\\
\iso &\Lambda((\W\times \W')^{-1}N(\I)\times N(\J))\\
\xrightarrow[\sim]{\mathrm{Lemma}~\ref{lem:localizationinftydg}}&L_{S\otimes S'}\Lambda(N(\I)\times N(\J))\\
\xrightarrow[\sim]{\mathrm{Lemma}~\ref{lem:Lambdaklinearization}} &L_{S\otimes S'}(k\I\otimes k\J)\\
\xrightarrow[\sim]{\mathrm{Lemma}~\ref{lem:dglocalizationtensor}} &L_{S}(k\I)\otimes L_{S'}(k\J)\\
\iso &L_{S}(\Lambda(N(\I)))\otimes L_{S'}(\Lambda(N(\J)))\\
\iso &\Lambda(\W^{-1}N(\I))\otimes \Lambda(\W'^{-1}N(\J))\\
\iso &\Lambda(X)\otimes \Lambda(Y).
\end{aligned}
\]
\end{proof}
\begin{lemma}
Suppose that $\A$ is a connective additive dg category and that $S$ is a set of morphisms in $H^0(\A)$ which admits a calculus of left fractions in the sense of \cite[Chapter 1, 2.2]{GabrielZisman67}.
Then we have an equivalence of $k$-categories $H^0(L_S\A)\rightarrow H^0(\A)[S^{-1}]$.
\end{lemma}
\begin{proof}
We assume that $\A$ is cofibrant and strictly connective.
The dg functor $\A\rightarrow H^0(\A)$ is universal among morphisms in $\Hqe$ from $\A$ to dg categories concentrated in degree 0.
Indeed, it follows from the fact that each morphism $\B_1\rightarrow \B_2$ in $\Hqe$ between dg categories concentrated in degree zero is given by a dg functor $F:\B_1\rightarrow \B_2$. 
Note that two dg functors $F,G:\B_1\rightarrow \B_2$ give rise to the same morphism in $\Hqe$ if and only if they are isomorphic as $k$-linear functors.

Let $\theta:\A\rightarrow L_S{\A}$ be the localization morphism in $\Hqe$.
By \cite[Chapter 1, 3.3]{GabrielZisman67}, the category $H^0(\A)[S^{-1}]$ is additive and the localization functor $u:H^0(\A)\rightarrow H^0(\A)[S^{-1}]$ is an additive functor. 
We have the following diagram in $\Hqe$
\[
\begin{tikzcd}
\A\ar[r,"\theta"]\ar[d,"p"swap]&L_S\A\ar[d,"q"]\ar[ldd,dashed, red,"t",bend left = 24ex]\\
H^0(\A)\ar[r,"H^0(\theta)"]\ar[d,"u"swap]&H^0(L_S\A)\ar[ld,shift left = 1.2ex,dashed, red,"r"]\\
H^0(\A)[S^{-1}]\ar[ru,shift left=1.2ex,dashed,red,"s"]&
\end{tikzcd}
\]
where $p$ and $q$ are the projection dg functors.
By the universal property of $\theta$, there exists a morphism $t:L_S\A\rightarrow H^0(\A)[S^{-1}]$ such that $t\theta=up$.
By the universal property of $q$, there exists a morphism $r:H^0(L_S\A)\rightarrow H^0(\A)[S^{-1}]$ such that $rq=t$.
By the universal property of $u$, there exists a functor $s:H^0(\A)[S^{-1}]\rightarrow H^0(L_S\A)$ such that $su=H^0(\theta)$. 
We view $s$ as a morphism in $\Hqe$. 
Then we have $rsup=rsH^0(\theta)=rq\theta=t\theta=up$. It follows that $rsu=u$ in $\Hqe$, or equivalently $rsu$ is isomorphic to $u$ as $k$-linear functors. Therefore $rs$ is isomorphic to $\Id$ as $k$-linear functors, or equivalently as dg functors. Hence $rs=\Id$ in $\Hqe$.
Similarly we have $srq\theta=st\theta=sup=H^0(\theta)p=q\theta$ and hence $sr=\Id$ in $\Hqe$.
\end{proof}
\begin{remark}\label{homotopylimit}
By Corollary~\ref{dgnerve}, we define an {\em $\mathcal{I}$-indexed diagram} in $\A$ to be an object in $\rep(k(\I),\A)$ and use it to define homotopy limits in $\A$.
The case for homotopy colimits is completely dual. 
For a small category $\I$, let $\I^{\vartriangleleft}$ be the join of $*$ and $\I$, i.e.~the category obtained from $\I$ by adding a new object $*$ together with a unique map from $*$ to each object in $\I$.
Then we have the following diagram of functors
\[
\begin{tikzcd}
*&\I\ar[r,"G"]\ar[l,"F"swap]&\I^{\vartriangleleft}&*\ar[l,"H"swap]
\end{tikzcd}
\]
where $H$ sends $*$ to $*$ in $\I^{\vartriangleleft}$ and $G$ is the canonical inclusion into the join. 
It induces the following adjunctions
\[
\begin{tikzcd}
\D(\A\otimes k\I^{op})\ar[r,shift left=0.8ex,"F_{*}"]&\D(\A)\ar[l,shift left=0.8ex,"F^{*}"]
\end{tikzcd},
\begin{tikzcd}
\D(\A\otimes k\I^{op})\ar[r,shift left=0.8ex,"G_{*}"]&\D(\A\otimes k(\I^{\vartriangleleft})^{op})\ar[l,shift left=0.8ex,"G^{*}"]
\end{tikzcd}
\]
and similarly for $H_*$ and $H^*$.
Let $M$ be an object in $\rep(k\I^{\vartriangleleft},\A)$.
Then $M$ corresponds to a morphism $F^{*}H^{*}M\rightarrow G^{*}M$ in $\C(\A\otimes k(\I^{\vartriangleleft})^{op})$.  
By adunction, we have a morphism $H^*M\rightarrow F_*G^*M$.

Let $N$ be an object in $\rep(k\I,\A)$. Then a pair $(M,\theta)$, where $M$ is an object in $\rep(k\I^{\vartriangleleft},\A)$ and $\theta:G^*M\rightarrow N$ is an isomorphism in $\rep(k\I,\A)$, is called the homotopy limit of $N$ if for each object $A\in \A$, the canonical map 
$H^*M\rightarrow F_*G^*M\xrightarrow{F_*(\theta)}F_*N$ in $\D(\A)$ induces a quasi-isomorphism of complexes
\[
\tau_{\leq 0}\RHom(A^{\wedge},H^*M)\rightarrow\tau_{\leq 0}\RHom(A^{\wedge},F_*N).
\] 
\end{remark}
We are mainly concerned with the case when $X=\Delta^1\times \Delta^1$.
The dg category $\Lambda(X)$ is the dg $k$-path category given by the graded quiver
\[
\begin{tikzcd}
0\ar[r,"f"]\ar[rd,"l"{yshift=-3pt}]\ar[rd,bend right=5ex,"h_1"{yshift=6pt,xshift=2pt,red},red,swap]\ar[rd,bend left=5ex,"h_2"{yshift=-5pt, red},red]\ar[d,"g"swap]&1\ar[d,"j"]\\
2\ar[r,"k"swap]&3
\end{tikzcd}
\]
where $|f|=|j|=|l|=|g|=|k|=0$, $|h_1|=|h_2|=-1$, $d(h_1)=l-kg$ and $d(h_2)=l-jf$.

The obvious dg functor $\Lambda(X)\rightarrow k(\mathrm{Sq})\iso\Lambda(\Delta^1)\otimes\Lambda(\Delta^1)$ is a quasi-equivalence.
The induced functor $\rep(\Lambda(X),\A)\rightarrow \rep(\Sq,\A)$ sends an object
\[
\begin{tikzcd}
U\ar[r,"f"]\ar[rd,"l"{yshift=-3pt}]\ar[d,"g"swap]\ar[rd,bend right=5ex,"h_1"{yshift=6pt,xshift=2pt,red},red,swap]\ar[rd,bend left=5ex,"h_2"{yshift=-5pt, red},red]&Y\ar[d,"j"]\\
W\ar[r,"k"swap]&Z
\end{tikzcd}
\]
to
\[
\begin{tikzcd}
U\ar[r,"\begin{bmatrix}f\\i\end{bmatrix}"]\ar[d,"\begin{bmatrix}g\\i\end{bmatrix}"swap]&Y\oplus IU\ar[d,"\begin{bmatrix}j\;\; h_2\;\;d(h_2)\end{bmatrix}"]\\
W\oplus IU\ar[r,"\begin{bmatrix}k\;\; h_1\;\;d(h_1)\end{bmatrix}"swap]&Z
\end{tikzcd}
\]
where $i:U\rightarrow IU$ is the inclusion of $U$ into $\Cone(\Id_{U})$.
Note that by Corollary~\ref{dgnerve} or using the fact that $\Lambda(X)$ is cofibrant, we may assume $U$, $Y$, $W$ and $Z$ are representable dg $\A$-modules.
Again by Corollary~\ref{dgnerve}, if the dg category $\A$ has contractible pre-envelopes and contractible 
pre-covers, then we may further assume $h_1=0$, $h_2=0$.

The morphism $\phi_{\Delta^1\times \Delta^1,\A}: \rep(\Lambda(\Delta^1\times\Delta^1),\A)\rightarrow h(\Map(\Delta^1\times\Delta^1,N_{dg}(\A)))$ sends an object in $\rep(\Lambda(\Delta^1\times \Delta^1),\A)$ which corresponds to the diagram
\[
\begin{tikzcd}
A^{\wedge}\ar[r,"f^{\wedge}"]\ar[rd,"l^{\wedge}"{yshift=-3pt}]\ar[d,"g^{\wedge}"swap]\ar[rd,bend right=5ex,"h_1^{\wedge}"{yshift=6pt,xshift=2pt,red},red,swap]\ar[rd,bend left=5ex,"h_2^{\wedge}"{yshift=-5pt, red},red]&B^{\wedge}\ar[d,"j^{\wedge}"]\\
C^{\wedge}\ar[r,"k^{\wedge}"swap]&D^{\wedge}
\end{tikzcd}
\]
to the object in $\Map(\Delta^1\times\Delta^1,N_{dg}(\A))$
\begin{equation} \label{dia:htpy-square}
\begin{tikzcd}
A\ar[r,"f"]\ar[rd,"l"{yshift=-3pt}]\ar[d,"g"swap]\ar[rd,bend right=5ex,"h_1"{yshift=6pt,xshift=2pt,red},red,swap]\ar[rd,bend left=5ex,"h_2"{yshift=-5pt, red},red]&B\ar[d,"j"]\\
C\ar[r,"k"swap]&D
\end{tikzcd}\;.
\end{equation}

\begin{lemma}\label{homotopypullback}
Let $X$ be an object of the form (\ref{dia:htpy-square}) where $h_1$ and $h_2$ vanish.
Then $X$ is a homotopy pullback square if and only if the corresponding object $X'$ in $\rep(\Sq)$ is homotopy cartesian in the sense of Definition \ref{maindef}.
\end{lemma}

We will prove the lemma after Theorem~\ref{holimsimplicial} below.

To compute homotopy pullbacks/pushouts in the $\infty$-category $N_{dg}(\A)$, we will need a homotopy equivalent description of the dg nerve functor, which we denote by $N^{big}_{dg}(\A)$, following \cite{Faonte17b}.

The category $\sMod k$ is a monoidal category with tensor products given by componentwise tensor product of $k$-modules.
We denote this tensor product by $(X, Y) \mapsto X\boxtimes Y$.
The Dold-Kan construction $DK$ is a right-lax monoidal functor (in the sense of \cite[Definition A.1.3.5]{Lurie09}). 
The truncation functor $\tau_{\leq 0}:\C(k)\rightarrow \C(k)^{\leq 0}$ and the forgetful functor $F:\sMod k\rightarrow \sSet$ are also right-lax monoidal.
Thus, by \cite[Remark A.1.4.3]{Lurie09}, from a dg $k$-category $\A$, one can produce a simplicial category $s(\A)$ having the same objects as $\A$ and where the simplicial set of morphisms from $x$ to $y$ is the
underlying simplicial set of the Dold--Kan construction applied to the truncation $\tau_{\leq 0}\A(x,y)$.
Notice that $s(\A)$ is locally fibrant in the sense that,
for each pair of objects $x,y$ in $\A$, the mapping space $s(\A)(x,y)$ is a Kan complex. 
This implies that $s(\A)$ is also fibrant as a simplicial category by~\cite[Theorem 1.1]{Bergner07}.
By \cite[Proposition 1.1.5.10]{Lurie09}, the {\em simplicial nerve} of $s(\A)$ denoted by
\[
N^{big}_{dg}(\A)=N(s(\A))
\]
 is an $\infty$-category. 

\begin{proposition}[\cite{LurieHA}, Proposition 1.3.1.17]\label{bignerve}
There is a canonical equivalence of $\infty$-categories $\theta:N^{big}_{dg}(\A)\rightarrow N_{dg}(\A)$.
\end{proposition}
In general, for complexes $U, V\in \C(k)^{\leq 0}$, the simplicial $k$-module map 
\[
DK(U)\boxtimes DK(V)\rightarrow DK(U\otimes V), 
\]
which defines the right-lax monoidal structure of the functor $DK$, is not very explicit.
But it is explicit at the level of $0$-simplices and this is enough for our purposes: it sends 
$f\boxtimes g$ to $f\otimes g$. 
Also the $k$-module $DK(V)_1$ can be identified with $V^{-1}\oplus V^0$ and $DK(V)_0$ can be identified with $V^0$. 
The maps $d_0, d_1:DK(V)_1\rightarrow DK(V)_0$ are then identified with maps
\[
d_0:(h,k)\mapsto d(h)+k,\;\; d_1:(h,k)\mapsto k
\] 
where $h\in V^{-1}$ and $k\in V^0$.
A $2$-simplex of $N^{big}_{dg}(\A)$ can be identified with a simplicial functor $\C(\Delta^2)\rightarrow s(\A)$, which in turn can be identified with a diagram in $\A$
\[
\begin{tikzcd}
A\ar[rr,"k"{name=1}]\ar[rd,"f"swap]\ar[rr,"h"description,bend right=6ex]&&C\\
&B\ar[ru,"g"{swap}]&
\end{tikzcd}
\]
where the morphisms $f$, $g$ and $k$ are closed of degree $0$, the map $h:A\rightarrow C$ is of degree $-1$ and $d(h)+k=g\circ f$.
\begin{theorem}[\cite{Lurie09}, Theorem 4.2.4.1]\label{holimsimplicial}
Let $\C$ and $\J$ be fibrant simplicial categories and $F:\J\rightarrow \C$ be a simplicial functor. 
Suppose we are given an object $C\in\C$ and a compatible family of maps $\{\eta_{I}:F(I)\rightarrow C\}_{I\in\J}$. 
The following conditions are equivalent:
\begin{itemize}
\item[(1)] The maps $\eta_{I}$ exhibit $C$ as a homotopy colimit of the diagram $F$. 
\item[(2)] Let $f:N(\J)\rightarrow N(\C)$ be the simplicial nerve of $F$ and $\overline{f}:N(\J)^{\vartriangleright}\rightarrow N(\C)$ the extension of $f$ determined by the maps $\{\eta_{I}\}$. 
Then $\overline{f}$ is a colimit diagram in $N(\C)$.
\end{itemize}
\end{theorem}
\begin{proof}[Proof of Lemma \ref{homotopypullback}]
By Proposition \ref{bignerve}, the object $X$ is homotopy cartesian in the $\infty$-category $N_{dg}(\A)$ if and only if it is so in the big dg nerve $N_{dg}^{big}(\A)$.
Let $\J$ be the fibrant simplicial category
\[
\begin{tikzcd}
&1\ar[d,"*"]\\
2\ar[r,"*"swap]&3
\end{tikzcd}\;\;.
\]
Then the object $X$ restricts to a cospan and this gives rise to a simplicial functor $F:\J\rightarrow s(\A)$ where $s(\A)$ is the simplicial category described before Proposition \ref{bignerve}.
Since we have $kg=jf$ in $Z^0(\A)$, the morphisms $f:A\rightarrow B$, $g:A\rightarrow C$ and $kg:A\rightarrow D$ form a compatible family of maps.
Note that $N(\J)^{\vartriangleleft}=\Delta^1\times\Delta^1$ and that $X$ is the extension of $N(F)$ determined by the family of maps $\{f,\;g,\;kg\}$.
 By Theorem \ref{holimsimplicial}, the object $X$ is a homotopy pullback diagram in $N(s(\A))=N_{dg}^{big}(\A)$ if and only if for each object $A'\in\A$, the following diagram
\[
\begin{tikzcd}
U(DK(\tau_{\leq 0}\Hom_{\A}(A',A)))\ar[r]\ar[d]&U(DK(\tau_{\leq 0}\Hom_{\A}(A',B)))\ar[d]\\
U(DK(\tau_{\leq 0}\Hom_{\A}(A',C)))\ar[r]&U(DK(\tau_{\leq 0}\Hom_{\A}(A',D)))
\end{tikzcd}
\]
is homotopy pullback in $\sSet_{Quillen}$. 
Note that here we are using the fact that the homotopies $h_1$ and $h_2$ in $X$ are both zero so that we could apply Theorem \ref{holimsimplicial}.
Since both functors $U$ and $DK$ are right Quillen functors which preserve and reflect weak equivalences, this is equivalent to the diagram
\[
\begin{tikzcd}
\tau_{\leq 0}\Hom_{\A}(A',A)\ar[r]\ar[d]&\tau_{\leq 0}\Hom_{\A}(A',B)\ar[d]\\
\tau_{\leq 0}\Hom_{\A}(A',C)\ar[r]&\tau_{\leq 0}\Hom_{\A}(A',D)
\end{tikzcd}
\]
being homotopy pullback in $\C(k)^{\leq 0}$. 
By Example \ref{complexpullback}, this is equivalent to saying that the map
\[
\tau_{\leq 0}\Hom_{\A}(A',A)\rightarrow \tau_{\leq 0}\Sigma^{-1}\Cone((\tau_{\leq 0}\Hom_{\A}(A',B)\oplus \tau_{\leq 0}\Hom_{\A}(A',C))\rightarrow \tau_{\leq 0}\Hom_{\A}(A',D))
\]
is an isomorphism in $\D(k)$.
The complex on the right hand side is quasi-isomorphic to 
\[
\tau_{\leq 0}\Sigma^{-1}\Cone(\Hom_{\A}(A',B)\oplus\Hom_{\A}(A',C)\rightarrow \Hom_{\A}(A',D))
\]
which is further isomorphic to 
\[
\tau_{\leq 0}\Hom_{\A}(A'^{\wedge},\Sigma^{-1}\Cone(B^{\wedge}\oplus C^{\wedge}\rightarrow D^{\wedge})).
\]
Thus $X$ is homotopy cartesian in $N_{dg}(\A)$ if and only if the canonical map
\[
\tau_{\leq 0}\RHom(A'^{\wedge},A^{\wedge})\rightarrow \tau_{\leq 0}\RHom(A'^{\wedge},\Sigma^{-1}\Cone(B^{\wedge}\oplus C^{\wedge}\rightarrow D^{\wedge}))
\]
is an isomorphism in $\D(k)$ for each $A'\in \A$, i.e.~the corresponding object $X'$ in $\rep(\Sq)$ is homotopy cartesian in the sense of Definition \ref{maindef}.
\end{proof}
\begin{remark}Similarly, for a family of objects $\{A_i\}_{i\in\I}$ in $\A$, the {\em homotopy coproduct} 
is an object $A\in \A$ with morphisms $A_i\rightarrow A$ in $Z^0(\A)$ such that the induced morphism of complexes
\[
\tau_{\leq 0}\A(A,B)\to \tau_{\leq 0}\prod_{i\in\I}\A(A_i,B)
\]
is a quasi-isomorphism of complexes for each $B\in \A$.
In particular, if the morphisms $A_i\rightarrow A$ in $Z^0(\A)$ yield a direct sum in $H^0(\A)$, then they form a homotopy coproduct. The dual assertion also holds. 
\end{remark}

The notion of exact $\infty$-category was introduced by Barwick in \cite{Barwick15} as a homotopy-theoretic
generalization of Quillen's notion of exact category. Here we adopt a reformulation of Barwick's
axioms due to Jasso. 
Recall that an $\infty$-category that admits finite products and finite coproducts is defined to be {\em additive} if its homotopy 
category is additive, cf.~\cite[Definition 2.2]{Barwick15}, also \cite[Definition 2.6]{GepnerGrothNikolaus15}.

\begin{definition}\label{def:exactinfinity}
Let $\E$ be an additive $\infty$-category. 
An exact structure on $\mathcal E$ is a class $\mathcal S\subseteq \Map(\Delta^1\times\Delta^1,\E)$ stable under isomorphisms, consisting of bicartesian squares (called admissible sequences)
\[
\begin{tikzcd}
x\ar[r, tail,"i"]\ar[d,"",two heads]&y\ar[d,"p",two heads]\\
0\ar[r,tail]&z
\end{tikzcd}
\]
where $\begin{tikzcd}x\ar[r,tail,"i"]&y\end{tikzcd}$ is called an inflation and $\begin{tikzcd}y\ar[r,two heads,"p"]&z\end{tikzcd}$ is called a deflation, such that the following axioms are satisfied
\begin{itemize}
\item[E0]For any $x\in \mathcal E$, $1_x:x\rightarrow x$ is both an inflation and a deflation.
\item[E1]Compositions of deflations (resp.~inflations) are deflations (resp.~inflations).
\item[E2]Given a deflation $p:y\rightarrow z$ and any map $g: z'\rightarrow z$ in $\E$, there exists a homotopy pullback 
\[
\begin{tikzcd} 
{y'}\ar[r,"{q}"]\ar[d,"{h}"swap]&{z'}\ar[d,"{g}"]\\
{y}\ar[r,"{p}"swap, two heads]&{z}
\end{tikzcd}
\]
and ${q}$ is also a deflation;
\item[E$2^{op}$]Given an inflation $i: x\rightarrow y$ and any map $f:x\rightarrow x'$ in $\E$, there exists a homotopy pushout 
\[
\begin{tikzcd}
{x}\ar[r,"{i}",tail]\ar[d,"{f}"swap]&{y}\ar[d,"{}"]\\
{x'}\ar[r,"{j}"swap]&{y'}
\end{tikzcd}
\]
and ${j}$ is also an inflation.
\end{itemize}
\end{definition}
Recall that a {\em subcategory} of an $\infty$-category $\E$ is a simplicial subset $\E'\subseteq \E$ such that for some subcategory 
$h(\E)'$ of the homotopy category $h(\E)$, the square
\[
\begin{tikzcd}
\E'\ar[r]\ar[d]&\E\ar[d]\\
N(h\E)'\ar[r,hook]&N(h\E)
\end{tikzcd}
\] 
is a pullback diagram of simplicial sets, cf.~\cite[1.2.11]{Lurie09}.
\begin{lemma} The above definition of exact $\infty$-category is equivalent to Barwick's. 
More precisely, the assignment sending an exact structure $\mathcal S$ to the pair of subcategories $(\E_{\dag},\E^{\dag})$, where $\E_{\dag}$ is the subcategory of inflations and $\E^{\dag}$ is the subcategory of deflations, gives a bijection between the class of exact structures in the sense of Definition \ref{def:exactinfinity} and the class of exact structures in the sense of Barwick \cite[Definition 3.1]{Barwick15}.
\end{lemma}
\begin{proof}
To show that the pair $(\E_{\dag},\E^{\dag})$ forms an exact structure in the sense of Barwick, 
we need to show that in a homotopy pullback square $X$ in $\E$
\[
\begin{tikzcd} 
{y'}\ar[r,"{q}"]\ar[d,"{h}"swap]&{z'}\ar[d,"{g}",tail]\\
{y}\ar[r,"{p}"swap, two heads]&{z}
\end{tikzcd}
\]
where $p$ is a deflation and $g$ is an inflation, the square is homotopy bicartesian and the morphism $h$ is an inflation. 
By Axiom $\rm{E}2$, the morphism $q$ is also a deflation and hence the homotopy fiber of $q$ is homotopy bicartesian. 
So the above square $X$ is homotopy bicartesian.
The homotopy cofiber of $h$ is a composition of $p$ and the homotopy cofiber of $g$ and hence is a deflation.
Therefore the morphism $h$ is an inflation.

Conversely, suppose we are given a pair of full subcategories $(\E_{\dag},\E^{\dag})$ defining an exact structure in the sense of Barwick.
We call morphisms in $\E_{\dag}$ (resp.~$\E^{\dag}$) inflations (resp.~deflations).
Consider a square in $\E$
\begin{equation}\label{square:infinity}
\begin{tikzcd}
x\ar[r,"i"]\ar[d,""]&y\ar[d,"p"]\\
0\ar[r]&z\mathrlap{.}
\end{tikzcd}
\end{equation}
If $i$ is an inflation and the square is a pushout, then by Barwick's axioms, the square is homotopy bicartesian and $p$ is a deflation since the morphism $x\rightarrow 0$ is a deflation.
Dually, if $p$ is a deflation and the square is a pullback, then the square is homotopy bicartesian and $i$ is an inflation since the morphism $0\rightarrow z$ is an inflation.
So the class $\mathcal S$ of homotopy bicartesian squares (\ref{square:infinity}), where $i$ is an inflation and $p$ is a deflation, defines an exact structure on $\E$ in the sense of Definition~\ref{def:exactinfinity}.
\end{proof} 

\begin{proof}[Proof of Theorem~\ref{nerve}]
The $\infty$-category $N_{dg}(\A)$ is an additive $\infty$-category if and only if $H^0(\A)$ is an additive category.
From now on, we may assume that $H^0(\A)$ is an additive category.
Note that the dg nerve functor $N_{dg}$ sends quasi-equivalences to equivalences of $\infty$-categories and
that there is a bijection between the classes of exact $\infty$-structures on two equivalent exact $\infty$-categories. 
The inclusion of $\A$ into its additive closure $\A'$ in $\pretr(\A)$ is a quasi-equivalence. 
Hence we may replace $\A$ by $\A'$ and assume that $Z^0(\A)$ is additive and that $\A$ has contractible pre-covers and contractible pre-envelopes.
Then any object in $\Map(\Delta^1\times\Delta^1,N_{dg}(\A))$ is isomorphic to an object of the form in Lemma \ref{homotopypullback}
and similarly for objects in $\rep_{dg}(\Lambda(\Delta^1\times\Delta^1),\A)$.

By Corollary~\ref{dgnerve} and Lemma \ref{homotopypullback}, the equivalence of categories 
\[
\delta=\delta_{\Delta^1\times\Delta^1,\A}: \rep(\Lambda(\Delta^1\times \Delta^1),\A)\rightarrow h(\Map(\Delta^1\times\Delta^1,N_{dg}(\A)))
\]
 induces a bijection between the classes of homotopy pullback squares.
Dually, it induces a bijection between the classes of homotopy pushout squares and hence also a bijection between the classes of homotopy bicartesian squares.

Recall that we have a fully faithful functor $\mathcal H_{3t}(\A)\hookrightarrow \rep(\Lambda(\Delta^1\times\Delta^1),\A)$. Its composition with $\delta$ sends a 3-term h-complex
\[
\begin{tikzcd}
A\ar[r,"f"]\ar[rr,"h"swap,bend right=8ex]&B\ar[r,"j"]&C
\end{tikzcd}
\]
to the object in $\Map(\Delta^1\times\Delta^1,N_{dg}(\A))$
\[
\begin{tikzcd}
A\ar[rr,"f"]\ar[d,"i"swap]& &B\ar[d,"j"]\\
IA\ar[rr,"{[}-h{,}\;-d(h){]}"swap]& & C
\end{tikzcd}
\]
where $i:A\rightarrow IA=\Cone(\Id_{A})$ is the canonical inclusion. 

Let $\mathcal S\subseteq \mathcal H_{3t}(\A)$ be an exact dg structure on $\A$ and let $\mathcal S'\subseteq \Map(\Delta^1\times\Delta^1,N_{dg}(\A))$ be the corresponding class of objects which is closed under isomorphisms. 
By the above discussion, the class $\mathcal S'$ consists of homotopy bicartesian squares.
Let us check the axioms $\rm{E0}$--$\rm{E2}^{op}$ for the class $\mathcal S'$.
Note that a morphism in $N_{dg}(\A)$ is identified with a morphism in $Z^0(\A)$ and that the class of inflations in $\mathcal S'$ is stable under isomorphisms in $\Map(\Delta^1,N_{dg}(\A))$.
 Axiom E0 follows from Proposition \ref{property}.
Suppose we have inflations $f:A\rightarrow B$ and $g:B\rightarrow C$ and a $2$-simplex
\[
\begin{tikzcd}
A\ar[r,"f"]\ar[rr,"j"swap, bend right=8ex]\ar[rr,"h"description,bend right=4ex]&B\ar[r,"g"]&C
\end{tikzcd}
\]
in $N_{dg}(\A)$ where $j-gf=d(h)$. 
If we view $f$ and $g$ as objects in $\Mor(\A)$, then they are inflations in $\mathcal S$ and so is $gf$.
Since $j$ is isomorphic to $gf$ in $H^0(\Mor(\A))$, the morphism $j$ is also an inflation in $\mathcal S$ and hence is an inflation in $\mathcal S'$ by definition. 
This proves that inflations are stable under compositions in $N_{dg}(\A)$.
Dually one proves that deflations are stable under compositions and this shows Axiom $\rm{E1}$.
Suppose we are in the context of Axiom $\rm{E2}$.
So we have a deflation $j:B\rightarrow C$ in $\mathcal S'$ and a morphism $g:C'\rightarrow C$ in $N_{dg}(\A)$.
Then the morphism $j$ is also an inflation in $\mathcal S$. 
Thus by Axiom $\Ex2$ for the class $\mathcal S$, the cospan
\[
\begin{tikzcd}
&C'\ar[d,"g"]\\
B\ar[r,"j"swap]&C
\end{tikzcd}
\]
 admits a homotopy pullback and the homotopy pullback of $j$ remains a deflation.
 Thus the image of this homotopy pullback in $\Map(\Delta^1\times\Delta^1,N_{dg}(\A))$ under the functor $\delta$ provides the required homotopy pullback. This proves Axiom $\rm{E2}$.
 Dually one shows Axiom $\rm{E2}^{op}$ for $\mathcal S'$.
 Therefore the class $\mathcal S'$ defines an exact $\infty$-structure on $N_{dg}(\A)$.
 
Similarly, we show that the class $\mathcal S$ defines an exact dg structure on $\A'$ if it corresponds to a class $\mathcal S'$ which defines an exact $\infty$-structure on $N_{dg}(\A)$.
\end{proof}
\newpage
\section{Embedding theorem for connective exact dg categories}
Throughout this section, let $\A$ be an exact dg category. 
By Remark~\ref{truncationexactdgstructure}, we may assume that $\A$ has cofibrant Hom complexes.
One of the main aims of this section is to show the following theorem.

\begin{theorem}\label{main}
Let $\A$ be an exact dg category. 
There exists a universal exact morphism $F:\A\rightarrow \D^b_{dg}(\A)$ in $\Hqe$ from $\A$ to a pretriangulated dg category $\D^b_{dg}(\A)$. 

If $\A$ is connective, this morphism satisfies:
\begin{itemize}
\item[1)]It induces a quasi-equivalence from $\tau_{\leq 0}{\A}$ to $\tau_{\leq 0}\D'$ for an extension-closed dg subcategory $\D'$ of $\D^b_{dg}(\A)$.
\item[2)]It induces a natural bijection $\mathbb E(C,A)\xrightarrow{\sim} \Ext^1_{\D^b(\A)}(FC,FA)$ for each pair of objects $C,A$ in $H^0(\A)$ where $\D^b(\A)=H^0(\D^b_{dg}(\A))$.
\end{itemize}
We call $\D^b_{dg}(\A)$ the {\em bounded dg derived category} of $\A$.
\end{theorem}
\begin{example} If $\A$ is a Quillen exact category, then $\D^b_{dg}(\A)$ is quasi-equivalent to the canonical dg enhancement of the bounded derived category of $\A$, whence the name.
\end{example}

\begin{remark} The analogous theorem for exact $\infty$-categories in the sense of Barwick is
due to Klemenc \cite{Klemenc22}.
\end{remark}

\subsection{Construction of the universal morphism}
Let $\mathcal N$ be the full triangulated subcategory of $\tr(\A)$ generated by the total dg modules $N$ of conflations 
\[
\begin{tikzcd}
A\ar[r,"f"]\ar[rr,"h"swap,bend right=8ex]&B\ar[r,"j"]&C
\end{tikzcd}. 
\]
So $N$ is defined by the following diagram in $\C_{dg}(\A)$
\begin{equation}\label{TR4}\tag{$\bigstar$}
\begin{tikzcd}[every label/.append style={font=\tiny}]
A\ar[r,"f"]\ar[d,"\begin{bmatrix}-f\\-h\end{bmatrix}",swap]&B\ar[r,"\begin{bmatrix}0\\1\end{bmatrix}"]\ar[d,equal]&U\ar[rd,"s"red,red]\ar[r,"{[}1{,}\;0{]}"]\ar[d,"{[}-h{,}j{]}",swap]&\Sigma A\ar[d]\\ 
V\ar[r,"{[}-1\ 0{]}"swap]&B\ar[r,"j"swap]&C\ar[r,"\begin{bmatrix}0\\1\end{bmatrix}"swap]\ar[d,"\begin{bmatrix}0\\0\\1\end{bmatrix}",swap]&\Sigma V\ar[d]\\ 
& &N\ar[r,equal]\ar[d]&N\ar[d]\\
&&\Sigma U\ar[r]&\Sigma^{2}A
\end{tikzcd}
\end{equation}
where we omit the symbol $\wedge$ for representable dg modules and where 
\[
s=\begin{bmatrix}0&0\\-1&0\end{bmatrix} \ko U=\Cone(f)\, \ko V=\Sigma^{-1}\Cone(j) \mbox{ and }
N=\Cone([-h,\;j]).
\]
Let $\mathcal N_{dg}$ be the full dg subcategory of $\pretr(\A)$ consisting of the objects in $\N$. 
Let $F$ be the canonical morphism from $\A$ to the Drinfeld dg quotient $\pretr\A/\mathcal N_{dg}$
(recall that we assume the morphism complexes of $\A$ to be cofibrant hence flat). We consider
$F$ as a morphism in the category $\Hqe$ obtained from the category of small dg categories
by localizing at the class of quasi-equivalences.

\begin{lemma}\label{univer} The morphism
$F:\A\rightarrow \pretr(\A)/\mathcal N_{dg}$ is the universal exact morphism (cf.~Definition~\ref{def:exactmorphism}) from $\A$ to a pretriangulated dg category.
\end{lemma}
\begin{proof}Put $\D^b(\A)=H^0(\pretr(A)/\mathcal N_{dg})$.
We first show that $F$ is exact.
For each conflation $X$ as follows
\[
\begin{tikzcd}
A\ar[r,"f"]\ar[rr,"h"swap,bend right=8ex]&B\ar[r,"j"]&C
\end{tikzcd},
\]
let $ [1,0]/[-h,j]$ be the following roof
\[
\begin{tikzcd}
&U\ar[ld,Rightarrow,"{[}-h{,}j{]}"swap]\ar[rd,"{[}1{,}0{]}"]&\\
C&&\Sigma A
\end{tikzcd}
\]
where $U$ is the cone of $f:A\rightarrow B$.
We have the following triangle in $\D^b(\A):$
\[
\begin{tikzcd}
A\ar[r,"f/1"]&B\ar[r,"j/1"]&C\ \ar[r,"{[1}{,}0{]}{/}{[}-h{,}j{]}"]&\ \Sigma A
\end{tikzcd}.
\]
So $F$ is exact.

We show that $F$ is universal.
Let $G:\A\rightarrow \B$ be an exact morphism in $\mathrm{Hqe}$ from $\A$ to a pretriangulated dg category $\B$.
By the universal property of $\A\rightarrow \pretr(\A)$, there exists a unique morphism from $\pretr(\A)$ to $\B$ in $\mathrm{Hqe}$ which extends $G$.
Since $G$ is an exact morphism, $H^0(G)$ sends the totalization $N$ of $X$ to a zero object in $H^0(\B)$.
By the universal property of Drinfeld dg quotient \cite{Tabuada10}, there exists a unique morphism in 
$\Hqe$ from $\pretr(A)/\mathcal N_{dg}$ to $\B$.
\end{proof}
Denote by $\mathcal H$ the heart of the canonical t-structure $(\D(\A)^{\leq 0},\D(\A)^{\geq 0})$, cf.~Lemma~\ref{lemma:tstructure}. We denote by $H^0:\D(\A)\rightarrow \mathcal H$ the associated homological functor.

\begin{lemma}\label{leftexactsequence}
A 3-term h-complex
\[
\begin{tikzcd}
A\ar[r,"f"]\ar[rr,bend right=8ex,"h"swap]&B\ar[r,"j"]&C
\end{tikzcd}
\]
is exact if and only if its totalization $N$ lies in $\mathcal H$. In this case, the totalization $N$ is the cokernel of $H^0(B^{\wedge})\xrightarrow{H^0(j^{\wedge})}H^0(C^{\wedge})$ in $\mathcal H$.
\end{lemma}
\begin{proof}
Consider the diagram (\ref{TR4}).
For $A'\in\A$ and $i\leq 0$, we have 
\[
\Hom_{\D(\A)}(A'^{\wedge},\Sigma^{i}A)\xrightarrow{\sim}\Hom_{\D(\A)}(A'^{\wedge},\Sigma^{i}V)
\]
and so we have $\Hom_{\D(\A)}(A'^{\wedge},\Sigma^{i-1}N)=0$ and hence $N\in \D^{\geq 0}$. 

Since representable dg modules lie in $\D^{\leq 0}$, it is clear from the construction that $N\in \D^{\leq 0}$. 
Therefore we have $N\in\mathcal H$.

From the rightmost column of the diagram (\ref{TR4}), we have $H^1(V)\xrightarrow{\sim}N$. 

From the second row we have an exact sequence where we omit the symbol ${\wedge}$ for representable dg modules.
\[
\begin{tikzcd}H^0(B)\ar[r,"H^0(\overline{\jmath})"]&H^0(C)\ar[r]&H^1(V)\ar[r]&H^1(B)=0\end{tikzcd}.
\]
So the second statement follows.
\end{proof}
\begin{remark}
We have an exact sequence in $\mathcal H:$
\[
\begin{tikzcd}
H^0(A)\ar[r,"H^0(f)"]&H^0(B)\ar[r,"H^0(\overline\jmath)"]&H^0(C)\ar[r]&N\ar[r]&0
\end{tikzcd}
\]
where, with the notations of  (\ref{TR4}), the image of $H^0(j)$ is $H^0(U)$.
Using the extriangulated category $(H^0(\A),\mathbb E,\mathfrak s)$ defined in subsection \ref{canonicalstructure}, this is equivalent to saying that $N$ is the {\em defect} (cf.~Definition \ref{defect}) of the conflation 
\[
A\xrightarrow{\overline{f}}B\xrightarrow{\overline \jmath}C 
\]
in $H^0(\A)$.
\end{remark}
\begin{definition}\label{def:defective}
Let $\tilde{\mathcal N}\subseteq \mathcal H$ be the essential image of $\mathcal N$ under the homological functor $H^0:\D(\A)\rightarrow \mathcal H$.
An object in $\mathcal H$ is {\em defective}, if it is the cokernel of a map 
\[
H^0(B)\xrightarrow{H^0(\overline{\jmath})}H^0(C)\ko
\]
where $j$ is a deflation, or equivalently if it is the totalization of some conflation.
\end{definition}
Let $X_i$, $i=1,2$, be homotopy short exact sequences
\[
\begin{tikzcd}
A_i\ar[r,"f_i"]\ar[rr,"h_i"swap,bend right=8ex]&B_i\ar[r,"j_i"]&C_i
\end{tikzcd}
\]
and $\alpha: X_1 \rightarrow X_2$ a morphism of ${\mathcal H}_{3t}(\A)$. Let
$X_3$ be the object 
\[
\begin{tikzcd}
A_1 \ar[r] \ar[d] \ar[rd] & B_1 \ar[d]\\
A_2 \ar[r] & B_2
\end{tikzcd}
\]
of $\rep(\ol{\Sq},\A)$ obtained by restricting a representative of $\alpha$.
Notice that by Lemma~\ref{squareepivalence}, up to (non canonical) isomorphism,
the object $X_3$ is independent of the choice of representative.
Suppose that $\alpha$ restricts to an isomorphism $C_1\rightarrow C_2$ in $H^0(\A)$. 
Then $X_3$ is homotopy cartesian by Proposition \ref{push}. 
Let $N_i$ be the totalization of $X_i$ for $i=1$, $2$, $3$. 

\begin{lemma}\label{exact}
\begin{itemize}
\item[a)] There exists a triangle in $\D(\A)$
\[
N_3\rightarrow N_1\rightarrow N_2\rightarrow \Sigma N_3
\]
\item[b)] If $X_1$ is a conflation, the triangle in a) induces a short exact sequence in $\mathcal H$
\[
0\rightarrow N_3\rightarrow N_1\xrightarrow{g} N_2\rightarrow 0
\]
where $g$ is the morphism induced by $\alpha$.
\end{itemize}
\end{lemma}
Notice that if $X_1$ is a conflation, then $X_2$ is a conflation by the dual of Proposition~\ref{property}~b) 
and hence $X_3$ is homotopy bicartesian by Lemma \ref{bothconflation}. 
\begin{proof} a)
Suppose $\alpha$ is given by the following diagram in $\A$
\[
\begin{tikzcd}
A_1\ar[r,"f_1"]\ar[rd,"s_1"{red,swap},red]\ar[rrd,"t"{blue,near start},blue]\ar[rr,bend left=8ex,"h_1"]\ar[d,"a"swap]&B_1\ar[rd,"s_2"red,red]\ar[r,"j_1"]\ar[d,"b"{swap,near end}]&C_1\ar[d,"c"]\\
A_2\ar[rr,bend right=8ex,"h_2"swap]\ar[r,"f_2"swap]&B_2\ar[r,"j_2"swap]&C_2
\end{tikzcd}.
\]
Since $\overline{c}$ is an isomorphism in $H^0(\A)$, we may replace $C_1$ by $C_2$, $j_1:B_1\rightarrow C_1$ by $c\circ j_1$ and $h_1$ by $c\circ h_1$. 
Thus we may assume $C_1=C_2$ and $c=\Id_{C_1}$ in $Z^0(\A)$.
Then we have the following diagram in $Z^0(\A)$
\[
\begin{tikzcd}
A_1\ar[rrd,"-t"blue,blue]\ar[rr,bend left=8ex,"s_1"]\ar[r,"\phi"]\ar[d,equal]& B_1\oplus A_2\ar[rd,"\psi"{red},red]\ar[r,"{[}b{,}-f_2{]}"]\ar[d,"{[}1{,}0{]}"{swap,near end}]& B_2\ar[d,"j_2"]\\
A_1\ar[rrd,"t"{swap,near end,blue},blue]\ar[rd,"s_1"{red,swap},red]\ar[r,"f_1"]\ar[d,"a"swap]& B_1\ar[rd,"s_2"{red},red]\ar[r,"j_1"]\ar[d,"b"{swap,near end}]& C_1\ar[d,equal]\\
A_2\ar[rr,bend right=8ex,"h_2"swap]\ar[r,"f_2"swap]&B_2\ar[r,"j_2"swap]&C_1
\end{tikzcd}
\]
where we have omitted a zero diagonal map and the homotopy $h_1:A_1\rightarrow C_1$ and where $\phi=\begin{bmatrix}f_1\\a\end{bmatrix}$, $\psi={[}-s_2{,}-h_2{]}$.
It induces a canonical diagram in $\C(\A)$
\[
\begin{tikzcd}
A_2\ar[r,"\begin{bmatrix}0\\0\\1\end{bmatrix}"]\ar[d,"\begin{bmatrix}-f_2\\-h_2\end{bmatrix}"swap]&\Cone(\phi)\ar[rd,"s'"{red},red]\ar[r,"\begin{bmatrix}1\ 0\ 0\\0\ 1\ 0\end{bmatrix}"]\ar[d,"u"]&\Cone(f_1)\ar[d,"{[}-h_1{,}j_1{]}"]\\
\Sigma^{-1}\Cone(j_2)\ar[r,"{[}1{,}0{]}"swap]&B_2\ar[r,"j_2"swap]&C_1
\end{tikzcd}
\]
where $u={[}-s_1{,}b{,}-f_2{]}$, $s'=[t,s_2,h_2]$ and where the first row is a graded-split short exact sequence in $\C_{dg}(\A)$. 
By the $3\times 3$-lemma, we have a triangle in $\D(\A)$ 
\[
N_3\rightarrow N_1\rightarrow N_2\rightarrow \Sigma N_3
\]
which gives the short exact sequence of b) in $\mathcal{H}$ when $X_1$ is a conflation.
\end{proof}
\subsection{The canonical t-structure on the category $\mathcal N$}

In the rest of this section, we keep the assumption that $\A$ is connective. In this case, we have 

\begin{itemize}
\item The left aisle $\D(\A)^{\leq 0}$ is formed by those dg modules whose cohomology is concentrated in non-positive degrees.
\item The heart $\mathcal H$ of the t-structure $(\D(\A)^{\leq 0},\D(\A)^{\geq 0})$ is equivalent to $\Mod H^0(\A)$ and the cohomological functor $H^0$ sends a dg module $M$ to $H^0(M):A\mapsto H^0(M(A))$.
\item In particular $H^0$ is fully faithful on the full subcategory of quasi-representable dg modules and $H^0(M)$ is projective in the heart for each quasi-representable dg module $M$.
\end{itemize}

Recall that a full subcategory of an abelian category is {\em wide} if it is stable under extensions,
kernels and cokernels. In particular, each wide subcategory is abelian and its inclusion is fully 
faithful and fully exact.
\begin{lemma}\label{wide}The full subcategory of $\mathcal H$, consisting of defective objects (cf.~Definition~\ref{def:defective}), is a wide subcategoy.
\end{lemma}

\begin{corollary}\label{t}An object in $\mathcal H$ lies in $\tilde{\mathcal N}$ (cf.~Definition~\ref{def:defective}) if and only if it is defective. The t-structure $(\D(\A)^{\leq 0},\D(\A)^{\geq  0})$ on $\D(\A)$ restricts to a bounded t-structure on $\N$.
\end{corollary}

\begin{proof}[Proof of the Corollary] Let $\mathcal N'$ be the full subcategory of $\N$ consisting of objects $N$ such that $H^i(N)$ is defective for each $i\in\mathbb Z$ and only finitely many of them are nonzero. 
By Lemma \ref{wide}, this is a triangulated subcategory of $\mathcal N$. It certainly contains totalizations of conflations and thus is equal to $\N$. 
Therefore the objects $N$ in $\N$ have the property that $H^i(N)$ is defective for each $i\in \mathbb Z$ and only finitely many of them is nonzero.
\end{proof}
\begin{proof}[Proof of Lemma \ref{wide}]
Denote by $\overline{\mathcal N}$ the full subcategory of $\mathcal H$ consisting of defective objects. Consider a 
three-term complex
\begin{align}
P\rightarrow Q\rightarrow R \label{short}
\end{align}
in $\Mod H^0(\A)$ which is exact in the middle. 

We first show that $\overline{\mathcal N}$ is closed under extensions.
Assume the sequence \ref{short} is a short exact sequence with $P$ and $R$ defective.
Suppose $P$ and $R$ are totalizations of $X_1$ and $X_2$ respectively, where $X_i$, $i=1,2$, is of the form
\[
\begin{tikzcd}
A_i\ar[r,"f_i"]\ar[rr,"h_i"swap,bend right=8ex]&B_i\ar[r,"j_i"]&C_i
\end{tikzcd}.
\]
Then $P$ is the cokernel of $H^0B_1\xrightarrow{H^0(\overline{\jmath}_1)} H^0C_1$ and $R$ is the cokernel of $H^0B_2\xrightarrow{H^0(\overline{\jmath}_2)} H^0C_2$.
By the horseshoe lemma, $Q$ is the cokernel of 
\[
H^0(B_1\oplus B_2)\xrightarrow{\begin{bmatrix}\overline{\jmath}_1\ \ \overline{h}\\0\ \ \overline{\jmath}_2\end{bmatrix}}H^0(C_1\oplus C_2)
\]
in $\mathcal H$, where $h:B_2\rightarrow C_1$ is a morphism in $Z^0(\A)$.

Since the morphism 
\[
B_1\oplus B_2\xrightarrow{\begin{bmatrix}{j}_1\ \ h\\0\ \ {j}_2\end{bmatrix}}C_1\oplus C_2
\]
 can be written as the following composition
\[
\begin{tikzcd}
B_1\oplus B_2\ar[r,"{\begin{bmatrix}j_1\ \ 0\\0\ \ 1\end{bmatrix}}"] &C_1\oplus B_2\ar[r,"{\begin{bmatrix}1\ \ h\\0\ \ 1\end{bmatrix}}"]&C_1\oplus B_2\ar[r,"{\begin{bmatrix}1\ \ 0\\0\ \ j_2\end{bmatrix}}"]&C_1\oplus C_2 \ko
\end{tikzcd}
\]
it is a deflation by the axioms ${\Ex0}$ and ${\Ex2}$. Therefore $Q$ is also defective.

We show that $\overline{\mathcal N}$ is closed under kernels.
Assume the sequence \ref{short} is a left short exact sequence in $\mathcal H$ with $Q$ and $R$ defective.

Suppose $Q$ and $R$ are totalizations of conflations $X_1$ and $X_2$ respectively, as described in $(1)$.

The morphism $Q\rightarrow R$ induces a commutative diagram in $\mathcal H$
\[
\begin{tikzcd}
H^0B_1\ar[r,"H^0(\overline{\jmath_1})"]\ar[d,dotted]&H^0C_1\ar[r]\ar[d,dotted,"H^0(\overline{c})"]&Q\ar[d]\ar[r]&0\\
H^0B_2\ar[r,"H^0(\overline{\jmath_2})"swap]&H^0C_2\ar[r]&R\ar[r]&0
\end{tikzcd}
\]
which in turn gives a commutative diagram in $\D(\A)$. 
This diagram lifts to a morphism $\mu: j_1\rightarrow j_2$ in $H^0(\Mor(\A))$ which induces a morphism $\alpha: X_1\rightarrow X_2$ in $\mathcal H_{3t}(\A)$. 
Let $\overline{a}:A_1\rightarrow A_2$ be the restriction of $\alpha$ to $A_1$. 
By Lemma \ref{fact}, the morphism $\alpha$ can be written as a composition $X_1\xrightarrow{\beta} {X}_3\xrightarrow{\gamma} X_2$ where $X_3$ is a conflation with ends $A_2$ and $C_1$, and the morphism $\beta$ restricts to $\overline a:A_1\rightarrow A_2$ and $\Id_{C_1}$, and the morphism $\gamma$ restricts to $\Id_{A_2}$ and $\overline{c}:C_1\rightarrow C_2$.
Now we apply Lemma \ref{exact} and its dual to the morphisms $X_1\rightarrow {X_3}$ and ${X_3}\rightarrow X_2$ respectively, and get short exact sequences in $\mathcal H$
\[
0\rightarrow N'\rightarrow Q\rightarrow N''\rightarrow0 
\]  
\[
0\rightarrow N''\rightarrow R\rightarrow N'''\rightarrow 0
\]
where, by Axiom ${\Ex2}$, $N'$ is the totalization of a conflation. 
This shows that $P$ is isomorphic to  $N'$ and thus is defective.

Now we show that $\overline{\mathcal N}$ is closed under cokernels.
Assume we have an exact sequence in $\overline{\mathcal N}$
\[
Q\rightarrow R\rightarrow S\rightarrow 0
\]
where $Q$ and $R$ are defective.
Using the notation in $(2)$, we find that $S$ is isomorphic to $N'''$. 
By Proposition \ref{property} (b), $N'''$ is the totalization of a conflation. 
This shows that $S$ is also defective.
\end{proof}
\begin{remark}
The proof of Lemma \ref{wide} only uses the axioms ${\Ex0}$, ${\Ex1}$ and ${\Ex2}$ and Proposition \ref{property} $(b)$.
So a similar result holds for a left exact dg category, cf.~Definition~\ref{def:leftstructure}.
\end{remark}
\begin{example}Suppose that $\A$ is an idempotent complete Quillen exact category endowed with the associated exact dg structure. We claim that then $\N$ is the category $\mathrm{ Ac}^b(\A)$ of acyclic complexes over $\A$, cf. \cite{Neeman21}, and that the t-structure given in Corollary \ref{t} is the one described in [loc.~cit.].
Indeed, as usual, we identify $\pretr(\A)$ with $\C^b_{dg}(\A)$. Then the totalization of a conflation which corresponds to the following sequence in $\A$
\[
\begin{tikzcd}
0\ar[r]&A\ar[r,"f"]&B\ar[r,"g"]&C\ar[r]&0
\end{tikzcd}
\]
is identified with the complex 
\[
\begin{tikzcd}
\cdots\ar[r]&0\ar[r]&A\ar[r,"f"]&B\ar[r,"g"]&C\ar[r]&0\ar[r]&\cdots
\end{tikzcd}
\]
where $A$ is in degree 0. 
By definition, $\mathcal N$ is the full triangulated subcategory of $\tr(\A)$ generated by the totalizations of conflations. 
So $\mathcal N$ is identified with $\rm{Ac}^b(\A)$.
By Corollary \ref{t}, we have
\[
\mathcal N^{\leq 0}=\{N\in \N{|} H^iN=0 \text{ for i $\geq 1$}\}
\]
and 
\[
\mathcal N^{\geq 0}=\{N\in \N{|} H^iN=0 \text{ for i $\leq -1$}\}
\]
It is clear that $\rm{Ac}^b(\A)^{\leq 2}$ is contained in $\mathcal N^{\leq 0}$ and $\rm{Ac}^b(\A)^{\geq 2}$ is contained in $\mathcal N^{\geq 0}$. Thus the two t-structures coincide.
\end{example}
The following lemma is standard.
\begin{lemma}[\cite{IyamaYang20}, Lemma 3.3]\label{lem:tstructureff}
Let $T$ be a triangulated category with a thick subcategory $\mathcal S$. 
Suppose that $\mathcal S$ admits a torsion pair (\cite[Definition 2.2]{IyamaYoshino08}) $\mathcal S=(\X, \Y)$.
Put $\P=\X\cap \Sigma^{-1}\Y$.
Then for any $A\in {^{\perp}\Y}$ and $B\in \Sigma^{-1}\X^{\perp}$, where the symbol $^{\perp}$ stands for the Hom orthogonal, the canonical map 
\[
\Hom_{\T/[\P]}(A,B)\rightarrow \Hom_{\T/\mathcal S}(A,B)
\]
is a bijection.
\end{lemma}
In the above lemma, we have $\P=0$ when $\mathcal S=(\X,\Y)$ is in particular a t-structure.
\begin{proof}[Proof of Theorem~\ref{main}]
By construction, the dg category $\D^b_{dg}(\A)$ is the dg quotient \linebreak $\pretr(\A)/\mathcal N_{dg}$ and $F:\A\rightarrow \D^b_{dg}(\A)$ is the canonical morphism in $\Hqe$. 
This morphism is universally exact by Lemma \ref{univer}. 
Put $\D^b(\A)=H^0(\D^b_{dg}(\A))$. 
For simplicity, for objects $A\in H^0(\A)$, we keep the same notation $A$ for the corresponding object in $\D^b(\A)$.
Note that for any object $A\in\A$ and any totalization $N$ of a conflation in $\A$, we have $\Hom_{\tr(\A)}(A, \Sigma^{i}N)=0$ and $\Hom_{\tr(\A)}(N,\Sigma^{i+2} A)=0$ for $i\leq -1$. 
It follows from Corollary~\ref{t} and Lemma~\ref{lem:tstructureff} that for objects $A$ and $C$ in $H^0(\A)$, we have
\[
\Hom_{\tr(\A)}(C,\Sigma^{i}A)\iso \Hom_{\D^b(\A)}(C,\Sigma^{i}A)
\]
for $i\leq 0$.

Consider the following triangle in $\D^b(\A)$:
\[
\begin{tikzcd}
A\ar[r] &E\ar[r]&C\ar[r,"u{/}t"]&\Sigma A
\end{tikzcd}
\]
where $u/t$ is the following roof
\[
\begin{tikzcd}
&W\ar[ld,Rightarrow,"t"swap]\ar[rd,"u"]&\\
C&&\Sigma A
\end{tikzcd}
\]
Similarly as above, we may assume that the mapping cone $N'=\Cone(t)$ belongs to $\mathcal N^{\leq 0}$. The composition 
\[
\begin{tikzcd}
\Sigma^{-1}N'\ar[r] &W\ar[r,"u"]& \Sigma A
\end{tikzcd}
\]
factors through $\Sigma^{-1}H^0(N')$. 
Therefore we may assume $N'$ lies in the heart of the t-structure on $\mathcal N$.
So $N'$ is the totalization of a conflation $X'$ of the form
\[
\begin{tikzcd}
A'\ar[r,"f'"]\ar[rr,bend right=8ex,"h'"swap]&B'\ar[r,"j'"]&C'
\end{tikzcd}
\]
Then the morphisms $C\rightarrow N'$ and $\Sigma^{-1}N'\rightarrow \Sigma A$ give rise to morphisms $c:C\rightarrow C'$ and $a:A'\rightarrow A$ in $H^0(\A)$.

Put $[X'']=c^*[X']$ and $[X''']=a_*[X'']\in \mathbb E(C,A)$. 
Let $N'',N'''$ be the totalizations of $X''$ and $X'''$ respectively. 
Consider the corresponding diagrams \ref{TR4}  for these conflations.
We have the following commutative diagram in $\tr(\A)$:
\[
\begin{tikzcd}
&C\ar[rr]\ar[ld]&&N''\ar[ld]\ar[r]\ar[dd]&N'''\ar[dd]\\
C'\ar[rr]&&N'\ar[rd]&&\\
&&&\Sigma^2 A'\ar[r]&\Sigma^2 A
\end{tikzcd}
\]
So the roof $u/t$ is equivalent to the roof 
\[
\begin{tikzcd} & U'''\ar[ld,Rightarrow]\ar[rd]\\
C&&\Sigma A
\end{tikzcd}
\]
which is given by the conflation $X'''$
\[
\begin{tikzcd}
A\ar[r,"f'''"]\ar[rr,"h'''"swap,bend right=8ex]&B'''\ar[r,"j'''"]&C
\end{tikzcd}
\]
So $E$ is isomorphic to $B'''$ and the image of $H^0(\A)$ under the functor $H^0(F)$ is an extension-closed subcategory of $\D^b(\A)$. 
This proves 1).

It remains to show that $H^0(F)$ induces natural isomorphisms $\mathbb E(C,A)\xrightarrow{\sim} \Ext^1_{\D^b(\A)}(C,A)$ for $C,A\in H^0(\A)$.
First we observe that if $[X]=[X']\in \mathbb E(C,A)$,  then $X$ and $X'$ give equivalent roofs. So we have a well-defined map 
\[
\mathbb E(C,A)\rightarrow \Ext^1_{\D^b(\A)}(C,A).
\]
From the above proof, this map is a surjection.
From the definition of the bifunctor $\mathbb E(-,-)$, it is clear that this map is natural in the variables $C$ and $A$.
From the definition of the addition operation on $\mathbb E(C,A)$, it is clear that this map is additive.
Now suppose that $[X]$ is an element in $\mathbb E(C,A)$ which is sent to zero in $\Ext^1_{\D^b(\A)}(C,A)$.  
Suppose that $X$ is of the form
\[
\begin{tikzcd}
A\ar[r,"f"]\ar[rr,bend right=8ex,"h"swap]&B\ar[r,"j"]&C.
\end{tikzcd}
\]
Then the morphism $\overline{f}/1:A\rightarrow B$ in $\D^b(\A)$ is a split monomorphism .
Since $H^0(F)$ is fully faithful, the morphism $\overline{f}:A\rightarrow B$ is a split monomorphism in $H^0(\A)$. By $2)\Rightarrow 1)$ in Proposition \ref{split} , $[X]$ is zero in $\mathbb E(C,A)$. This shows 2).
\end{proof}
\subsection{The canonical extriangulated structure on $H^0(\A)$}\label{canonicalstructure}
We are now ready to show that $H^0(\A)$ carries a canonical extriangulated structure as announced
in Remark~\ref{rk:extriangulated-structure}.
We define the realisation $\mathfrak s$ of the bifunctor $\mathbb E$ as follows: for an element $\delta=[X]\in \mathbb E(C,A)$, where $X$ is as follows
\[
\begin{tikzcd}
A\ar[r,"f"]\ar[rr,bend right=8ex,"h"swap]&B\ar[r,"j"]&C
\end{tikzcd},
\]
put $\mathfrak s(\delta)=[A\xrightarrow{\overline{f}} B\xrightarrow{\overline{\jmath}} C]$. 
This is well-defined by Corollary \ref{middleterm}.
\begin{lemma}The map $\mathfrak s$ defined above is an additive realization of $\mathbb E$.
\end{lemma}
\begin{proof}We first show that $\mathfrak s$ is a realization.
Suppose we have $\delta=[X]\in \mathbb E(C,A)$, $\delta'=[X']\in \mathbb E(C',A')$, $\overline{a}:A\rightarrow A'$ and $\overline{c}:C\rightarrow C'$ such that $\overline{a}_*[X]=\overline{c}^*[X']=[X'']\in \mathbb E(C,A')$. Put $\mathfrak s(\delta)=[A\xrightarrow{\overline{f}} B\xrightarrow{\overline{\jmath}} C]$ and $\mathfrak s(\delta')=[A'\xrightarrow{\overline{f'}} B'\xrightarrow{\overline{\jmath'}} C']$.

We have a composition of morphism of conflations $X\rightarrow X''\rightarrow X'$ which restricts to $\overline{a}:A\rightarrow A'$ and $\overline{c}:C\rightarrow C'$. 
This morphism then gives a morphism $\overline{b}:B\rightarrow B'$ which makes the following diagram in $H^0(\A)$ commutative
\[
\begin{tikzcd}
A\ar[r,"\overline{f}"]\ar[d,"\overline{a}"swap]& B\ar[r,"\overline{\jmath}"]\ar[d,"\overline{b}"] &C\ar[d,"\overline{c}"]\\
A'\ar[r,"\overline{f'}"swap]&B'\ar[r,"\overline{\jmath'}"swap]&C'
\end{tikzcd}.
\]

We now show that $\mathfrak s$ is additive.
By definition, for $0\in\mathbb E(C,A)$, we have 
\[
\mathfrak s(0)=[A\xrightarrow{\begin{bmatrix}0\\1\end{bmatrix}}A\oplus C\xrightarrow{[0{,}1]}C].
\]

By Lemma \ref{sum}, we have $\mathfrak s(\delta\oplus \delta')=\mathfrak s(\delta)\oplus \mathfrak s(\delta')$.
\end{proof}
\begin{remark}\label{truncationextriangulated}
By Remark \ref{truncationexactdgstructure}, the exact dg structures on $\A$ are the same as 
the exact dg structures on $\tau_{\leq 0}\A$. 
Note that we have $H^0(\tau_{\leq 0}\A)=H^0(\A)$. 
By definition, the triple $(H^0(\A), \mathbb E,\mathfrak s)$ remains unchanged when we replace $\A$ by $\tau_{\leq 0}\A$. 
\end{remark}
By the above remark, the following theorem is a direct consequence of the fully exact embedding of $H^0(\A)$ into $\D^b(\A)$ given by Theorem \ref{main}.  One can also prove it as follows: By Theorem~\ref{nerve}, 
the dg nerve $N_{dg}(\A)$ carries a canonical exact structure in the sense of Barwick; this induces an 
extriangulated structure on its homotopy category $h(N_{dg}(\A))$ as shown by Nakaoka--Palu \cite[Theorem 1.1]{NakaokaPalu20} and we know that
the homotopy category $h(N_{dg}(\A))$ is isomorphic to $H^0(\A)$.
Below, we give a direct proof of the theorem.
\begin{theorem}The triple $(H^0(\A),\mathbb E,\mathfrak s)$ forms an extriangulated category.
\end{theorem}
\begin{proof}

$(\mathrm{ET3})$ Let $\delta=[X]\in \mathbb E(C,A)$, $\delta'=[X']\in \mathbb E(C',A')$. 
Suppose they are realized as $\mathfrak s(\delta)=[A\xrightarrow{\overline{f}} B\xrightarrow{\overline{\jmath}} C]$ and $\mathfrak s(\delta')=[A'\xrightarrow{\overline{f'}} B'\xrightarrow{\overline{\jmath'}} C']$.
For any commutative square in $H^0(\A)$
\[
\begin{tikzcd}
A\ar[r,"\overline{f}"]\ar[d,"a"swap]&B\ar[r,"\overline{\jmath}"]\ar[d,"b"]&C\\
A'\ar[r,"\overline{f'}"swap]&B'\ar[r,"\overline{\jmath'}"swap]&C'
\end{tikzcd},
\]
the left hand square can be lifted (in a non-unique way) to a morphism from $f:A\rightarrow B$ to $f':A'\rightarrow B'$ in $H^0(\Mor(\A))$. 
It then induces a morphism $X\rightarrow X'$ in $\mathcal H_{3t}(\A)$. 
The restriction gives a morphism $c:C\rightarrow C'$ in $H^0(\A)$.
By  Lemma \ref{fact},  we have $a_*\delta=c^*\delta'$. 
This proves (ET3). 
Axiom \text{(ET3)$^{op}$} is proved dually.

$(\mathrm{ET4})$ Consider $\delta=[X]\in \mathbb E(D,A)$ and $\delta'=[X']\in\mathbb E(F,B)$ which are realized respectively by $A\xrightarrow{\overline{f}}B\xrightarrow{\overline{\jmath}}D $ and $B\xrightarrow{\overline{f'}}C\xrightarrow{\overline{\jmath'}}F$. 

Since a composition of inflations is an inflation, the object $f'f:A\rightarrow C$ is an inflation and there exists a conflation $X''$ such that $[X'']$ is realized by $A\xrightarrow{\overline{f'f}}C\xrightarrow{\overline{h}}E$.
The commutative diagram 
\[
\begin{tikzcd}
A\ar[r,"\overline{f}"]\ar[d,equal]&B\ar[d,"\overline{f'}"]\\
A\ar[r,"\overline{f'f}"swap]&C
\end{tikzcd}
\]
in $\D(\A)$ can be lifted to a morphism in $H^0(\Mor(\A))$ from $f:A\rightarrow B$ to $f'f:A\rightarrow C$. 
This morphism induces a morphism from $X$ to $X''$ which restricts to a homotopy bicartesian square $Y_1$ whose image in $\Fun(\mathrm{Sq},\D(\A))$ is isomorphic to
\[
\begin{tikzcd}
B\ar[r,"\overline{\jmath}"]\ar[d,"\overline{f'}"swap]&D\ar[d,"\overline{d}"]\\
C\ar[r,"\overline{h}"swap]&E
\end{tikzcd}\;.
\]
Let $\overline{\jmath}_* [X']=[X''']\in \mathbb E(F,B)$. 
We have a morphism $X'\rightarrow X'''$ which, by restriction to $f':B\rightarrow C$, yields a homotopy bicartesian square $Y_2$. 
Both $Y_1$ and $Y_2$ are homotopy pushouts of $S:$
\[
\begin{tikzcd}
B\ar[r,"j"]\ar[d,"f'"swap]&D\\
C&
\end{tikzcd}.
\] 
So there exists an isomorphism $Y_1\rightarrow Y_2$ in $\D(\mathrm{Sq})$ which is compatible with the isomorphisms $R(Y_i)\rightarrow S$ where $R:\D(\mathrm{Sq})\rightarrow \D(\mathrm{Sp})$ is the restriction functor.
So $[X''']$ can be realized by $D\xrightarrow{\overline{d}}E\xrightarrow{\overline{e}}F$ such that $\overline{\jmath'}=\overline{e}\overline{h}$. This shows (ET4). Axiom \text(ET4$^{op}$) can be shown dually.
\end{proof}

\begin{proposition-definition}\label{algebraic}Let $\C$ be an extriangulated category. The following statements are equivalent:
\begin{itemize}
\item[1)]$\C$ is equivalent, as an extriangulated category, to a full extension-closed subcategory of an algebraic triangulated category.
\item[2)]$\C$ is equivalent, as an extriangulated category, to $\B/(\mathcal P_0)$ for a Quillen exact category $\B$ and a class $\mathcal P_0$ of projective-injective objects.
\item[3)]$\C$ is equivalent, as an extriangulated category, to $H^0(\A)$ for an exact dg category $\A$.
\end{itemize}
If one of the above equivalent conditions holds, then $\C$ is called an {\em algebraic extriangulated category}.
\end{proposition-definition}
\begin{proof}
Let us prove that 1) implies 2). Suppose the extriangulated category $(\C,\mathbb E,\mathfrak s)$ is equivalent to a full extension-closed subcategory of $\underline{\E}$ where $\E$ is a Quillen exact category. 
Let $F:\C\rightarrow \underline{\E}$ be the fully exact functor which is an equivalence onto its essential image.
Let $\pi:\E\rightarrow \underline{\E}$ be the canonical quotient functor.
We form the following diagram
\[
\begin{tikzcd}
\B\ar[r,dashed,hook]\ar[d,dashed,two heads]&\E\ar[d,two heads,"\pi"]\\
\C\ar[r,hook,"F"swap]&\underline{\E}\mathrlap{\;.}
\end{tikzcd}
\]
Let $\B$ be the full subcategory of $\E$ consisting of objects $X$ such that $\pi(X)$ belongs to the essential image of $F$.
The category $\B$ is extension-closed in $\E$ since the essential image of $F$ is closed under extension and $\pi$ is a fully exact functor. 
Note also that $\B$ contains all projective-injective objects in $\E$.
So $\B$ is canonically exact and the full subcategory
\[
\P_0=\{P\in\B| \text{$P$ is projective-injective in $\E$}\}
\]
consists of projective-injective objects of $\B$ (but perhaps not all).
By \cite[Proposition 3.30]{NakaokaPalu19}, $\B/[\P_0]$ has the structure of an extriangulated category, induced from that of $\B$.
The functor $F$ induces a canonical equivalence of categories 
\[
G:\C\iso \B/[\P_0],\;\; A\mapsto F(A)
\]
which, by definition, is a fully exact functor.

Let us prove that conversely, condition 2) implies 1). Suppose $\C$ is equivalent to $\B/[\P_0]$ as extriangulated categories.
Then we have a fully exact fully faithful functor
\[
\C\iso\B/[\P_0]\hookrightarrow \D^b(\B)/\mathrm{thick}(\P_0).
\]
The triangle quotient $\D^b(\B)/\mathrm{thick}(\P_0)$ is an algebraic triangulated category and $\C$ identifies with an extension-closed subcategory of $\D^b(\B)/\mathrm{thick}(\P_0)$.

Let us show that 1) implies 3). Suppose $\C$ is equivalent to a full extension-closed subcategory of $H^0(\A')$ for a pretriangulated dg category $\A'$.
Let $F:\C\rightarrow H^0(\A')$ be the fully exact functor which is an equivalence onto its essential image.
Let $\A$ be the full dg subcategory of $\A'$ whose objects are those in the essential image of $F$.
Then $\A$ is an extension-closed subcategory of $\A'$, since $F$ is a fully exact functor. 
By Example \ref{exactdgextensionclosed}, $\A$ inherits an exact dg structure from that of $\A'$.
By the definition of the biadditive functor $\mathbb E$ associated with an exact dg structure, 
we see that the inclusion $H^0(\A)\hookrightarrow H^0(\A')$ is a fully exact fully faithful functor.
Thus $\C$ is equivalent to $H^0(\A)$ as an extriangulated category where $\A$ is an exact dg category.

Let us show that conversely, condition 3) implies 1).
Suppose $\C$ is equivalent to $H^0(\A)$ for an exact dg category $\A$.
By Remark \ref{truncationexactdgstructure}, we replace $\A$ with $\tau_{\leq 0}\A$ and assume that $\A$ is a connective dg category.
Then by Theorem \ref{main}, we have a fully exact fully embedding $H^0(\A)\hookrightarrow \D^b(\A)$ where $\D^b(\A)$ is an algebraic triangulated category.
\end{proof}
\begin{remark}Item 1) shows that the class of algebraic extriangulated categories is closed under forming extension-closed subcategories.
Item 2) shows that the class of algebraic extriangulated categories is closed under forming ideal quotients by projective-injective objects.
\end{remark}
In fact, the class of extriangulated categories can also be described using a seemingly more
general version of item 2) in Proposition-Definition~\ref{algebraic}, cf.~Corollary~\ref{algebraic2} for another proof based 
on a similar idea. This generalization is suggested by a result of \cite{FangGorskyPaluPlamondonPressland22}, 
stating that if $\C$ is an extriangulated category and $\I$ is an ideal generated by a family of morphisms 
$f:I\rightarrow P$ from injectives to projectives, then $\C/\I$ is naturally extriangulated.
In the context of algebraic extriangulated categories, we have the following result.
\begin{proposition}
Let $\C$ be an algebraic extriangulated category. 
Let $\{f:I\rightarrow P\}$ be a family of morphisms from injectives to projectives in $\C$. 
Denote by $[I\rightarrow P]$ the ideal generated by this family of morphisms.
Then $\C/[I\rightarrow P]$ is still an algebraic extriangulated category.
\end{proposition}
\begin{proof} By item 2) of the preceding Proposition-Definition, we may assume
that $\C$ is a Quillen exact category.
In \cite[Proposition 1.7]{DraxlerReitenSmaloSolberg99}, it is shown that given a full subcategory $\N\subset \E$ of a Quillen exact category $\E$, the class of conflations which make the objects in $\N$ projective-injective 
is a Quillen exact structure on $\E$.
We endow $\Mor(\C)$ with the componentwise exact structure and view $\{f:I\rightarrow P\}$ as a 
full subcategory of $\Mor(\C)$.
Let $\F$ be the exact category whose underlying category is $\Mor(\C)$ and whose exact structure is 
obtained from $\Mor(\C)$ by making the objects $f:I\rightarrow P$ projective-injective.
Then $\F$ is an exact category.
We have the inclusion functor 
\[
\theta:\C\rightarrow \F
\]
sending an object $X$ to $(\Id:X\rightarrow X)$. Clearly, it induces a functor
\[
\phi:\C/[I\rightarrow P]\rightarrow \F/(I\rightarrow P).
\]
We have the following observations
\begin{itemize}
\item $\theta$ is an exact functor. 
Indeed, a conflation in $\C$ 
\[
\begin{tikzcd}
0\ar[r]&X\ar[r]&Y\ar[r]&Z\ar[r]&0
\end{tikzcd}
\]
is sent, by $\theta$, to a conflation in $\Mor(\C)$
\[
\begin{tikzcd}[row sep=small]
&&&I\ar[dd,blue]\ar[rd,blue]\ar[ld,dashed,red]&&\\
0\ar[r]&X\ar[r]\ar[dd,equal]&Y\ar[rr]\ar[dd,equal]&&Z\ar[r]\ar[dd,equal]&0\\
&&&P\ar[rd,blue]\ar[ld,dashed,red]&&\\
0\ar[r]&X\ar[r]&Y\ar[rr]&&Z\ar[r]&0
\end{tikzcd}
\]
For any morphism from $I\rightarrow P$ to $\Id:Z\rightarrow Z$, the morphism $P\rightarrow Z$ factors through $Y$ since $P$ is projective in $\C$.
This shows that the morphism from $I\rightarrow P$ to $\Id:Z\rightarrow Z$ factors through $\Id:Y\rightarrow Y$.
Hence the objects $I\rightarrow P$ are projective with respect to the conflations in the essential image of $\theta$. 
Dually, one shows that they are injective. 
Hence the functor $\theta$ sends conflations to conflations in $\F$.
\item The induced functor $\phi$ is also fully faithful. This follows by definition.
\item The full embedding $\phi:\C/[I\rightarrow P]\rightarrow \F/(I\rightarrow P)$ is fully exact.
Let $\Id_X:X\rightarrow X$ and $\Id_Z:Z\rightarrow Z$ be two objects in the image of the functor $\phi$.
Consider a conflation in $\F$
\[
\begin{tikzcd}
0\ar[r]&X\ar[r]\ar[d,equal]&Y\ar[r]\ar[d,"\alpha"]&Z\ar[r]\ar[d,equal]&0\\
0\ar[r]&X\ar[r]&Y'\ar[r]&Z\ar[r]&0\mathrlap{.}
\end{tikzcd}
\]
The map $\alpha$ is an isomorphism in $\C$. 
Hence $\alpha:Y\rightarrow Y'$ is isomorphic to $\Id:Y\rightarrow Y$ and the above conflation is equivalent to a conflation in the essential image of $\theta$.
\item By definition, the category $\F/(I\rightarrow P)$ is an algebraic extriangulated category.
\end{itemize}
Therefore, the functor $\phi$ is a fully exact fully faithful functor from $\C/[I\rightarrow P]$ to an algebraic extriangulated category and hence $\C/[I\rightarrow P]$ is also algebraic.
\end{proof}
\subsection{Higher extensions}
Let $(\C,\mathbb E,\mathfrak s)$ be a small extriangulated category.
In their paper \cite{GorskyNakaokaPalu21}, for $n\geq 1$, Gorsky--Nakaoka--Palu defined the higher 
extension bimodule $\mathbb E^{n}(?,-)$ on $\C$ as the $n$th tensor power over $\C$ of $\mathbb E$.
Recall that a right $\C$-module $F:\C^{op}\rightarrow \Ab$ is {\em weakly effaceable} if for each $C\in \C$ and each element $x\in F(C)$, there exists a deflation $p:B\rightarrow C$ such that $F(p)(x)=0$, cf.~\cite[A.2]{Keller90} and also~\cite[Definition 2.6]{KaledinLowen15}.
Dually one has the notion of weakly effaceable left module.
A $\C$-$\C$-bimodule $G:\C^{op}\otimes \C\rightarrow \Ab$ is {\em weakly effaceable}, if for each 
$C\in\C$, the left module $G(C,-)$ and the right module $G(?,C)$ are both weakly effaceable.
By \cite[Corollary 3.5]{NakaokaPalu19}, for each $A\in \C$, the right $\C$-module 
$\mathbb E(?,A):\C^{op}\rightarrow \Ab$ is weakly effaceable. It follows that the tensor powers
$\mathbb E^n(?,-)$ are weakly effaceable.

In the rest of this section, we will omit the symbol $^{\wedge}$ in the representable dg module $A^{\wedge}$ for $A\in\A$.
Recall the extriangulated category $(H^0(\A),\mathbb E,\mathfrak s)$ associated with the exact dg category $\A$.
Clearly $(\Ext^n_{\D^b(\A)}(?,-))_{n\geq 0}$ together with the connecting morphisms given by triangles in $\D^b(\A)$ induced by $\mathbb E$-triangles in $H^0(\A)$, is a $\delta$-functor  (cf.~Definition~\ref{def:deltafunctor}) for $H^0(\A)$.
It clearly satisfies the assumptions~\ref{assumption:delta2} and~\ref{assumption:delta3}.
\begin{proposition}\label{higher}
Let $F:\A\rightarrow \D^{b}_{dg}(\A)$ be the universal embedding of $\A$ 
into a pretriangulated category. 
Then we have a canonical isomorphism of $\delta$-functors $\alpha:\mathbb E^n(?,-)\xrightarrow{\sim} \Ext^n_{\D^b(\A)}(?,-)$.\end{proposition}

\begin{proof}
Put $\C=H^0(\A)$. 
Let us show the $\C$-$\C$-bimodule $\Ext^n_{\D^b(\A)}(?,-)$ is effaceable for each $n\geq 1$.

By construction, we have $\D^b(\A)=\tr(\A)/\mathcal N$. 
Consider a morphism from $C$ to $\Sigma^{n}(A)$ in $\D^b(\A)$ given  by the roof 
\[
\begin{tikzcd}
&M\ar[Rightarrow,ld,"s",swap]\ar[rd,"w"]&\\
C&&\Sigma^n (A)
\end{tikzcd}
\]
where $s$ and $w$ are morphisms in $\tr(\A)$ and $N=\Cone(s)$ lies in $\N$.  
We may assume $N\in \N^{\leq 0}$. 
By Corollary \ref{t}, $N$ is an extension of its homology which are defective objects in $\mathcal H$. 
For defective objects $\tilde{N}$ we know that $\RHom_{\A}(\tilde{N}, A')$ is concentrated in degree $2$ for $A'\in \A$.  
So we may assume $N\in \mathcal N^{[-n+1,0]}$ and write $N=(\Sigma^{n-1}N_{n-1})*(\Sigma^{n-2}N_{n-2})*\cdots*N_{0}$ where $N_i$ is defective for each $i$. 
Suppose for each $i$, $N_{i}$ is the totalization of a conflation $X_i$ of the form
\[
\begin{tikzcd}
A_i\ar[r,"f_i"]\ar[rr,"h_i"swap,bend right=8ex]&B_i\ar[r,"j_i"]&C_i
\end{tikzcd}. 
\]
We have a canonical map $p:N_{n-2}\rightarrow\Sigma^{2}N_{n-1}$. 
Consider the following diagram in $\tr(\A)$
\[
\begin{tikzcd}
\Sigma A_{n-2}\ar[r]&\Sigma V_{n-2}\ar[r]&N_{n-2}\ar[r]\ar[d,"p"near start, swap]&\Sigma^2 A_{n-2}\ar[ld,dotted]\ar[lld,dotted]\\
\Sigma^2 U_{n-1}\ar[r]&\Sigma^2 C_{n-1}\ar[r]&\Sigma^2 N_{n-1}\ar[r]&\Sigma^3U_{n-1}
\end{tikzcd}
\]
Since $\A$ is connective, we have $\Hom_{\D(\A)}(V_{n-2},\Sigma N_{n-1})=0$ and $\Hom_{\D(\A)}(A_{n-2},\Sigma U_{n-1})=0$. 
Thus we have a morphism $b: A_{n-2}\rightarrow C_{n-1}$ which induces the morphism $p$. 

The map from $\Sigma^{-1}N$ to $\Sigma^n A$ gives rise to a morphism $N_{n-1}\rightarrow \Sigma^2(A)$. 
It induces a map $a:A_{n-1}\rightarrow A$ which is well-defined up to a morphism factoring through $\overline{f}_{n-1}$. 
Thus, we get an element $a_*[X_{n-1}]=\delta_{n-1}\in\mathbb E(C_{n-1},A)$.

Let $u$ denote the composition $C\rightarrow N\rightarrow N_{\geq -n+2}$. 
The map $\Sigma^{n-2}N_{n-2}\rightarrow N_{\geq -n+2}$ gives rise to a canonical map $N_{\geq -n+2}\rightarrow \Sigma^{n}A_{n-2}$. 
It induces a canonical map $v:\Cone(u)\rightarrow \Sigma^{n} A_{n-2}$. 
We have the following commuative diagram
\[
\begin{tikzcd}
\Sigma^{-1}N_{n-1}\ar[d,equal]\ar[r]&\Sigma^{-n}N\ar[r]\ar[d]&\Sigma^{-n}N_{\geq -n+2} \ar[r]\ar[d]&N_{n-1}\ar[d,equal]\\
\Sigma^{-1}N_{n-1}\ar[r]\ar[d,equal]&\Sigma^{-(n-1)}M\ar[r]\ar[d,dashed,"x"]&\Sigma^{-n}\Cone(u)\ar[r]\ar[d,"y"]&N_{n-1}\ar[d,equal]\\
\Sigma^{-1}N_{n-1}\ar[r]&\Sigma A_{n-1}\ar[r]&\Sigma V_{n-1}\ar[r]&N_{n-1}
\end{tikzcd}
\]
where $y$ is the composition $\Sigma^{-n}\Cone(u)\rightarrow A_{n-2}\rightarrow C_{n-1}\rightarrow \Sigma V_{n-1}$.
Since the spaces $\Hom_{\D(\A)}(\Sigma^{-n}N_{\geq -n+2},\Sigma A)$ and $\Hom_{\D(\A)}(\Sigma^{-n+1}C,\Sigma A)$ both vanish, the space 
\[
\Hom_{\D(\A)}(\Sigma^{-n}\Cone(u),\Sigma A)
\]
 also vanishes and hence the composition $\Sigma^{-n+1}M\xrightarrow{x}\Sigma A_{n-1}\xrightarrow{\Sigma a} \Sigma A$ equals to the morphism $\Sigma^{-n+1}(w)$.
 We have the following diagram in $\D(\A)$
 \[
 \begin{tikzcd}
 &&&\Sigma^{-n+1}M\ar[llldd,Rightarrow,"\Sigma^{-n+1}(s)"swap,bend right=6ex]\ar[rrdd,"x",bend left=9ex]\ar[rrrdd,"\Sigma^{-n+1}(w)",bend left=9ex]\ar[ld,Rightarrow]&&&\\
 &&\Sigma^{-n}\Cone(u)\ar[lld,Rightarrow,"\Sigma^{-n+1}(t)"]\ar[d,"\Sigma^{-n}(v)"]&&\Sigma V_{n-1}&&\\
 \Sigma^{-n+1}C&&A_{n-2}\ar[r,"b"]&C_{n-1}\ar[ru]&&\Sigma A_{n-1}\ar[lu,Rightarrow]\ar[r,"\Sigma(a)"]&\Sigma A
 \end{tikzcd}
 \]
which shows that the morphism $w/s$ is equal to the composition of $\Sigma^{-1}(v)/t$ and a morphism $\Sigma^{n-1}A_{n-2}\rightarrow \Sigma^{n}A$. 
 Hence the left $\C$-module $\Ext^n_{\D^b(\A)}(C,-)$ is effaceable for each $n\geq 1$ and $C\in \C$. 
 Similarly one shows that the right $\C$-module $\Ext^n_{\D^b(\A)}(?,A)$ is effaceable for each $n\geq 1$ and $A\in \C$.
 The conclusion follows from Corollary~\ref{cor:effaceablebimodule}.
\end{proof}
Let $\A$ be an exact dg category and $N_{dg}(\A)$ the exact $\infty$-category endowed with the 
corresponding exact structure, cf.~Theorem~\ref{nerve}.
Then the extriangulated structures on $H^0(\A)$ given by \ref{canonicalstructure}, 
which we denote by $(H^0(\A),\mathbb E_{\A},\mathfrak s)$, 
and given by \cite[Theorem 4.22]{NakaokaPalu20}, 
which we denote by $(H^0(\A),\mathbb E_{N_{dg}(\A)},\mathfrak s')$, coincide with each other. 
Indeed, it is straightforward to verify that the canonical map $\mathbb E_{\A}(C,A)\rightarrow \mathbb E_{N_{dg}(\A)}(C,A)$ determines a natural isomorphism of biadditive bifunctors.

To each exact $\infty$-category $\E$, Klemenc \cite{Klemenc22} associates its {\em stable hull} $\D^b_{\infty}(\E)$ ($=\mathcal H^{st}(\E)$ in Klemenc's notation) together with a fully exact fully faithful functor $\eta_{\E}:\E\rightarrow \D^b_{\infty}(\E)$ whose essential image is closed under extensions.
Put $\T=h(\D^b_{\infty}(\E))$ and $\C=h(\E)$. 
The category $\C$ carries a canonical extriangulated structure, cf.~\cite[Theorem 4.22]{NakaokaPalu20}.
Then by construction $\T$ is the Verdier quotient of $\mathcal S\W(\P_{\Sigma,f}(\E))$ by its full triangulated subcategory of acyclic objects.
By the same reasoning as in Proposition~\ref{higher}, we see that $\Ext^n_{\T}(?,-)$ is a weakly effaceable $\C$-$\C$-bimodule for $n\geq 1$ and thus coincides with $\mathbb E^n$.
Let $\E$ be the exact $\infty$-category $N_{dg}(\A)$. 
Let $F:\A\rightarrow \D^b_{dg}(\A)$ be the universal exact morphism into a pretriangulated dg category in $\Hqe$.
Then the morphism $N_{dg}(F): N_{dg}(\A)\rightarrow N_{dg}(\D^b_{dg}(\A))$ induces a morphism $\varphi_{\A}:\D^b_{\infty}(N_{dg}(\A))\rightarrow N_{dg}(\D^b_{dg}(\A))$ of stable $\infty$-categories.
By the above, this morphism is an equivalence of stable $\infty$-categories.

 \begin{theorem}
There is a commutative diagram up to a canonical isomorphism
 \[
 \begin{tikzcd}
 \{\text{Exact dg categories}\}\ar[d,"\D^b_{dg}"swap]\ar[r,"N_{dg}"]&\{\text{Exact $\infty$-categories}\}\ar[d,"{\D^b_{\infty}}"]\ar[ld,Rightarrow,"\varphi"]\\
 \{\text{Pretriangulated dg categories}\}\ar[r,"N_{dg}"swap]&\{\text{Stable $\infty$-categories}\}
 \end{tikzcd}
 \]
 \end{theorem}
\begin{remark}
 In the preprint \cite{BorveTrygsland21}, for an exact $\infty$-category $\E$, B{\o}rve--Trygsland defined an $\infty\mbox{-}$category $\E xt_{\E}^{n}(B,A)$ for each pair of objects $(B, A)$ in $\E$ and $n\geq 0$ whose $\pi_0$ is naturally isomorphic to the $n\mbox{-}$th extension $\Ext^n_{h(\E)}(B,A)$ of the extriangulated category $h(\E)$.
\end{remark}
\begin{example}\label{exm:twotermderived}
Let $A$ be a  $k$-algebra and $n\geq 1$ an integer. 
Let $\mathcal{H}^{[-n,0]}(\proj A)$ be the full subcategory of $\T=\mathcal{H}^b(\proj A)$ consisting of complexes of projectives concentrated in degrees from $-n$ to $0$. 
Let $\A'=\C^{[-n,0]}_{dg}(\proj A)$ be the canonical dg enhancement of $\mathcal{H}^{[-n,0]}(\proj A)$. 
We endow it with the exact dg structure inherited from $\C^{b}_{dg}(\proj A)$ as an extension-closed subcategory. 
Let $\A=\tau_{\leq 0}\A'$ with the induced exact dg structure. 
We claim that $\D^b_{dg}(\A)$ is quasi-equivalent to $\C^{b}_{dg}(\proj A)$. 
Indeed, since $\C^{b}_{dg}(\proj A)$ is pretriangulated, the exact inclusion 
$\C^{[-n,0]}_{dg}(\proj A) \to \C^{b}_{dg}(\proj A)$ extends to a canonical exact dg functor 
$\D^b_{dg}(\A) \to \C^{b}_{dg}(\proj A)$. Its image contains $\proj A$ so to show that it is an
equivalence, it suffices to show that it is fully faithful. For this, observe that
the category $\mathcal{H}^{[-n,0]}(\proj A)$ has enough projective objects and the full 
subcategory of projective objects is $\proj A$. 
Then for each $P^{\cdot}\in\mathcal{H}^{[-n,0]}(\proj A)$ we have a conflation 
\[
\Sigma^{-1}P^{\leq -1}\rightarrow P^0\rightarrow P^{\cdot}.
\] 
Let $Q$ be an object in $\A'$. 
Then for $i\geq 1$ we have 
\[
\mathbb E^{i}(P^0, Q)\xrightarrow{\sim}\Ext_{\T}^{i}(P^0,Q)=0.
\] 
By induction we have 
\[
\mathbb E^{i}(\Sigma^{-1}P^{\leq -1},Q)\xrightarrow{\sim}\Ext_{\T}^{i}(\Sigma^{-1}P^{\leq -1},Q).
\] 
Then by \cite{GorskyNakaokaPalu21} Corollary 3.21, $\mathbb E^{i}(P^{\cdot},Q)$ is the cokernel of the map 
\[
\mathbb E^{i-1}(P^0,Q)\rightarrow \mathbb E^{i-1}(\Sigma^{-1}P^{\leq -1},Q).
\] 
Since the space $\Ext^i_{\T}(P^0,Q)$ vanishes, we have $\mathbb E^{i}(P^{\cdot},Q)\xrightarrow{\sim}\Ext^i_{\T}(P^{\cdot},Q)$.
\end{example}
\subsection{Quotients with respect to projective-injective objects}

\begin{lemma}Let $\C$ be a connective dg category and $\C'\subseteq\C$ a full dg subcategory. 
Let $\B=\tr(\C)/\tr(\C')$. 
Then for $A,B\in \C$, we have $\Hom_{\B}(A,\Sigma^n B)=0$ for $n>0$ and $\Hom_{\B}(A,B)=(H^0(\C)/[\C'])(A,B)$. 
\end{lemma}
\begin{proof}Since $\C$ is connective, by Proposition \ref{cot} we have a canonical co-t-structure $(\T^{\geq 0}, \T^{\leq 0})$ on $\T{\coloneqq}\tr(\C)$ which restricts to a co-t-structure on $\tr(\C')$. 

Suppose we have a morphism in $\B$ as follows
\[
\begin{tikzcd}&M\ar[ld,Rightarrow,"s"swap]\ar[rd,"b"]&\\
A&&\Sigma^nB
\end{tikzcd}
\]
where $s$ and $b$ are morphisms in $\tr(\C)$ and the cone of $s$ lies in $\tr(\C')$. 

Take a weight decomposition of $P=\mathrm{Cone}(s)$: $\sigma_{>-1}(P)\rightarrow P\rightarrow \sigma_{\leq -1}P\rightarrow \Sigma \sigma_{\geq -1}P$. 
Since $A$ lies in the coheart of the co-t-structure, $A\rightarrow P$ factors through $\sigma_{>-1}(P)$. 
Hence we may assume $\Cone(s)\in \T^{>-1}$. 

Since for $n\geq 0$ we have $\Sigma^{n}B\in \T^{-n}$, the map $\Sigma^{-1}\Cone(s)\rightarrow\Sigma^{n}B$ vanishes and hence the map $b$ factors through $A$. 
So for $n>0$, we have $\Hom_{\B}(A,\Sigma^{n}B)=0$. 

Now assume $n=0$. 
If a map $f:A\rightarrow B$ in $\tr(\C)$ factors through an object in $\tr(\C')$, then it factors through an object in the coheart, i.e. a retract of an object $W\in \C'$ in $\tr(\C')$. 
Therefore it factors through an object in $\C'$. 

So $\Hom_{\B}(A,B)=(H^0(\C)/[\C'])(A,B)$.  
\end{proof}

\begin{definition}An object $P$ in $H^0(\A)$ is said to be {\em projective} if for each conflation 
\[
\begin{tikzcd}
X\ar[r,"f"]\ar[rr,"h"swap, bend right=8ex]&Y\ar[r,"j"]&Z
\end{tikzcd} 
\]
the morphism $\overline{\jmath}$ induces a surjection $\Hom_{\D(\A)}(P,Y)\rightarrow \Hom_{\D(\A)}(P,Z)$, or equivalently, if $P$ is projective in the extriangulated category $H^0(\A)$.

Dually we define {\em injective} objects in $H^0(\A)$.
\end{definition}
\begin{example}\label{exm:exactdgpretriangulated}
Let $\A$ be a connective dg category such that $H^0(\A)$ is additive. 
Suppose that $(\A,\mathcal S)$ is an exact dg category with the associated extriangulated category $H^0(\A)$ being triangulated. 
We show that $\A$ is quasi-equivalent to the truncation of a pretriangulated dg category.

Since each morphism can be completed to a triangle, each morphism $f:A\rightarrow B$ in $Z^0(\A)$ admits a homotopy kernel and a homotopy cokernel which are both homotopy short exact sequences.
Moreover $\mathcal S$ must be the class of all homotopy short exact sequences.
The canonical functor $H^0(\A)\rightarrow \D^b(\A)=H^0(\D^b_{dg}(\A))$ is a fully faithful triangle functor and its essential image contains a family of generators for $\D^b(\A)$. 
Hence it is a triangle equivalence. 
So the connective dg category $\A$ is quasi-equivalent to $\tau_{\leq 0}\D^b_{dg}(\A)$.
\end{example}

\begin{lemma}\label{lem:leftorthogonal}
When $\A$ is connective, an object $P$ in $H^0(\A)$ is projective if and only if $\Hom_{\tr(\A)}(P,\mathcal N)=0$ where $\mathcal N$ is the full triangulated subcategory of $\tr(\A)$ generated by the total dg modules of conflations.
\end{lemma}
\begin{proof}Notice that the totalization $N$ of a conflation
\[
\begin{tikzcd}
X\ar[r,"f"]\ar[rr,"h"swap,bend right=6ex]&Y\ar[r,"j"]\ar[r]&Z
\end{tikzcd} 
\]
lies in the heart of the canonical t-structure of $\D(\A)$ and is the cokernel of the map
\[
\Hom_{\D(\A)}(-,Y)\rightarrow \Hom_{\D(\A)}(-,Z).
\]
Also for objects $A'$ in $H^0(\A)$, we have $\Hom_{\tr(\A)}(A', \Sigma^i N)=0$ for $i\neq 0$. 

So $P$ is projective if and only if $\Hom_{\D(\A)}(P,N)=0$ for each totalization $N$ of a conflation.
\end{proof}

Let $F:\A\rightarrow \D^b_{dg}(\A)$ be the universal exact morphism from a connective exact dg category $\A$ into a pretriangulated dg category $\D^b_{dg}(\A)$. 
Recall that we put $\D^b(\A)=H^0(\D^b_{dg}(\A))$. 
Let $(\T^{>0}, \T^{\leq 0})$ be the canonical co-t-structure  on $\T=\tr(\A)$ introduced in Proposition \ref{cot}.

Since we have
\[
\T^{\geq 0}=\bigcup_{i\geq0}\mathcal{H}[-i]*\mathcal{H}[-i+1]*\cdots*\mathcal{H}
\]
 where $\mathcal{H}$ is the coheart of the co-t-structure consisting of direct summands of $A$ in $\tr(\A)$ for $A\in\A$, it follows from Theorem~\ref{main} that $\tau_{\leq 0}\RHom_{\A}(A, M)$ is quasi-isomorphic to $\tau_{\leq 0}\Hom_{\D^b_{dg}(\A)}(A,M)$ for $A\in \A$ and $M\in \T^{\geq 0}$. 

If $P$ is a projective object in the extriangulated category $H^0(\A)$, then by Lemma~\ref{lem:leftorthogonal} $P$ lies in the left orthogonal of $\N$ and we have 
\[
\Hom_{\tr(\A)}(P,X)\xrightarrow{\sim} \Hom_{\D^b(\A)}(FP, X)
\]
 for any $X\in\pretr(\A)$. 

\begin{lemma}\label{lem:quotientprojectiveinjective}
Let $\P$ be a full dg subcategory of $\A$ consisting of projective-injective objects in $H^0(\A)$. 
Let $\B=\tr(\A)/\tr(\P)$ and $\B'=\D^b(\A)/\tr(\P)$.  
Then for $A,B\in\A$, the following statements hold:
\begin{itemize}
\item[1)]The canonical map $\Hom_{\B}(A, \Sigma^n B)\rightarrow\Hom_{\B'}(A,\Sigma^n B)$ is a bijection for $n\leq 0$.
\item[2)]The canonical map $\Hom_{\D^b(\A)}(A,\Sigma^{n}B)\rightarrow\Hom_{\B'}(A,\Sigma^nB)$ is a bijection for $n\geq 1$.
\end{itemize}
\end{lemma}
\begin{proof}Consider a morphism in $\B'$
\[
\begin{tikzcd}&M\ar[ld,Rightarrow,"s",swap]\ar[rd,"b"]&\\
A&&\Sigma^nB
\end{tikzcd}
\]
where $s$ and $b$ are morphisms in $\D^b(\A)$ and $P=\Cone(s)\in \tr(\P)$. 

Put $\T=\tr(\A)$. Since objects in $\P$ are injective, the quotient functor $F:\T\rightarrow \D^b(\A)=\T/\mathcal N$ induces bijections $\Hom_{\T}(X,Q)\xrightarrow{\sim} \Hom_{\D^b(\A)}(X,Q)$ for $X\in \T$ and $Q\in \tr(\P)$. We may assume $s=F(s')$ for some morphism $s'$ in $\T$. 
Let $(\T^{\geq 0}, \T^{\leq 0})$ be the canonical co-t-structure on $\T$. 
Since $A\in \T^{\geq 0}$, we may assume $P\in \T^{\geq 0}$. 

1) Assume $n\leq 0$. Consider the triangle $\Sigma^{-1}P\rightarrow M\xrightarrow{s} A\rightarrow P$ in $\T$ and the induced long exact sequence
\[
\begin{tikzcd}
[cramped,sep=small]\cdots\ar[r]&\T(M,\Sigma^{-1}B)\ar[r]\ar[d]&\T(P,B)\ar[r]\ar[d,"\sim"]&\T(A,B)\ar[r]\ar[d,"\sim"]&\T(M,B)\ar[r]\ar[d]&\T(P,\Sigma{B})=0\\
\cdots\ar[r]&\D^b(\A)(M,\Sigma^{-1}B)\ar[r]&\D^b(\A)(P,B)\ar[r]&\D^b(\A)(A,B)\ar[r]&\D^b(\A)(M,B)\ar[r]&\D^b(\A)(P,\Sigma{B})=0
\end{tikzcd}.
\]
We have 
\begin{align}
\Hom_{\T}(M,\Sigma^{n}B)\xrightarrow{\sim} \Hom_{\D^b(\A)}(M,\Sigma^{n}B) \label{projinj}
\end{align}
for $n\leq 0$. 
So we have $b=F(b')$ for some morphism $b'$ in $\T$. 

This shows that $\Hom_{\B}(A, \Sigma^n B)\rightarrow\Hom_{\B'}(A,\Sigma^n B)$ is a surjection for $n\leq 0$. 

Assume the morphism $b=F(b'):M\rightarrow \Sigma^n B$ in $\D^b(\A)$ factors through an object in $\tr(\P)$.  
By the bijection \ref{projinj} and since $\P$ consists of projective-injective objects, the morphism $b'$ in $\T$ also factors through an object in $\tr(\P)$.
 
So the canonical map $\Hom_{\B}(A, \Sigma^n B)\rightarrow\Hom_{\B'}(A,\Sigma^n B)$ is a bijection for each $n\leq 0$.

2) For $n\geq 1$, since $\Hom_{\D^b(\A)}(\Sigma^{-1}P, \Sigma^{n} B)=0$ for $P\in \T^{\geq 0}$, $b:M\rightarrow \Sigma^{n}B$ factors through $A$. 
So $\Hom_{\D^b(\A)}(A,\Sigma^{n}B)\rightarrow\Hom_{\B'}(A,\Sigma^nB)$ is a surjection.

Let $f:A\rightarrow \Sigma^n B$ be a morphism in $\D^b(\A)$ factors through an object $Q$ in $\tr(\P)$. 
We may assume $Q\in \tr(\P)^{\geq 0}$. 
Then $\Hom_{\D^b(\A)}(Q,\Sigma^n B)=0$ and hence the morphism $f$ is zero.

So $\Hom_{\D^b(\A)}(A,\Sigma^{n}B)\rightarrow\Hom_{\B'}(A,\Sigma^nB)$ is a bijection for $n\geq 1$.
\end{proof}

Let $\P$ be a full dg subcategory of a connective exact dg category $\A$ consisting of projective-injective objects. 
Then we have the following diagram
\[
\begin{tikzcd}&\tr(\A)\ar[r,"F_1"]\ar[d,"G_1"swap]&\D^b(\A)\ar[d,"G_2"]\\
H^0(\A)/[\P]\ar[r, hook,"H"swap]&\tr(\A)/\tr(\P)\ar[r,"F_2"swap]&\D^b(\A)/\tr(\P)
\end{tikzcd}
\]
In summary we have
\begin{proposition}\label{prop:dgsingularitycategory}
Let $A,B\in \A$. The following statements hold:
\begin{itemize}
\item[1)]$H$ is fully faithful.
\item[2)]For $n\geq 1$, $\Hom_{\tr(\A)/\tr(\P)}(A,\Sigma^n B)=0$.
\item[3)]For $n\leq 0$, $\Hom_{\tr(\A)/\tr(\P)}(A,\Sigma^n B)\xrightarrow{\sim}\Hom_{\D^b(\A)/\tr(\P)}(A,\Sigma^n B)$.
\item[4)]For $n\geq 1$, $\Hom_{\D^b(\A)}(A,\Sigma^n B)\xrightarrow{\sim}\Hom_{\D^b(\A)/\tr(\P)}(A,\Sigma^n B)$. 
In particular, the essential image of $F_2\circ H$ is closed under extension in $\D^b(\A)/\tr(\P)$. 
\item[5)]For $n\leq 0$, $\Hom_{\tr(\A)}(A,\Sigma^n B)\xrightarrow{\sim} \Hom_{\D^{b}(\A)}(A,\Sigma^n B)$.
\end{itemize}
\end{proposition}
For an extriangulated category $\C$ with a subcategory $\P$ consisting of (not necessarily all) projective-injective objects in $\C$, Nakaoka--Palu showed that the ideal quotient $\C/\P$ has the structure of an extriangulated category, induced from that of $\C$, cf.~\cite[Proposition 3.30]{NakaokaPalu19}. In the context of exact dg categories, we have the following theorem. 
\begin{theorem}\label{quot}
Let $\A$ be a connective exact dg category and $\P$ a full dg subcategory of $\A$ consisting of projective-injective objects in $\A$.
Let $\mathcal S_{dg}$ be the canonical dg enhancement of $\D^b(\A)/\tr(\P)$.
Then the dg quotient $\A/\P$ carries a canonical exact dg structure induced from that of $\A$ and its dg derived category is quasi-equivalent to $\mathcal S_{dg}$.
\end{theorem}
\begin{proof}
By the items $(1)$ and $(3)$ of Proposition~\ref{prop:dgsingularitycategory}, we have $H^0(\A/\P)=H^0(\A)/[\P]$ and the canonical morphism $\varphi:\A/\P\rightarrow \mathcal S_{dg}$ in $\Hqe$ is quasi-fully faithful. 
By the item (4), the quasi-essential image of $\varphi$ is closed under extensions in $\mathcal S_{dg}$ and therefore $\A/\P$ carries a canonical exact dg structure.
Again by the item $(4)$ and Proposition \ref{higher}, the bounded dg derived category of $\A/\P$ is quasi-equivalent to $\mathcal S_{dg}$.
\end{proof}
\begin{example}[\cite{Jin20}]\label{exm:CMdgmodule}
Let $k$ be a field. 
A dg $k$-algebra $A$ is {\em proper} if $\sum_{i\in \mathbb Z}\dim_{k}{H^i(A)<\infty}$. 
It is {\em Gorenstein} if the thick subcategory $\per(A)$ of the derived category $\D(A)$ 
generated by $A$ coincides with the thick subcategory $\thick(DA)$ generated by $DA$, where $D=\Hom_k(-,k)$ is the $k\mbox{-dual}$.
Let $A$ be a connective proper Gorenstein dg algebra.
A dg $A\mbox{-}$module is {\em perfectly valued} if its total cohomology is finite-dimensional and we denote by $\pvd(A)$ the triangulated category of perfectly valued dg $A\mbox{-}$modules.
A dg $A\mbox{-}$module $M$ in $\pvd(A)$ is {\em {Cohen--Macaulay}} if $H^i(M)=0$ and $\Hom_{\D(A)}(M,\Sigma^i A)=0$ for $i>0$.
Let $\CM A$ be the subcategory of $\pvd(A)$ consisting of Cohen-Macaulay dg $A\mbox{-}$modules.
By \cite[Theorem 2.4]{Jin20}, the category $\CM A$ is an $\Ext\mbox{-}$finite Frobenius extriangulated category with $\proj(\CM A)=\add A$. 
Let $\pvd_{dg}(A)$ be the canonical dg enhancement of $\pvd(A)$ and $\A'\subseteq \pvd(A)$ the full dg subcategory consisting of Cohen--Macaulay dg modules. 
Put $\A=\tau_{\leq 0}\A'$. 
Since $\A'$ is extension-closed in $\pvd_{dg}(A)$, the dg category $\A$ inherits a canonical exact structure whose corresponding extriangulated category is $\CM A$. Note that for projective objects $P\in \add(A)$, the Hom spaces $\Hom_{\D(A)}(P,\Sigma^{i} M)$ vanishes for $i>0$ and $M\in\CM A$. 
By \cite[Lemma 3.9 (2)]{Jin20}, the triangulated category $\pvd(A)$ is generated by $\CM A$.
By the same reasoning in Example~\ref{exm:twotermderived}, the bounded dg derived category of $\A$ is quasi-equivalent to $\pvd_{dg}(A)$. By Proposition~\ref{prop:dgsingularitycategory}, the stable category $\CM(A)/[\proj(\CM A)]$ is equivalent to the singularity category $\pvd(A)/\per(A)$.
\end{example}

\subsection{Categories with enough projectives}
Assume that the extriangulated category $H^0(\A)$ has enough projectives and let $\P$ be the full subcategory consisting of projective objects. 
For a morphism $b:B\rightarrow B'$ in $H^0(\A)$, denote by $[b]$ the corresponding morphism in $H^0(\A)/[\P]$. 
Put $\Omega^0{\coloneqq}\Id:H^0(\A)/[\P]\rightarrow H^0(\A)/[\P]$. 
For $n\geq 1$, we define syzygz functors $\Omega^n$ as $n$-iterations of $\Omega: H^0(\A)/[\P]\rightarrow H^0(\A)/[\P]$ which is defined as follows. 

For each object $B\in H^0(\A)$, we pick once and for all an $\mathbb E$-extension $\delta_B$ which is realized by $\Omega B\rightarrow P^{-1}\rightarrow B$ where $P^{-1}$ is projective. 
Since $P^{-1}$ is projective, a morphism $b:B\rightarrow B'$ in $H^0(\A)$ gives rise to a morphism $(r,b)$ of $\mathbb E$-extensions
\[
\begin{tikzcd}
\Omega B\ar[r]\ar[d,"\exists r"swap]&P^{-1}\ar[r]\ar[d]&B\ar[d,"b"]\\
\Omega B'\ar[r]&P'^{-1}\ar[r]&B'
\end{tikzcd}
\]
\begin{lemma}[\cite{IyamaNakaokaPalu18}, Definition-Proposition 1.25] $B\mapsto \Omega B$ and $[b]\mapsto [r]$ define a functor $\Omega: H^0(\A)/[\P]\rightarrow H^0(\A)/[\P]$, which will be called the {\em syzygy functor}. 
\end{lemma}
\begin{proof}
By Corollary 3.12 of \cite{NakaokaPalu19}, we have a long exact sequence
\[
\begin{tikzcd}
\cdots\ar[r]&\Hom_{H^0(\A)}(P^{-1},\Omega B')\ar[r]&\Hom_{H^0(\A)}(\Omega B,\Omega B')\ar[r]&\mathbb E(B,\Omega B')\ar[r]&\cdots
\end{tikzcd}
\]
The morphism $r:\Omega B\rightarrow \Omega B'$ satisfies $r_{*}\delta_B=b^{*}\delta_{B'}\in \mathbb E(B,\Omega B')$, so $[r]$ is independent of the choice of $r$. 
Also $b^{*}\delta_{B'}=b'^{*} \delta_{B'}$ if $[b]=[b']$. 

Therefore $B\mapsto \Omega B$ and $[b]\mapsto [r]$ define the syzygy functor $\Omega: H^0(\A)/[\P]\rightarrow H^0(\A)/[\P]$. 
\end{proof}

For each object $B\in H^0(\A)$, we have the following diagram in the dg category $\A$
\begin{equation}\label{total}
\begin{tikzcd}
&&&\ \ar[dddd,dashed, no head, red]&\ \ar[dd, green, dashed, no head]&\ \ar[ddd, dashed, no head, red]&          \   \ar[d, dashed, green, no head]                     &\ \ar[dd, dashed, no head, red]&\\
&&&&&&\Omega B\ar[r]\ar[d, equal]                                                                   &P^{-1}\ar[r]                                                       &B\\
&&&&\Omega^2 B\ar[r]\ar[d, equal]&P^{-2}\ar[r]&    \Omega B       \ar[r, green, no head, dashed,"X^{-1}"]&\; \ar[r, dashed, red, no head,"P^{\geq -1}_B"red]                  &\ \\
&&\Omega^3 B\ar[r]\ar[d, equal]&P^{-3}\ar[r]&\Omega^2 B\ar[r, green, dashed, no head,"X^{-2}"]&\ \ar[rrr, dashed, no head, red,"P^{\geq -2}_B"]&&&\ \\
&\cdots\ar[r]&\Omega^3 B&\ \ar[rrrrr, dashed, no head, red,"P^{\geq -3}_B"red]   &       &     &          &       &\;
\end{tikzcd}
\end{equation}
where the rows are shifted totalized conflations and where we have omitted the homotopies in order to make it more readable. 
We denote its totalization by $P_B$.
The totalization $P_B$ carries an increasing filtration by the $P^{\geq -i}_B$, $i\in \mathbb N$.
Notice that $P_B^{\geq -i}$ is homotopy equivalent to $X^{-i}$.

Claim 1: for each $Q\in \tr(\P)$, we have $\Hom_{\tr(\A)}(Q,P_B^{\geq -n})=0$ for $n>>0$.

Indeed for each $n\geq 1$, we have a triangle
\begin{equation}\label{omega}
\Sigma^{n}\Omega^{n}B\rightarrow X^{-n}\rightarrow N\rightarrow \Sigma^{n+1}\Omega^nB
\end{equation}
in $\tr(\A)$ where $N$ is in the triangulated subcategory $\N$ of $\tr(\A)$ generated by totalizations of conflations.
Thus we have $\Hom_{\tr(\A)}(Q,X)=0$.  
Note that for $n>>0$,  we have $\Hom_{\tr(\A)}(Q,\Sigma^n Y)=0$ for any $Y\in H^0(\A)$. 
So we have $\Hom_{\tr(\A)}(Q, X^{-n})=0$ for $n>>0$. 
This proves Claim 1.

Claim 2: for each $Q\in \tr(\P)$, we have $\Hom_{\D(\A)}(Q, P_B)=0$.

 Indeed, we have a graded split short exact sequence in $\C(\A)$
 \[
 0\rightarrow \bigoplus_{n\geq 0} P_B^{\geq -n}\rightarrow \bigoplus_{n\geq 0}P_B^{\geq -n}\rightarrow P_B\rightarrow 0
 \] 
 and hence we have a triangle in $\D(\A)$. 
 \[
 \bigoplus_{n\geq 0} P_B^{\geq -n}\rightarrow \bigoplus_{n\geq 0}P_B^{\geq -n}\rightarrow P_B\rightarrow \Sigma\bigoplus_{n\geq 0}P_B^{\geq -n}
 \]

By Claim 1 we have 
\[
\Hom_{\D(\A)}(Q,P_B)=\colim_{n\geq 0} \Hom_{\tr(\A)}(Q,P_B^{\geq -n})=0.
\] 
This proves Claim 2.

Now let $A\in H^0(\A)$ and fix $m\geq 0$. 
Then from the triangle 
\[
\Sigma^{n}P_B^{-n-1}\rightarrow P_B^{\geq -n}\rightarrow P_B^{\geq -n-1}\rightarrow \Sigma^{n+1}P_B^{-n-1}
\]
in $\tr(\A)$, for $n\geq m+1$ we have 
\[
\Hom_{\tr(\A)}(\Sigma^{m}A, P_B^{\geq -n})\xrightarrow{\sim}\Hom_{\tr(\A)}(\Sigma^{m}A,P_B^{\geq -n-1}).
\]
It follows that we have
\[
\begin{aligned}
\Hom_{\D(\A)}(\Sigma^{m}A, P_B)&=\Hom_{\tr(\A)}(\Sigma^{m}A,P_B^{\geq -m-1})\\
&=\cok(\Hom_{\tr(\A)}(A,P^{-m-1})\rightarrow \Hom_{\tr(\A)}(A,\Omega^{m}B))\\
&=\Hom_{H^0(\A)/[\P]}(A,\Omega^{m}B)
\end{aligned}
\] 
where we used the bijection 
\[
\Hom_{\tr(\A)}(\Sigma^mA, P^{\geq -m}_B)\iso \Hom_{\tr(\A)}(A,\Omega^mB)
\]
 given by the triangle \ref{omega}.
\begin{proposition}
Keep the assumptions as above. 
Then the canonical map 
\[
 \Hom_{\D(\A)}(\Sigma^m A, P_B)\rightarrow \Hom_{\tr(\A)/\tr(\P)}(A,\Sigma^{-m} B)
\]
is a bijection for each pair of objects $A$, $B\in \A$ and $m\in\mathbb Z$. 
In particular, the canonical map
\[
\Hom_{H^0(\A)/[\P]}(A,\Omega^m B)\rightarrow \Hom_{\tr(\A)/\tr(\P)}(A,\Sigma^{-m}B)
\]
is a bijection for each pair of objects $A, B\in H^0(\A)$ and $m\geq 0$.
\end{proposition}
\begin{proof}
By the above Claim 1, $\{B\rightarrow P_B^{\geq -n},n\geq 1\}$ is cofinal in the filtered category of morphisms with objects given by $B\xrightarrow{s} U$ in $\tr(\A)$ such that $\Cone(s)\in \tr(\P)$.
 Thus we have 
\[
\begin{aligned}
\Hom_{\tr(\A)/\tr(\P)}(A,\Sigma^{-m}B)&=\colim_{n\geq 0}\Hom_{\tr(\A)}(A, \Sigma^{-m}P_{B}^{\geq -n})\\
&=\Hom_{\D(\A)}(A,\Sigma^{-m}P_B).
\end{aligned}
\]
\end{proof}
Right triangulated categories were introduced in \cite{KellerVossieck87} (see also \cite{Keller90,BeligiannisMarmaridis94}) under the name ``suspended categories". 
In loc. cit, they applied Heller's stabilization  procedure $\mathcal S$ (\cite{Heller68}) to a right triangulated category and showed that this yields a triangulated category with a certain universal property. 
They also showed that the stable category $\underline{\mod}{\mbox{}} A$ of the category $\mod \mbox{}A$ of finite dimensional right modules over a finite dimensional algebra $A$ is left triangulated, and that its stabilization 
$\mathcal S(\underline{\mod}\mbox{} A)$ is equivalent to the singularity category $\D_{sg}(A){\coloneqq}\D^b(A)/\mathcal H^b(\proj\mbox{}A)$. 

Similarly, $H^0(\A)/[\P]$ is canonically a left triangulated category. 
Here we indicate what are the triangles. 

For a morphism $[b]:B\rightarrow B'$ in $H^0(\A)/[\P]$, we have a morphism $(\Id_{\Omega B'}, b): b^{*}\delta_{B'}\rightarrow \delta_{B'}$ of $\mathbb E$-extensions
\[
\begin{tikzcd}
\Omega B'\ar[r,"u"]\ar[d,equal] &E\ar[r,"v"]\ar[d]& B\ar[d,"b"]\\
\Omega B'\ar[r]& P'^{-1}\ar[r]&B'
\end{tikzcd}
\]
where $b^{*}\delta_{B'}$ is realized by $\Omega B'\xrightarrow{u}E\xrightarrow{v}B$.

Thus we get a sequence 
\[
\begin{tikzcd}
\Omega B'\ar[r,"{[}u{]}"]&E\ar[r,"{[}v{]}"]&B\ar[r,"{[}b{]}"]&B'
\end{tikzcd}
\]
in $H^0(\A)/[\P]$. We call such sequences standard triangles.

The distinguished triangles are defined to be those sequences which are isomorphic to standard triangles.

\begin{proposition}\label{stab}The canonical functor $H^0(\A)\rightarrow \D^{b}(\A)/\tr(\P)$ identifies $\D^b(\A)/\tr(\P)$ with the stablization of $H^{0}(\A)/[\P]$.
\end{proposition}

Recall for each object $B\in H^0(\A)$, we have a diagram \ref{total} in the dg category $\A$ whose totalization is $P_B$. 

Let $Q_B$ be the totalization of the diagram which is obtained from \ref{total} by deleting the object $B$ and such that $P^{-1}$ is in degree zero. 

For each $n\geq 1$, we take the totalization $Q^{\geq -n}$ of the diagram obtained from $P^{\geq -n}_B$ by deleting the object $B$ and such that $P^{-1}$ is in degree zero.

We have the following triangle
\[
\begin{tikzcd}
Q_{B}^{\geq -i}\ar[r]&B\ar[r]&P_B^{\geq -i}\ar[r]&\Sigma Q_{B}^{\geq -i}
\end{tikzcd}
\]
in $\tr(\A)$.

The following lemma is a direct consequence of the above Claim 1.
\begin{lemma}
For $A,B\in H^0(\A)$, we have 
\[
\Hom_{\D^b(\A)/\tr(\P)}(A,\Sigma^j B)=\colim_{n\geq 0}\Hom_{\D^b(\A)}(A,\Sigma^j P_B^{\geq -n}).
\]
\end{lemma}
\begin{proof}[Proof of Proposition \ref{stab}]It is clear that the canonical functor $H^0(\A)/[\P]\rightarrow \D^b(\A)/\tr(\P)$ is exact.
We only need to show that for each pair of objects $A,B\in H^0(\A)$ and $j\in\mathbb Z$, the following canonical map is a bijection 
\[
\begin{tikzcd}
\colim_{i\geq 0, i-j\geq 0}\colim_{n\geq 0}\Hom_{\tr(\A)}(\Sigma^i\Omega^iA,\Sigma^{j}P_B^{\geq -n})\ar[r,"\alpha"]
&\colim_{n\geq 0}\Hom_{\D^b(\A)}(A,\Sigma^j P_B^{\geq -n})
\end{tikzcd}
\]
where an element in the left hand side which is represented by a morphism $f:\Sigma^i\Omega^iA\rightarrow \Sigma^j P^{\geq -n}_B$ with $n\geq i-j+1$, is sent by the map $\alpha$ to the morphism 
\[
A\rightarrow P_A^{\geq -i}\iso \Sigma^i\Omega^i A\rightarrow \Sigma^jP_B^{\geq -n}
\]
in $\D^{b}(\A)$.
On the left hand side, the map
\[
\Hom_{\tr(\A)}(\Sigma^i\Omega^iA,\Sigma^{j}P_B^{\geq -n})\rightarrow \Hom_{\tr(\A)}(\Sigma^i\Omega^iA,\Sigma^{j}P_B^{\geq -{(n+1)}})
\]
is induced by the canonical morphism $P_B^{\geq -n}\rightarrow P_B^{\geq -(n+1)}$.  
The map
\[
\colim_{n\geq 0}\Hom_{\tr(\A)}(\Sigma^i\Omega^iA,\Sigma^{j}P_B^{\geq -n})\rightarrow \colim_{n\geq 0}\Hom_{\tr(\A)}(\Sigma^{i+1}\Omega^{i+1}A,\Sigma^{j}P_B^{\geq -n})
\]
sends an element represented by $f:\Sigma^i\Omega^iA\rightarrow \Sigma^{j}P_B^{\geq -n}$ with $n\geq i-j+1$, to the element represented by the morphism 
\[
\Sigma^{i+1}\Omega^{i+1}A\rightarrow \Cone(\Sigma^iP^{-i-1}_A\rightarrow\Sigma^i\Omega^iA)\rightarrow\Sigma^jP_B^{\geq -n}
\]
where the latter morphism is induced by $f$. 
Notice that for $n\geq i-j+1$ we have 
\[
\Hom_{\tr(\A)}(\Sigma^{i}P^{-i-1}, \Sigma^j P_B^{\geq -n})=0.
\]

For each $i\geq 0$ and $A\in H^0(\A)$, we have the triangle in $\tr(\A)$
\[
\begin{tikzcd}
Q_{A}^{\geq -i}\ar[r]&A\ar[r]&P_A^{\geq -i}\ar[r]&\Sigma Q_{A}^{\geq -i}
\end{tikzcd}.
\]
For each $n\geq 0$, we have the triangle in $\tr(\A)$
\[
\begin{tikzcd}
\Sigma^n \Omega^n B\ar[r]&P_B^{\geq -n}\ar[r]&X\ar[r]&\Sigma^{n+1} \Omega^n B
\end{tikzcd}.
\]
where $X$ lies in $\mathcal N^{[-n+1,0]}$. Here we adopt the convention that $P_B^{\geq 0}=B$ and in this case $X$ is a zero object.
Therefore, for $n\geq i-j$ we have 
\[
\Hom_{\D^b(\A)}(Q_{A}^{\geq -i},\Sigma^j P_B^{\geq -n})=\Hom_{\D^b(\A)}(Q_{A}^{\geq -i}, \Sigma^{n+j}\Omega^n B)=0.
\] 
For $A'\in H^0(\A)$, we have $\Hom_{\tr(\A)}(A', \Sigma^{m}X)=0$ for $m\leq -n$.
So for $0\leq n\leq i-j$, we have the following commutative diagram 
\[
\begin{tikzcd}
\Hom_{\tr(\A)}(A', \Sigma ^{n+j-i}\Omega^n B)\ar[d,"\sim"swap]\ar[r,"{\sim}"] &\Hom_{\D^b(\A)}(A',\Sigma^{n+j-i}\Omega^{n} B)\ar[d,"\sim"]\\
\Hom_{\tr(\A)}(A', \Sigma^{j-i}P_B^{\geq -n})\ar[r]&\Hom_{\D^b(\A)}(A', \Sigma^{j-i}P_B^{\geq -n})
\end{tikzcd}.
\]
So the bottom row is also a bijection. 

Claim: when $n=i-j+1$, the bottom row is an injection.
Indeed, let $r:A'\rightarrow \Sigma^{j-i}P_B^{\geq -n}$ be a morphism in $\tr(\A)$ which factors through an object $N$ in $\mathcal N$. 
We may assume $N\in \mathcal N^{\leq 0}$. Since we have $n+j-i=1\leq 2$, the Hom space $\Hom_{\tr(\A)}(N, \Sigma^{j-i}P_B^{\geq -n})$ is zero. 
Therefore $r$ is zero in $\tr(\A)$ and the bottom row is an injection.

We are ready to prove that $\alpha$ is bijective.
\begin{itemize}
\item[1)] $\alpha$ is surjective: Let $f:A\rightarrow \Sigma^j P_B^{\geq -n}$ be a morphism in $\D^b(\A)$. We may assume $n\geq -j$ and $n\geq 0$.
Let $i=n+j$. Then $f$ factors through $P_A^{\geq -i}$ in $\D^b(\A)$. 
In the category $\D^b(\A)$, the object $P_A^{\geq -i}$ is isomorphic to $\Sigma^i \Omega^i A$. 
Apply the above bijection to $A'=\Omega^i A$, and we get a morphism $\Omega^i A\rightarrow \Sigma^{j-i}P_B^{\geq -n}$ in $\tr(\A)$. 
This shows $\alpha$ is surjective.

\item[2)] $\alpha$ is injective: Let $g:\Omega^i A\rightarrow \Sigma^{j-i}P_B^{\geq -n}$ be a morphism in $\tr(\A)$ such that $\alpha$ sends the equivalence class of $g$ to zero.
So we may assume $n\geq i-j+1$ and the map 
\[
A\rightarrow P_A^{\geq -i}\xrightarrow{\sim}\Sigma^i\Omega^i A\rightarrow \Sigma^jP_B^{\geq -n}
\]
in $\D^b(\A)$ induced by $g$ is zero. So the map $P_A^{\geq -i}\rightarrow \Sigma^j P_B^{\geq -n}$ factors through $\Sigma Q_A^{\geq -i}$. 
Since we have $\Hom_{\D^b(\A)}(\Sigma Q_{A}^{\geq -i},\Sigma^j P_B^{\geq -n})=0$, the morphism $\Omega^i A\rightarrow \Sigma^{j-i}P_B^{\geq -n}$ is zero in $\D^b(\A)$.

Put $i'=n+j-1$. 
Then the map $g':\Omega^{i'}A\rightarrow \Sigma^{j-i'}P_B^{\geq -n}$ in $\tr(\A)$ obtained from $g$ is zero by the above claim. 
This shows $\alpha$ is injective.
\end{itemize}
\end{proof}
\subsection{Reproduction of exact dg categories}
We show that ``functor dg categories" with exact target and connective source carry canonical exact structures. Namely
\begin{theorem}\label{fun}
Let $\A$ be a small exact dg category. 
If $\B$ is a small connective dg category, then $\rep_{dg}(\B,\A)$ is canonically an exact dg category. 
\end{theorem}
In this subsection, it is convenient to work with $\rep(\Sq)$ instead of $\mathcal H_{3t}(\A)$. 
Recall that by Lemma \ref{equivalences}, we have a fully faithful functor $\mathcal H_{3t}(\A)\hookrightarrow \rep(\Sq)$ whose essential image consists of homotopy bicartesian squares $X$
\[
\begin{tikzcd}
X\ar[r]\ar[d]&Y\ar[d]\\
N\ar[r]&Z
\end{tikzcd}
\]
where $N$ is acyclic.
Then an exact dg structure on an additive dg category $\A$ can be defined to be a class $\mathcal S\subset \rep(\Sq)$ of homotopy bicartesian squares $X$, where $N$ is acyclic, which satisfies similar axioms as those of Definition~\ref{exactdgstructure}.
We will use this definition in the rest of the section.
\begin{remark}
Let $\B$ be the dg $k$-category $k\I$ for a small category $\I$ and $\A$ a Quillen exact category considered as an exact dg category. 
Let us compare the exact dg structure on $\rep_{dg}(\B,\A)$ with the Quillen exact structure on $\Fun(\I,\A)$.

Note that for objects $X$ and $Y$ in $\rep(\B,\A)$, we have 
\[
\Hom_{\D(\A\otimes\B^{op})}(X,\Sigma^i Y)\iso \Hom_{\D(\B^{e})}(\B,\RHom_{\A}(X,\Sigma^i Y))
\]
and the space $\tau_{\leq 0}\RHom_{\A}(X,\Sigma^i Y)$ is zero if $i<0$, since the dg category $\A$ is concentrated in degree $0$.
Thus we have a quasi-equivalence of dg categories
\[
\tau_{\leq 0}\rep_{dg}(\B,\A)\iso H^0(\rep_{dg}(\B,\A))\iso \rep(\B,\A).
\]
Thus the canonical exact dg structure on $\rep_{dg}(\B,\A)$ provided by Theorem \ref{fun} induces a Quillen exact structure on $\rep(\B,\A)$. 
 
 The heart of the canonical t-structure on $\D(\A\otimes \B^{op})$ is equivalent to $\Mod(\A\otimes \B^{op})$.
 Since the objects $X$ and $Y$ both have cohomology concentrated in degree 0, they belong to the heart.
 So they are given by functors $\A^{op}\otimes \B\rightarrow \Mod k$ and furthermore given by functors $F,G:\B\rightarrow \A$ because $X(-,B)$ and $Y(-,B)$ are quasi-representables for each $B\in\B$.
 So we have an inclusion
 \begin{align*}
 \Fun(\I,\A)=\Fun_k(\B,\A)&\hookrightarrow \Fun_{k}(\B,\Mod \A)\iso\Fun_{k}(\A^{op}\otimes \B,\Mod k)\\
 &\iso \Mod(\A\otimes \B^{op})\hookrightarrow \D(\A\otimes \B^{op})
 \end{align*}
whose essential image is exactly $\rep(\B,\A)$.
 
If we identify the categories $\Fun(\I,\A)$ and $\rep(\B,\A)$ via the above functor, we see that the Quillen exact structure on $\rep(\B,\A)$ is identified with the componentwise Quillen exact structure on $\Fun(\I,\A)$.
\end{remark}
 For a dg category $\A$, denote by $\overline{\A}$ the full dg subcategory of $\C_{dg}(\A)$ consisting of cofibrant quasi-representable dg $\A$-modules. 
 The Yoneda dg functor $Y:\A\rightarrow \C_{dg}(\A)$ induces a quasi-equivalence from $\A$ to $\overline{\A}$, which we still denote by $Y$.
 
 We may assume $\B$ is cofibrant as a dg $k$-category. 
 For each $B\in \B$, let $R_B:\rep_{dg}(\B,\A) \rightarrow \overline{\A}$ be the dg functor given by $X\mapsto X(-,B)$. 
 We have a morphism in $\mathrm{Hqe}$ given by the following roof of dg functors 
 \[
 \begin{tikzcd}
 &\overline{\A}&\\
 \rep_{dg}(\B,\A)\ar[ru,"R_B"]&&\A\ar[lu,"Y"swap]
 \end{tikzcd}.
 \]
 It then induces an (isomorphism class of) functor $\rep(k\mathcal I, \rep_{dg}(\B,\A))\rightarrow \rep(k\mathcal{I},\A)$ for each small category $\mathcal I$, which we denote by $F^{\mathcal I}_B$ or $F_B$ if there is no ambiguity.
 \begin{lemma}\label{suff}Let $X$ be an object in $\rep(\mathrm{Sq},\rep_{dg}(\B,\A))$. 
If for each $B\in\B$, $F_B(X)\in \rep(\mathrm{Sq},\A)$ is a homotopy cocartesian (resp.~cartesian) square, then $X$ is also a homotopy cocartesian (resp.~cartesian) square. 
 \end{lemma}
 \begin{proof}
We show the case when $F_B(X)$ is homotopy cocartesian for each $B\in\B$. The other case can be shown similarly. 

By Corollary \ref{rep}, we may suppose $X$ is of the form
 \[
 \begin{tikzcd}
 X_{0}^{\wedge}\ar[r,"f^{\wedge}"]\ar[d]&X_{1}^{\wedge}\ar[d,"j^{\wedge}"]\\
 IX_0^{\wedge}\ar[r]&X_{2}^{\wedge}
 \end{tikzcd}
 \]
 where $X_i\in \rep_{dg}(\B,\A)$ for $i=0,1,2$.
 Consider the following diagram in $\C(\B^{op}\otimes \A)$
 \[
 \begin{tikzcd}
 X_0\ar[r,"f"]\ar[d,"s=\begin{bmatrix}-f\\0\end{bmatrix}"swap]&X_1\ar[r,"\begin{bmatrix}0\\1\end{bmatrix}"]\ar[d,equal]&C(f)\ar[d,"r={[}0{,}j{]}"]\\
 \Sigma^{-1}C(j)\ar[r,"{[}-1{,}0{]}"swap]&X_1\ar[r,"j"swap]&X_2
 \end{tikzcd}.
 \]
 Let $W\in \rep_{dg}(\B,\A)$.
 By assumption, for each pair of objects $B_1,B_2$ in $\B$, and $i\leq 0$, we have
 \[
 \Hom_{\A}(X_2(-,B_1),\Sigma^{i}W(-,B_2))\xrightarrow{\sim}\Hom_{\A}(C(f)(-,B_1),\Sigma^{i}W(-,B_2))
 \]
 and hence for $j\leq 0$ we have
 \[
 \tau_{\leq 0}\RHom_{\A}(X_2,\Sigma^{j}W)\xrightarrow{\sim}\tau_{\leq 0}\RHom_{\A}(C(f),\Sigma^{j}W).
 \]
 Thus we have 
\[
\begin{aligned}
\Hom_{\rep_{dg}(\B,\A)}(X_2^{\wedge}, \Sigma^{j}W^{\wedge})&\xrightarrow{\sim}\Hom_{\B^{op}\otimes \A}(X_2, \Sigma^{j}W)\\
&\xrightarrow{\sim}\Hom_{\B^{e}}(\B,\RHom_{\A}(X_2,\Sigma^jW))\\
&\xleftarrow{\sim}\Hom_{\B^{e}}(\B,\tau_{\leq 0}\RHom_{\A}(X_2,\Sigma^{j}W))\\
&\xrightarrow{\sim}\Hom_{\B^{e}}(\B,\tau_{\leq 0}\RHom_{\A}(C(f),\Sigma^{j}W))\\
&\xrightarrow{\sim}\Hom_{\B^e}(\B,\RHom_{\A}(C(f),\Sigma^{j}W))\\
&\xrightarrow{\sim}\Hom_{\B^{op}\otimes\A}(C(f),\Sigma^{j})\\
&\xrightarrow{\sim}\Hom_{\rep_{dg}}(\B,\A)(C(f^{\wedge}),\Sigma^{j}W^{\wedge}).
\end{aligned}
\]
This shows that $X$ is homotopy cocartesian.
 \end{proof}
 \begin{lemma}\label{pointwise}Let $S$ be an object in $\rep(k\mathrm{Cosp}, \rep_{dg}(\B,\A))$ (resp.~$\rep(k\mathrm{Sp}, \rep_{dg}(\B,\A))$).
If for each $B\in \B$, the object $F_B(S)$ in $\rep(k\mathrm{Cosp}, \A)$ (resp. $\rep(k\mathrm{Sp}, \A)$) admits a homotopy pullback (resp.~homotopy pushout). 
Then $S$ admits a homotopy pullback (resp.~homotopy pushout).
 \end{lemma}
 \begin{proof}We show the case when $S\in\rep(k\mathrm{Cosp}, \rep_{dg}(\B,\A))$. The other case can be shown dually. 
 
 By Corollary \ref{rep}, we may assume $S$ is of the form
 \[
 \begin{tikzcd}
 &X_1^{\wedge}\ar[d,"j^{\wedge}"]\\
 PX_2^{\wedge}\ar[r]&X_2^{\wedge}
 \end{tikzcd}
 \]
 where $X_i\in \rep_{dg}(\B,\A)$, $i=1,2$ and $j$ is a graded-split surjection. 
 
 Since $j$ and $PX_2\rightarrow X_2$ are graded-split surjections, by assumption, for each $B\in\B$, $F_B(S)$ admits a homotopy pullback
 \[
 \begin{tikzcd}
 H(B)\ar[r,""]\ar[d]&X_1(-,B)\ar[d,"j_{(-,B)}"]\\
 PX_2(-,B)\ar[r]&X_2(-,B)
 \end{tikzcd}
 \]
 where $H(B)\in H^0(\A)$.
 
 Let $V$ be the kernel of the map $X_1\oplus PX_2\rightarrow X_2$ in $\C(\B^{op}\otimes \A)$. 
 Then for each $B\in\B$, we have a canonical map $s_B: H(B)\rightarrow V(-,B)$.
 By definition $G(B){\coloneqq}\Sigma\Cone(s_B)$ lies in the heart of the canonical t-structure on $D(\A)$, which is equivalent to $\Mod(H^0(\A))$. 
 Note that the map $\Sigma V(-,B)\rightarrow G(B)$ is well-defined up to a unique isomorphism.
 
 For each morphism $b:B\rightarrow B'$ in $H^0(\B)$, we have a canonical morphism $H(b):H(B)\rightarrow H(B')$ in $H^0(\B)$ which makes the following diagram in $\D(\A)$ commute 
 \[
 \begin{tikzcd}
 H(B)\ar[r]\ar[d]&H(B')\ar[d]\\
 V(-,B)\ar[r]&V(-,B').
 \end{tikzcd}
 \]
and it induces a canonical morphism $G(b):G(B)\rightarrow G(B')$ in $\D(\A)$.

Therefore we have a canonical functor $G: H^0(\B)\rightarrow \Mod(H^0(\A))$ which sends $B$ to $G(B)$ and $b:B\rightarrow B'$ to $G(b):G(B)\rightarrow G(B')$.

By adjunction, $G$ corresponds to an object in $\Mod(H^0(\B)^{op}\otimes H^0(\A))$, which is equivalent to the heart of the canonical t-structure on $\D(\B^{op}\otimes \A)$.
We still denote by $G$ the corresponding object in $\D(\B^{op}\otimes \A)$. We may assume $G$ is cofibrant.

Our next aim is to find a morphism from $\Sigma V$ to $G$ in $\D(\B^{op}\otimes \A)$.
We have the canonical isomorphism
\begin{align}
\Hom_{\D(\B^{op)}\otimes\A}(\Sigma V, G)&\xrightarrow{\sim}\Hom_{\D(\B^{e})}(\B, \RHom_{\A}(\Sigma V, G)).\label{bimodule1}\tag{41.1}
\end{align}
We have that $\RHom_{\A}(\Sigma V, G)$ is concentrated in degrees 0 and 1.
So
\begin{align}
\Hom_{\D(\B^{e})}(\B, \RHom_{\A}(\Sigma V, G))\xrightarrow{\sim} \Hom_{(H^0\B)^e}(H^0(\B), H^0\RHom_{\A}(\Sigma V, G)). \label{bimodule2}\tag{41.2}
\end{align}
For each pair of objects $B,B'\in H^0(\B)$, we have the map
\[
\alpha_{B,B'}:\Hom_{H^0(\B)}(B,B')\rightarrow \Hom_{\D(\A)}(\Sigma V(-,B), G(-,B'))
\]
where $b:B\rightarrow B'$ is sent to $\Sigma V(-,B)\rightarrow G(-,B)\xrightarrow{G(-,b)} G(-,B')$ where the first map is the canonical map.
It is not hard to check 
\[
\alpha_{-,-}: \Hom_{H^0(\B)}(-,-)\rightarrow \Hom_{\A}(\Sigma V, G)
\]
defines an $H^0(\B)-H^0(\B)$-bimodule morphism.  
By \ref{bimodule1} and \ref{bimodule2}, this defines a morphism $u:\Sigma V\rightarrow G$ in $\D(\B^{op}\otimes \A)$. By definition, $u_{(-,B)}$ is the canonical morphism $\Sigma V(-,B)\rightarrow G(-,B)$.

Let $X_0$ be $\Sigma^{-1}\Cone(u)$. 
Then $X_0(-,B)$ is quasi-isomorphic to $H(B)$ for each $B\in\B$ and hence $X_0\in \rep_{dg}(\B, \A)$. 
Consider the composition morphism $X_0\rightarrow V\rightarrow X_1\oplus PX_2$.
The sequence $X_0\rightarrow X_1\oplus PX_2\rightarrow X_2$ then corresponds to a square
\[
\begin{tikzcd}
X_0^{\wedge}\ar[r]\ar[d]&X_1^{\wedge}\ar[d]\\
PX_2^{\wedge}\ar[r]&X_2^{\wedge}
\end{tikzcd}
\]
which is homotopy cocartesian and hence a homotopy pushout of $S$.
\end{proof}
 \begin{proof}[Proof of Theorem \ref{fun}]
 Let $\mathcal S$ be the class of homotopy bicartesian squares \linebreak $X\in \rep(\mathrm{Sq}, \rep_{dg}(\B,\A))$ such that $F_B(X)\in\rep(\mathrm{Sq},\A)$ is a conflation for each $B\in\B$. By Lemma \ref{suff} and Lemma \ref{pointwise}, one checks that $(\rep_{dg}(\B,\A),\mathcal S)$ is an exact dg category.
\end{proof}
\begin{remark}
From the definition of the canonical exact dg structure on $\rep_{dg}(\B,\A)$, it is clear that an exact morphism $F:\A\rightarrow \A'$ in $\Hqe$ induces an exact morphism $F_*:\rep_{dg}(\B,\A)\rightarrow\rep_{dg}(\B,\A')$ in $\Hqe$.
\end{remark}
Let $\A$ be an exact dg category. 
Recall that there is a canonical extriangulated structure $(H^0(\A),\mathbb E,\mathfrak s) $ on $H^0(\A)$.
Let $\{I\xrightarrow{\overline{f}} P\}$ be family of morphisms in $H^0(\A)$ from injectives to projectives. 
Let $\B$ be the dg category $k\Mor=k(0\rightarrow 1)$.
Recall that if a morphism $f:I\rightarrow P$ in $Z^0(\A)$ is homotopic to $f'$, then the objects $f:I\rightarrow P$ and $f':I\rightarrow P$ are isomorphic in $\rep(\B,\A)$. 
Thus we may consider the family of objects $I\xrightarrow{{f}} P$ in $\rep(\B,\A)$. 
We may assume $\{I\rightarrow P\}$ forms an additive subcategory of $\rep(\B,\A)$.  
Let $\J$ be the ideal in $H^0(\A)$ consisting of morphisms which factor through morphisms in the family $\{I\rightarrow P\}$. 

By Theorem \ref{fun}, $\rep_{dg}(\B,\A)$ is canonically an exact dg category. Consider the exact dg substructure of $\rep_{dg}(\B,\A)$ which makes the objects represented by $I\rightarrow P$ projective-injective. We have an equivalence of categories $\rep(\B,\A)\xrightarrow{\sim} H^0(\rep_{dg}(\B,\A))$ where $H^0(\rep_{dg}(\B,\A))$ is identified with the full subcategory of $\D(\rep_{dg}(\B,\A))$ of cofibrant quasi-representable dg modules via the Yoneda functor.
Denote by $\mathbb E'$ the corresponding bifunctor of the extriangulated structure on $\rep(\B,\A)$. 

Denote by $F$ (resp.~$G$) the dg functor $k\rightarrow \B=k(0\rightarrow 1)$, sending the object $*$ to $0$ (resp.~$1$). 
Let $M$ be the $\B$-$k$-bimodule $\Hom_{\B}(F(-), -)$ and $N$ the $k$-$\B$-bimodule $\Hom_{\B}(-, G(-))$. 
Then we have a morphism of $k$-$k$-bimodules $\Hom_{\B}(F(-), G(-))\rightarrow \Hom_{k}(-,-)$ which gives rise to a natural isomorphism $_{\B}M\otimes^{\mathbb L}_{k}-\rightarrow \RHom_{k}(_{k}N_{\B},-)$. 
We denote this (isomorphism class of) functor by $H:\D(\A)\rightarrow \D(\B^{op}\otimes\A)$ which sends $X\in \D(\A)$ to $X\xrightarrow{\Id}X$.  
It restricts to a functor $H^0(\A)\rightarrow \rep(\B,\A)$ which we still denote by $H$.

Let $X\xrightarrow{f} Y$ be an object in $\D(\B^{op}\otimes \A)$ and $Z$, $Z'$ objects in $\D(\A)$.  
Then we have 
\begin{align}
\Hom_{\B^{op}\otimes \A}(H(Z), X\rightarrow Y)\xrightarrow{\sim} \Hom_{\A}(Z, X)\label{Mor1}
\end{align}
 and 
 \begin{align}
 \Hom_{\B^{op}\otimes \A}(X\rightarrow Y, H(Z'))\xrightarrow{\sim}\Hom_{\A}(Y,Z').\label{Mor2}
 \end{align}
 Hence a morphism $H(Z)\rightarrow H(Z')$ factors through the object $X\xrightarrow{f} Y$ if and only if the corresponding morphism $Z\rightarrow Z'$ factors through the morphism $X\xrightarrow{\overline{f}} Y$.

The dg functor $F$ induces a quasi-fully faithful exact morphism  $\A\xrightarrow{\sim}\rep_{dg}(k,\A)\rightarrow\rep_{dg}(\B,\A)$ in $\mathrm{Hqe}$. 
By taking 0-th cohomology, we obtain the (isomorphism class of) functor $H: H^0(\A)\rightarrow \rep(\B,A)$. 

\begin{proposition}\label{inj-proj}
Let $\B$ be the dg category $k\Mor$.
The functor $H:H^0(\A)\rightarrow \rep(\B,\A)$ induces a fully exact embedding $H^0(\A)/\J\rightarrow \rep(\B, \A)/[I\rightarrow P]$.  
\end{proposition}
\begin{proof}
Recall $H^0(\A)$ is identified with the full subcategory of $\D(\A)$ consisting of cofibrant quasi-representable dg modules. By the above discussion the functor $H^0(\A)\rightarrow \rep(\B,\A)$ induces a fully faithful functor $H^0(\A)/\mathcal J\rightarrow \rep(\B,\A)/[I\rightarrow P]$.

Let $X, Z$ be two objects in $H^0(\A)$. 

Claim: The functor $H^0(\A)\rightarrow \rep(\B,\A)$ induces a canonical map 
\[
\alpha_{Z,X}:\mathbb E(Z,X)\rightarrow \mathbb E'(H(Z), H(X))
\]
which is a bijection. 

To see this, first notice that, by \ref{Mor1} and \ref{Mor2}, conflations in $\rep(\Mor,\A)$ which are images of conflations in $\A$ make $I\rightarrow P$ projective-injective. 
So they are in the exact substructure and hence the above is a well-defined map. 
Since $H$ is fully faithful, the map $\alpha_{Z,X}$ is an injection by Proposition \ref{split}. 

Suppose we have an $\mathbb E'$-extension $\delta'\in \mathbb E'(H(Z), H(X))$ which is realized by the following conflation
\[
\begin{tikzcd}
X\ar[d,equal,""{name=1}]&Y\ar[d,""{name=2}]&Z\ar[d,equal,""{name=3}]\\
X&Y'&Z\ar[r,from=1,to=2,blue,shorten >=1.5ex]\ar[r,from=2,to=3,blue,shorten >=1.5ex]
\end{tikzcd}.
\]
Then the map $Y\rightarrow Y'$ is a homotopy equivalence. 

By definition of the exact dg structure on $\rep_{dg}(\B,\A)$, we have an exact functor $G:\rep(\B,\A)\rightarrow H^0(\A)$ sending $X\rightarrow Y$ to $X$. We have that $G\circ H= \Id$. 
The above $\mathbb E'$-extension $\delta'$ is sent under the exact functor $G$ to an $\mathbb E$-extension $\delta\in \mathbb E(Z,X)$ which is realized by 
\[
\begin{tikzcd}
X\ar[r]&{Y}\ar[r]&Z
\end{tikzcd}.
\]
We have a morphism from $H(Y)$ to $Y\rightarrow Y'$ which corresponds to $\Id_Y$ via \ref{Mor1}. 
It is an isomorphism since $Y\rightarrow Y'$ is a homotopy equivalence.
We have the following commutative diagram in $\rep(\B,\A)$
\[
\begin{tikzcd}
H(X)\ar[d,equal]\ar[r] & H(Y)\ar[d]\\
  H(X)              \ar[r]   & (Y\rightarrow Y')
\end{tikzcd}
\]
So $\delta'=H(c)^{*}H(\delta)=H(c^*\delta)$ for some automorphism $c:Z\rightarrow Z$. 
Thus the map $\alpha_{Z,X}$ is a bijection and the functor $H^0(\A)/\mathcal J\rightarrow \rep(\Mor,\A)/[I\rightarrow P]$ is a fully exact embedding.
\end{proof}
Combining Proposition-Definition \ref{algebraic}, Theorem \ref{quot}, Theorem \ref{fun} and Proposition \ref{inj-proj}, we have 
\begin{corollary}\label{algebraic2}
$H^0(\A)/\mathcal J$ is an algebraic extriangulated category.
\end{corollary}

Let $\B$ be a connective dg category. 
Let $F:\A\rightarrow \D^b_{dg}(\A)$ be the universal exact morphism from $\A$ to a pretriangulated dg category. 
The canonical morphism $F_*:\rep_{dg}(\B,\A)\rightarrow\rep_{dg}(\B,\D^b_{dg}(\A))$ induces a canonical morphism in $\Hqe$
\[
\phi: \D^b_{dg}(\rep_{dg}(\B,\A))\rightarrow \rep_{dg}(\B,\D^b_{dg}(\A)).
\]
For objects $X$ and $Y$ in $\rep_{dg}(\B,\A)$, the canonical map
\[
\tau_{\leq 0}\RHom(X,Y)\iso \tau_{\leq 0}\RHom(F_*X,F_*Y).
\]
is an isomorphism in $\D(k)$.
Also, by Corollary \ref{strictify}, the canonical map
\begin{equation}\label{bij:repderived}
\mathbb E(X, Y)\iso \Hom_{\rep(\B, \D^b_{dg}(\A))}(F_*X, \Sigma F_* Y)
\end{equation}
is a bijection. 
We do not know whether the same is true for higher extension groups (and hence do not know whether $H^0(\phi)$ is a triangle equivalence).
It is interesting to study the morphism $\phi$ as it is closely related to the theory of triangle derivators.

Let $(\A,\mathcal S)$ and $(\A',\mathcal S')$ be small exact dg categories. 
Recall that a morphism $\A\rightarrow \A'$ in $\Hqe$ is {exact} if the induced functor $\mathcal H_{3t}(\A)\rightarrow \mathcal H_{3t}(\A')$ sends objects in $\mathcal S$ to objects in $\mathcal S'$.
Recall that we denote by $\Hqe_{ex}(\A,\A')$ the subset of $\Hqe(\A,\A')$ consisiting of exact morphisms. 
Let $\C$ be an exact dg category. 
A morphism $\mu: \A\otimes \A'\rightarrow \C$ in $\Hqe$ is {\em biexact}, if for each object $A\in \A$ and $A'\in \A'$, the induced morphisms 
\[
\mu_{A,-}:\A'\rightarrow \C,\;\;\mu_{-,A'}:\A\rightarrow \C
\]
are both exact morphisms.
Suppose $\A$ and $\A'$ are connective. 
We have universal fully exact embeddings $\A\rightarrow \D^b_{dg}(\A)$ and $\A'\rightarrow \D^b_{dg}(\A')$ into pretriangulated dg categories.
We denote by $\rep_{dg}^{ex}(\A,\A')$ the full dg subcategory of $\rep_{dg}(\A,\A')$ consisting of the exact morphisms.
\begin{lemma}\label{lem:internalhom}
Suppose $\A$ and $\A'$ are small connective exact dg categories.
The dg category $\rep_{dg}^{ex}(\A,\A')$ is stable under extensions in $\rep_{dg}(\A,\A')$. In particular it inherits a canonical exact dg structure.
\end{lemma}
\begin{proof}
We have a canonical bijection
\[
\Hqe_{ex}(\A,\A')\iso \Hqe_{ex}(\A,\D^b_{dg}(\A'))
\]
 and hence by Lemma~\ref{trun} we have a canonical quasi-equivalence
\[
\tau_{\leq 0}\rep_{dg}^{ex}(\A,\A')\iso \tau_{\leq 0}\rep_{dg}^{ex}(\A,\D^b_{dg}(\A')).
\]
Since $\D^b_{dg}(\A')$ is a pretriangulated dg category, 
we have that $\rep^{ex}_{dg}(\A,\D^b_{dg}(\A'))$ is extension-closed in $\rep_{dg}(\A,\D^b_{dg}(\A'))$. 
Therefore by the bijection~\ref{bij:repderived} we have that $\rep_{dg}^{ex}(\A,\A')$ is an 
extension-closed dg subcategory of $\rep_{dg}(\A,\A')$.
\end{proof}

Let $\A\boxtimes \A'$ be the $\tau_{\leq 0 }$-truncation of the extension closure of the quasi-essential image of the morphism in $\Hqe$
\[
F: \A\otimes \A'\rightarrow \D^b_{dg}(\A)\otimes\D^b_{dg}(\A')\rightarrow \pretr(\D^b_{dg}(\A)\otimes\D^b_{dg}(\A'))
\]

\begin{proposition}\label{prop:universalbilinear}
 The natural morphism $\A\otimes \A'\rightarrow \A\boxtimes \A'$ is the universal biexact morphism in $\Hqe$ from $\A\otimes \A'$ to an exact dg category.
\end{proposition}
\begin{proof}
Let $\C$ be an exact dg category and $G:\A\otimes \A'\rightarrow \C$ a biexact morphism in $\Hqe$. 
Since $\A\otimes \A'$ is connective, the morphism $G$ factors through $\tau_{\leq 0}\C$.
We have the following diagram
\[
\begin{tikzcd}
\A\otimes \A'\ar[r,"F"]\ar[d]\ar[dd,bend right=10ex]&\pretr(\D^b_{dg}(\A)\otimes \D^b_{dg}(\A'))\ar[dd,dashed,bend left=10ex]\\
\Im(F)\ar[r]\ar[ru,hook]&\A\boxtimes \A'\ar[u,hook]\ar[ld,dashed,red]\\
\tau_{\leq 0}\C\ar[d,hook]\ar[r,hook]&\D_{dg}(\tau_{\leq 0}\C)\\
\C&
\end{tikzcd}.
\] 
Since the essential image of the inclusion $\tau_{\leq 0}{\C}\rightarrow \D^b_{dg}(\tau_{\leq 0}\C)$ is extension-closed, the conclusion follows immediately.
\end{proof}
\begin{corollary}\label{cor:exactadjunction}
Let $\A$ and $\B$ be small connective exact dg categories. Let $\C$ be a small exact dg category. 
Then we have the following natural bijection
\[
\Hqe_{ex}(\A\boxtimes \B,\C)\iso \Hqe_{ex}(\A,\rep_{dg}^{ex}(\B,\C)).
\]
\end{corollary}
\begin{proof}
The required bijection follows from the following composition of bijections
\[
\Hqe_{ex}(\A\boxtimes\B,\C)\iso \Hqe_{biexact}(\A\otimes \B,\C)\iso\Hqe_{ex}(\A,\rep_{dg}^{ex}(\B,\C)).
\]
\end{proof}
\begin{corollary}Keep the same assumptions as above.
We have the following quasi-equivalence of exact dg categories
\[
\tau_{\leq 0}\rep^{ex}_{dg}(\A\boxtimes\B,\C)\iso\tau_{\leq 0}\rep^{ex}_{dg}(\A,\rep_{dg}^{ex}(\B,\C)).
\]
\end{corollary}
\subsection{Subcategories, closed subbifunctors and exact dg substructures}
Let $(\C,\mathbb E,\mathfrak s)$ be an extriangulated category. 

\begin{definition}[\cite{NakaokaPalu19}, Definition 2.17]
Let $\D\subset \C$ be a full additive subcategory, closed under isomorphisms. 
The subcategory $\D$ is {\em extension-closed} if, for any conflation 
\[
A\rightarrow B\rightarrow C
\]
which satisfies $A, C\in \D$, we have $B\in \D$.
\end{definition}
\begin{definition}[\cite{HerschendLiuNakaoka21}, Lemma 3.15]Let $\mathbb F\subseteq \mathbb E$ be an additive subbifunctor. The following are equivalent.
\begin{itemize}
\item[(1)]$\mathbb F$ is closed on the right, i.e.~for any $\mathfrak s|_{\mathbb F}$-conflation $A\xrightarrow{f} B\xrightarrow{j} C$, the sequence
\[
\mathbb F(-,C)\xRightarrow{f_{*}} \mathbb F(-,B)\xRightarrow{j_{*}} \mathbb F(-,A)
\]
is exact.
\item[(2)]$\mathbb F$ is closed on the left, i.e.~for any $\mathfrak s|_{\mathbb F}$-conflation $A\xrightarrow{f} B\xrightarrow{j} C$, the sequence
\[
\mathbb F(A,-)\xRightarrow{j^{*}} \mathbb F(B,-)\xRightarrow{f^{*}} \mathbb F(C,-)
\]
is exact.

\end{itemize}
Thus, we simply say $\mathbb F\subseteq \mathbb E$ is {\em closed}, if either of the conditions is satisfied. 
\end{definition}
\begin{proposition}[\cite{HerschendLiuNakaoka21}, Proposition 3.16]\label{closedsubbifunctor}
For any additive subbifunctor $\mathbb F\subseteq\mathbb E$, the following are equivalent.
\begin{itemize}
\item[(1)]$(\C,\mathbb F,\mathfrak s|_{\mathbb F})$ is extriangulated.
\item[(2)]$\mathfrak s|_{\mathbb F}$-inflations are closed under composition.
\item[(3)]$\mathfrak s|_{\mathbb F}$-deflations are closed under composition.
\item[(4)]$\mathbb F$ is closed.
\end{itemize}
\end{proposition}
\begin{definition}[\cite{Ogawa21}, Definition 2.4]\label{defect}
Let $(\C,\mathbb E,\mathfrak s)$ be a small extriangulated category. 
Let $\delta\in\mathbb E(C,A)$ be an $\mathbb E$-extension. 
Take a realization $A\xrightarrow{x}B\xrightarrow{y}C$ of $\delta$ 
and define $\tilde{\delta}$ to be the cokernel of $\C(?,y):\C(?,B)\rightarrow \C(?,C)$ in $\Mod \C$. 
The functor $\tilde{\delta}$ is called the defect of $\delta$, 
or the defect of an $\mathfrak s$-conflation $A\xrightarrow{x}B\xrightarrow{y}C$. 
We denote by $\Def \mathbb E$ the subcategory of $\Mod \C$ consisting of $\C$-modules which are isomorphic to defects of some $\mathfrak s$-conflations.
\end{definition}
\begin{remark}\label{rmk:defectses}
Suppose we have an $\mathbb E$-extension $\delta\in\mathbb E(C,A)$ and a morphism $c:C'\rightarrow C$ in $\C$.
By Lemma~\ref{LN}, we have the following diagram
\begin{equation}\label{dia:defectses}
\begin{tikzcd}
A\ar[r,"u"]\ar[d,equal]&E\ar[d,"g"]\ar[r,"v"]&C'\ar[d,"c"]\ar[r,dashed,"c^*(\delta)=\theta"]&\,\\
A\ar[d,"u"swap]\ar[r,"x"]&B\ar[d,"\begin{bmatrix}1\\0\end{bmatrix}"]\ar[r,"y"]\ar[d,""]&C\ar[d,equal]\ar[r,dashed,"\delta"]&\,\\
E\ar[r,"\begin{bmatrix}g\\v\end{bmatrix}"swap]&B\oplus C'\ar[r,"{[}y{,}\,-c{]}"swap]&C\ar[r,dashed,"u_*(\delta)=\mu"swap]&\,
\end{tikzcd}
\end{equation}
A diagram chasing shows that there is a short exact sequence of defects
\[
0\rightarrow \tilde{\theta}\rightarrow \tilde{\delta}\rightarrow \tilde{\mu}\rightarrow 0
\]
\end{remark}
\begin{theorem}[\cite{Enomoto21}, Theorem B]\label{thm:closedsubbifunctorserresubcategory}
Let $(\C,\mathbb E,\mathfrak s)$ be a small extriangulated category. 
Then the map $\mathbb F\mapsto \Def\mathbb F$ defines an isomorphism between the following posets, where the poset structures are given by inclusion.\begin{itemize}
\item[(1)]The poset of closed subbifunctors of $\mathbb E$.
\item[(2)]The poset of Serre subcategories of $\Def\mathbb E$.
\end{itemize}
\end{theorem}

\begin{theorem}\label{bijectionstructures}
Let $(\A,\mathcal S)$ be an exact dg category and $(H^0(\A),\mathbb E,\mathfrak s)$ the associated algebraic extriangulated category. The following posets, where the poset structures are given by inclusion, are isomorphic.
\begin{itemize}
\item[(1)]The poset of exact substructures of $(\A,\mathcal S)$.
\item[(2)]The poset of closed subbifunctors of $\mathbb E$.
\item[(3)]The poset of Serre subcategories of $\Def \mathbb E$.
\end{itemize}
\end{theorem}
\begin{proof}
The isomorphism between items (2) and (3) is given by Theorem~\ref{thm:closedsubbifunctorserresubcategory}.
We show the isomorphism between items (1) and (2).

Let $\mathbb F\subset \mathbb E$ be a closed subbifunctor. 
Let $\mathcal S_{\mathbb F}\subseteq\mathcal S$ be the class of conflations $X$
\[
\begin{tikzcd}
A\ar[r,"f"]\ar[rr,bend right=8ex,"h"swap]&B\ar[r,"j"]&C
\end{tikzcd}
\]  
whose equivalence class $[X]\in \mathbb F(C,A)\subseteq\mathbb E(C,A)$.
We first show that the class $\mathcal S_{\mathbb F}$ is closed under isomorphisms.
Let $X$ be a conflation in $\mathcal S_{\mathbb F}$ and $\alpha:X\rightarrow X'$ an isomorphism in $\mathcal H_{3t}(\A)$.
Then $X'$ is a conflation in $\mathcal S$.
By Lemma \ref{fact}, the isomorphism $\alpha$ factorises as the composition of isomorphisms $X\rightarrow \tilde{X}$ and $\tilde{X}\rightarrow X'$ and $[\tilde{X}]=\overline{a}_*[X]=\overline{c}^*[X']$ for some isomorphisms $\overline a$ and $\overline c$ in $H^0(\A)$. 
So we have $[X']=(\overline{c}^{-1})^*\overline{a}_*[X]$ and $X'$ also belongs to $\mathcal S_{\mathbb F}$.

Since $\mathbb F$ is an additive subbifunctor, Axiom $\Ex0$ for $\mathcal S_{\mathbb F}$ is satisfied.

By Proposition \ref{closedsubbifunctor}, $\mathfrak s|_{\mathbb F}$-deflations are closed under compositions. 
Since we have showed $\mathcal S_{\mathbb F}$ is closed under isomorphisms, deflations are closed under compositions.
Thus Axiom ${\Ex1}$ is satisfied.

Suppose we are in the context of Axiom $\Ex2$. 
So we have a cospan
\[
\begin{tikzcd}
&C'\ar[d,"c"]\\
B\ar[r,"p"swap]&C
\end{tikzcd}
\]
where $p:B\rightarrow C$ is a deflation in $\mathcal S_{\mathbb F}$.
Put $[X']=\overline{c}^*[X]$. 
The morphism from $X'$ to $X$ then restricts to the homotopy pullback of the above cospan.
Since $X'$ also belongs to $\mathcal S_{\mathbb F}$, the homotopy pullback of $p$ is also a deflation in $\mathcal S_{\mathbb F}$.
Thus Axiom ${\Ex2}$ is satisfied.
Dually, one shows Axiom ${\Ex2^{op}}$ for $\mathcal S_{\mathbb F}$.

Let $\mathcal S'\subseteq \mathcal S$ be an exact dg substructure. 
Let $(H^0(\A),\mathbb F_{\mathcal S'},\mathfrak s')$ be the corresponding extriangulated structure.
The inclusion $\mathcal S'\subseteq \mathcal S$ induces a natural transformation of additive bifunctors $\mathbb F_{\mathcal S'}\rightarrow \mathbb E$ which is compatible with the realizations $\mathfrak s$ and $\mathfrak s'$.
By Proposition \ref{split}, the natural transformation $\mathbb F_{\mathcal S'}\rightarrow \mathbb E$ is an inclusion and thus defines a closed subbifunctor of $\mathbb E$ by Proposition \ref{closedsubbifunctor}.

It is direct to check that the maps $\mathbb F\mapsto \mathcal S_{\mathbb F}$ and $\mathcal S'\mapsto \mathbb F_{\mathcal S'}$ are inverse to each other.
\end{proof}
\begin{corollary}\label{cor:exactsubstructure}
Let $(\C,\mathbb E,\mathfrak s)$ be an algebraic extriangulated category.
Let $\mathbb F\subset\mathbb E$ be a closed subbifunctor.
Then the extriangulated category $(\C,\mathbb F,\mathfrak s|_{\mathbb F})$ is still an algebraic extriangulated category.
\end{corollary}
\begin{remark}
We have an injection 
\[
\{\text{Exact dg structures on $\A$}\}\to \{\text{Extriangulated structures on $H^0(\A)$}\}.
\]
This map is in general not a surjection. 
For example, when $\A$ is concentrated in degree $0$, 
then the exact dg structures on $\A$ are exactly the Quillen exact structures on $\A$. 
But $\A$ may have nontrivial extriangulated structures, e.g.~if $\A$ is an algebraic triangulated category.
\end{remark}
\begin{definition}[\cite{IyamaNakaokaPalu18}, Definitions 2.1, 2.6]
Let $(\C,\mathbb E,\mathfrak s)$ be an extriangulated category. 
A non-split (i.e.~non zero) $\mathbb E$-extension $\delta\in\mathbb E(C,A)$ is said to be {\em almost split} if it satisfies the following conditions.
\begin{itemize}
\item[(AS1)] $a_*(\delta)=0$ for any non-section $a\in\C(A,A')$.
\item[(AS2)] $c^{*}(\delta)=0$ for any non-retraction $c\in\C(C',C)$.
A sequence of morphisms $A\xrightarrow{x}B\xrightarrow{y}C$ in $\C$ is called an {\em almost split sequence} if it realizes some almost split extension $\delta\in \mathbb E(C,A)$.
\end{itemize}
\end{definition}

A non-zero object $A\in\C$ is {\em endo-local} if $\End_{\C}(A)$ is local.
For any almost split $\mathbb E$-extension $\delta\in\mathbb E(C,A)$, we have that $A$ and $C$ are both endo-local, cf.~\cite[Proposition 2.5]{IyamaNakaokaPalu18}.
Almost split extensions are unique in the following sense.
\begin{proposition}[\cite{IyamaNakaokaPalu18}, Proposition 2.11]
Let $A\in\C$ be any object. If there are non-split extensions $\delta\in\mathbb E(C,A)$ and $\delta'\in\mathbb E(C',A)$ satisfying (AS2), then there is an isomorphism $c\in\C(C',C)$ such that $c^{*}(\delta)=\delta'$.
Dually, if non-split extensions $\rho\in\mathbb E(C,A)$ and $\rho'\in\mathbb E(C,A')$ satisfy (AS1) for some $A$, $A'$, $C\in \C$, then there is an isomorphism $a\in\C(A,A')$ such that $a_{*}(\rho)=\rho'$.
\end{proposition}

\begin{definition}[\cite{IyamaNakaokaPalu18}, Definition 3.1]
We say that $\C$ has {\em right almost split extensions} if for any endo-local non-projective object $A\in\C$, there exists an almost split extension $\delta\in\mathbb E(A,B)$ for some $B\in\C$.
Dually, we say that $\C$ has {\em left almost split extensions} if if for any endo-local non-injective object $B\in\C$, there exists an almost split extension $\delta\in\mathbb E(A,B)$ for some $A\in\C$.
We say that $\C$ has {\em almost split extensions} if it has right and left almost split extensions.

\end{definition}
Let $k$ be a field.  Suppose $\C$ is Krull--Schmidt.
Let $A\xrightarrow{x}B\xrightarrow{y}C$ be a conflation realizing an almost split $\mathbb E$-extension $\delta\in\mathbb E(C,A)$.
Then its defect $\tilde{\delta}$, i.e.~the cokernel of $\C(?,y):\C(?,B)\rightarrow \C(?,C)$ in $\Mod\C$, is the simple module at $C$.
So we have $\tilde{\delta}(C)=\End_{\C}(C)/\rad\End_{\C}(C)$ and $\tilde{\delta}(D)=0$ for any indecomposable $D\in\C$ which is not isomorphic to $C$.
Conversely, if a simple module $S_{C}$ lies in $\Def \mathbb E$ for an indecomposable object $C\in\C$, then by definition there is a defect $\delta\in \mathbb E(C,A)$ such that $\tilde{\delta}$ is isomorphic to $S_C$. 
It is clear that $\delta$ satisfies (AS2).
By \cite[Lemma 2.14]{IyamaNakaokaPalu18}, we may assume $A$ is indecomposable and then $\delta$ is an almost split extension, whose defect is necessarily isomorphic to $S_C$.
The following corollary generalizes an analogous result in exact categories, cf.~\cite[Theorem 2.26]{FangGorsky22}.
\begin{corollary}
Let $k$ be a field. 
Let $(\A,\mathcal S)$ be an exact dg category such that the following conditions hold
\begin{itemize}
\item[1)] The $k$-category $H^0(\A)$ is Krull--Schmidt.
\item[2)] The extriangulated category $(H^0(\A),\mathbb E,\mathfrak s)$ is {\em admissible}, i.e.~every object in $\Def \mathbb E$ has finite length (cf.~\cite{Enomoto18}).
\end{itemize}
Then the lattice of exact substructures of $(\A,\mathcal S)$ is isomorphic to the Boolean lattice of the subsets of the set of isomorphism classes of indecomposable non-projective objects in $(H^0(\A),\mathbb E,\mathfrak s)$.
\end{corollary}
\begin{proof}
Let $\M$ be a set of representatives for each isomorphism class of non-projective indecomposable objects in $\C=H^0(\A)$. 
Let us show that for each $C\in \M$, there exists an almost split extension with ending term $C$.
Indeed, since $C$ is non-projective and $\C$ is Krull--Schmidt, there exists a non-zero extension $\delta\in\mathbb E(C,A)$ with $A$ indecomposable.
By item 2), each defect has finite length. Let us choose $\delta$ such that the length of $\tilde{\delta}$ is minimal among such extensions.
Let $c:C'\rightarrow C$ be a non-retraction morphism in $\C$.
Let us consider the diagram~\ref{dia:defectses} in Remark~\ref{rmk:defectses}. 
So there is a surjection from $\tilde{\delta}$ to $\tilde{\mu}$ where $\mu=u_*(\delta)$ for some morphism $u:A\rightarrow E$.
Since $C$ is indecomposable and $\delta$ is non-zero and $c$ is a non-retraction, the $\mathbb E$-extension $\mu$ is nonzero. 
We decompose $E$ into indecomposables $E=\oplus_{i} E_i$.
Then for some $p_i:E\rightarrow E_i$, we have $\mu_i=(p_i)_*(\mu)\neq 0$.
There is a surjection $\tilde{\mu}\rightarrow \tilde{\mu_i}$. So the composition $\tilde{\delta}\twoheadrightarrow \tilde{\mu}\twoheadrightarrow\tilde{\mu_i}$ is an isomorphism by the minimality of the length of $\tilde{\delta}$.
Then $\tilde{\delta}$ is isomorphic to $\tilde{\mu}$ and hence $\tilde{\theta}=0$. 
It follows that $c^*(\theta)=0$ and hence we have that $\tilde{\delta}$ is isomorphic to $S_C$ and that $\delta$ is an almost split extension.

We choose an almost split extension $\delta_C$ for each $C\in\M$ and put $\mathcal R=\{\delta_C,C\in\M\}$.
Since any defect takes zero value at each projective object in $\C$, 
the category $\Def\mathbb E$ can be identified with a full subcategory of $\Mod (\add \M)$.  
The defects $\tilde{\delta_{C}}$ are the simple modules in $\Mod (\add \M)$ and hence we have that $\Def \mathbb E$ is the full subcategory of $\Mod (\add\M)$ consisting of objects which have finite lengths. 
Recall that in a small finite length abelian category, the lattice of Serre subcategories is isomorphic to the lattice of the subsets of the set of isomorphism classes of simple objects.
Therefore, by Theorem~\ref{bijectionstructures}, the lattice of exact substructures of $(\A,\mathcal S)$ is isomorphic to the Boolean lattice of the subsets of the set of isomorphism classes of indecomposable non-projective objects in $\C$.
\end{proof}
\begin{example}
Let $A$ be the path $k$-algebra $k(\overrightarrow{A_2})$ of the quiver $A_2$. 
Let $\C=\mathcal H^{[-1,0]}(\proj A)$ be the full subcategory of $\T=\mathcal H^b(\proj A)$ consisting of two-term complexes $P^{-1} \to P^0$ of finitely generated projective $A$-modules. 
Let $\A=\C^{b}_{dg}(\proj A)$ be the canonical enhancement of
${\mathcal H}^b(\proj A)$ and $\A'$ its full dg subcategory on the objects $P^{-1} \to P^0$.
Then $\A'$ is stable under extensions in the pretriangulated dg category $\A$ and thus inherits a canonical exact dg structure.
Since the algebra $A$ is hereditary, it is the greatest exact dg structure on $\A'$, cf.~Theorem~\ref{thm:maximalexactdgstructure}. 
Let $(\C=H^0(\A'),\mathbb E,\mathfrak s)$ be the corresponding extriangulated category.
We denote by $S_2$ the two-term complex $P_1\rightarrow P_2$. 
Then the AR-quiver of $\C$ is as follows
\[
\begin{tikzcd}
&P_2\ar[rd]&&\Sigma P_1\ar[rd]\ar[ll,dashed,no head]&\\
P_1\ar[ru]&&S_2\ar[ru]\ar[ll,dashed,no head]&&\Sigma P_2\ar[ll,dashed,no head]
\end{tikzcd}
\]
where $P_1$ and $P_2$ are projective objects in $\C$.
We take the following AR-sequences
\[
\begin{tikzcd}
\alpha: P_2\ar[r,tail]&S_2\ar[r,two heads]&\Sigma P_1,
\end{tikzcd}
\]
\[
\begin{tikzcd}
\beta: P_1\ar[r,tail]&P_2\ar[r,two heads]&S_2,
\end{tikzcd}
\]
\[
\begin{tikzcd}
\gamma:S_2\ar[r,tail]&\Sigma P_1\ar[r,two heads]&\Sigma P_2.
\end{tikzcd}
\]
Then the lattice of exact structures on $\A'$ is isomorphic to the following lattice
\[
\begin{tikzcd}
&\{\alpha,\beta,\gamma\}&\\
\{\alpha,\beta\}\ar[ru]&&\{\beta,\gamma\}\ar[lu]\\
&\{\alpha,\gamma\}\ar[uu,dashed]&\\
\{\alpha\}\ar[uu]\ar[ru,dashed]&\{\beta\}\ar[luu]\ar[ruu]&\{\gamma\}\ar[uu]\ar[lu,dashed]\\
&\emptyset\ar[u]\ar[lu]\ar[ru]&
\end{tikzcd}
\]
We adopt the notation in Lemma~\ref{extensionclosed}.
Then we have the following correspondence between closed subbifunctors and exact structures.
\[
\begin{tikzcd}
 \mathbb E_{\add{\Sigma P_1}}=\mathbb E^{\add{P_1}}\ar[r,leftrightarrow,dashed]& \{\alpha,\beta\}&\mathbb E^{\add{S_2}}\ar[r,leftrightarrow,dashed]& \{\gamma\} \\
  \mathbb E_{\add \Sigma P_2}=\mathbb E^{\add{P_2}}\ar[r,leftrightarrow,dashed]& \{\alpha,\gamma\} & \mathbb E_{\add S_2}\ar[r,leftrightarrow,dashed]& \{\beta\}\\
     \mathbb E_{0}=\mathbb E^0\ar[r,leftrightarrow,dashed]& \emptyset & \mathbb E_{\C}\ar[r,leftrightarrow,dashed]& \{\alpha,\beta,\gamma\}
 \end{tikzcd}
 \]
\end{example}

\subsection{DG relative cluster categories and Higgs categories}\label{subsection:higgscategories}
One of the motivations for our work is Yilin Wu's construction of the Higgs category associated with a finite ice quiver with potential (see \cite{Wu23a}\cite{Wu21} for the Jacobi-finite case and \cite{KellerWu22} for the Jacobi-infinite case). 

Let $(Q,F,W)$ be a Jacobi-finite ice quiver with potential. 
Denote by $\Gamma_{rel}$ the relative Ginzburg dg algebra $\Gamma_{rel}(Q,F,W)$. 
Let $e=\sum_{i\in F}e_i$ be the idempotent associated with all frozen vertices. 
Let $\pvd_{e}(\Gamma_{rel})$ be the full subcategory of $\pvd(\Gamma_{rel})$ of the dg $\Gamma_{rel}$-modules whose restriction to frozen vertices is acyclic.

Then the {\em relative cluster category} $\C=\C(Q,F,W)$ associated to $(Q,F,W)$ is defined as the Verdier quotient of triangulated categories
\[
\per(\Gamma_{rel})/\pvd_{e}(\Gamma_{rel}).
\]
The {\em relative fundamental domain} $\mathcal F^{rel}_{\Gamma_{rel}}=\mathcal F^{rel}$ associated to $(Q,F,W)$ is defined as the following subcategory of $\per(\Gamma_{rel})$
\[
\mathcal F^{rel}{\coloneqq}\{\Cone(X_1\xrightarrow{f}X_0) | X_i\in\add(\Gamma_{rel}) \text{ and } \Hom(f,I) \text{ is surjective}, \forall I\in \mathcal P{\coloneqq}\add(e\Gamma_{rel})\}.
\]
Let $\pi^{rel}:\per(\Gamma_{rel})\rightarrow \C$ be the canonical quotient functor. 

By \cite[Proposition 5.20]{Wu23a}, the restriction of the quotient functor $\pi^{rel}$ to $\mathcal F^{rel}$ is fully faithful. 
In particular for $X, Y\in \mathcal F^{rel}$, $\pi^{rel}$ induces an injection 
\[
\Ext^1_{\per(\Gamma_{rel})}(X,Y)\rightarrow\Ext^1_{\C}(\pi^{rel}X,\pi^{rel}Y).
\]
The {\em Higgs category} (\cite[Definition 5.22]{Wu23a}) $\mathcal H$ is the image of $\mathcal F^{rel}$ in $\C$ under the quotient functor $\pi^{rel}$. 
Indeed, we have the following
\begin{lemma}\label{equi}Let $X,Y$ be objects in $\mathcal F^{rel}$.
Then $\pi^{rel}$ induces a canonical isomorphism 
\[
\Hom_{\per(\Gamma_{rel})}(X,\Sigma^{i}Y)\xrightarrow{\sim}\Hom_{\C}(\pi^{rel}X,\Sigma^{i}\pi^{rel}Y).
\] for any $i\leq 0$ or for $Y\in \add\Gamma_{rel}$ and $i\leq 1$.

\end{lemma}
\begin{proof}We first show the case when $Y=\Gamma_{rel}$ and $i\leq 1$. 

Recall the relative t-structure $(\D(\Gamma_{rel})^{rel}_{\leq 0}, \D(\Gamma_{rel})^{rel}_{\geq 0})$ on $\D(\Gamma_{rel})$ which restricts to t-structures on $\per \Gamma_{rel}$ and on $\pvd_{e}(\Gamma_{rel})$, cf.~\cite[Propositions 4.10, 4.11]{Wu23a}.

We show the injectivity of the map.
Let $f:X\rightarrow \Sigma^{i}\Gamma_{rel}$ be a map in $\per(\Gamma_{rel})$ which factors through an object $M$ in $\pvd_{e}(\Gamma_{rel})$. 

We want to show that $f$ is indeed a zero map.
By definition $\Gamma_{rel}\in \per(\Gamma_{rel})^{rel}_{\leq 0}$ and hence $X\in \add \Gamma_{rel}*\add \Sigma\Gamma_{rel}\subset \per(\Gamma_{rel})^{rel}_{\leq 0}$.
So we may suppose $M\in \pvd_{e}(\Gamma_{rel})^{rel}_{\leq 0}$.

By \cite[Corollary 3.13]{Wu23a}, we have 
\[
\Hom_{\Gamma_{rel}}(M,\Sigma^i\Gamma_{rel})\xrightarrow{\sim}D\Hom_{\Gamma_{rel}}(\Gamma_{rel}, \Sigma^{3-i}M)=0
\] 
where $D=\Hom(-,k)$. 

This shows the injectivity of the morphism. 
We now show the surjectivity of the map. Consider a roof $b/s$ 
\[
\begin{tikzcd}
&N\ar[ld,Rightarrow,"s"swap]\ar[rd,"b"]&\\
X&&\Sigma^i\Gamma_{rel}
\end{tikzcd}.
\]
Since $X\in\per(\Gamma_{rel})^{rel}_{\leq 0}$, we may assume $\Cone(s)\in \pvd(\Gamma_{rel})^{rel}_{\leq 0}$.
Then since we have
\[
\Hom_{\Gamma_{rel}}(\Sigma^{-1}\Cone(s),\Sigma^{i} \Gamma_{rel})\iso D\Hom_{\Gamma_{rel}}(\Gamma,\Sigma^{2-i} \Cone(s))=0,
\] 
the morphism $b$ factors through $s$ and this shows the surjectivity of the map.

For general $Y\in \mathcal F^{rel}$ and $i\leq 0$, it follows from the above case by applying the homological functor $\Hom(X,-)$ to the defining triangle for $Y$.
\end{proof}
Let $\per_{dg}(\Gamma_{rel})$ (resp. $\pvd_{e}(\Gamma_{rel})_{dg}$) be the canonical dg enhancement of $\per(\Gamma_{rel})$ (resp. $\pvd_{e}(\Gamma_{rel})$). 
Let $\C_{dg}$ be the canonical dg enhancement of $\C=\C(Q,F,W)$ given by the dg quotient $\per_{dg}(\Gamma_{rel})/\pvd_{e}(\Gamma_{rel})_{dg}$. 
Let $\mathcal H_{dg}$ be the full dg subcategory of $\C_{dg}$ consisting of objects in $\mathcal H$.
Let $\F^{rel}_{dg}$ be the dg full subcategory of $\per_{dg}(\Gamma_{rel})$ consisting of objects in $\F^{rel}$.

By \cite[Definition 5.12, Proposition 5.39]{Wu23a}, the dg category $\F^{rel}_{dg}$ (resp. $\mathcal H_{dg}$) is extension-closed in $\per_{dg}(\Gamma_{rel})$ (resp. $\C_{dg}$) and thus inherits a canonical exact dg structure. 
By Lemma \ref{equi}, the canonical morphism from $\per_{dg}(\Gamma_{rel})$ to $\C_{dg}$ induces a quasi-equivalence 
\[
\tau_{\leq 0}\F_{dg}^{rel}\rightarrow \tau_{\leq 0} \mathcal H_{dg}.
\]

By \cite[Section 5.8]{Wu23a}, Theorem \ref{main} and Proposition \ref{higher}, the canonical morphism 
\[
\tau_{\leq 0}\mathcal H_{dg}\rightarrow \C_{dg}
\]
characterises $\C_{dg}$ as the bounded dg derived category of $\tau_{\leq 0}\mathcal H_{dg}$.

Since $\Gamma_{rel}$ is projective in the extriangulated category $\F^{rel}$, it follows that higher extensions $\mathbb E^i$ of $\F^{rel}$ vanish for $i\geq 2$.
Also we have $\Ext^{i}_{\per(\Gamma_{rel})}(X,Y)=0$ for $X,Y\in\F^{rel}$ and $i\geq 2$.
Thus by  Theorem \ref{main} and Proposition \ref{higher}, the canonical morphism 
\[
\tau_{\leq 0}\mathcal F^{rel}_{dg}\rightarrow \per_{dg}(\Gamma_{rel})
\]
characterises $\per_{dg}(\Gamma_{rel})$ as the bounded dg derived category of $\tau_{\leq 0}\mathcal F^{rel}_{dg}$.

Claim: Via the canonical quasi-equivalence $\tau_{\leq 0}\F^{rel}_{dg}\rightarrow \tau_{\leq 0} \mathcal H_{dg}$, the exact dg structure on $\tau_{\leq 0}\F^{rel}_{dg}$ is identified with the exact substructure on $\tau_{\leq 0}\mathcal H_{dg}$ left generated by $\add \Gamma_{rel}$.

Indeed, let $g:\pi^{rel}X\rightarrow \Sigma \pi^{rel}Y$ be a morphism in $\mathcal H$. 
Then by Lemma \ref{equi}, $g$ is of the form $\pi^{rel}(f)$ for some $f:X\rightarrow \Sigma Y$ if and only if $g$ factors though an object in $\add \Gamma_{rel}$. 

Note that $\Gamma_{rel}$ is cluster-tilting in $\mathcal H$, cf.~\cite[Proposition 5.49]{Wu23a}. In general we have 

\begin{proposition}\label{ct}
Let $(\A,\mathcal S)$ be a connective exact dg category such that $(H^0(\A),\mathbb E)$ is a Hom-finite, Krull--Schmidt extriangulated category with a basic cluster-tilting object $T$. 
Let $\Gamma$ be the dg endomorphism algebra of $T$ in $\A$. 
Let $(\A,\mathcal S')$ be the connective exact dg category with exact substructure left generated by $\add T$.
Let $\D^b_{\mathcal S'}(\A){\coloneqq}H^0(\D^b_{dg}(\A))$ be the bounded derived category of $\A$ with the exact structure $\mathcal S'$.
Then we have an equivalence of triangulated categories up to direct summand
\[
\D^b_{\mathcal S'}(\A)\rightarrow \per\Gamma.
\]  
\end{proposition}
\begin{proof}The dg functor 
\[
\Hom(T,-):\D^b_{dg}(\A)\rightarrow \C_{dg}(\Gamma)
\]
 induces a functor $D^b_{\mathcal S'}(\A)\rightarrow \per\Gamma$. 
Notice that $T$ is projective in $(H^0(\A),\mathbb E')$ with $\mathbb E'$ given by the exact structure $\mathcal S'$.

Since $T$ is cluster-tilting, the extriangulated category $(H^0(\A),\mathbb E')$ has enough projective objects $\add T$ and is hereditary in the sense of \cite{GorskyNakaokaPalu23}. 
Also, $T$ generates $\D^b_{\mathcal S'}(\A)$ as a triangulated category.

Now the proposition follows from Theorem \ref{main} and Proposition \ref{higher}, by a d\'evissage argument, cf.~\cite[Lemma 4.2]{Keller94}\cite[1.3]{Keller98c}.
\end{proof}
\begin{remark}In Proposition \ref{ct}, if we take different basic cluster-tilting objects $T$ in $(H^0(\A),\mathbb E)$, then the exact substructures thus constructed are different because they have different projective objects.
\end{remark}
We end this section by constructing a class of homotopy bicartesian squares in the dg category $\mathcal F^{rel}_{dg}$ (or, by Remark \ref{truncationexactdgstructure}, equivalently in the dg category $\tau_{\leq 0}\F^{rel}_{dg}$) which are not necessarily homotopy bicartesian in the ambient pretriangulated dg category $\per_{dg}(\Gamma_{rel})$.
The idea is that we have a quasi-equivalence of dg categories
\[
\tau_{\leq 0}\F^{rel}_{dg}\simeq\tau_{\leq 0}\mathcal H_{dg}
\]
which identifies the exact structure on the left hand side with a substructure of that on the right hand side. 
So we only need to find conflations in $\tau_{\leq 0}\mathcal H_{dg}$ which are not in $\tau_{\leq 0}\F^{rel}_{dg}$.

Let $M$ be an object in $\mathcal F^{rel}$ of the form $\Cone(P\xrightarrow{f}Q)$. 
Denote by $M^{*}$ the dual $\RHom_{\Gamma_{rel}}(M,\Gamma_{rel})\in \per(\Gamma^{op}_{rel})$ of $M$.
Then we have $H^iM^{*}=0$ for $i\geq 2$ and that $H^1M^{*}$ is finite-dimensional.
Thus $N=\tau_{\leq 0}M^{*}$ lies in $\per(\Gamma^{op}_{rel})$. 

Claim 3: We have $H^1M^{*}\in \pvd_{e}(\Gamma^{op}_{rel})$ and $M\xrightarrow{\sim} \tau_{\leq 0}N^{*}$.

Indeed, we apply the cohomological functor $\Hom(-,\Gamma_{rel})$ to the triangle 
\[
P\xrightarrow{f}Q\rightarrow M\rightarrow \Sigma P
\]
and we get the following exact sequence
\[
\begin{tikzcd}
\Hom(\Sigma Q, \Sigma e\Gamma_{rel})\ar[r, two heads]&\Hom(\Sigma P,\Sigma e\Gamma_{rel})\ar[r]&\Hom(M,\Sigma e\Gamma_{rel})\ar[r]&\Hom(Q,\Sigma e\Gamma_{rel})=0.
\end{tikzcd}
\]
Thus we have $eH^1M^*=\Hom(M,\Sigma e\Gamma_{rel})=0$ and $H^1M^*\in \pvd_{e}(\Gamma^{op}_{rel})$.
We have a canonical triangle
\[
\begin{tikzcd}
(\Sigma^{-1}H^1M^{*})^{*}\ar[r]&M^{**}\ar[r]& N^{*}\ar[r]&\Sigma(\Sigma^{-1}H^1M^*)^*.
\end{tikzcd}
\]
Then the claim follows by observing $(H^1M^*)^*\in \D(\Gamma^{op}_{rel})_{\geq 3}$. 
Indeed, we have 
\[
(H^1M^*)^*\in \D(\Gamma^{op}_{rel})_{\geq 3}\cap\D(\Gamma^{op}_{rel})_{\leq 3}
\]
which implies that $N^{*}\in \D(\Gamma_{rel}^{op})_{\leq 1}^{rel}$.

Claim 4: We have $\Hom(N,\Sigma^{i} \Gamma_{rel}e)\xrightarrow{\sim}H^i(M)e=0$ for $i\geq 1$ and $\Hom(N,T)\xrightarrow{\sim} \Hom(H^1M^{*}, \Sigma^2 T)$ for $T\in \D(\Gamma^{op}_{rel})_{\leq -2}$.

Indeed, the claim follows by applying the cohomological functors $\Hom(-,\Sigma \Gamma_{rel}e)$ and $\Hom(-,T)$ to the triangle in $\per(\Gamma_{rel}^{op})$
\[
N=\tau_{\leq 0}M^{*}\rightarrow M^{*}\rightarrow \Sigma^{-1}H^1M^{*}\rightarrow \Sigma N.
\]
We also have $\Hom(\Gamma_{rel}e,\Sigma^{\geq 1}N)=0$.

Summarising the above discussion, we have, by \cite[Definition 5.12]{Wu23a}, that $N\in \F^{rel}_{\Gamma^{op}_{rel}}$ if and only if $\Hom(H^1M^{*},T)=0$ for any $T\in \D(\Gamma^{op}_{rel})_{\leq -4}$.

\begin{proposition}
Keep the notations as above. 
Suppose that $N$ lies in $\mathcal F^{rel}_{\Gamma^{op}_{rel}}$ and has a defining triangle
\[
\begin{tikzcd}
R\ar[r,"g"] &S\ar[r]&N\ar[r]&\Sigma R
\end{tikzcd}
\]
where $R$ and $S$ belongs to $\add\Gamma_{rel}$ and the morphism $g$ induces a surjection 
\[
\Hom_{\Gamma^{op}_{rel}}(S,e\Gamma^{op}_{rel})\rightarrow \Hom_{\Gamma^{op}_{rel}}(R,e\Gamma^{op}_{rel}).
\] 
Then the following diagram in $\F^{rel}\subset \D(\F^{rel}_{dg})$
\begin{equation}\label{diagram}
\begin{tikzcd}
M\ar[r,"u"]\ar[d]&S^{*}\ar[d,"g^{*}"]\\
0\ar[r]& R^*
\end{tikzcd},
\end{equation}
where the map $u$ is given by 
\[
M\xrightarrow{\sim}\tau_{\leq 0}N^*\rightarrow N^*\rightarrow S^{*},
\]
lifts to an object in $\rep(\mathrm{Sq},\F^{rel}_{dg})$, which is homotopy bicartesian, under the canonical functor $\Dia:\D((k\mathrm{Sq})^{op}\otimes \F^{rel}_{dg})\rightarrow \Fun(\mathrm{Sq}, \D(\F^{rel}_{dg}))$. 
Moreover, it is homotopy bicartesian with respect to $\per_{dg}(\Gamma_{rel})$ if and only if $\Hom(M,\Sigma\Gamma_{rel})=0$.
\end{proposition}
\begin{proof}
We may assume $N^*$ is strictly concentrated in degrees $\leq 1$ and $M=\tau_{\leq 0} N^{*}$. 
We also assume the triangle 
\[
N^{*}\rightarrow S^*\rightarrow R^*\rightarrow \Sigma N^*
\]
is induced by a short exact sequence of dg $\Gamma^{op}_{rel}$-modules
\[
0\rightarrow N^*\rightarrow S^*\rightarrow R^*\rightarrow 0.
\]
Then we have the following commutative diagram in $\C(\Gamma_{rel}^{op})$
\[
\begin{tikzcd}
M\ar[r,tail]\ar[d,equal] &N^*\ar[r,two heads]\ar[d,tail]& \Sigma^{-1} H^1N^*\ar[d,tail]\\
M\ar[r,tail]&S^*\ar[r, two heads]\ar[d,two heads]&U\ar[d,two heads]\\
&R^*\ar[r,equal]& R^*
\end{tikzcd}.
\]
Since $\Hom(H^1N^*,\Sigma^{\leq 1}X)=0$ and $\Hom(X,\Sigma^{\leq -1}H^1N^*)=0$ for $X\in \F^{rel}_{\Gamma_{rel}^{op}}$, the following object in $\rep(\mathrm{Sq},\F^{rel}_{\Gamma^{op}_{rel}})$
\[
\begin{tikzcd}
M\ar[r]\ar[d]&S^*\ar[d]\\
0\ar[r]&R^*
\end{tikzcd}
\] 
is a lift of the diagram \ref{diagram} which is homotopy bicartesian.
\end{proof}
\begin{example}Let $(Q,F,W)$ be the following ice quiver with potential
\[
\begin{tikzcd}
&4\ar[rd,"\beta"]&\\
\textcolor{blue}3\ar[ru,"\alpha"]&&\textcolor{blue}1\ar[ld,"\gamma"blue, blue]\\
&\textcolor{blue}2\ar[lu,"\delta"blue, blue]&
\end{tikzcd}
\]
where the ice part $F$ is given by the blue vertices and blue arrows and the potential $W=\delta\gamma\beta\alpha$.
 Then the underlying graded quiver of the corresponding relative Ginzburg algebra $\Gamma_{rel}(Q,F,W)$ is given as follows
 \[
 \begin{tikzcd}
 &4\ar[rd,"\beta", shift left=1ex]\ar[loop, looseness=4, "t_4"swap, green]\ar[ld,"\alpha^{*}",shift left=1ex,red]&\\
\textcolor{blue}3\ar[ru,"\alpha", shift left =1ex]&&\textcolor{blue}1\ar[ld,"\gamma"blue, blue]\ar[lu,"\beta^{*}", shift left=1ex,red]\\
&\textcolor{blue}2\ar[lu,"\delta"blue, blue]&
 \end{tikzcd}
 \]
 where $|\alpha^*|=-1$, $|\beta^{*}|=-1$ and $|t_4|=-2$. The differential $d$ is determined by the values on generators, which are given as follows
 \begin{gather*}
 d(\alpha)=d(\beta)=d(\gamma)=d(\delta)=0,\\
  d(\alpha^*)=\delta\gamma\beta, d(\beta^*)=\alpha\beta\gamma, d(t_4)=\alpha\alpha^*-\beta^*\beta.
 \end{gather*}
 
 We have $e=e_1+e_2+e_3$. Denote by $P_i= e_i\Gamma_{rel}$.
 
 Let $M=S_1$. Then we have the following triangle
 \[
 P_4\rightarrow P_1\rightarrow S_1\rightarrow \Sigma P_4
 \]
 such that 
 \[
 \Hom(P_1, e\Gamma_{rel})\twoheadrightarrow \Hom(P_4,e\Gamma_{rel}).
 \]
 So $M\in \F^{rel}$.
 
 Let $L$ be the homotopy fiber of the map $P_3\rightarrow P_4$ induced by $\alpha$. 
 Then we have a triangle
 \[
 D\rightarrow L\rightarrow \Sigma^{-1} S_4\rightarrow \Sigma D
 \]
 where $D$ is the mapping cone of the map $P_4\rightarrow P_1$ induced by $\beta$, which is isomorphic to $S_1$.

 So $N=\tau_{\leq 0} M^*$ is isomorphic to $S_3'$. 
 Similarly we have the following triangle
  \[
 P_4^*\rightarrow P_3^*\rightarrow S_3'\rightarrow \Sigma P_4^*
 \]
 which shows that $N\in \F^{rel}$.
 
 We thus have the following diagram in $\F^{rel}$
 \[
 \begin{tikzcd}
 S_1\ar[r]\ar[d]& P_3\ar[d]\\
 0\ar[r] &P_4
 \end{tikzcd}
 \]
 which can be lifted to a homotopy bicartesian square in $\F^{rel}_{dg}$.
 \end{example}

\newpage
\section{The greatest exact structure on an additive dg category}
We follow Rump's idea \cite{Rump11,Rump15} to show the existence of the greatest exact structure. 

Let $\A$ be an additive connective dg category. For simplicity, we may assume that $Z^0(\A)$ is an additive category, cf.~Remark \ref{truncationexactdgstructure}.
\begin{definition}\label{def:leftstructure}
A {\em left exact structure} on $\A$ is a class $\mathcal{S}\subseteq \mathcal H_{3t}(\A)$ stable under isomorphisms, consisting of homotopy short exact sequences (called {\em conflations})
\[
\begin{tikzcd}
A\ar[r, tail,"i"]\ar[rr,bend right=8ex,"h"swap]&B\ar[r,two heads, "p"]&C\\
\end{tikzcd} 
\]
where $i$ is called an {\em inflation} and $p$ is called a {\em deflation}, such that the following axioms are satisfied
\begin{itemize}
\item[Ex0]$\Id_{0}$ is a deflation.
\item[{Ex}1]Compositions of deflations are deflations.
\item[{Ex}2]Given a deflation $p:B\rightarrow C$ and any map $c: C'\rightarrow C$ in $Z^0(\A)$, the object
\[
\begin{tikzcd}
&C'\ar[d,"c"]\\
B\ar[r,"p"swap]&C
\end{tikzcd}
\]
admits a homotopy pullback 
\[
\begin{tikzcd} 
{B'}\ar[r,"{p'}"]\ar[d,"{b}"swap]\ar[rd,"s"blue,blue]&{C'}\ar[d,"{c}"]\\
{B}\ar[r,"{p}"swap]&{C}
\end{tikzcd}
\]
and ${p'}$ is also a deflation.
\item[{Ex3}]If a morphism $g:B\rightarrow C$ admits a homotopy kernel and for some $h:B'\rightarrow B$ the composition $gh:B'\rightarrow C$ is a deflation, then $g$ is a deflation.
\end{itemize}
We call $(\A,\mathcal {S})$ or simply $\A$ a {\em left exact dg category}.
\end{definition}
Dually, a right exact structure on $\A$ is given by a class of homotopy short exact sequences which satisfies the axioms ${\Ex0}$, ${\Ex1}^{op}$, ${\Ex2}^{op}$ and ${\Ex3}^{op}$.
\begin{remark}By the axioms $\Ex0$ and $\Ex2$, the morphism $\Id_A$ is a deflation for each $A\in\A$. 
By the axioms $\Ex0$ and $\Ex3$, the morphism $A\rightarrow 0$ is a deflation for each $A\in\A$.
By Axiom $\Ex3$, in the context of Axiom $\Ex2$, the 3-term h-complex
\[
\begin{tikzcd}
B'\ar[r,"\begin{bmatrix}b\\p'\end{bmatrix}"]\ar[rr,bend right=8ex,"s"swap]&B\oplus C'\ar[r,"{[}-p{,}\;c{]}"]&C
\end{tikzcd}
\]
is a conflation.
\end{remark}
\begin{proposition}\label{leftplusright}
Let $\A$ be an additive connective dg category equipped with a left exact structure $\mathcal{S}_1$ and a right exact structure $\mathcal{S}_2$. 
Then the class of homotopy short exact sequences  
\[
\begin{tikzcd}
A\ar[r, tail,"f"]\ar[rr,"h"swap,bend right=8ex]&B\ar[r,two heads, "j"]&C
\end{tikzcd} 
\]
where $f$ is an inflation in $\mathcal S_2$ and $j$ is a deflation in $\mathcal S_1$, makes $\A$ into an exact dg category. 
\end{proposition}
\begin{proof}
Consider a homotopy short exact sequence $X$
\[
\begin{tikzcd}
A\ar[r,"f"]\ar[rr,"h"swap,bend right=8ex]&B\ar[r,"j"]&C
\end{tikzcd} 
\]
where $j$ is a deflation in $\mathcal S_1$ and $f$ is an inflation in $\mathcal S_2$.
We call such homotopy short exact sequences conflations and denote by $\mathcal S$ the class of conflations.
We will show that the class $\mathcal S$ defines an exact dg structure on $\A$.
It is clear that $\mathcal S$ is closed under isomorphisms in $\mathcal H_{3t}(\A)$.
Let us point out that the axioms $\Ex2$ and $\Ex2^{op}$ can be verified directly, using Axiom $\Ex3$ for $\mathcal S_1$ and Axiom $\Ex3^{op}$ for $\mathcal S_2$.
So it remains to show Axiom $\Ex1$ for $\mathcal S$.

Let $X'$ be another conflation
\[
\begin{tikzcd}
E\ar[r,"f'"]\ar[rr,"h'"swap,bend right=8ex]&C\ar[r,"j'"]&F\mathrlap{.}
\end{tikzcd} 
\]
We want to show that the composition $j'j$ remains a deflation in $\mathcal S$.
Note that Lemma \ref{fact} remains true for the dg category $\A$ with the class $\mathcal S$ of homotopy short exact sequences.
 The composition $j'j$ is a deflation in $\mathcal S_1$, so it admits a homotopy kernel $\tilde{X'}$
 \[
 \begin{tikzcd}
 {D}\ar[r,"{f''}"]\ar[rr,"h''"swap, bend right=8ex]&B\ar[r,"j'j"]&F
 \end{tikzcd}
 \]
 which is homotopy short exact. 
 By Lemma \ref{strictmorphismlift}, the obvious morphism from $B\xrightarrow{j'j} F$ to $C\xrightarrow{j'} F$ in $H^0(\Mor(\A))$ induces a morphism $\mu:\tilde{X'}\rightarrow X'$ in $\mathcal H_{3t}(\A)$
 \[
 \begin{tikzcd}
 D\ar[rr,bend left=8ex,"h''"]\ar[rrd,"t"{blue},blue]\ar[d,"l"swap]\ar[rd,"s_2"{red,swap},red]\ar[r,"f''"]&B\ar[rd,"0"red,red]\ar[d,"{j}"swap]\ar[r,"j'j"]&F\ar[d,equal]\\
 E\ar[rr,bend right=8ex,"h'"swap]\ar[r,"f'"swap]&C\ar[r,"j'"swap]&F
 \end{tikzcd}
 \]
which restricts to a homotopy pullback $Y$ of the cospan $L$
 \[
 \begin{tikzcd}
 &E\ar[d,"f'"]\\
 B\ar[r,"j"swap]&C\mathrlap{.}
 \end{tikzcd}
 \] 
 Apply Axiom ${\Ex2}$ to the cospan $L$ and we get a morphism $\theta:\tilde{X}\rightarrow X$ of conflations in $\mathcal S_1$ (cf.~Proposition \ref{cons})
\[
\begin{tikzcd}
A\ar[rd,"s_1"{red,swap},red]\ar[rrd,"t"{blue},bend left=2ex,blue]\ar[rr,bend left=8ex,"\tilde{h}"]\ar[r,"k"]\ar[d,equal]&D\ar[r,"l"]\ar[d,"f''"swap]\ar[rd,"s_2"red,red]&E\ar[d,"f'"]\\
A\ar[r, "f"swap]\ar[rr,"h"swap,bend right=8ex]&B\ar[r,"j"swap]&C
\end{tikzcd}
\]
 Note that the object $Y$ is a homotopy bicartesian square, cf.~the proof of Proposition~\ref{property}.
 
We have the following object in $\rep(\overline{\Sq},\A)$ 
\[
\begin{tikzcd}
A\ar[rd,"u"]\ar[r,"k"]\ar[d,"f"swap]&D\ar[d,"\begin{bmatrix}f''\\l\end{bmatrix}"]\\
B\ar[r,"\begin{bmatrix}1\\0\end{bmatrix}"swap]&B\oplus E
\end{tikzcd}
\]
where $u=[s_1,\tilde{h}]$. 
It is a graded-split extension of 
\[
\begin{tikzcd}
A\ar[rd,"\tilde{h}"swap]\ar[r,"k"]\ar[d,"0"swap]&D\ar[d,"l"]\\
0\ar[r,"0"swap]&E
\end{tikzcd}
\]
by 
\[
\begin{tikzcd}
0\ar[rd,""]\ar[r,""]\ar[d,""swap]&0\ar[d,""]\\
B\ar[r,equal]&B
\end{tikzcd}
\]
and is thus homotopy bicartesian by Lemma \ref{leftexactsequence}.
So $\begin{bmatrix}f''\\l\end{bmatrix}:D\rightarrow B\oplus E$ is an inflation in $\mathcal S_2$.

Similarly, the following object
\[
\begin{tikzcd}
D\ar[rd,"v"]\ar[d,"\begin{bmatrix}f''\\l\end{bmatrix}"swap]\ar[r,"f''"]     &B\ar[d,"\begin{bmatrix}1\\j\end{bmatrix}"]      \\
B\oplus E\ar[r,swap,"\begin{bmatrix}1\ 0\\0\ f'\end{bmatrix}"]           &                            B\oplus C
\end{tikzcd}
\]
where $v=[0,-s_2]^{\intercal}$, is a homotopy bicartesian square.
Therefore $\begin{bmatrix}1\\j\end{bmatrix}:B\rightarrow B\oplus C$ is an inflation in $\mathcal S_2$. 
Notice that $\begin{bmatrix}1&0\\0&f'\end{bmatrix}:B\oplus E\rightarrow B\oplus C$ is an inflation in $\mathcal S_2$ as a homotopy pushout of the inflation $f':E\rightarrow C$. 
So $f''$ is also an inflation in $\mathcal S_2$ by ${\Ex3}^{op}$.
Therefore $j'j$ is a deflation in $\mathcal S$ and this finishes the proof.
\end{proof}

Let $\R'$ be a class of objects $p:Y\rightarrow Z$ in $H^0(\Mor(\A))$ which is closed under isomorphisms.
The class $\R'$ is {\em stable under pullbacks}, provided that for any object in $\R'$, the cospan along any morphism in $Z^0(\A)$ admits a homotopy pullback and that $\R'$ is stable under forming homotopy pullbacks. 
We define $P\R'\subset \R'$ to be the class of objects $f:C\rightarrow F$ such that for any $h:E\rightarrow F$ in $H^0(\Mor(\A))$, the cospan $L$
\begin{equation}\label{cospan:max}
\begin{tikzcd}
&C\ar[d,"f"]\\
E\ar[r,"h"swap]&F
\end{tikzcd}
\end{equation}
admits a homotopy pullback $X$
\begin{equation}\label{pullback:max}
\begin{tikzcd}
 {B}\ar[rd,"s"]\ar[d,"e"swap]\ar[r," {b}"]& {C}\ar[d,"  f"]\\
 {E}\ar[r,"h"swap]& {F} 
\end{tikzcd}
\end{equation}
such that ${e}:B\rightarrow E$ belongs to $\R'$. 
In particular, objects in $P\R'$ admits homotopy kernels.
Note that $P\R'=\R'$ if and only if $\R'$ is stable under pullbacks.

We write $Q\R'\subset \R'$ for the class of objects $f:C\rightarrow F$ such that if a homotopy pullback $X$ as above exists, then $ {b}\in \R'$ implies $ {h}\in \R'$. 
The classes $P\R'$ and $Q\R'$ are again closed under isomorphisms in $H^0(\Mor(\A))$.
Note that for an object $A\in \A$, if the object $\Id_A\in H^0(\Mor(\A))$ belongs to $\R'$, then it also belongs to the classes $P\R'$ and $Q\R'$.

Let $\R$ be the class of objects $p:Y\rightarrow Z$ in $H^0(\Mor(\A))$ which admits a homotopy kernel which is homotopy bicartesian. Clearly, the class $\R$ is closed under isomorphisms in $H^0(\Mor(\A))$ and contains $\Id_{A}:A\rightarrow A$ for $A\in\A$.
It is clear that to define a left exact structure, it is enough to specify the class $\R'$ of deflations, which is contained in $\R$, satisfying the axioms $\Ex0$--$\Ex3$.
Note that whenever $\R'$ defines a left exact structure, we have $Q\R'=\R'$.

\begin{proposition}\label{operationP}
Keep the assumptions as above. The following statements hold.
\begin{itemize}
\item[1)]We have $PP(\R')=P(\R')$, i.e.~$P\R'$ is stable under pullbacks;
\item[2)]If the class $\R'$ is stable under forming compositions, the same holds for $P\R'$.
\end{itemize}
\end{proposition}
\begin{proof}
1). Let $f:C\rightarrow F$ be an object in $P\R'$ and $h:E\rightarrow F$ any object in $H^0(\Mor(\A))$. 
Then the cospan $L$ (\ref{cospan:max}) admits a homotopy pullback $X$ of the form (\ref{pullback:max}) such that $ {e}\in \mathcal R'$. 

Claim: The object ${e}: B\rightarrow E$ belongs to $P\R'$. 
Indeed, let $ {g}: {D}\rightarrow  {E}$ be any object in $H^0(\Mor(\A))$, then the cospan $L'$
\[
\begin{tikzcd}
& {C}\ar[d,"  f"]\\
 {D}\ar[r,"  h  g"swap]& {F}
\end{tikzcd}
\]
 admits a homotopy pullback $X'$
 \[
 \begin{tikzcd}
 A'\ar[rd,"s'"]\ar[r,"c'"]\ar[d,"d'"swap]&{C}\ar[d,"f"]\\
D\ar[r,"hg"swap]&F\mathrlap{.}
\end{tikzcd}
\]
The obvious morphism $L'\rightarrow L$ between cospans induces a morphism $X'\rightarrow X$ (cf.~Lemma~\ref{strictmorphismlift}). Thus, by restriction we obtain a homotopy pullback square (cf.~Proposition~\ref{push})
\[
\begin{tikzcd}
A'\ar[r,"c'"]\ar[rd,"s"]\ar[d,"b'"swap]&D\ar[d,"g"]\\
B\ar[r,"e"swap]&E
\end{tikzcd}
\]
Now apply Corollary~\ref{Comp} and we see that $b':A'\rightarrow B$ belongs to $\R'$. 
Thus we have ${e}\in P\mathcal R'$ and this proves 1).
 
2). Let $g:D\rightarrow E$ and $h:E\rightarrow F$ be two objects in $P\mathcal R'$ and $f:C\rightarrow F$ any object in $H^0(\Mor(\A))$. 
Then the cospan (\ref{cospan:max}) admits a homotopy pullback (\ref{pullback:max}) where $b$ belongs to $\R'$.
The cospan 
\[
\begin{tikzcd}
&B\ar[d,"e"]\\
D\ar[r,"g"swap]&E
\end{tikzcd}
\]
also admits a homotopy pullback square 
\[
\begin{tikzcd}
A'\ar[r,"c'"]\ar[rd,"s"]\ar[d,"b'"swap]&D\ar[d,"g"]\\
B\ar[r,"e"swap]&E
\end{tikzcd}
\]
where $b':A'\rightarrow B$ belongs to $\R'$ and hence the composition $bb':A'\rightarrow C$ also belongs to $\R'$.
 By Corollary \ref{Comp} we see that $hg$ also belongs to $P\R'$.  

\end{proof}
The class $\R$ is not closed under compositions in general. 
By Lemma \ref{deflationcomposition}, the class $P\R$ is closed under compositions. 
Indeed, in the context of Lemma \ref{deflationcomposition}, suppose that $j:B\rightarrow C$ and $j':C\rightarrow D$ are two objects in $P\R$.
Then the object $v:E\rightarrow A'$ belongs to $\R$ and hence the sequence on the leftmost column is homotopy short exact.
Then the pullback of the cospan $L$ is homotopy bicartesian and hence the sequence in the middle row is homotopy short exact.
 Thus the composition $j'j:B\rightarrow D$ belongs to $\R$.
 
 Notice that the homotopy pullback of $j:B\rightarrow C$ along any morphism in $Z^0(\A)$ is in $P\R$.
 Thus, the homotopy pullback of $j'j$ along any morphism in $Z^0(\A)$ is a composition of objects in $P(\R)$ and hence is in $\R$ by the above discussion.
 Therefore $j'j$ is in $P\R$. 
 So we have proved
 \begin{lemma}The class $P\R$ is closed under compositions. 
 \end{lemma}
\begin{proposition}\label{obscureaxiom}
Suppose that the class $\R'$ contains $\Id_{A}:A\rightarrow A$ for each $A\in \A$, and is stable under compositions and pullbacks.
 If the morphism $A\rightarrow 0$ belongs to $\R'$ for each $A\in \A$, then $Q\R'$ is closed under compositions and satisfies Axiom $\Ex3$: if a composition $A\xrightarrow{a}B\xrightarrow{b}C $ belongs to $Q\R'$ and $b$ admits a homotopy kernel, then $b\in Q\R'$.
\end{proposition}
\begin{remark}The class $P\R$ satisfies the conditions in Proposition \ref{obscureaxiom}.
\end{remark}
\begin{proof}[Proof of Proposition \ref{obscureaxiom}]
Since the cospan of an object in $\R'$ along any morphism in $Z^0(\A)$ admits a homotopy pullback, it is clear that $Q\R'$ is closed under compositions by Corollary~\ref{Comp}.

We show that $Q\R'$ satisfies Axiom ${\Ex3}$.
Let $g:D\rightarrow E$ and $h:E\rightarrow F$ be two morphisms in $Z^0(\A)$ such that the composition $c=hg$ belongs to $Q\R'$ and that $h$ admits a homotopy kernel
\[
\begin{tikzcd}
K\ar[r,"k"]\ar[rr,bend right=8ex,"u"swap]&E\ar[r,"h"]&F\mathrlap{.}
\end{tikzcd}
\]
Then the object in $\rep(\overline{\Sq},\A)$
\[
\begin{tikzcd}
K\oplus D\ar[r,"{[}0{,}\;1{]}"]\ar[d,"{[}k{,}\;g{]}"swap]\ar[rd,"{[}u{,}\;0{]}"{swap}]&D\ar[d,"c"]\\
E\ar[r,"h"swap] &F
\end{tikzcd}
\]
is homotopy cartesian.
Since the object $K\oplus D\xrightarrow{[0,\;1]}D$ belongs to $\R'$, by definition we have that $h$ belongs to $\R'$.

Let $f:C\rightarrow F$ be arbitrary morphism in $Z^0(\A)$. 
Consider the homotopy pullback (\ref{pullback:max}) of $h$ along $f$ and the homotopy pullback of $hg$ along $f$.
Assume $e\in \R'$.
By Corollary~\ref{pastinglaw:second}, we infer that $f\in \R'$ since $\R'$ is stable under pullbacks and $hg \in Q\R'$.
Therefore $h\in Q\R'$ and this finishes the proof.
\end{proof}
\begin{proposition}\label{generatingleftexactstructure}
Suppose that the class $\R'$ is stable under compositions and pullbacks. Assume that the morphism $A\rightarrow 0$ belongs to $\R'$ for each $A\in \A$.
Then the class $PQ\R'$ is stable under compositions and pullbacks and satisfies Axiom $\Ex3$, i.e.~the class $PQ\R'$ defines a left exact dg structure on $\A$.
\end{proposition}
\begin{proof}
By Proposition~\ref{obscureaxiom}, the class $Q\R'$ is stable under compositions.
So by Proposition~\ref{operationP}, the class $PQ\R'$ is stable under pullbacks and compositions.
We show that $PQ\R'$ satisfies Axiom $\Ex3$.

Let $a:A\rightarrow B$ and $b:B\rightarrow C$ be morphisms in $Z^0(\A)$ such that $b$ admits a homotopy kernel and that the composition $c=ba$ belongs to $PQ\R'$. 
By Proposition \ref{obscureaxiom}, the morphism $b$ belongs to the class $Q\R'$.
In particular, the cospan of $b$ along any morphism in $Z^0(\A)$ admits a homotopy pullback.
Now the claim follows from Corollary \ref{pastinglaw:second} and the fact that $Q\R'$ satisfies Axiom $\Ex3$ which is proved in Proposition \ref{obscureaxiom}. 
\end{proof}
The class $P\R$ satisfies the conditions of Proposition \ref{generatingleftexactstructure}.  So we have
\begin{corollary}\label{maximalleftexact}
The class $PQP\R$ defines a left exact dg structure on $\A$, which is the unique maximal left exact dg structure on $\A$.
\end{corollary}
Combining Corollary \ref{maximalleftexact} and Proposition \ref{leftplusright}, we have
\begin{theorem}\label{thm:maximalexactdgstructure}
For each additive dg category $\A$, the greatest exact dg structure exists.
\end{theorem}
The above theorem generalizes Rump's result on the maximal exact structure on an additive category, cf.~\cite{Rump11} and also~\cite{SiegWegner11,Crivei12}.

Let us turn to the special case when $H^0(\A)$ is {\em divisive} (=weakly idempotent complete), i.e.~each split epimorphism has a kernel.
The following lemma shows that an additive dg category with homotopy kernels is divisive.
\begin{lemma}
Let $\A$ be a dg category.
Let $\overline{p}:Y\rightarrow Z$ be a split epimorphism in $H^0(\A)$. If the object $p:Y\rightarrow Z$ in $\Mor(\A)$ admits a homotopy kernel, then $\overline{p}$ admits a kernel in $H^0(\A)$.
\end{lemma}
\begin{proof}
Suppose we have a homotopy left exact sequence
\[
\begin{tikzcd}
X\ar[r,"f"]\ar[rr,bend right,"h"swap]&Y\ar[r,"p"]&Z
\end{tikzcd}
\]
where $d(h)=-pf$. 
For each object $A\in\A$, it induces a long exact sequence
\[
\begin{tikzcd}[column sep=tiny]
\ldots\ar[r]&H^{-1}\A(A,X)\ar[r]&H^{-1}\A(A,Y)\ar[r,two heads]&H^{-1}\A(A,Z)\ar[r]&H^0\A(A,X)\ar[r]&H^0\A(A,Y)\ar[r,two heads]&H^0\A(A,Z).
\end{tikzcd}
\]
Since $\overline{p}:Y\rightarrow Z$ is an split epimorphism in $H^0(\A)$, we have that the induced map $H^{-1}\A(A,Y)\rightarrow H^{-1}\A(A,Z)$ is also surjective for each $A\in\A$.
Therefore by the above long exact sequence, we have that $\overline{f}$ is a kernel of $\overline{p}$.
 \end{proof}
\begin{proposition}\label{prop:leftexactembedding}
Let $(\A,\mathcal S)$ be a left exact dg category such that $H^0(\A)$ is divisive. 
Then there is a universal exact morphism in $\Hqe$
\[
F:\A\rightarrow\D^b_{dg}(\A)
\]
 to a pretriangulated dg category. Put $\D^b(\A)=H^0(\D^b_{dg}(\A))$. If $\A$ is connective, then this morphism satisfies the following properties
 \begin{itemize}
 \item[a)] It induces a quasi-fully faithful morphism from $\tau_{\leq 0}\A\rightarrow \tau_{\leq 0}\D^b_{dg}(\A)$. 
 Let $\A'\subseteq \D^b_{dg}(\A)$ be the essential image.
 \item[b)] The conflations of $\A$ are inherited from $\D^b_{dg}(\A)$.
 \item[c)] The inclusion $H^0(\A)\rightarrow \D^b(\A)$ is stable under kernels of retracts. 
 \item[d)] The inclusion $\A'\hookrightarrow \D^b_{dg}(\A)$ is stable under special pullbacks, i.e.~if we are given a diagram of the following form
 \[
 \begin{tikzcd}
 X\ar[r,dashed]\ar[d,equal]&Y'\ar[r,dashed]\ar[d,dashed]&W\ar[d]\\
 X\ar[r]&Y\ar[r]&Z
 \end{tikzcd}
 \]
 where the objects $X$, $Y$, $Z$ and $W$ are in $\A'$, and the second row is a conflation in $\D^b_{dg}(\A)$, and the right square is a homotopy pullback square in $\D^b_{dg}(\A)$, then $Y'$ is also an object in $\A'$. 
 \end{itemize}
\end{proposition}
\begin{proof}[Sketch of proof]
By construction $\D^b_{dg}(\A)$ is the dg quotient $\pretr(\A)/\N_{dg}$ where $\N_{dg}$ is the full dg subcategory of $\pretr(\A)$ generated by the totalizations of conflations in $\A$.
The proof of a) is similar to the proof of Theorem \ref{main}.
The assertion c) follows from the fact that $H^0(\A)$ is divisive.

We show the assertion b).  
Consider a triangle in $\D^b(\A)$
\[
\begin{tikzcd}
A\ar[r,"\overline{f}"]&B\ar[r,"\overline{\jmath}"]&C\ar[r,"u/t"]&\Sigma A
\end{tikzcd}
\]
where $A$, $B$ and $C$ are objects in $\A$ and we keep the same notation for the corresponding objects in $\D^b(\A)$.
We proceed as in the proof of Theorem \ref{main}.
We assume the roof $u/t$ is of the form
\[
\begin{tikzcd}
&W\ar[ld,Rightarrow,"t"swap]\ar[rd,"u"]&\\
C&&\Sigma A
\end{tikzcd}
\]
where $N'=\Cone(t)$ is the totalization of a conflation $X'$ of the form
\[
\begin{tikzcd}
A'\ar[r,"f'"]\ar[rr,"h"swap, bend right=8ex]&B'\ar[r,"j'"]&C'.
\end{tikzcd}
\]
Then the morphisms $C\rightarrow N'$ and $\Sigma^{-1}N'\rightarrow \Sigma A$ give rise to morphisms $\overline{c}:C\rightarrow C'$ and $\overline{a}:A'\rightarrow A$ in $H^0(\A)$.
Claim: the morphism $j:B\rightarrow C$ admits a homotopy kernel which is homotopy short exact in $\A$.
Indeed, there is an isomorphism from the above triangle to the following triangle
\[
\begin{tikzcd}
\Sigma^{-1}\Cone(j)\ar[r]&B\ar[r,"\overline{\jmath}"]&C\ar[r,""]&\Cone(j)
\end{tikzcd}
\]
which restricts to identities of $B$ and $C$. 
Let $s:A\rightarrow \Sigma^{-1}\Cone(j)$ be the restricted morphism in $\D^b(\A)$.
By the quasi-full faithfulness, this morphism lifts to a morphism in $\pretr(\A)$ and from this morphism we get a homotopy kernel of $j$ which is homotopy short exact. Thus it is enough to show that $j$ is a deflation.

Apply Axiom $\Ex2$ and we get the following diagram
\[
\begin{tikzcd}
A'\ar[r,"i"]\ar[d,equal]&E\ar[r,"p"]\ar[d]&C\ar[d,"c"]\\
A'\ar[r,"f'"swap]&B'\ar[r,"j'"swap]&C'
\end{tikzcd}
\]
where for simplicity we omit the homotopies and where $p$ is a deflation. 
Then we have the following diagram in $\D^b(\A)$
\[
\begin{tikzcd}
A'\ar[r,"\overline{i}"]\ar[d,"\overline{a}"swap]&E\ar[r,"\overline{p}"]\ar[d]&C\ar[d,equal]\\
A\ar[r,"\overline{f}"swap]&B\ar[r,"\overline{\jmath}"swap]&C.
\end{tikzcd}
\]
By the above claim, the morphism $j:B\rightarrow C$ admits a homotopy kernel in $\A$.
By the quasi-full faithfulness of $F$ and Axiom $\Ex3$, we infer that $j$ is a deflation.

It is direct to verify the assertion d), using Axiom $\Ex2$ and the assertion b).
\end{proof}
\begin{remark}
The converse of Proposition \ref{prop:leftexactembedding} also holds. Let $\D$ be a pretriangulated dg category. If $\A\subseteq \D$ is a full dg subcategory which is closed under special pullback in the sense of Proposition \ref{prop:leftexactembedding} d) and which is stable under kernels of retracts, then $\A$ with the conflations inherited from $\D$ is a divisive left exact dg category.
\end{remark}
\begin{proposition}\label{divisivekernel}
Let $\A$ be an additive dg category such that $H^0(\A)$ is divisive. 
Let $\R'$ be a class which is stable under pullbacks. 
Let $a:A\rightarrow B$ and $b:B\rightarrow C$ be morphisms in $Z^0(\A)$. 
If the composition $c=ba$ belongs to $\R'$, then $b$ admits a homotopy kernel. 
\end{proposition}
\begin{proof}
The cospan of $c$ along $b$ admits a homotopy pullback $X$
\[
\begin{tikzcd}
D\ar[r,"e"]\ar[d,"d"swap]\ar[rd,"s"]&B\ar[d,"b"]\\
A\ar[r,"c"swap]&C\mathrlap{.}
\end{tikzcd}
\]
We view $X$ as a morphism from $e$ to $c$ in $H^0(\Mor(\A))$.
By Lemma \ref{univ}, the obvious morphism from $a$ to $c$ factors through the morphism $X$ and hence the morphism $\overline{d}:D\rightarrow A$ is a split epimorphism in $H^0(\A)$.
Since $H^0(\A)$ is divisive, it admits a kernel $K$ and hence the object $d\in H^0(\Mor(\A))$ is isomorphic to the canonical projection $K\oplus A\rightarrow A$ and hence admits a homotopy kernel.
By Corollary \ref{cok}, the object $b$ also admits a homotopy kernel.
\end{proof}
\begin{theorem}\label{divisivedgcategory}
Let $\A$ be an additive dg category such that $H^0(\A)$ is divisive.
Then $P\R$ defines the greatest left exact structure on $\A$.
\end{theorem}
\begin{proof}
It is clear that the class $P\R$ is stable under compositions and pullbacks.
It remains to check Axiom $\Ex3$ for $P\R$.

Let $a:A\rightarrow B$ and $b:B\rightarrow C$ be morphisms in $Z^0(\A)$ such that the composition $c=ba$ belongs to $P\R$. We claim that $b\in \R$.
Note that $b$ admits a homotopy kernel $X$ by Proposition \ref{divisivekernel}
\[
\begin{tikzcd}
K\ar[r,"k"]\ar[rr,"h"swap, bend right=8ex]&B\ar[r,"b"]&C
\end{tikzcd}
\]

We have a morphism of 3-term h-complexes
\[
\begin{tikzcd}
K\ar[rr,"0",bend left=8ex]\ar[d,equal]\ar[r,"v"]\ar[rd,"0"{red,swap},red]\ar[rrd,"0"{blue,swap,near end},blue,bend left=2ex]&K\oplus A\ar[r,"{[}0{,}\;1{]}"]\ar[d,"u"swap]\ar[rd,"{[}-h{,}\;0{]}"{red,near start},red]&A\ar[d,"c"]\\
K\ar[r,"k"swap]\ar[rr,"h"swap, bend right=8ex]&B\ar[r,"b"swap]&C
\end{tikzcd}
\]
where $u=[k,\;a]$, $v=\begin{bmatrix}1\\0\end{bmatrix}$. 
It is not hard to see that the right-hand square is homotopy cartesian.
So the morphism $u$ belongs to $\R$. 
By Corollary \ref{cok} and Proposition \ref{push}, the right-hand square is homotopy bicartesian.
By Proposition \ref{push}, the 3-term h-complex $X$ is homotopy short exact. 
Therefore $b$ belongs to $\R$.

It remains to show that $b$ belongs to $P\R$.
Let $f:F\rightarrow C$ be any object in $H^0(\Mor(\A))$.
Consider the following diagram
\[
\begin{tikzcd}
A'\ar[r,dashed,"a'"]\ar[d,dashed]&B'\ar[r,dashed,"b'"]\ar[d,dashed]&F\ar[d,"f"]\\
A\ar[r,"a"swap]&B\ar[r,"b"swap]&C\mathrlap{.}
\end{tikzcd}
\]
By Corollary \ref{pastinglaw:second},  the homotopy pullback $c'$ of $c$ along $f$ is homotopic to the composition of some morphism $a'$ with the homotopy pullback $b'$ of $b$ along $f$. 
By the above, we have that $b'$ belongs to $\R$. 
Thus $b$ belongs to $P\R$.
\end{proof}
A morphism $b:B\rightarrow C$ is a {\em semi-stable} homotopy cokernel if it belongs to $P\R$. 
Dually, one defines the notion of {\em semi-stable} homotopy kernel. 
A homotopy short exact sequence 
\[
\begin{tikzcd}
A\ar[r,"f"]\ar[rr,"h"swap,bend right=8ex]&B\ar[r,"j"]&C
\end{tikzcd}
\]
is {\em stable} if $f$ is a semi-stable homotopy kernel and if $g$ is a semi-stable homotopy cokernel.
If the additive dg category $\A$ is such that $H^0(\A)$ is divisive, then by Theorem \ref{divisivedgcategory}, the homotopy short exact sequences in the greatest exact dg structure on $\A$ are exactly the stable ones. 
\begin{remark}
In~\cite{SiegWegner11}, the authors showed that in pre-abelian categories, i.e.~categories with kernels and cokernels, the greatest exact structure consists of the class of all stable kernel-cokernel pairs.
In~\cite{Crivei12}, it was generalized to a more general class of additive categories, namely 
the weakly idempotent complete ones.
In~\cite{Rump15}, Rump gave a counterexample where not every stable short exact sequence in an additive category is in the greatest exact structure.
\end{remark}
\begin{example}
Let $\C$ be an abelian category having enough projectives and enough injectives. 
A full subcategory of $\C$ is {\em resolving} if it is stable under extensions, stable under kernels of epimorphisms, stable under taking direct summands and contains all the projectives, cf.~\cite[Section 3]{AuslanderReiten91}. Dually one defines the notion of {\em coresolving} subcategory.
Let $\E$ be a both resolving and coresolving subcategory of $\C$. 
Then it is direct to check that a kernel in $\E$ is a monomorphism in $\C$. 
Therefore a kernel-cokernel pair in $\E$ is a short exact sequence in $\C$. 
Thus the inherited exact structure on $\E$ from $\C$ is the maximal exact structure on $\E$.
For example, the category $\rep_{\rm{l.f.}}(H)$ of locally free modules over the generalized path algebra $H$ is a both resolving and coresolving subcategory of the ambient abelian category $\rep(H)$, cf.~\cite[Theorem 3.9]{GeissLeclercSchroer17}, and thus the exact structure of $\rep_{\rm{l.f.}}(H)$ inherited from $\rep(H)$ is the maximal exact structure.

Let $\C$ be an arbitrary abelian category. 
If $\E$ is an extension-closed subcategory of $\C$ which is stable under the passage to subobjects, then a kernel-cokernel pair in $\E$ is necessarily a short exact sequence in $\C$. Thus the inherited exact structure on $\E$ from $\C$ is the maximal exact structure. For example, let $Q$ be an acyclic quiver and $\B$ an abelian category. Let $\rep(Q,\B)$ be the category of representations of $Q$ in $\B$, i.e.~the category of functors $Q\rightarrow \B$, where $Q$ is considered as a category in the natural way. The monomorphism category, denoted by $\mathrm{mono}(Q,\B)$, is a full subcategory of $\rep(Q,\B)$ which consists of all representations $(M_i,M_{\alpha})_{i\in Q_0, \alpha\in Q_1}$ for which 
\[
\begin{tikzcd}
M_{i,in}:\bigoplus_{\alpha\in Q_1,t(\alpha)=i} M_{s(\alpha)}\ar[r,"(M_{\alpha})_{\alpha}"]& M_i
\end{tikzcd}
\]
is a monomorphism for all vertices $i$, where $s(\alpha)$ and $t(\alpha)$ denote the source and target of the arrow $\alpha$.
It is clear that $\mathrm{mono}(Q,\B)$ is stable under extensions and the passage to subobjects in $\rep(Q,\B)$, cf.~\cite{GaoKulshammerKvammePsaroudakis23}. Thus the exact structure of $\mathrm{mono}(Q,\B)$ inherited from $\rep(Q,\B)$ is the maximal exact structure.
\end{example}
\begin{example}We recall the setting in Example \ref{example:2-term}.
Let $A$ be a finite-dimensional $k$-algebra. 
Let $\mathcal H^{[-1,0]}(\proj A)$ be the full subcategory of $\T=\mathcal H^b(\proj A)$ consisting of two-term complexes $P^{-1} \to P^0$ of finitely generated projective $A$-modules. 
Let $\A=\C^{b}_{dg}(\proj A)$ be the canonical enhancement of
${\mathcal H}^b(\proj A)$ and $\A'$ its full dg subcategory on the objects $P^{-1} \to P^0$.
For an $A$-module $M$, we denote by $M^*$ its $A$-dual.

Consider the algebra $A$ given by a quiver with relations as follows
\[
\begin{tikzcd}
 &2\ar[rd,""{name=b,swap,near start},""{name=b',swap, near end},"b",""{name=b'',swap}]&\\
1\ar[ru,""{name= a,swap,near end},"a"]&&3\mathrlap{.}\ar[ll,""{name=c, swap,near start},"c"]\arrow[from = a, to = b,bend right=20, dashed, no head]\ar[from=b'',to=c,bend right=20,dashed,no head]
\end{tikzcd}
\]
There exists a homotopy exact squence $X$ in $\A'$ which is not homotopy short exact in $\A$
\[
\begin{tikzcd}
S_2\ar[r]\ar[rr,bend right=8ex,"0"swap]&P_3\ar[r]&P_1
\end{tikzcd}
\]
where $S_2$ stands for the two-term complex $P_1\rightarrow P_2$.
We claim that this homotopy short exact sequence is stable. 
As a consequence, the natural exact dg structure on $\A'$ is not the greatest one.

The key point is that the map $P_3\rightarrow P_1$ is a right $\add(P_2\oplus P_3)$-approximation of $P_1$ and the map $P_2^{*}\rightarrow P_1^*$ is a right $\add (P_2^*\oplus P_3^*)$-approximation of $P_1^*$. 

Let $T=(T^{-1}\rightarrow T^0)$ be an object in $\A'$.
Consider a morphism $S_2\rightarrow P_3$ in $Z^0(\A')$.
Take the $A$-dual and we get
\[
\begin{tikzcd}
&((T^{0})^*\rightarrow (T^{-1})^*)\ar[d]\\
(P_3^*\rightarrow 0)\ar[r]&(P_2^*\rightarrow P_1^*)\mathrlap{.}
\end{tikzcd}
\]
We have the following diagram
\[
\begin{tikzcd}
0\ar[r]&S_3'\ar[d,tail]\ar[r,tail]&P_2^*\ar[r]\ar[d,tail]&P_1^*\ar[d,equal]\\
0\ar[r]&M\ar[r]\ar[d,two heads]&P_2^*\oplus (T^{-1})^*\ar[r]\ar[d,two heads]&P_1^*\\
&N\ar[r,tail]&(T^{-1})^*&
\end{tikzcd}
\]
The image of the map $P_2^*\oplus (T^{-1})^*\rightarrow P_1^*$ is either $P_1^*$ so that the second row is split exact, or equals the image of the map $P_2^*\rightarrow P_1^*$ so that $N=(T^{-1})^*$ and the first column is split exact.

In the first case, $M$ is a projective module and $P_1$ is a direct summand of $T^{-1}$.
Assume that $T^{-1}$ decomposes as $P_1\oplus K$. 
Then the cospan $L$
\[
\begin{tikzcd}
S_2\ar[d]\ar[r]&P_3\\
T&
\end{tikzcd}
\]
admits a homotopy pushout which is homotopy bicartesian. 
So we have the following diagram
\[
\begin{tikzcd}
S_2\ar[d]\ar[r]&P_3\ar[d]\ar[r]&P_1\ar[d,equal]\\
(P_1\oplus K\rightarrow T^0)\ar[r]&(P_2\oplus K\rightarrow T^0\oplus P_3)\ar[r]&P_1
\end{tikzcd}
\]
and it is direct to check the second row is homotopy left exact and hence homotopy exact.

In the second case, we have $M=S_3'\oplus (T^{-1})^*$ and hence the cospan $L$ admits a homotopy pushout which is homotopy bicartesian. 
Similarly as above, one checks that the pushout of $X$ along $S_2\rightarrow T$ is homotopy short exact. 
Therefore, the morphism $S_2\rightarrow P_3$ is a semi-stable homotopy kernel.
Similarly, the morphism $P_3\rightarrow P_1$ is a semi-stable homotopy cokernel and hence the homotopy short exact sequence $X$ is stable.
\end{example}
\newpage
\appendix

\section{Extriangulated categories}
In this appendix, we give two related constructions of closed subbifunctors, one of them generalizing the construction by Ogawa \cite[Proposition 2.1]{Ogawa22} in the case of triangulated categories, and some diagram lemmas in extriangulated categories.
\subsection{Closed subbifunctors associated with extension-closed subcategories}
\begin{lemma}\label{extensionclosed}
Let $\N\subset \C$ be an extension-closed subcategory.
Then we have the following closed subbifunctors of $\mathbb E$
\[
\mathbb E^{\N}: \C^{op}\times \C\rightarrow \Ab,\;\; (C,A)\mapsto\sum_{f:N\rightarrow A, N\in \N}\Im(\mathbb E(C,N)\rightarrow \mathbb E(C,A)) 
\]
and
\[
\mathbb E_{\N}: \C^{op}\times \C\rightarrow \Ab,\;\; (C,A)\mapsto\sum_{f:C\rightarrow N, N\in \N}\Im(\mathbb E(N,A)\rightarrow \mathbb E(C,A))
\]
such that $\mathbb E^{\N}|_{\mathbb \N}=\mathbb E|_{\mathbb \N}=\mathbb E_{\N}|_{\mathbb \N}$.
The closed subbifunctor $\mathbb E^{\N}$ (resp.~$\mathbb E_{\N}$) is called the subbifunctor {\em left generated} (resp.~{\em right generated}) by $\N$. 
\end{lemma}
\begin{remark}
When $\C$ is a Quillen exact category, this gives a way to construct Quillen exact substructures of $\C$.
\end{remark}
\begin{proof}[Proof of Lemma~\ref{extensionclosed}]
We show the case for $\mathbb F=\mathbb E_{\N}$. 
The case for $\mathbb E^{\N}$ follows by duality. 
By \cite[Proposition 3.16]{HerschendLiuNakaoka21}, it is enough to show that $\mathfrak s|_{\mathbb F}$-deflations are closed under compositions.

Since $\N$ is an additive subcategory, it is direct to check that an $\mathfrak s|_{\mathbb F}$-conflation $\delta\in\mathbb F(C,A)$ is in the image of a map
\[
\mathbb E(f,A): \mathbb E(N,A)\rightarrow \mathbb E(C,A)
\]
for some object $N\in \N$ and some map $f:C\rightarrow N$.

Suppose $f:X\rightarrow Y$ and $g:Y\rightarrow Z$ are two $\mathfrak s|_{\mathbb F}$-deflations.
We have the following diagrams made of $\mathbb E$-triangles
\[
\begin{tikzcd}
U\ar[r,"e"]\ar[d,equal]&X\ar[d,"j'"]\ar[r,"f"]&Y\ar[d,"j"]\ar[r,"j^*(\delta)",dashed]&\;\\
U\ar[r] &X'\ar[r,"f'"swap]&N\ar[r,"\delta"swap,dashed]&\;
\end{tikzcd}\;,
\]
\[
\begin{tikzcd}
V\ar[r,"d"]\ar[d,equal]&Y\ar[d,"k'"]\ar[r,"g"]&Z\ar[d,"k"]\ar[r,"k^*(\mu)",dashed]&\;\\
V\ar[r] &Y'\ar[r,"g'"swap]&N'\ar[r,"\mu"swap,dashed]&\;
\end{tikzcd}\;.
\]
By \cite[Proposition 1.20]{LiuNakaoka19}, we may choose the maps $j'$ and $k'$ such that we have the following $\mathbb E$-triangles
\[
\begin{tikzcd}
X\ar[r,"\begin{bmatrix}j'\\f\end{bmatrix}"]&X'\oplus Y\ar[r,"{[}f'{,}-j{]}"]&N\ar[r,"e_{*}(\delta)",dashed]&\;
\end{tikzcd},
\]
\[
\begin{tikzcd}
Y\ar[r,"\begin{bmatrix}k'\\g\end{bmatrix}"]&Y'\oplus Z\ar[r,"{[}g'{,}-k{]}"]&N'\ar[r,"d_{*}(\mu)",dashed]&\;
\end{tikzcd}.
\]
So we have the following diagram
\[
\begin{tikzcd}
Y\ar[r,"\begin{bmatrix}k'\\g\end{bmatrix}"]\ar[d,"j"swap]&Y' \oplus Z\ar[d,dashed,red,"{[}u{,}\;v{]}"]\ar[r,"{[}g'{,}-k{]}"]&N'\ar[r,dashed,"d_{*}(\mu)"]\ar[d,equal]&\;\\
N\ar[r,"x"swap]& W\ar[r,"y"swap]&N'\ar[r,dashed,"j_*d_*(\mu)"swap]&\;
\end{tikzcd}.
\]
Since $\N$ is an extension-closed subcategory, we have $W\in\N$.
By Lemma~\ref{LN}, there exists a morphism $[u,\;v]:Y'\oplus Z\rightarrow W$ such that the following sequence is an $\mathbb E$-triangle
\[
\begin{tikzcd}
Y\ar[r,"\begin{bmatrix}-j\\k'\\g\end{bmatrix}"]&N\oplus Y'\oplus Z\ar[r,"{[}x{,}\;u{,}\; v{]}"]&W\ar[rr,dashed,"\theta=y^{*}j_{*}d_{*}(\mu)"]&&\;.
\end{tikzcd}
\]
We have the following isomorphism of sequences
\[
\begin{tikzcd}
X\ar[r,"\begin{bmatrix}j'\\f\\k'f \end{bmatrix}"]\ar[d,equal]& X'\oplus Y\oplus Y'\ar[d,"s"]\ar[r,"\begin{bmatrix}f'\;\;-j\;\;0\\0\;\;-k'\;\;1\end{bmatrix}"]&N\oplus Y'\ar[d,equal]&\\
X\ar[r,"\begin{bmatrix}j'\\f\\0 \end{bmatrix}"swap]&X'\oplus Y\oplus Y'\ar[r,"\begin{bmatrix}f'\;\;-j\;\;0\\0\;\;0\;\;1\end{bmatrix}"swap]&N\oplus Y'\ar[r,dashed,"e_*(\delta)\oplus 0"]&\;
\end{tikzcd}
\]
where $s=\begin{bmatrix}1&0&0\\0&1&0\\0&-k'&1\end{bmatrix}$.
Thus the first row is also a conflation.

Now we apply the pasting law for homotopy cartesian squares (cf.~\cite{HeXieZhou23}) to the following diagram
\[
\begin{tikzcd}
X\ar[r,"f"]\ar[d,"\begin{bmatrix}j'\\k'f\end{bmatrix}"swap]&Y\ar[d,"\begin{bmatrix}j\\k'\end{bmatrix}"]\ar[r,"g"]&Z\ar[d,"-v"]\\
X'\oplus Y'\ar[r,"\begin{bmatrix}f'\;\;0\\0\;\;1\end{bmatrix}"swap]&N\oplus Y'\ar[r,"{[}-x{,}\; u{]}"swap]&W
\end{tikzcd}
\]
and by Lemma~\ref{lem:extrihomotopycartesiandeflation}, the morphism $[-x,\;u]$ is an $\mathfrak s$-deflation. 
We see that $gf$ is an $\mathfrak s|_{\mathbb F}$-deflation.

\end{proof}

\begin{proposition}\label{prop:ogawasubbifunctors}
Let $(\C,\mathbb E,\mathfrak s)$ be an extriangulated category and $\N$ an extension-closed subcategory.
Then we have the following closed subbifunctors of $\mathbb E$ 
\[
\mathbb E^{L}_{\N}:\C^{op}\times \C\to \mathrm{Ab},\quad (C,A)\mapsto \{\delta\in \mathbb E(C,A)| \text{for any $N\in \N$ and $s:N\rightarrow C$, $s^{*}(\delta)\in \mathbb E^{\N}(N,A)$})\}
\]
and 
\[
\mathbb E^{R}_{\N}:\C^{op}\times \C\to \mathrm{Ab},\quad (C,A)\mapsto \{\delta\in \mathbb E(C,A)| \text{for any $N\in\N$ and $t:A\rightarrow N$, $t_{*}(\delta)\in \mathbb E_{\N}(C,N)$})\}
\]
where $\mathbb E^{\N}$ and $\mathbb E_{\N}$ are the closed subbifunctors defined in Lemma~\ref{extensionclosed}.
\end{proposition}
\begin{proof}
We prove that $\mathbb E^{L}_{\N}$ is a closed subbifunctor. The other case is similar.

Put $\mathbb E'=\mathbb E^{\mathcal N}$ and $\mathbb F=\mathbb E^{L}_{\N}$.
Let 
\[
\begin{tikzcd}
E\ar[r,"f"]&D\ar[r,"j"]&C\ar[r,dashed,"\delta"]&\;
\end{tikzcd},
\]
\[
\begin{tikzcd}
H\ar[r,"r"]&B\ar[r,"p"]&D\ar[r,dashed,"\psi"]&\;
\end{tikzcd}
\]
 be $\mathfrak s|_{\mathbb F}\mbox{-}$triangles.
 Let $N$ be an object in $\N$ and $s:N\rightarrow C$ any morphism in $\C$.
 Then we have the following diagram in $\C$
 \[
 \begin{tikzcd}
 A\ar[r,"i"]&B\ar[r,"jp"]\ar[d,"p"]&C\ar[d,equal]\ar[r,dashed,"\theta"]&\;\\
E\ar[r,"f"swap]&D\ar[r,"j"swap]&C\ar[r,dashed,"\delta"swap]&\;
\end{tikzcd}
 \]
  where the first row is an $\mathfrak s\mbox{-}$triangle.
  By Axiom $(\mathrm{ET3})^{op}$, there exists a morphism $(a,\Id_{C}):\theta\rightarrow \delta$ which is realized by $(a,p,\Id_{C})$ and in particular, we have $a_*(\theta)=\delta$.
  
  We have the following diagram in $\C$ made of $\mathfrak s\mbox{-}$triangles
   \[
 \begin{tikzcd}
 A\ar[r,"u"]\ar[d,equal]&F\ar[r,"q"]&N\ar[r,dashed,"s^{*}(\theta)"]\ar[d,"s"]&\;\\
 A\ar[r,"i"swap]&B\ar[r,"jp"swap]&C\ar[r,dashed,"\theta"swap]&\;
\end{tikzcd}.
 \]
 By Lemma~\ref{LN}, there exists a morphism $z:F\rightarrow B$ such that 
\begin{equation}\label{dia:Etriangle}
 \begin{tikzcd}
 F\ar[r,"\begin{bmatrix}z\\q\end{bmatrix}"]&B\oplus N \ar[r,"{[}jp{,}\;-s{]}"]&C\ar[r,dashed,"u_{*}(\theta)"]&\;
\end{tikzcd}
\end{equation}
is an $\mathbb E\mbox{-}$triangle. 

 By definition, we have $s^*(\delta)\in\mathbb E'(N,E)$ and hence we have some morphism $t:N' \rightarrow E$ and an $\mathbb E$-extension $\mu\in \mathbb E(N,N')$ with $N'\in \mathcal N$ and $t_*(\mu)=s^*(\delta)$. So we have the following diagram
 \[
 \begin{tikzcd}
 N'\ar[r,"a"]\ar[d,"t"]&N''\ar[d,"m"]\ar[r,"b"]&N\ar[d,equal]\ar[r,"\mu",dashed]&\;\\
 E\ar[r]\ar[d,equal]&G\ar[d,"n"]\ar[r]&N\ar[d,"s"]\ar[r,dashed,"s^{*}(\delta)"]&\;\\
 E\ar[r,"f"]&D\ar[r,"j"]&C\ar[r,dashed,"\delta"]&\;
 \end{tikzcd}.
 \]
 By definition, we have $(nm)^{*}(\psi)\in \mathbb E'(N'',H)$ and hence we have some morphism $g:\tilde{N}\rightarrow H$ and an $\mathbb E$-extension $\gamma\in \mathbb E(N'', \tilde{N})$ with $\tilde{N}\in \N$ and $g_*(\gamma)=(nm)^{*}(\psi)$. 
 So we have the following diagram
 \[
 \begin{tikzcd}
 \tilde{N}\ar[r,"v"]\ar[d,"g"]&\tilde{N'}\ar[d,"k"]\ar[r,"w"]&N''\ar[d,equal]\ar[r,"\gamma",dashed]&\;\\
 H\ar[r]\ar[d,equal]&G\ar[d,"l"]\ar[r]&N''\ar[d,"nm"]\ar[r,dashed,"(nm)^{*}(\psi)"]&\;\\
 H\ar[r,"r"]&B\ar[r,"p"]&D\ar[r,dashed,"\psi"]&\;
 \end{tikzcd}.
 \]
 So we have the following diagram
 \[
 \begin{tikzcd}
 \tilde{N'}\ar[r, "w"]\ar[rd,dashed,"c"description]\ar[ddr,bend right=6ex,"lk"swap]&N''\ar[rd,"b"]&\\
 &F\ar[r,"q"]\ar[d,"z"]&N\ar[d,"s"]\\
 &B\ar[r,"jp"swap]&C
 \end{tikzcd}
 \]
By the $\mathbb E$-triangle~\ref{dia:Etriangle}, there exists a (non-unique) morphism $c:\tilde{N'}\rightarrow F$ making the above diagram commute.
We apply axiom $(\mathrm{ET3})^{op}$ to the following diagram
\[
\begin{tikzcd}
\tilde{N}\ar[r,"v"]\ar[d,dashed,"d"swap]&\tilde{N'}\ar[d,"c"]\ar[r,"w"]&N''\ar[d,"b"]\ar[r,"\gamma",dashed]&\;\\
 A\ar[r,"u"swap]&F\ar[r,"q"swap]&N\ar[r,dashed,"s^{*}(\theta)"swap]&\;
\end{tikzcd}
\]
 and obtain a morphism of $\mathbb E$-triangles $(d,b):\gamma\rightarrow s^*(\theta)$.
 This implies that $b^*(s^*(\theta))$ is an $\mathbb E'$-extension.
 Since $b$ is an $\mathbb E'$-deflation, by \cite[Lemma 3.14]{HerschendLiuNakaoka21} we have $s^*(\theta)\in \mathbb E'(N,A)$ and hence $jp:B\rightarrow C$ is an $\mathbb F$-deflation. This finishes the proof.
\end{proof}
\subsection{Some diagram lemmas }
Let $(\mathcal C,\mathbb E,\mathfrak s)$ be an extriangulated category. 
The following lemma is called (ET4-2) in \cite{KongLinWang21}.
\begin{lemma}[\cite{LiuNakaoka19}, Proposition 1.20]\label{LN}
Let  
\[
\begin{tikzcd}
A\ar[r,"x"]&B\ar[r,"y"]&C\ar[r,dashed,"\delta"]&\ 
\end{tikzcd}
\]
be any $\mathbb E\mbox{-}$triangle, let $f:A\rightarrow D$ any morphism, and let
\[
\begin{tikzcd}
D\ar[r,"d"]&E\ar[r,"e"]&C\ar[r,dashed,"f_{*}\delta"]&\ 
\end{tikzcd}
\]
be any $\mathbb E\mbox{-}$triangle realizing $f_{*}\delta$.
Then there is a morphism $g$ which gives a morphism of $\mathbb E\mbox{-}$triangles
\[
\begin{tikzcd}
A\ar[r,"x"]\ar[d,"f"swap]&B\ar[r,"y"]\ar[d,"g"red,dashed]&C\ar[d,equal]\ar[r,dashed,"\delta"]&\ \\
D\ar[r,"d"swap]            &E\ar[r,"e"swap]              &C\ar[r,dashed,"f_{*}\delta"swap]           &\ 
\end{tikzcd}
\]
and moreover, the sequence
\[
\begin{tikzcd}
A\ar[r,"\begin{bmatrix} {-}f \\ x \end{bmatrix}"] & D\oplus B\ar[r,"{[}d{,}\; g{]}"]&E\ar[r,dashed,"e^{*}\delta"]&\ 
\end{tikzcd}
\]
becomes an $\mathbb E\mbox{-}$triangle.
\end{lemma}

\begin{definition}[\cite{He19}, Definition~3.1]\label{def:extrihomotopycartesian}
A commutative square 
\begin{equation}\label{dia:extrihomotopycartesian}
\begin{tikzcd}
A\ar[r,"f"]\ar[d,"a"swap]&B\ar[d,"b"]\\
C\ar[r,"u"swap]&D
\end{tikzcd}
\end{equation}
is a {\em homotopy cartesian square} if there exists an $\mathbb E$-triangle
\[
\begin{tikzcd}
A\ar[r,"\begin{bmatrix}f\\a\end{bmatrix}"]&B\oplus C \ar[r,"{[}-b{,}\; u{]}"]&D\ar[r,dashed,"\delta"]&\;
\end{tikzcd}
\] 
in $\C$. The $\mathbb E$-extension $\delta$ is called a {\em differential} for the square.
\end{definition}
\begin{lemma}\label{lem:extrihomotopycartesiandeflation}
Consider a homotopy cartesian square~\ref{dia:extrihomotopycartesian} with a differential $\delta$.
Then $f$ is an $\mathfrak s$-deflation (resp.~$\mathfrak s$-inflation) if and only if so is $u$.
\end{lemma}
\begin{proof}
We only prove the deflation part of the statement. 
Suppose that $f$ is an $\mathfrak s$-deflation. Let 
\[
\begin{tikzcd}
E\ar[r,"g"]&A\ar[r,"f"]&B\ar[r,dashed,"\theta"]&\;
\end{tikzcd}
\]
be an $\mathbb E$-triangle.
We apply \cite[Proposition 3.17]{NakaokaPalu19} to the following diagram
\[
\begin{tikzcd}
E\ar[r,"g"]&A\ar[r,"f"]\ar[d,"\begin{bmatrix}f\\a\end{bmatrix}"]&B\ar[r,dashed,"\theta"]\ar[d,equal]&\;\\
C\ar[r]&B\oplus C\ar[d,"{[}-b{,}\; u{]}"]\ar[r,"{[}1{,}\;0{]}"]&B\ar[r,"0",dashed]&\;\\
D\ar[r,equal]&D\ar[d,dashed,"\delta"]&&\\
&\;&&
\end{tikzcd}
\]
and hence $u:C\rightarrow D$ is also an $\mathfrak s$-deflation. 

Conversely, suppose that $u$ is an $\mathfrak s$-deflation.
Let 
\[
\begin{tikzcd}
E\ar[r,"h"]&C\ar[r,"u"]&D\ar[r,dashed,"\mu"]&\;
\end{tikzcd}
\]
be an $\mathbb E$-triangle.
We have the following diagram
\[
\begin{tikzcd}
E\ar[r,"g'"]\ar[d,equal]&A'\ar[r,"f'"]\ar[d,"a'"red,dashed]&B\ar[d,"b"l]\ar[r,dashed,"b^{*}(\mu)"]&\ \\
E\ar[r,"h"swap]            &C\ar[r,"u"swap]              &D\ar[r,dashed,"\mu"swap]           &\ 
\end{tikzcd}
\]
By the dual of Lemma~\ref{LN}, there exists a morphism $a':A'\rightarrow C$ making the above diagram commute and such that the right hand side square is homotopy cartesian with a differential given by $g'_{*}(\mu)$. 
It follows that there exists an isomorphism $j:A\rightarrow A'$ such that $f'j=f$ and hence $f$ is also an $\mathfrak s$-deflation.
\end{proof}
The following lemma is called (ET4-5) in \cite{KongLinWang21}.
\begin{lemma}[\cite{KongLinWang21}, Theorem 3.3]
Let  
\[
\begin{tikzcd}
A_1\ar[r,"x_1"]&B_1\ar[r,"y_1"]&C\ar[r,dashed,"\delta"]&\ 
\end{tikzcd}
\]
and 
\[
\begin{tikzcd}
A_2\ar[r,"x_2"]&B_2\ar[r,"y_2"]&C\ar[r,dashed,"\delta'"]&\ 
\end{tikzcd}
\]
be $\mathbb E-$triangles and $g:B_1\rightarrow B_2$ be a morphism such that $y_2g=y_1$. 
Then there is a morphism $f:A_1\rightarrow A_2$ which gives a morphism of $\mathbb E\mbox{-}$triangles
\[
\begin{tikzcd}
A_1\ar[r,"x_1"]\ar[d,"f"swap,dashed]&B_1\ar[r,"y_1"]\ar[d,"g"]&C\ar[d,equal]\ar[r,dashed,"\delta"]&\ \\
A_2\ar[r,"x_2"swap]                         &B_2\ar[r,"y_2"swap]              &C\ar[r,dashed,"\delta'"swap]                  &\ 
\end{tikzcd}
\]
and moreover, the squence
\[
\begin{tikzcd}
A_1\ar[r,"\begin{bmatrix} {-}f \\ x_1 \end{bmatrix}"] & A_2\oplus B_1\ar[r,"{[}x_2{,}g{]}"]&B_2\ar[r,dashed,"y_2^{*}\delta"]&\ 
\end{tikzcd}
\]
becomes an $\mathbb E\mbox{-}$triangle.
\end{lemma}

Using the above lemma, we prove a slightly stronger version of \cite[Proposition 3.17]{NakaokaPalu19}.
\begin{lemma}
Suppose we are given $\mathbb E\mbox{-}$triangles
\[
\begin{tikzcd}
A_1\ar[r,"x_1"]&B_1\ar[r,"y_1"]&C\ar[r,dashed,"\delta"]&,
\end{tikzcd}
\]
\[
\begin{tikzcd}
A_2\ar[r,"x_2"]&B_2\ar[r,"y_2"]&C\ar[r,dashed,"\delta'"]&,
\end{tikzcd}
\]
\[
\begin{tikzcd}
B_1\ar[r,"g"]&B_2\ar[r,"h"]&D\ar[r,dashed,"\delta''"]&\ 
\end{tikzcd}
\]
satisfying $y_2\circ g=y_1$. Then there is an $\mathbb E\mbox{-}$triangle 
\[
\begin{tikzcd}
A_1\ar[r,"f"]&A_2\ar[r,"k"]&D\ar[r,dashed,"\theta"]&\ 
\end{tikzcd}
\]
which makes the following diagram commutative in $\mathcal C$
\[
\begin{tikzcd}
A_1\ar[r,"x_1"]\ar[d,"f"swap,dashed]&B_1\ar[r,"y_1"]\ar[d,"g"]                    &C\ar[d,equal]\\
A_2\ar[r,"x_2"]\ar[d,"k"swap,dashed]                         &B_2\ar[r,"y_2"] \ar[d,"h"]             &C\\
D\ar[r,equal]&D&
\end{tikzcd}
\]
and satisfy the following equalities:
\begin{itemize}
\item[(i)] $f_{*}(\delta)=\delta'$,
\item[(ii)] $(x_1)_*(\theta)=\delta''$,
\item[(iii)] $h^{*}(\theta)+y_2^{*}(\delta)=0$
\end{itemize}
and moreover, the sequence
\[
\begin{tikzcd}
A_1\ar[r,"\begin{bmatrix} {-}f \\ x_1 \end{bmatrix}"] & A_2\oplus B_1\ar[r,"{[}x_2{,}g{]}"]&B_2\ar[r,dashed,"y_2^{*}\delta"]&\ 
\end{tikzcd}
\]
becomes an $\mathbb E\mbox{-}$triangle.
\end{lemma}
\begin{proof}
By the previous lemma, there exists a morphism $f:A_1\rightarrow A_2$ which gives a morphism of $\mathbb E-$triangles
\[
\begin{tikzcd}
A_1\ar[r,"x_1"]\ar[d,"f"swap,dashed]&B_1\ar[r,"y_1"]\ar[d,"g"]&C\ar[d,equal]\ar[r,dashed,"\delta"]&\ \\
A_2\ar[r,"x_2"swap]                         &B_2\ar[r,"y_2"swap]              &C\ar[r,dashed,"\delta'"swap]                  &\ 
\end{tikzcd}
\]
and moreover, the sequence
\[
\begin{tikzcd}
A_1\ar[r,"\begin{bmatrix} {-}f \\ x_1 \end{bmatrix}"] & A_2\oplus B_1\ar[r,"{[}x_2{,}g{]}"]&B_2\ar[r,dashed,"y_2^{*}\delta"]&\ 
\end{tikzcd}
\]
becomes an $\mathbb E\mbox{-}$triangle.

Thus we have the following diagram in $\mathcal C$
\[
\begin{tikzcd}
                                                                                      &B_1\ar[r,equal]\ar[d,"\begin{bmatrix}0\\1\end{bmatrix}"swap]         &B_1\ar[d,"g"]\\
A_1\ar[d,equal]\ar[r,"\begin{bmatrix} {-}f \\ x_1 \end{bmatrix}"] & A_2\oplus B_1\ar[r,"{[}x_2{,}g{]}"]\ar[d,"{[}1{,}0{]}"swap]     &B_2\ar[d,"h"]\\
A_1\ar[r,"-f"]&A_2\ar[r,dashed,"k"]&D
\end{tikzcd}
\]
By the dual of \cite[Proposition 3.17]{NakaokaPalu19}, there is an $\mathbb E-$triangle 
\[
\begin{tikzcd}
A_1\ar[r,"-f"swap]&A_2\ar[r,"k"swap]&D\ar[r,dashed,"-\theta"]&\ 
\end{tikzcd}
\]
which makes the above diagram commutative in $\mathcal C$ and $h^*(-\theta)=y_2^{*}\delta$ and 
\[
\begin{bmatrix}-f\\x_1\end{bmatrix}_{*}(-\theta)+\begin{bmatrix}0\\1\end{bmatrix}_{*}(\delta'')=0.
\]
 Applying the morphism $[0,1]_{*}$ to the last equality, we see that $(x_1)_{*}(\theta)=\delta''$.
\end{proof}
Similarly, we can prove a slightly stronger version of \cite[Proposition 3.15]{NakaokaPalu19}.
\begin{lemma}\label{lem:diagramlemmaextriangulated4}
Suppose we are given $\mathbb E\mbox{-}$triangles
\[
\begin{tikzcd}
A_1\ar[r,"x_1"]&B_1\ar[r,"y_1"]&C\ar[r,dashed,"\delta"]&,
\end{tikzcd}
\]
\[
\begin{tikzcd}
A_2\ar[r,"x_2"]&B_2\ar[r,"y_2"]&C\ar[r,dashed,"\delta'"]&.
\end{tikzcd}
\]
Then there is a commutative diagram in $\mathcal C$: 
\[
\begin{tikzcd}
                                               &A_2\ar[r,equal]\ar[d,"m_2"]          &A_2\ar[d,"x_2"]\\
A_1\ar[r,"m_1"]\ar[d,equal]     &M\ar[r,"e_1"]\ar[d,"e_2"]              &B_2\ar[d,"y_2"]\\
A_1\ar[r,"x_1"swap]                        &B_1\ar[r,"y_1"swap]                             &C
\end{tikzcd}
\]
which satisfies
\[
\mathfrak s(y_2^{*}\delta)=[A_1\xrightarrow{m_1} M\xrightarrow{e_1} B_2],
\]

\[
\mathfrak s(y_1^{*}\delta')=[A_2\xrightarrow{m_2} M\xrightarrow{e_2} B_1],
\]
\[
(m_1)_{*}(\delta)+(m_2)_{*}(\delta')=0
\]
and moreover, the sequence
\[ 
\begin{tikzcd}
M\ar[r,"\begin{bmatrix}{e_1}\\e_2\end{bmatrix}"]&B_1\oplus B_2\ar[r,"{[}y_1{,}-y_2{]}"]&C\ar[r,dashed,"(m_1)_*(\delta)"]&\,
\end{tikzcd}
\]
is an $\mathbb E\mbox{-}$triangle.
\end{lemma}
\begin{proof}
We apply the dual of Lemma \ref{LN} to the diagram
\[
\begin{tikzcd}
A_1\ar[d,equal]\ar[r,"m_1"]&M\ar[r,"e_1"]\ar[d,dashed,"e_2"]&B_2\ar[r,dashed,"y_2^{*}(\delta)"]\ar[d,"y_2"]&\,\\
A_1\ar[r,"x_1"swap]&B_1\ar[r,"y_1"swap]&C\ar[r,dashed,"\delta"swap]&\,
\end{tikzcd}
\]
Then there exists a morphism $e_2:M\rightarrow B_1$ such that the above diagram commutes in $\mathcal C$ and
\[
\begin{tikzcd}
M\ar[r,"\begin{bmatrix}{e_2}\\e_1\end{bmatrix}"]&B_1\oplus B_2\ar[r,"{[}y_1{,}-y_2{]}"]&C\ar[r,dashed,"(m_1)_*(\delta)"]&\,
\end{tikzcd}
\]
is an $\mathbb E\mbox{-}$triangle.

We then apply \cite[Proposition 3.17]{NakaokaPalu19} to the diagram
\[
\begin{tikzcd}
A_2\ar[r,"x_2"]\ar[d,"m_2",dashed]&B_2\ar[r,"-y_2"]\ar[d,"\begin{bmatrix}0\\1\end{bmatrix}"]                    &C\ar[d,equal]\\
M\ar[r,"\begin{bmatrix}e_2\\e_1\end{bmatrix}"]\ar[d,"k",dashed]                         &B_1\oplus B_2\ar[r,"{[}y_1{,}-y_2{]}"] \ar[d,"{[}1{,}0{]}"]             &C\\
B_1\ar[r,equal]&B_1&
\end{tikzcd}.
\]
So we have an $\mathbb E\mbox{-}$triangle
\[
\begin{tikzcd}
A_2\ar[r,"m_2"]&M\ar[r,"k"]&B_1\ar[r,dashed,"\theta"]&\,
\end{tikzcd}
\]
such that the above diagram commutes (so we have $k=e_2$) and 
\begin{itemize}
\item[(i)] $(m_2)_{*}(-\delta')=(m_1)_*(\delta)$,
\item[(ii)] $(x_2)_{*}(\theta)=0$,
\item[(iii)] $[1,0]^{*}(\theta)+[y_1,-y_2]^{*}(-\delta')=0$.
\end{itemize}
Applying the morphism $\begin{bmatrix}1\\0\end{bmatrix}^{*}$ to the last equality, we get $\theta=y_1^*(\delta')$.
\end{proof}

We also have the following slightly stronger version of \cite[Lemma 3.14]{NakaokaPalu19}, which is called (ET4-3) in \cite{KongLinWang21}.
\begin{lemma}\label{lem:ET4strongform}
Let
\[
\begin{tikzcd}
A\ar[r,"f"]&B\ar[r,"f'"]&D\ar[r,dashed,"\delta_f"]&,
\end{tikzcd}
\]
\[
\begin{tikzcd}
A\ar[r,"h"]&C\ar[r,"h_0"]&E_0\ar[r,dashed,"\delta_h"]&,
\end{tikzcd}
\]
\[
\begin{tikzcd}
B\ar[r,"g"]&C\ar[r,"g'"]&F\ar[r,dashed,"\delta_g"]&,
\end{tikzcd}
\]
be $\mathbb E\mbox{-}$triangles satisfying $h=gf$. Then there are morphisms $d_0:D\rightarrow E_0$ and $e_0:E_0\rightarrow F$ such that the following 
\[
\begin{tikzcd}
A\ar[r,"f"]\ar[d,equal]                    &B\ar[d,"g"]\ar[r,"f'"]                                        &D\ar[r,"\delta_{f}"]\ar[d,"d_0"]            &\,\\
A\ar[r,"h"]                   &C\ar[d,"g'"]\ar[r,"h_0"]                                   &E_0\ar[r,"\delta_{h}"] \ar[d,"e_0"]          &\,\\                                                  
\,                               &F\ar[d,dashed,"\delta_{g}"]\ar[r,equal]          &F\ar[d,dashed,"f'_{*}(\delta_g)"]  &\,\\
\,                                &\,                                                               &\,                                                   &\,
\end{tikzcd}
\]
is a commutative diagram whose third column is an $\mathbb E\mbox{-}$triangle, moreover, $d_0^*(\delta_h)=\delta_f$ and $f_*\delta_h=e_0^*(\delta_g)$ and the sequence
\[ 
\begin{tikzcd}
B\ar[r,"\begin{bmatrix}g\\f{'}\end{bmatrix}"]           &   C\oplus D\ar[r,"{[}h_0{,}-d_0{]}"]             &E_0\ar[r,dashed,"f_{*}(\delta_h)"]&\,
\end{tikzcd}
\]
is an $\mathbb E\mbox{-}$triangle.
\end{lemma}
\begin{proof}
By the dual of Lemma \ref{LN}, there exists a morphism $d_0:D\rightarrow E_0$ such that 
\[
\begin{tikzcd}
B\ar[r,"\begin{bmatrix}g\\f'\end{bmatrix}"]&C\oplus D\ar[r,"{[}h_0{,}-d_0{]}"]&E_0\ar[r,dashed,"\theta=f_*(\delta_h)"]&\,
\end{tikzcd}
\]
is an $\mathbb E\mbox{-}$triangle and $(\Id_{A},g,d_0)$ is a morphism of $\mathbb E\mbox{-}$triangles.
 In particular, $d_0^{*}(\delta_h)=\delta_f$.
 
Consider the following commutative diagram
\[
\begin{tikzcd}
&D\ar[r,equal]\ar[d,"\begin{bmatrix}0\\-1\end{bmatrix}"swap]&D\ar[d,dashed,"d_0"]&\\
B\ar[r,"\begin{bmatrix}g\\f'\end{bmatrix}"]\ar[d,equal]&C\oplus D\ar[r,"{[}h_0{,}-d_0{]}"]\ar[d,"{[}1{,}0{]}"swap]&E_0\ar[r,dashed,"\theta"]\ar[d,dashed,"e_0"]&\,\\
B\ar[r,"g"swap]&C\ar[r,"g'"swap] \ar[d,dashed,"0"swap]&F\ar[d,dashed, "\mu"]\ar[r,dashed,"\delta_g"swap]&\,\\
      \,&\,&\,
\end{tikzcd}
\]
By the dual of \cite[Proposition 3.17]{NakaokaPalu19}, there exists an $\mathbb E\mbox{-}$triangle
\[
\begin{tikzcd}
D\ar[r,"\tilde{d_0}"]&E_0\ar[r,"e_0"]&F\ar[r,dashed,"\mu"]&\,
\end{tikzcd}
\]
which makes the above diagram commute (hence $d_0=\tilde{d_0}$) and $\begin{bmatrix}0\\-1\end{bmatrix}_{*}(\mu)+\begin{bmatrix}g\\f'\end{bmatrix}_{*}(\delta_g)=0$, $e_0^{*}(\delta_g)=\theta$.

Thus we have $\mu=f'_{*}(\delta_g)$ and $e_0^{*}(\delta_g)=f_*(\delta_{h})$.
\end{proof}
\begin{remark}
The triangulated case for Lemma~\ref{lem:ET4strongform} is proved in \cite[Appendix B]{Hubery16}.
\end{remark}

\begin{definition}
A {\em contravariant connected sequence of functors} is a pair $(T,\epsilon)$ where $T=(T^i)_{i\geq 0}$ is sequence of right $\C$-modules and $\epsilon$ is a collection of morphisms $\epsilon_{\delta}^i:T^i(A)\rightarrow T^{i+1}(C)$ for each $\mathbb E\mbox{-}$extension $\delta\in\mathbb E(C,A)$ and $i\geq 0$ which is natural with respect to morphisms of $\mathbb E$-extensions.
It is a {\em right $\delta$-functor} if for any $\mathfrak s\mbox{-}$triangle 
\[
\begin{tikzcd}
A\ar[r,"k"]&B\ar[r,"p"] &C\ar[r,dashed,"\delta"]&\;
\end{tikzcd},
\]
 the associated sequence (which is a complex by the naturality of $\epsilon$, cf.~\cite[Proposition 3.20]{GorskyNakaokaPalu21})
\[
\ldots \rightarrow T^{n}(C)\xrightarrow{T^n(p)} T^n(B)\xrightarrow{T^n(k)}T^n(A)\xrightarrow{\epsilon^{n}_{\delta}} T^{n+1}(C)\rightarrow \ldots 
\] 
is exact.
A morphism $(T,\epsilon)\rightarrow (\tilde{T},\tilde{\epsilon})$ of right $\delta$-functors is a family of morphisms of right $\C$-modules $\theta^i:T^i\rightarrow \tilde{T}^i$ which is compatible with the connecting morphisms, i.e.~for each $\delta\in \mathbb E(C,A)$ and each $i\geq 0$, the following diagram is commutative
\[
\begin{tikzcd}
T^{i}(A)\ar[r,"\epsilon^i_{\delta}"]\ar[d,"\theta^i(A)"swap]&T^{i+1}(C)\ar[d,"\theta^{i+1}(C)"]\\
\tilde{T}^i(A)\ar[r,"\tilde{\epsilon}^i_{\delta}"swap]&\tilde{T}^{i+1}(C)
\end{tikzcd}.
\]
Dually one defines the notion of {\em covariant connected sequence of functors} and {\em left $\delta$-functors}.
\end{definition}

We have the following extriangulated analogue of \cite[Proposition 2.1]{Grothendieck57}.
\begin{proposition}\label{prop:effaceableuniversal}
Let $(T,\epsilon)$ and $(R,\eta)$ be right $\delta$-functors such that $R^i$ is weakly effaceable for each $i>0$.
Then each morphism $R^0\rightarrow T^0$ extends uniquely to a morphism of right $\delta$-functors $(R,\eta)\rightarrow (T,\epsilon)$.
\end{proposition}
\begin{proof}
Suppose we have constructed natural morphisms $\theta^j:R^j\iso T^j$ which are compatible with the connecting morphisms for $0\leq j<i$.

Let $C$ be an object in $\C$ and $x\in R^i(C)$ an arbitrary element.
Since $R^i$ is effaceable, there exists a $\mathbb E$-triangle 
\[
\begin{tikzcd}
A\ar[r,"k"]&B\ar[r,"p"] &C\ar[r,dashed,"\delta"]&\;
\end{tikzcd},
\]
 such that $R^i(p)(x)=0$.
We have the following diagram
\[
\begin{tikzcd}
R^{i-1}(B)\ar[r,"R^{i-1}(k)"]\ar[d,"\theta^{i-1}_{B}"swap]&R^{i-1}(A)\ar[r,"\eta_{\delta}^{i-1}"]\ar[d,"\theta^{i-1}_{A}"swap]&R^i(C)\ar[r,"R^i(p)"]\ar[d,dashed]&R^i(B)\\
T^{i-1}(B)\ar[r,"T^{i-1}(k)"swap]&T^{i-1}(A)\ar[r,"\epsilon_{\delta}^{i-1}"swap]&T^i(C)\ar[r]&T^i(B)\\
\end{tikzcd}
\]
By the exactness of the top horizontal sequence, there exists an element $y\in R^{i-1}(A)$ such that $\eta_{\delta}^{i-1}(y)=x$.
Since the left square commutes by the construction of $\theta^{i-1}$, it is clear that the element $\epsilon_{\delta}^{i-1}\circ \theta_{A}^{i-1}(y)$ is independent of the choice of $y$. 
Let 
\[
\begin{tikzcd}
A'\ar[r,"l"]&B'\ar[r,"q"] &C\ar[r,dashed,"\delta'"]&\;
\end{tikzcd}
\] 
be another such $\mathbb E$-triangle. 
By Lemma~\ref{lem:diagramlemmaextriangulated4}, we have the following diagram
\[
\begin{tikzcd}
&A'\ar[r,equal]\ar[d,"f"]&A'\ar[d,"l"]&\;\\
A\ar[r,"m"]\ar[d,equal]&E\ar[r,"r"]\ar[d,"g"]&B'\ar[d,"q"]\ar[r,dashed,"q^{*}(\delta)"]&\;\\
A\ar[r,"k"swap]&B\ar[r,"p"swap] &C\ar[r,dashed,"\delta"swap]&\;
\end{tikzcd}
\]
where $m_*(\delta)=-f_*(\delta')$ and the sequence
\[
\begin{tikzcd}
E\ar[r,"\begin{bmatrix}r\\g\end{bmatrix}"]&B'\oplus B\ar[r,"{[}-q{,}\;p{]}"] &C\ar[r,dashed,"m_{*}(\delta)"]&\;
\end{tikzcd}
\] 
is an $\mathbb E$-triangle. So we have the following diagram made of $\mathbb E$-triangles
\[
\begin{tikzcd}
A\ar[r,"k"]\ar[d,"m"swap]&B\ar[r,"p"] \ar[d,"\begin{bmatrix}1\\0\end{bmatrix}"swap]&C\ar[r,dashed,"\delta"]\ar[d,equal]&\;\\
E\ar[r,"\begin{bmatrix}g\\r\end{bmatrix}"swap]&B\oplus B'\ar[r,"{[}-q{,}\;p{]}"swap] &C\ar[r,dashed,"m_{*}(\delta)"swap]\ar[d,equal]&\;\\
A'\ar[r,"l"swap]\ar[u,"-f"]&B'\ar[r,"q"swap] \ar[u,"\begin{bmatrix}0\\-1\end{bmatrix}"]&C\ar[r,dashed,"\delta'"swap]&\;
\end{tikzcd}
\]
This implies that $\epsilon_{\delta}^{i-1}\circ \theta_{A}^{i-1}(y)$ is independent of the choice of the $\mathbb E$-triangle. 
Similarly as above, one shows that the map $\theta^{i}_{C}: R^{i}(C)\rightarrow T^{i}(C)$ is a morphism of abelian groups.
It is direct to verify that $\theta^i:R^i\rightarrow T^{i}$ is a natural transformation 
and it follows directly from the construction that $\theta^i$ is compatible with the connecting morphisms. 
By induction on $i$, we have natural transformations $\theta^i:R^i\rightarrow T^i$ for $i\geq 0$ which are compatible with the connecting morphisms. The uniqueness is also clear from the construction of the $\theta^i$.
\end{proof}

\begin{definition}[\cite{GorskyNakaokaPalu21}, Definitions 4.5]\label{def:deltafunctor}
A {\em $\delta$-functor} is a triple $(T,\epsilon,\eta)$ where $T=(T^i)_{i\geq 0}$ be a sequence of $\C$-$\C$-bimodules
and $\epsilon$ and $\eta$ are collections of morphisms $\epsilon_{\delta}^i:T^i(A,-)\rightarrow T^{i+1}(C,-)$ and $\eta_{\delta}^i:T^i(?,C)\rightarrow T^{i+1}(?,A)$ for each $\mathbb E$-extension $\delta\in\mathbb E(C,A)$ and $i\geq 0$ which are natural with respect to morphisms of $\mathbb E$-extensions, and such that for each $\mathfrak s$-triangle
\[
\begin{tikzcd}
A\ar[r,"k"]&B\ar[r,"p"] &C\ar[r,dashed,"\delta"]&\;
\end{tikzcd},
\]
 the associated complexes
 \[
\ldots \rightarrow T^{n}(C,-)\xrightarrow{T^n(p,-)} T^n(B,-)\xrightarrow{T^n(k,-)}T^n(A,-)\xrightarrow{\epsilon^{n}_{\delta}} T^{n+1}(C,-)\rightarrow \ldots 
\] 
\[
\ldots \rightarrow T^{n}(?,A)\xrightarrow{T^n(?,k)} T^n(?,B)\xrightarrow{T^n(?,p)}T^n(?,C)\xrightarrow{\eta^{n}_{\delta}} T^{n+1}(?,A)\rightarrow \ldots 
\] 
are exact in $\Mod \C$, resp.~in $\C\Mod$. 
A morphism $(T,\epsilon,\eta)\rightarrow (\tilde{T},\tilde{\epsilon},\tilde{\eta})$ of $\delta$-functors is a family of morphisms of $\C$-$\C$-bimodules $\theta^i:T^i\rightarrow \tilde{T}^i$ which is compatible with the connecting morphisms.
\end{definition}
Let $(T,\epsilon,\eta)$ a $\delta$-functor. 
We assume that
\begin{equation}\label{assumption:delta1}\tag{A.1}
 \text{$T^0=\Hom(?,-)$ and $T^i$ are effaceable for $i>0$.}
 \end{equation}

By Proposition~\ref{prop:effaceableuniversal}, we have that $T^i$ is isomorphic to $\mathbb E^i$ for $i>0$ which is compatible with the connecting morphisms $\epsilon$.
We see that $\epsilon_{\delta}^0:\Hom(A,-)\rightarrow \mathbb E(C,-)$ is given by 
\[
(\epsilon_{\delta}^0)_U:\Hom(A,U)\rightarrow \mathbb E(C,U),\;\; f\mapsto f_{*}\delta
\]
for each object $U\in\C$. 
\begin{equation}\label{assumption:delta2}\tag{A.2}
\text{We assume that a similar description holds for $\eta_{\delta}^0$.}
\end{equation}  
For each $i\geq 2$ and any two $\mathfrak s$-triangles
\[
\begin{tikzcd}
A\ar[r,"k"]&B\ar[r,"p"] &C\ar[r,dashed,"\delta"]&\;
\end{tikzcd},
\]
\[
\begin{tikzcd}
F\ar[r,"l"]&G\ar[r,"q"] &H\ar[r,dashed,"\psi"]&\;
\end{tikzcd}
\]
\begin{equation}\label{assumption:delta3}\tag{A.3}
\text{we assume the following diagram~\ref{dia:bivariantdeltafunctor} is commutative}
\end{equation}
\begin{equation}\label{dia:bivariantdeltafunctor}
\begin{tikzcd}
 T^{i-2}(A,H)\ar[r,"\epsilon_{\delta}^{i-2}"]\ar[d,"\eta_{\psi}^{i-2}"swap]&T^{i-1}(C,H)\ar[d,"\eta_{\psi}^{i-1}"]\\
T^{i-1}(A,F)\ar[r,"\epsilon_{\delta}^{i-1}"swap]&T^i(C,F)
\end{tikzcd}\;.
\end{equation}

Put $\mathbb E^0=\Hom_{\C}(?,-)$. 
From the definition of $\mathbb E^n$ and the construction of the connecting morphisms, cf.~\cite[Claim 3.11]{GorskyNakaokaPalu21}, it is clear that $(\mathbb E^n)_{n\geq 0}$ together with its connecting morphisms satisfy the above assumptions~\ref{assumption:delta1}--\ref{assumption:delta3}.
\begin{corollary}\label{cor:effaceablebimodule}
Keep the notations as above. 
Let $(R,\gamma,\mu)$ be another $\delta$-functor which satisfies the above assumptions A.1--A.3. 
 Then there is a unique isomorphism of $\delta$-functors $\theta^i: R^i\rightarrow T^i$
 for $i\geq 0$ such that $\theta^0=\Id_{\Hom_{\C}(?,-)}$.
\end{corollary}
\begin{proof}
By Proposition~\ref{prop:effaceableuniversal}, we have that for each $i\geq 0$, the right $\C$-module $R^i(?,A)$ is naturally isomorphic to $T^i(?,A)$ in $\Mod \C$ and the isomorphism is natural in the variable $A$ and compatible with the connecting morphisms. 
Hence we have an isomorphism of $\C$-$\C$-bimodules $\theta^i: R^i\rightarrow T^i$ for each $i\geq 0$.
Similarly $R^i(C,-)$ is naturally isomorphic to $T^i(C,-)$ and this also gives an isomorphism of $\C$-$\C$-bimodules $\lambda^i:R^i\rightarrow T^i$ for each $i\geq 0$.

We identify both $R^1$ and $T^1$ with $\mathbb E$ and by assumption the natural isomorphisms $\theta^1$ and $\lambda^1$ are both given by the identity.
We only need to check that the natural isomorphisms $\theta^i$ and $\lambda^i$ coincide with each other for each $i\geq 2$ and we proceed by induction on $i$.

Let $C$ and $F$ be objects in $\C$ and $x\in R^{i}(C,F)$ for $i\geq 1$.
Since $R^i$ and $R^{i-1}$ are both effaceable, there exists $\mathfrak s$-triangles
\[
\begin{tikzcd}
A\ar[r,"k"]&B\ar[r,"p"] &C\ar[r,dashed,"\delta"]&\;
\end{tikzcd},
\]
\[
\begin{tikzcd}
F\ar[r,"l"]&G\ar[r,"q"] &H\ar[r,dashed,"\psi"]&\;
\end{tikzcd}
\]
such that there exists an element $y\in R^{i-2}(A,H)$ with $x=\mu_{\psi}^{i-1}\gamma_{\delta}^{i-2}(y)=\gamma_{\delta}^{i-1}\mu_{\psi}^{i-2}(y)$ where the second equality is given by the commutativity of the diagram~\ref{dia:bivariantdeltafunctor} for the bimodules $R^{i}$, $R^{i-1}$ and $R^{i-2}$.
Since by induction $\theta^{j}$ coincides with $\lambda^{j}$ for $j=i-2$ and $i-1$, we obtain that the isomorphisms $\theta^{i}$ and $\lambda^i$ coincide with each other and this finishes the proof. 
\end{proof}
	
	\newpage
	\def\cprime{$'$} \def\cprime{$'$}
	\providecommand{\bysame}{\leavevmode\hbox to3em{\hrulefill}\thinspace}
	\providecommand{\MR}{\relax\ifhmode\unskip\space\fi MR }
	% \MRhref is called by the amsart/book/proc definition of \MR.
	\providecommand{\MRhref}[2]{%
		\href{http://www.ams.org/mathscinet-getitem?mr=#1}{#2}
	}
	\providecommand{\href}[2]{#2}
	%\begin{thebibliography}{10}
%		
	%\end{thebibliography}
	\bibliographystyle{amsplain}
	\bibliography{stanKeller}
\end{document}